\documentclass[11pt,a4paper]{amsart}
\usepackage[margin=2.5cm,marginpar=2cm]{geometry}
\usepackage{amscd}

\usepackage[bookmarksopen,bookmarksdepth=2,hidelinks,colorlinks=false]{hyperref}
\usepackage[nameinlink]{cleveref}

\usepackage{array}
\usepackage{amsmath}
\usepackage{mathrsfs}
\usepackage{mathbbol,mathabx} 
\usepackage{amsthm}
\usepackage{graphicx}
\usepackage{braket}
\usepackage{mathtools}
 
\usepackage[alphabetic]{amsrefs}

\usepackage[all,cmtip]{xy}
\usepackage{rotating}
\usepackage{leftidx}

\usepackage{pifont}
\makeatletter
\newcommand*{\circnuma}[1]{%
  \ifnum#1<1 %
    \@ctrerr
  \else
    \ifnum#1>20 %
      \@ctrerr
    \else
      \mbox{\ding{\numexpr 171+(#1)\relax}}%
     \fi
  \fi
}
\makeatother

\usepackage[rgb,table,dvipsnames]{xcolor}

\usepackage{imakeidx}

\makeindex

\usepackage[normalem]{ulem}

\usepackage[centertableaux]{ytableau}

\usepackage{enumitem}
\newlist{enumC}{enumerate}{1} 
\setlist[enumC,1]{label=(\alph*),wide,ref=(\alph*)}
\crefname{enumCi}{condition}{conditions}
\Crefname{enumCi}{Condition}{Conditions}
\newlist{enumT}{enumerate}{3} 
\setlist[enumT]{label=(\roman*),wide}
\setlist[enumT,1]{label=(\roman*),wide}
\setlist[enumT,2]{label=(\alph*),ref ={(\roman{enumTi}.\alph*)},left=2em}
\setlist[enumT,3]{label*=.(\arabic*), ref ={(\roman{enumTi}.\alph{enumTii}.\alph*)}}
\crefname{enumTi}{}{}
\Crefname{enumTi}{Item}{Items}
\crefname{enumTii}{}{}
\Crefname{enumTii}{Item}{Items}
\crefname{enumTiii}{}{}
\Crefname{enumTiii}{Item}{Items}
\newlist{enumPF}{enumerate}{3}
\setlist[enumPF,1]{label=(\roman*),wide}
\setlist[enumPF,2]{label=(\alph*),left=2em}
\setlist[enumPF,3]{label=\arabic*).,left=1em}
\newlist{enumS}{enumerate}{3} 
\setlist[enumS]{label=\roman*)}
\setlist[enumS,1]{label=\roman*)}
\setlist[enumS,2]{label=\alph*)}
\setlist[enumS,3]{label=\arabic*.}
\newlist{enumI}{enumerate}{3} 
\setlist[enumI,1]{label=\roman*),leftmargin=*}
\setlist[enumI,2]{label=\alph*), leftmargin=*}
\setlist[enumI,3]{label=\arabic*), leftmargin=*}
\newlist{enumIL}{enumerate*}{1} 
\setlist*[enumIL]{label=\roman*)}
\newlist{enumR}{enumerate}{1} 
\setlist[enumR]{label=\arabic*.,wide,labelwidth=!, labelindent=0pt}
\crefname{enumRi}{remark}{remarks}

\newlist{enuma}{enumerate}{1} 
\setlist[enuma]{label=(\alph*),nosep,leftmargin=*}

\colorlet{srcol}{black!15}

\crefname{equation}{}{}
\Crefname{equation}{Equation}{Equations}
\Crefname{lem}{Lemma}{Lemma}
\Crefname{thm}{Theorem}{Theorem}
\Crefname{conv}{Convention}{Convention}

\newlist{des}{enumerate}{1}
\setlist[des]{font=\upshape\sffamily\bfseries, label={}}

\long\def\delete#1{}

\newcommand{\trivial}[2][]{\if\relax\detokenize{#1}\relax
  {
      \color{orange} \vspace{0em} $[$  #2 $]$
      \color{black}
  }
  \else
\ifx#1h
\ifcsname showtrivial\endcsname
{
    \color{orange} \vspace{0em}  $[$ #2 $]$
    \color{black}
}
\fi
\else {\red Wrong argument!} \fi
\fi
}

\newcommand{\byhide}[2][]{\if\relax\detokenize{#1}\relax
{\color{orange} \vspace{0em} Plan to delete:  #2}
\else
\ifx#1h\relax\fi
\fi
}

\newcommand{\pr}{\mathrm{pr}}

\newcommand{\AC}{\mathrm{AC}}

\def\YD{{\mathsf{YD}}}
\def\SYD{{\mathsf{SYD}}}
\def\MYD{{\mathsf{MYD}}}

\def\Im{\operatorname{Im}}



\makeatletter
\def\inn#1#2{\left\langle
      \def\ta{#1}\def\tb{#2}
      \ifx\ta\@empty{\;} \else {\ta}\fi ,
      \ifx\tb\@empty{\;} \else {\tb}\fi
      \right\rangle}
\def\binn#1#2{\left\lAngle
      \def\ta{#1}\def\tb{#2}
      \ifx\ta\@empty{\;} \else {\ta}\fi ,
      \ifx\tb\@empty{\;} \else {\tb}\fi
      \right\rAngle}
\makeatother

\makeatletter
\def\binn#1#2{\overline{\inn{#1}{#2}}}
\makeatother

\usepackage{xparse}
\def\usecsname#1{\csname #1\endcsname}
\def\useLetter#1{#1}
\def\usedbletter#1{#1#1}

\ExplSyntaxOn

\def\mydefcirc#1#2#3{\expandafter\def\csname
  circ#3{#1}\endcsname{{}^\circ {#2{#1}}}}
\def\mydefvec#1#2#3{\expandafter\def\csname
  vec#3{#1}\endcsname{\vec{#2{#1}}}}
\def\mydefdot#1#2#3{\expandafter\def\csname
  dot#3{#1}\endcsname{\dot{#2{#1}}}}

\def\mydefacute#1#2#3{\expandafter\def\csname a#3{#1}\endcsname{\acute{#2{#1}}}}
\def\mydefbr#1#2#3{\expandafter\def\csname br#3{#1}\endcsname{\breve{#2{#1}}}}
\def\mydefbar#1#2#3{\expandafter\def\csname bar#3{#1}\endcsname{\bar{#2{#1}}}}
\def\mydefhat#1#2#3{\expandafter\def\csname hat#3{#1}\endcsname{\hat{#2{#1}}}}
\def\mydefwh#1#2#3{\expandafter\def\csname wh#3{#1}\endcsname{\widehat{#2{#1}}}}
\def\mydeft#1#2#3{\expandafter\def\csname t#3{#1}\endcsname{\tilde{#2{#1}}}}
\def\mydefu#1#2#3{\expandafter\def\csname u#3{#1}\endcsname{\underline{#2{#1}}}}
\def\mydefr#1#2#3{\expandafter\def\csname r#3{#1}\endcsname{\mathrm{#2{#1}}}}
\def\mydefb#1#2#3{\expandafter\def\csname b#3{#1}\endcsname{\mathbb{#2{#1}}}}
\def\mydefwt#1#2#3{\expandafter\def\csname wt#3{#1}\endcsname{\widetilde{#2{#1}}}}
\def\mydefbf#1#2#3{\expandafter\def\csname bf#3{#1}\endcsname{\mathbf{#2{#1}}}}
\def\mydefc#1#2#3{\expandafter\def\csname c#3{#1}\endcsname{\mathcal{#2{#1}}}}
\def\mydefsf#1#2#3{\expandafter\def\csname sf#3{#1}\endcsname{\mathsf{#2{#1}}}}
\def\mydefs#1#2#3{\expandafter\def\csname s#3{#1}\endcsname{\mathscr{#2{#1}}}}
\def\mydefcks#1#2#3{\expandafter\def\csname cks#3{#1}\endcsname{{\check{
        \csname s#2{#1}\endcsname}}}}
\def\mydefckc#1#2#3{\expandafter\def\csname ckc#3{#1}\endcsname{{\check{
      \csname c#2{#1}\endcsname}}}}
\def\mydefck#1#2#3{\expandafter\def\csname ck#3{#1}\endcsname{{\check{#2{#1}}}}}

\cs_new:Npn \mydeff #1#2#3 {\cs_new:cpn {f#3{#1}} {\mathfrak{#2{#1}}}}

\cs_new:Npn \doGreek #1
{
  \clist_map_inline:nn {alpha,beta,gamma,Gamma,delta,Delta,epsilon,varepsilon,zeta,eta,theta,vartheta,Theta,iota,kappa,lambda,Lambda,mu,nu,xi,Xi,pi,Pi,rho,sigma,varsigma,Sigma,tau,upsilon,Upsilon,phi,varphi,Phi,chi,psi,Psi,omega,Omega,tG} {#1{##1}{\usecsname}{\useLetter}}
}

\cs_new:Npn \doSymbols #1
{
  \clist_map_inline:nn {otimes,boxtimes} {#1{##1}{\usecsname}{\useLetter}}
}

\cs_new:Npn \doAtZ #1
{
  \clist_map_inline:nn {A,B,C,D,E,F,G,H,I,J,K,L,M,N,O,P,Q,R,S,T,U,V,W,X,Y,Z} {#1{##1}{\useLetter}{\useLetter}}
}

\cs_new:Npn \doatz #1
{
  \clist_map_inline:nn {a,b,c,d,e,f,g,h,i,j,k,l,m,n,o,p,q,r,s,t,u,v,w,x,y,z} {#1{##1}{\useLetter}{\usedbletter}}
}

\cs_new:Npn \doallAtZ
{
\clist_map_inline:nn {mydefsf,mydeft,mydefu,mydefwh,mydefhat,mydefr,mydefwt,mydeff,mydefb,mydefbf,mydefc,mydefs,mydefck,mydefcks,mydefckc,mydefbar,mydefvec,mydefcirc,mydefdot,mydefbr,mydefacute} {\doAtZ{\csname ##1\endcsname}}
}

\cs_new:Npn \doallatz
{
\clist_map_inline:nn {mydefsf,mydeft,mydefu,mydefwh,mydefhat,mydefr,mydefwt,mydeff,mydefb,mydefbf,mydefc,mydefs,mydefck,mydefbar,mydefvec,mydefdot,mydefbr,mydefacute} {\doatz{\csname ##1\endcsname}}
}

\cs_new:Npn \doallGreek
{
\clist_map_inline:nn {mydefck,mydefwt,mydeft,mydefwh,mydefbar,mydefu,mydefvec,mydefcirc,mydefdot,mydefbr,mydefacute} {\doGreek{\csname ##1\endcsname}}
}

\cs_new:Npn \doallSymbols
{
\clist_map_inline:nn {mydefck,mydefwt,mydeft,mydefwh,mydefbar,mydefu,mydefvec,mydefcirc,mydefdot} {\doSymbols{\csname ##1\endcsname}}
}

\cs_new:Npn \doGroups #1
{
  \clist_map_inline:nn {GL,Sp,rO,rU,fgl,fsp,foo,fuu,fkk,fuu,ufkk,uK} {#1{##1}{\usecsname}{\useLetter}}
}

\cs_new:Npn \doallGroups
{
\clist_map_inline:nn {mydeft,mydefu,mydefwh,mydefhat,mydefwt,mydefck,mydefbar} {\doGroups{\csname ##1\endcsname}}
}

\cs_new:Npn \decsyms #1
{
\clist_map_inline:nn {#1} {\expandafter\DeclareMathOperator\csname ##1\endcsname{##1}}
}

\decsyms{Mp,id,SL,Sp,SU,SO,GO,GSO,GU,GSp,PGL,Pic,Lie,Mat,Ker,Hom,Ext,Ind,reg,res,inv,Isom,Det,Tr,Norm,Sym,Span,Stab,Spec,PGSp,PSL,tr,Ad,Br,Ch,Cent,End,Aut,Dvi,Frob,Gal,GL,Gr,DO,ur,vol,ab,Nil,Supp,rank,Sign}

\def\abs#1{\left|{#1}\right|}

\NewDocumentCommand\cent{o m }{
  \IfValueTF{#1}{
    \mathop{Z}_{#1}{(#2)}}
  {\mathop{Z}{(#2)}}
}

\def\fsl{\mathfrak{sl}}

\doallAtZ
\doallatz
\doallGreek
\doallGroups
\doallSymbols
\ExplSyntaxOff

\usepackage{diagbox}
\usepackage{arydshln}
\usepackage[mode=buildnew]{standalone}

\usepackage{tikz,etoolbox}
\usetikzlibrary{matrix,arrows,positioning,backgrounds}
\usetikzlibrary{decorations.pathmorphing,decorations.pathreplacing}
\usetikzlibrary{cd}

\usepackage{upgreek}

\usepackage{listings}
\newcommand{\BC}{{\mathbb {C}}}

\newcommand{\BH}{{\mathbb {H}}}

\newcommand{\BN}{{\mathbb {N}}}

\newcommand{\BR}{{\mathbb {R}}}

\newcommand{\CA}{{\mathcal {A}}}

\newcommand{\CE}{{\mathcal {E}}}
\newcommand{\CF}{{\mathcal {F}}}

\newcommand{\CJ}{{\mathcal {J}}}
\newcommand{\CK}{{\mathcal {K}}}

\newcommand{\CM}{{\mathcal {M}}}

\newcommand{\CO}{{\mathcal {O}}}
\newcommand{\CP}{{\mathcal {P}}}
\newcommand{\CQ}{{\mathcal {Q}}}

\newcommand{\CU}{{\mathcal {U}}}

\newcommand{\CX}{{\mathcal {X}}}

\newcommand{\CZ}{{\mathcal {Z}}}

\DeclareMathOperator{\ind}{ind}

\newcommand{\slt}{\operatorname{SL}_2(\mathbb{R})}

\newcommand{\od}{\operatorname{d}}

\newcommand{\oH}{\operatorname{H}}
\newcommand{\oO}{\operatorname{O}}
\newcommand{\oS}{\operatorname{S}}

\newcommand{\oU}{\operatorname{U}}

\newcommand{\gl}{\mathfrak g \mathfrak l}

\newcommand{\g}{\mathfrak g}
\newcommand{\h}{\mathfrak h}
\newcommand{\p}{\mathfrak p}
\newcommand{\Z}{\mathbb{Z}}
\DeclareDocumentCommand{\C}{}{\mathbb{C}}
\newcommand{\R}{\mathbb R}

\newcommand{\X}{\mathbf{X}}

\def\eDD{\overline{\nabla}}

\def\DD{\nabla}

\newcommand{\la}{\langle}
\newcommand{\ra}{\rangle}

\newcommand{\be}{\begin {equation}}
\newcommand{\ee}{\end {equation}}

\numberwithin{equation}{section}

\def\flushl#1{\ifmmode\makebox[0pt][l]{${#1}$}\else\makebox[0pt][l]{#1}\fi}
\def\flushr#1{\ifmmode\makebox[0pt][r]{${#1}$}\else\makebox[0pt][r]{#1}\fi}

\newtheorem*{thm*}{Theorem}
\newtheorem{thm}{Theorem}[section]

\newtheorem{lem}[thm]{Lemma}

\newtheorem*{lem*}{Lemma}

\newtheorem{prop}[thm]{Proposition}

\newtheorem{cor}[thm]{Corollary}

\newtheorem*{claim*}{Claim}
\newtheorem{defn}[thm]{Definition}

\newtheorem{eg}[thm]{Example}

\theoremstyle{remark}
\newtheorem*{remark}{Remark}

\def\slt{\fsl_2(\bC)}

\ExplSyntaxOn
\clist_map_inline:nn {J,L,C,X,Y,H,c,e,f,h,}{
  \expandafter\def\csname #1slt\endcsname{{\mathring{#1}}}}
\ExplSyntaxOff

\def\subset{\subseteq}
\def\Nil{\overline{\mathrm{Nil}}}

\NewDocumentCommand{\NilP}{t'}{
\IfBooleanTF{#1}{\Nil_{\fpp'}}{\Nil_\fpp}
}

\NewDocumentCommand{\KTW}{o g}{
  \IfValueTF{#2}{
    \left.\varsigma_{\IfValueT{#1}{#1}}\right|_{#2}}{
    \varsigma_{\IfValueT{#1}{#1}}}
}

\NewDocumentCommand{\CHI}{o g}{
  \IfValueTF{#1}{
    {\chi}_{\left[#1\right]}}{
    \IfValueTF{#2}{
      {\chi}_{\left(#2\right)}}{
      {\chi}}
  }
}
\NewDocumentCommand{\PR}{g}{
  \IfValueTF{#1}{
    \mathop{\pr}_{\left(#1\right)}}{
    \mathop{\pr}}
}
\NewDocumentCommand{\XX}{g}{
  \IfValueTF{#1}{
    {\cX}_{\left(#1\right)}}{
    {\cX}}
}
\NewDocumentCommand{\PP}{g}{
  \IfValueTF{#1}{
    {\fpp}_{\left(#1\right)}}{
    {\fpp}}
}
\NewDocumentCommand{\LL}{g}{
  \IfValueTF{#1}{
    {\bfL}_{\left(#1\right)}}{
    {\bfL}}
}
\NewDocumentCommand{\ZZ}{g}{
  \IfValueTF{#1}{
    {\cZ}_{\left(#1\right)}}{
    {\cZ}}
}

\NewDocumentCommand{\WW}{g}{
  \IfValueTF{#1}{
    {\bfW}_{\left(#1\right)}}{
    {\bfW}}
}

\NewDocumentCommand\KK{g}{
\IfValueTF{#1}{K_{(#1)}}{K}}
\NewDocumentCommand\XXo{d()}{
\IfValueTF{#1}{\cX^\circ_{(#1)}}{\cX^\circ}}

\NewDocumentCommand\ZZo{g}{
\IfValueTF{#1}{\cZ^\circ_{(#1)}}{\cZ^\circ}}

\NewDocumentCommand{\bcO}{t'}{
  \overline{\cO\IfBooleanT{#1}{'}}}

\NewDocumentCommand{\oliftc}{g}{
\IfValueTF{#1}{\boldsymbol{\vartheta} (#1)}{\boldsymbol{\vartheta}}
}
\NewDocumentCommand{\oliftr}{g}{
\IfValueTF{#1}{\vartheta_\bR(#1)}{\vartheta_\bR}
}
\NewDocumentCommand{\olift}{g}{
\IfValueTF{#1}{\vartheta(#1)}{\vartheta}
}
\NewDocumentCommand{\tlift}{g}{
\IfValueTF{#1}{\wtvartheta(#1)}{\wtvartheta}
}

\DeclareDocumentCommand{\NN}{g}{
\IfValueTF{#1}{\fN(#1)}{\fN}
}

\def\lsign{{}^l\mathrm{Sign}}

\def\ssign{\mathrm{Sign}}
\NewDocumentCommand{\sign}{m}{
  \mathrm{Sign}(#1)
}

\NewDocumentCommand\lnn{t+ t- g}{
  \IfBooleanTF{#1}{{}^l n^+\IfValueT{#3}{(#3)}}{
    \IfBooleanTF{#2}{{}^l n^-\IfValueT{#3}{(#3)}}{}
  }
}

\makeatletter
\def\bcO{\def\O@@{\cO}\@ifnextchar'\@Op\@Onp}
\def\@Opnext{\@ifnextchar^\@Opsp\@Opnsp}
\def\@Op{\afterassignment\@Opnext\let\scratch=}
\def\@Opnsp{\def\O@@{\cO'}\@Otsb}
\def\@Onp{\@ifnextchar^\@Onpsp\@Otsb}
\def\@Opsp^#1{\def\O@@{\cO'^{#1}}\@Otsb}
\def\@Onpsp^#1{\def\O@@{\cO^{#1}}\@Otsb}
\def\@Otsb{\@ifnextchar_\@Osb{\@Ofinalnsb}}
\def\@Osb_#1{\overline{\O@@_{#1}}}
\def\@Ofinalnsb{\overline{\O@@}}

\def\defpcmd#1{
  \def\nn@tmp{#1}
  \def\nn@np@tmp{@np@#1}
  \expandafter\let\csname\nn@np@tmp\expandafter\endcsname \csname\nn@tmp\endcsname
  \expandafter\def\csname @pp@#1\endcsname`##1{{}^{##1}{\csname @np@#1\endcsname}}
  \expandafter\def\csname #1\endcsname{\,\@ifnextchar`{\csname
      @pp@#1\endcsname}{\csname @np@#1\endcsname}}
}

\defpcmd{bfV}
\def\KK{\bfK}\defpcmd{KK}
\defpcmd{bfG}
\defpcmd{A}
\defpcmd{K}
\def\G{G}\defpcmd{G}
\defpcmd{J}
\defpcmd{L}
\defpcmd{eps}
\defpcmd{pp}
\defpcmd{wtK}
\makeatother

\NewDocumentCommand\LW{g}{
\IfValueTF{#1}{L_{W_{#1}}}{L_{W}}}

\def\floor#1{{\lfloor #1 \rfloor}}

\def\UU{\rU}

\def\Thetab{\bar{\Theta}}

\def\Spin{{\mathrm{Spin}}}

\def\totimes{\widehat{\otimes}}

\def\cf{\emph{cf.} }

\def\Irr{\mathrm{Irr}}

\def\Unip{\mathrm{Unip}}

\def\PBPe{\mathrm{PBP}^{\mathrm{ext}}}
\def\PBPes{\mathrm{PBP}^{\mathrm{ext}}_{\star}}
\def\PBPeg{\mathrm{PBP}^{\mathrm{ext}}_{G}}
\def\PBPesp{\mathrm{PBP}^{\mathrm{ext}}_{\star'}}

\def\DDn{\DD_{\mathrm{naive}}}

\makeatletter
\pgfkeys{/ytableau/options,
  noframe/.default = false,
  noframe/.is choice,
  noframe/true/.code = {%
    \global\let\vrule@YT=\vrule@none@YT
    \global\let\hrule@YT=\hrule@none@YT
  },
  noframe/false/.code = {%
    \global\let\vrule@YT=\vrule@normal@YT
    \global\let\hrule@YT=\hrule@normal@YT
  },
  noframe/on/.style = {noframe/true},
  noframe/off/.style = {noframe/false},
}

\def\hrule@enon@YT{%
  \hrule width  \dimexpr \boxdim@YT + \fboxrule *2 \relax
  height 0pt
}
\def\vrule@enon@YT{%
  \vrule height \dimexpr  \boxdim@YT + \fboxrule\relax
     width \fboxrule
}

\def\enon{\omit\enon@YT}
\newcommand{\enon@YT}[2][clear]{%
  \def\thisboxcolor@YT{#1}%
  \let\hrule@YT=\hrule@enon@YT
  \let\vrule@YT=\vrule@enon@YT
  \startbox@@YT#2\endbox@YT
  \nullfont
}

\makeatother
\ytableausetup{noframe=off,boxsize=1.3em}
\let\ytb=\ytableaushort

\newcommand{\tytb}[1]{{\tiny\ytb{#1}}}

\makeatletter
\newcommand{\dotminus}{\mathbin{\text{\@dotminus}}}

\newcommand{\@dotminus}{%
  \ooalign{\hidewidth\raise1ex\hbox{.}\hidewidth\cr$\m@th-$\cr}%
}
\makeatother

\def\ckcOp{\ckcO^{\prime}}
\def\ckcOpp{\ckcO^{\prime\prime}}

\def\cOp{\cO^{\prime}}
\def\cOpp{\cO^{\prime\prime}}

\def\uptaup{\uptau^{\prime}}
\def\uptaupp{\uptau^{\prime\prime}}

\def\taup{\tau^{\prime}}

\def\BOX{\mathrm{Box}}
\def\ckDD{{\check\DD}}

\def\PP{\mathrm{PP}}

\def\PBP{\mathrm{PBP}}

\def\Vs{V_{\sfss}}
\def\Vsp{V_{\sfss'}}

\def\maltese{\mathrm{T}}

\NewDocumentCommand\KC{o}{
  \IfNoValueTF{#1}{{K_{\bC}}}{{K_{#1,\bC}}}
}

\def\YD{{\mathsf{YD}}}
\def\SYD{{\mathsf{SYD}}}

\def\MYD{{\mathsf{MYD}}}
\def\AND{\quad\text{and}\quad}

\def\AOD{\mathrm{AOD}}

\def\pac#1{\ac_{#1}^+}
\def\nac#1{\ac_{#1}^-}

\def\AC{\mathrm{AC}}

\def\ac{\cL}
\def\lotimes{\otimes}

\def\sqii{\sqrt{-1}}
\def\St#1{\mathrm{St}_{#1}}

\def\Ls{L_\sfss}

\usepackage{subfiles}

\title[Special unipotent representations]{Special unipotent representations of real classical groups: construction and unitarity}

\author [D. Barbasch] {Dan Barbasch}
\address{Department of Mathematics\\
  310 Malott Hall, Cornell University, Ithaca, New York 14853 }
\email{barbasch@math.cornell.edu}

\author [J.-J. Ma] {Jia-Jun Ma}
\address{School of Mathematical Sciences\\
  Xiamen University, Xiamen, 361005, China}
  \address{Department of Mathematics, Xiamen University Malaysia campus, Sepang, Selangor Darul Ehsan, 43900,  Malaysia}
 \email{hoxide@xmu.edu.cn}

\author[B. Sun]{Binyong Sun}
\address{Institute for Advanced Study in Mathematics and New Cornerstone Science Laboratory\\  Zhejiang University,  Hangzhou, 310058, China}
\email{sunbinyong@zju.edu.cn}

\author [C.-B. Zhu] {Chen-Bo Zhu}
\address{Department of Mathematics\\
  National University of Singapore\\
  10 Lower Kent Ridge Road, Singapore 119076} \email{matzhucb@nus.edu.sg}

\subjclass[2020]{22E45, 22E46} \keywords{Unitary representation, nilpotent orbit, special unipotent representation, Barbasch-Vogan duality, classical group, theta lifting, moment map, associated cycle}

\begin{document}

\begin{abstract} 
Let $G$ be a real classical group (including the real metaplectic group). 
We consider a nilpotent adjoint orbit $\check{\mathcal O}$ of $\check G$, the Langlands dual of $G$ (or the metaplectic Langlands dual of $G$ when $G$ is a real metaplectic group). We classify all special unipotent representations of $G$ attached to $\check{\mathcal O}$, in the sense of Arthur and Adams-Barbasch-Vogan. When $\check{\mathcal O}$ has good parity in the sense of M{\oe}glin, we construct all such representations of $G$ via the method of theta lifting. As a consequence of the construction and the classification, we conclude that all special unipotent representations of $G$ are unitarizable, as predicted by the Arthur-ABV conjecture. We also determine the precise structure of the associated cycles of special unipotent representations of $G$. The paper is the second in a series of two papers on the classification of special unipotent representations of real classical groups.
\end{abstract}

\maketitle

\hspace*{\fill}\emph{To Roger E. Howe with gratitude and admiration}\hspace*{\fill}

\vskip 20pt 

\tableofcontents

\section{Introduction}\label{sec:intro}

\subsection{Background and goals}
In \cite{ArPro, ArUni}, Arthur predicted the existence of a collection of automorphic representations, the so-called unipotent automorphic representations, for the purpose of building arbitrary automorphic representations (from tempered and unipotent ones). In particular, it suggested the existence of certain collections of representations of reductive algebraic groups over $\BR$ or $\BC$. Arthur’s desired representations, the special unipotent representations in the title of this paper, were defined by Barbasch-Vogan in \cite{BVUni} for groups over $\BC$, and by Adams-Barbasch-Vogan (abbreviated as ABV hereafter) in \cite[Chapter 27]{ABV} for groups over $\BR$.
As special unipotent representations are expected to be the archimedean components of unipotent automorphic representations, it is natural to expect that all of them are unitarizable. This is the Arthur-ABV conjecture (\cite[Section 4]{ArUni}, \cite{ABV}*{Introduction}) for unipotent parameters.

In the present article, we will be concerned with a real classical group $G$ (see the list of groups in \eqref{tableg}). 
Our main goal is to construct and classify all special unipotent representations of $G$ attached to $\check \CO$. Here $\check \CO$ is a nilpotent adjoint orbit in $\check \g$, the Lie algebra of the Langlands dual of $G$ (or the metaplectic Langlands dual of $G$ when $G$ is a real metaplectic group, see \cite{Weis}). 
As a consequence of the construction and the classification, we conclude that all special unipotent representations of $G$ are unitarizable, as predicted by the Arthur-ABV conjecture. See \Cref{thmA}.

We derive a further consequence of our results. It is relatively easy to classify all genuine special unipotent representations of the (real, complex, quaternionic) spin groups, which are easily seen to be unitarizable (see \cite{BMSZ3}). 
Combining this with Barbasch's work in \cite{B.Class} for complex classical groups and the authors' work in \cite{BMSZ2} for simple linear Lie groups of type $A$, we may further conclude the unitarity of special unipotent representations for any real form of a connected reductive complex Lie group of a classical type. See \Cref{thmA'} for the precise statement. 
We remark that Adams, Miller, van Leeuwen and Vogan have recently established the same unitarity result in the case of real exceptional groups, with the Atlas algorithm of Adams-van Leeuwen-Trapa-Vogan being a key tool. See \cite{AMVV}.

\subsection{Approach}
In an earlier paper \cite{BMSZa}, the authors have found a precise count for the number of special unipotent representations of $G$ attached to $\check \CO$, and have also reduced the problem of constructing special unipotent representations attached to $\check \CO$ to the case when $\check \CO$ has good parity in the sense of \cite{Mo11}. We will thus focus our attention on the case when $\check \CO$ has good parity.

We outline the main ingredients of our approach to the classification, which is completely explicit. It starts with a certain combinatorially defined set $\PBPeg(\check \CO)$, called the extended set of painted bipartitions attached to $(G, \check \CO)$, which counts the number of special unipotent representations of $G$ attached to $\check \CO$. We then define a combinatorial operation, called the descent that takes $\uptau \in \PBPeg(\check \CO)$ to a $\uptau '\in \mathrm{PBP}^{\mathrm{ext}}_{G'}(\check \CO')$, where $G'$ is a ``smaller" classical group determined by $\uptau$ such that $(G, G')$ is a reductive dual pair in the sense of \cite{Howe79}, and $\check \CO'$ is a nilpotent adjoint orbit in $\check \g'$ determined by $\check \CO$. 
This descent operation allows us to construct inductively a special unipotent representation $\pi_{\uptau}$ of $G$ from a representation $\pi_{\uptau'}$ of $G'$ via theta lifting (\cite{Howe89}).

The main challenge in concluding the classification (i.e. exhaustion by our construction) is to distinguish the constructed representations, which we overcome by determining the associated cycles of representations (\cite{Vo89}) in every step of the induction process. It turns out that the combinatorially defined descent map also carries with it a host of geometric information, notably a double fibration of moment maps. This ultimately allows us to compute inductively the associated cycle of $\pi_{\uptau}$ in terms of the associated cycle of $\pi_{\uptau'}$, by geometric theta lifting (a process analogous to, but simpler than, theta lifting). 

To complete the aforementioned task on the associated cycles and to conclude the unitarizability of all special unipotent representations, we make crucial technical advances  on two fundamental issues of theta lifting: nonvanishing and unitarity preservation. Our approach to the issue of nonvanishing is via a  bridge we set up relating the theory of theta lifting and the theory of associated varieties, which is a key part of Vogan’s formulation of the orbit method for real reductive groups (\cite{VoBook,Vo98}). On the other hand, our approach to the issue of unitarity preservation exploits properties of representations with almost $L^2$ matrix coefficients (\cite{CHH}) and is quite general.

\subsection{Prior and recent work}

Important local consequences of Arthur's conjecture were established by Barbasch and Vogan for groups over $\BC$ (\cite{BVUni}) and by Adams, Barbasch, and Vogan for groups over $\BR$ (\cite{ABV}). In particular, (integral) special unipotent representations for complex semisimple groups were classified in \cite{BVUni}. For the group of real points of a connected reductive algebraic group defined over $\BR$, Adams, Barbasch and Vogan defined a finite set of irreducible admissible representations for any local Arthur parameter, now called an ABV-packet. They established a number of properties of such representations expected by Arthur but not their unitarity. See Problems A--F, \cite{ABV}. In the direction of unitarity, the best general results were those of Barbasch for complex classical groups (\cite{B.Class}). See also \cite{Mil} for some cases in exceptional groups. 

The other thread linking with the subject matter of this article is  Howe's theory of theta lifting (\cite{Howe79,Howe89}). By the end of 1980s, it was well recognized that the theory has much relevance for constructing unitary representations of
classical groups (see for example, \cite{HoweRank,Li89}). Since then there have been many attempts to construct special unipotent representations (or more generally the so-called unipotent representations, see \cite{VoBook})
in the framework of theta lifting. See \cites{Sa,ZH,HL,Br,He,Tr,PT,LM,B17, Mo17}.
From the vantage point of this article, particularly worth mentioning is the work of Przebinda (\cites{Pz1,Pz2}) in which a double
fibration of moment maps was exploited in the framework of theta lifting, and the work of He (\cite{He}) in which an innovative technique called quantum
induction was devised to show the nonvanishing of the lifted representations.

We now discuss the related developments which are more recent. In a monumental work (called the theory of endoscopic classification) \cite{ArEnd}, Arthur has defined the global and local A-packets in the cases of quasi-split orthogonal and symplectic groups.  For these groups, Adams, Arancibia and Mezo (\cite{AAM}),  at about the same time as the current paper was completed, showed that archimedean A-packets coincide with the ABV-packets constructed earlier by Adams, Barbasch and Vogan in \cite{ABV}. More recently, Arancibia and Mezo (\cite{AM}) have proved the same results for quasi-split unitary groups, based on the endoscopic classification for quasi-split unitary groups established by Mok in \cite{Mok}. As the set of special unipotent representations attached to $\check \CO$ is a union of unipotent ABV-packets (\cite[Corollary 27.13]{ABV}), and representations in A-packets are unitary (\cite{ArEnd,Mok}), the main results of \cite{AAM} and \cite{AM} provide an independent proof of the unitarity of special unipotent representations in the case of  quasi-split classical groups. We remark further that the internal structure of local A-packets has been a subject of intense interest. 
In the archimedean case, we refer the reader to several works \cite{Mo17,MR18a,MR18b,MR19,MR20} of M{\oe}glin, and M{\oe}glin and Renard  for results in this direction.

\section{The main results}
\subsection{Special unipotent representations of real classical groups}\label{secsu} 
Let $\star$ be one of the 10 labels, let $G$ be a classical Lie group of type $\star$, and let $\check G$ be its Langlands dual group (\cite{Borel}) or a variation, as in the following table ($n,p,q\in \bN:=\{0,1,2, \dots\}$). 

\be\label{tableg}
  \begin{aligned}
    &\textrm{Label $\star$}&& \textrm{Real classical group $G$} & & \textrm{Langlands dual group  }\check G\\ 
    & A^\R&&\GL_n(\R) &&\GL_n(\C)\\
    & A^\bH&&\GL_{\frac{n}{2}}(\bH)\ \,  (n \textrm{ even}) && \GL_{n}(\C)\\
    & A &&\oU(p,q) &&\GL_{p+q}(\C)\\
    &\widetilde{A}&&\widetilde \oU(p,q) &&\GL_{p+q}(\C)/\{\pm 1_{p+q}\}\quad(\textrm{$1_{p+q}$ is the identity matrix})\\
    & B &&\oO(p,q)\ \, (p+q\, \textrm{ odd}) && \Sp_{p+q-1}(\C)\\
    &D&&\oO(p,q)\  \, (p+q\, \textrm{ even}) &&\oO_{p+q}(\C)\\
    &C&&\Sp_{2n}(\R)&& \oO_{2n+1}(\C)\\
    &\widetilde{C}&&\widetilde \Sp_{2n}(\R) &&\Sp_{2n}(\C)\\
    &D^*&& \oO^*(2n) &&\oO_{2n}(\C)\\
    &C^* && \Sp(\frac{p}{2},\frac{q}{2})  \ \, (p,q\, \textrm{ even}) &&\oO_{p+q+1}(\C)  
  \end{aligned}
\ee

In this table $\widetilde \Sp_{2n}(\R)$ denotes  the metaplectic double cover of the symplectic
 group $\Sp_{2n}(\R)$ (it does not split for $n>0$), and  $\widetilde \rU(p,q)$ denotes
 the (linear) double cover of the unitary group  $\rU(p,q)$ defined by a square root of the
 determinant character.  

 Let $\g$ denote the complexified Lie algebra of $G$, and let $\check \g$ denote the Lie algebra of $\check G$. Then $\check \g$ is the
  Langlands dual of  $\g$, except for the case of $\wtC$. In this exceptional case,  $\check \g=\mathfrak{sp}_{2n}(\C)$, to be called the metaplectic Langlands dual of $\g=\mathfrak{sp}_{2n}(\C)$ (see  \cite{Weis,BMSZ1}). In all cases, the  universal Cartan subalgebra of $\check \g$ is identified with the dual of the universal Cartan subalgebra of $\g$ (for the case of $\wtC$, the identification  is induced by the half of the trace form on the
 Lie algebra $\g=\mathfrak{sp}_{2n}(\C)$).  
Thus by the Harish-Chandra isomorphism (or a slight variation of Harish-Chandra isomorphism), there is an identification (\cf \cite{BMSZa}*{Section 2.3})
\begin{equation}\label{idenhach}
   \Hom_{\C-\mathrm{alg}}(\CU(\g)^{G_\C}, \C)=\check G\backslash\{\textrm{semisimple elements in $\check \g$}\}. 
\end{equation}
Here and henceforth,  $\CU$ indicates the universal enveloping algebra, a superscript group indicates the fixed point set of a group action, $\Hom_{\C-\mathrm{alg}}$ indicates the set of the $\C$-algebra homomorphisms, and $G_\C$ is the complexification of $G$ which is respectively
\[
  \begin{array}{c}
  \GL_n(\C), \ \GL_n(\C),\  \GL_{p+q}(\C), \ \widetilde{\GL}_{p+q}(\C), \\  
  \oO_{p+q}(\C),\ \oO_{p+q}(\C),\ \Sp_{2n}(\C),\  \Sp_{2n}(\C),\  \oO_{2n}(\C),\ \textrm{or}\   \Sp_{p+q}(\C).
  \end{array}
\]
In the above, $\widetilde{\GL}_{p+q}(\C)$
denotes the  double cover of $\GL_{p+q}(\C)$ defined by a square root of the
 determinant character. 

The Lie algebra $\check \g$ is viewed as a space of matrices as usual, and we denote by
$\mathrm{Nil}(\check \g)$ the set of nilpotent matrices in $\check \g$. 
Let $\check \CO$ be a $\check G$-orbit in $\mathrm{Nil}(\check \g)$.  
Consider a semisimple element of $\check \g$ that equals half of 
 the neutral element in any
 $\mathfrak s\mathfrak l_2$-triple attached to $\ckcO$, as in \cite{BVUni}*{Section 5}.  Using the identification \eqref{idenhach},   this semisimple element determines a $\C$-algebra homomorphism
 \[
   \chi_{\check \CO}: \CU(\g)^{G_\C}\rightarrow \C.
 \]
 The homomorphism $\chi_{\check \CO}$ is independent of the choice of the $\mathfrak s\mathfrak l_2$-triple. 
 
Let $\CZ(\g)$ denote the center of $\CU(\g)$. Then $\CZ(\g)\supsetneq \CU(\g)^{G_\C}$ when $G_\C$ is a nontrivial even orthogonal group, and in all other cases $\CZ(\g)=\CU(\g)^{G_\C}$.  Recall the following result of Duflo (see \cite{Dix}, \cite[Section 3]{Bor}): for every algebraic character $\chi$ of $\CZ(\g)$, there exists a unique maximal ideal of $\CU(\g)$ that contains the kernel of $\chi$. 
As an easy consequence, there is a unique maximal $G_\C$-stable ideal of $\CU(\g)$ that contains the kernel of $\chi_{\check \CO}$. Write $I_{\check \CO}$ for this ideal.

Recall that a  smooth Fr\'echet representation of  moderate growth  of a real reductive group is called a Casselman-Wallach representation (\cite[Chapter 11]{Wa2}) if its Harish-Chandra module has  finite length. Let $\Irr(G)$ denote the set of isomorphism classes of irreducible Casselman-Wallach representations of $G$. 
Following Barbasch-Vogan (\cite{BVUni}) and  Adams-Barbasch-Vogan (\cite[Chapter 27]{ABV}), define the set of the special unipotent representations of $G$
 attached to $\ckcO$ by
 \begin{equation}\label{eq:defuni}
   \begin{split}
     \Unip_{\ckcO}(G)
     :=& \begin{cases}
       \{\pi\in \Irr(G) \,:\,\pi \textrm{ is genuine  and annihilated by } I_{\check \CO}\}, & \text{if } \star\in \{\widetilde A, \widetilde C\};\\
       \{\pi\in \Irr(G) \,:\, \pi \textrm{ is annihilated by } I_{\check \CO}\}, & \text{otherwise}.\\
     \end{cases}
   \end{split}
\end{equation}
 Here ``genuine" means that the central subgroup $\{\pm 1\}$ of $G$, which is the kernel of the covering homomorphism $\widetilde \oU(p,q)\rightarrow  \oU(p,q)$ or $\widetilde \Sp_{2n}(\R)\rightarrow \Sp_{2n}(\R)$, acts on $\pi$ through the nontrivial character. 
 Note that an irreducible representation $\pi\in \Irr(G)$ is annihilated by $I_{\check \CO}$ if and only if   $\CU(\g)^{G_\C}$ acts on it through the character $\chi_{\check \CO}$ and the complex associated variety of $\pi$ is contained in the Zariski closure of $\CO$. Here $\CO$ is the Barbasch-Vogan dual of $\check \CO$ (see \Cref{bvdual}). 
 
 \begin{remark} The nilpotent orbit $\CO$ is special in the sense of Lusztig (\cite{Lsp}) when $\star \ne \widetilde C$, and is called metaplectic special when $\star =\widetilde C$. See \cite{BVUni} and \cite{BMSZ1}. 
 \end{remark}
 
 \begin{remark} In \cite{BMSZa}, we treat special orthogonal groups rather than orthogonal groups. We will examine the relationship of their special unipotent representations by applying Clifford theory. See \Cref{sorth}.  
  \end{remark}

\begin{remark} Special unipotent representations of complex classical groups are well-understood (see \cite{BVUni}, \cite{B.Class}) and are known to be unitarizable (\cite{B.Class}). The method of this paper (with minor modifications) applies to complex classical groups. We refer the reader to \cite[Section 2.9]{BMSZa} for a review both of their counting and their constructions, which is in the same framework as this article. Note also that Losev, Mason-Brown and Matvieievskyi (\cite{LMBM}) have recently proposed a notion of unipotent representations for a complex reductive group, extending the notion of special unipotent representations.  
\end{remark} 

    When $\star\in \{A^\R, A^\mathbb H\}$ so that $G$ is a general linear group, all special unipotent representations of $G$  are obtained via normalized smooth parabolic induction from quadratic characters (see \cite{VGL}*{Page 450}). As an immediate consequence, all  special unipotent representations of $G$ are unitarizable. We refer the reader to \cite[Section 7]{BMSZa} for a review of their classification.
 In the rest of this article, we assume that $\star\notin \{A^\R, A^\mathbb H\}$.

Combining reduction to good parity in our earlier paper \cite{BMSZa} (see \Cref{thmB}) and construction and classification in the case of good parity in the current paper (see  \Cref{thmC}), we will have constructed all the representations in $\Unip_{\check \CO}(G)$. As a consequence, we will have the following unitarity result, as predicted by the Arthur-ABV conjecture. For completeness, we have included the two cases of $\star \in \{A^\R, A^\mathbb H\}$, which are due to Vogan (\cite{VGL}). 

\begin{thm}\label{thmA}
Let $G$ be a real classical group from the list \eqref{tableg}. 
Let $\check \CO$ be a $\check G$-orbit  in $\mathrm{Nil}(\check \g)$. Then all representations in $\Unip_{\ckcO}(G)$ are unitarizable. 
\end{thm}
\begin{remark}
As mentioned in the introduction, for a quasi-split orthogonal group or a real symplectic group $G$ (resp. a quasi-split unitary group), the unitarity of the representations in $\Unip_{\ckcO}(G)$ is independently proved in \cite{AAM} (resp. \cite{AM}), from the perspective of the endoscopic classification of representations (\cite{ArEnd, Mok}).   

\end{remark}

We now temporarily shift gears and consider the following general class of Lie groups. Let $G_\BC$ be a connected  reductive complex Lie group, and let $G$ be a real form of $G_\BC$, namely the fixed point group of an anti-holomorphic involutive automorphism $\sigma$ of $G_\BC$.  Note that this class of Lie groups includes any real classical group 
from the list \eqref{tableg} except a real metaplectic group, and it also includes complex classical groups (namely complex general linear, orthogonal and symplectic groups), as well as (real, complex, quaternionic) spin groups. 
Denote by $\check G$ the Langlands dual group of $ G_\BC$, and by $\check \g$ the Lie algebra of $\check G$. For a nilpotent $\check G$-orbit $\check \CO$ in $\check \g$, let $\mathrm{Unip}_{\check \CO}(G)$ be the set of isomorphism classes of special unipotent representations of $G$ attached to $\check \CO$. The Arthur-ABV conjecture for unipotent parameters asserts that all representations in $\mathrm{Unip}_{\check \CO}(G)$ are unitarizable. It is easy to see that the conjecture is reduced to the case when $G_\BC$ is simply connected, and the (real) Lie algebra $\mathrm{Lie}(G)$ of $G$ is simple.  In this case the Lie algebra of $G_\BC$, denoted by $\g$, is either simple or the product of two isomorphic simple Lie algebras. Below is the list of groups $G_\BC$ and $G$ satisfying the aforementioned properties with $\g$ of a classical type ($n,m,p,q\in \BN$, $n\geq 2$, $m\geq 7$).  
\[
  \begin{aligned}
    & \textrm{$G_\BC$} & \ &  G\\
    & \SL_n(\C) &\ & \SL_n(\R), \ \  \mathrm{SU}(p,q) \ (p+q=n),  \ \ \SL_{\frac{n}{2}}(\BH)\ (n\textrm{ is even})\\
      & \Spin_m(\C) &\ & \Spin (p,q) \ (p+q=m),  \ \ \Spin^*(m)  \ (m\textrm{ is even})\\
       & \Sp_{2n}(\C) &\ & \Sp(\frac{p}{2},\frac{q}{2}) \ (\textrm{$p,q$ are even}, \, p+q=2n),  \ \ \Sp_{2n}(\R)   \\
         & \SL_n(\C)\times \SL_n(\C) &\ & \SL_n(\C)\\
           & \Spin_m(\C)\times \Spin_m(\C) &\ & \Spin_m(\C)\\
       & \Sp_{2n}(\C)\times \Sp_{2n}(\C) &\ & \Sp_{2n}(\C)\\
  \end{aligned}
\]
\trivial[h]{{\bf Type $A$} ($n\geq 2$): 
\begin{itemize}
    \item $G_\BC: \SL_n(\BC) \text{ or } \SL_n(\BC)\times \SL_n(\BC)$. 
\item $G=\SL_n(\R), \ \mathrm{SU}(p,q) \ (p+q=n),  \ \SL_{\frac{n}{2}}(\BH)\ (n\textrm{ is even}),  \textrm{ or } \SL_n(\BC)$.
\end{itemize}

{\bf Type $B$ or $D$} ($m\geq 7$): 
\begin{itemize}
    \item $G_\BC: \Spin_m(\BC) \text{ or } \Spin_m(\BC) \times \Spin_m(\BC)$.  
\item $G=\Spin (p,q) \ (p+q=m),  \ \Spin^*(m) \ (m\textrm{ is even}), \ \textrm{or}\ \Spin_m(\BC)$.
\end{itemize}

{\bf Type $C$} ($n\geq 2$): 
\begin{itemize}
    \item $G_\BC: \Sp_{2n}(\BC) \text{ or }  \Sp_{2n} (\BC) \times  \Sp_{2n}(\BC)$.
    \item $G=\Sp(p,q) \ (p+q=n),  \ \Sp_{2n}(\BR), \ \textrm{or} \ \Sp_{2n}(\BC)$.
\end{itemize}
}

\Cref{thmA}, together with the results of Barbasch for complex classical groups (\cite{B.Class}) and the results of the present authors for spin groups (\cite[Theorems 6 and 7]{BMSZ3}) and for simple linear Lie groups of type $A$ (\cite[Corollary 1.2]{BMSZ2}), thus implies the following unitarity result, in accordance with the Arthur-ABV conjecture. 

\begin{thm}\label{thmA'} 
Let $G_\BC$ be a connected  reductive complex Lie group, and $G$  a real form of $G_\BC$. Assume that every simple factor of the Lie algebra $\g$ of $G_\BC$ is of a classical type. Let $\check \CO$ be a nilpotent $\check G$-orbit  in $\check \g$. Then all representations in $\Unip_{\ckcO}(G)$ are unitarizable. 
\end{thm}

\subsection{Reduction to good parity: a review}\label{subsec:red}
From now on, let $\star$ be one of the following 8 labels: $A, \wtA, B, D, C, \wtC, C^{*}, D^{*}$. Note that if $\star\in \{C, \wtC, D^*\}$, then $n$ (in \eqref{tableg}) equals the rank of $\g$. In all other cases, we also let $n$ denote the rank of $\g$.
We say that an integer has good parity (depends on $\star$ and $n$) if it has the same parity as
\begin{equation}\label{parity}
  \begin{cases}
    n, &  \text{if $\star =A$}; \\
    1+ n, &  \text{if $\star = \wtA$}; \\
   1, & \text{if } \star \in \set{C,C^{*},D,D^{*}};\\
   0, & \text{if } \star \in \set{B,\wtC}.\\
  \end{cases}
\end{equation}
Otherwise we say that the integer has bad parity.

When there is no confusion, we will not distinguish between a nilpotent orbit $\check \CO$ and the Young diagram corresponding to $\check \CO$. Given two Young diagrams $\imath$ and $\jmath$, write $\imath \stackrel{r}{\sqcup} \jmath$ for
the Young diagram whose multiset of nonzero row lengths equals the union of
those of $\imath$ and $\jmath$. Also write $2\imath =\imath \stackrel{r}{\sqcup} \imath$. 

We have a Young diagram decomposition 
\begin{equation*} 
     \ckcO=\check \CO_\mathrm g  \stackrel{r}{\sqcup} \check \CO_\mathrm b
\end{equation*}
 such that $\check \CO_\mathrm g$ has good parity in the sense that all its nonzero row lengths have good parity, and $\check \CO_\mathrm b$ has bad parity in the sense that all its nonzero row lengths have bad parity.
 We assume that
 \be\label{doubleco}
 \check \CO_\mathrm b =2 \check \CO'_\mathrm b 
 \ee
 for some Young diagram $\check \CO'_\mathrm b$. This is automatic when  $\star\in \set{B, C,\wtC, C^{*},D,D^{*}}$, and is necessary when $\star \in\{A, \wtA\}$ and the set $\Unip_{\ckcO}(G)$ is nonempty (see \cite[Theorem 2.15]{BMSZa}).
 
 For every Young diagram $\imath$, write $\abs{\imath}\in \BN$ for the number of boxes in $\imath$ (see \eqref{eq:BOX22}). Put
 \[
   n_\mathrm b:=\abs{\check \CO'_\mathrm b},
 \]
 and define a group
 \begin{equation}\label{Gpb}
  G'_\mathrm b := \begin{cases}
  \GL_{n_\mathrm b}(\C),  & \text{if } \star \in \set{A, \wtA}; \\
    \GL_{n_\mathrm b}(\bR), & \text{if } \star \in \set{B,C,D}; \\
       \widetilde{ \GL}_{n_\mathrm b}(\bR), & \text{if } \star =\wtC; \\
    \GL_{\frac{n_\mathrm b}{2}}(\bH), & \text{if } \star \in \set{C^{*},D^{*}}. \\
  \end{cases}
\end{equation}
Here $ \widetilde{ \GL}_{n_\mathrm b}(\bR)$ is the double cover of $ \GL_{n_\mathrm b}(\bR)$ that fits the following Cartesian diagram of Lie groups:
\begin{equation}\label{wgll}
\begin{CD}
 \widetilde{ \GL}_{n_\mathrm b}(\bR)@>>>  \GL_{n_\mathrm b}(\bR)\\
  @VVV @VV g\mapsto \textrm{ sign of $\det(g)$} V\\
  \{\pm 1, \pm \sqrt{-1}\} @> x\mapsto x^2 >> \{\pm 1\}. \\
\end{CD}
\end{equation}

Consider $\check \CO'_\mathrm b$ as a $\GL_{n_\mathrm b}(\C)$-orbit in the set $\mathrm{Nil}(\gl_{n_\mathrm b}(\C))$ of nilpotent matrices. Similar to $\Unip_{\ckcO}(G)$,  let $\Unip_{\check \CO'_\mathrm b}(G'_\mathrm b)$ denote the set of isomorphism classes of special unipotent representations of $G'_\mathrm b$ attached to $\check \CO'_\mathrm b$. When $\star \in \set{A, \wtA}$ so that $G'_\mathrm b=\GL_{n_\mathrm b}(\C)$, the set  $\Unip_{\check \CO'_\mathrm b}(G'_\mathrm b)$ is a singleton whose  unique element is the parabolic induced representation of a trivial character (see \cite[Section 2.9]{BMSZa}). When $\star=\wtC$ so that $G'_\mathrm b= \widetilde{ \GL}_{n_\mathrm b}(\bR)$, we have a  bijective map
 \[
    \Unip_{\check \CO'_\mathrm b}(\GL_{n_\mathrm b}(\bR))\rightarrow  \Unip_{\check \CO'_\mathrm b}(\widetilde{ \GL}_{n_\mathrm b}(\bR)), \quad \pi\mapsto \pi\otimes \tilde \chi_{n_\mathrm b},
 \]
 where $\tilde \chi_{n_\mathrm b}$ is the character given by the left vertical arrow of \eqref{wgll}.
 
We further assume that 
\be\label{existgl}
  \textrm{either $\star\in \{A, \wtA, B,D, C^*\}\ $ and $\ p,q\geq n_\mathrm b,\quad \ $ or  $\quad\  \star\in \{C, \wtC, D^*\}$}.
\ee
Otherwise the set $\Unip_{\ckcO}(G)$  is empty (\cite[Theorems 2.15 and 2.20]{BMSZa}).
Put  \be\label{gg00}
  G_\mathrm g :=
  \begin{cases}
  \oU(p-n_\mathrm b, q-n_\mathrm b),  & \text{if } \star =A; \\
  \widetilde \oU(p-n_\mathrm b, q-n_\mathrm b),  & \text{if } \star =\wtA; \\
    \oO(p-n_\mathrm b,q-n_\mathrm b), & \textrm{if $\star\in \set{B,D}$};\\
    \oO^{*}(2n-2n_\mathrm b), &\textrm{if $\star = D^{*}$};\\
    \Sp_{2n-2n_\mathrm b}(\bR), &\textrm{if $\star = C$};\\
    \wtSp_{2n-2n_\mathrm b}(\bR), &\textrm{if $\star = \wtC$};\\
    \Sp(\frac{p-n_\mathrm b}{2},\frac{q-n_\mathrm b}{2}), &\textrm{if $\star = C^{*}$}.\\
  \end{cases}
\ee

Define a complex group
\be\label{gg0022}
 \check  G_\mathrm g :=
  \begin{cases}
  \GL_{n-2n_\mathrm b}(\C),  & \text{if } \star =A; \\
 \GL_{n-2n_\mathrm b}(\C)/\{\pm 1_{n-2n_\mathrm b}\},  & \text{if } \star=\wtA; 
  \\
    \Sp_{2n-2n_\mathrm b}(\C), & \textrm{if $\star\in\{B,\wtC\}$};\\
    \oO_{2n-2 n_\mathrm b}(\C), &\textrm{if $\star\in \set{D,D^*}$};\\
    \oO_{2n+1-2n_\mathrm b}(\C), &\textrm{if $\star\in \{C, C^*\}$}.
  \end{cases}
\ee
Then $\check \CO_\mathrm g$ is identified with a $\check G_\mathrm g$-orbit in $\mathrm{Nil}(\check \g_\mathrm g)$, where $\check \g_\mathrm g$ denotes the Lie algebra of $\check G_\mathrm g$.
As before we have the set $\Unip_{\ckcO_{\mathrm g}}( G_{\mathrm g})$ of special unipotent  representations of $G_{\mathrm g}$ attached to $\ckcO_{\mathrm g}$. 

Up to conjugation, $G$ has a unique parabolic subgroup, to be denoted by $Q$, whose Levi component is 
isomorphic to   $G'_\mathrm b\times G_\mathrm g$ (or   $(G'_\mathrm b\times G_\mathrm g)/\{\pm 1\} $ when $\star=\wtC$). Based on \cite[Theorems 2.16 and 2.20]{BMSZa}, we will prove the following result in \Cref{modify}.   

\begin{thm}
\label{thmB}
 Suppose that the conditions  \eqref{doubleco} and \eqref{existgl} hold. Then the normalized smooth parabolic induction from $Q$ to $G$ yields
   a well-defined bijection
   \[
    \begin{array}{rccc}
 &\Unip_{\ckcO'_{\mathrm b}}( G'_{\mathrm b}) \times   \Unip_{\ckcO_{\mathrm g}}( G_{\mathrm g})  &         \longrightarrow &\Unip_{\ckcO }(G), \\
                &   (\pi',\pi_\mathrm g) & \mapsto & \pi'\rtimes \pi_\mathrm g.
    \end{array}
 \]
 Otherwise the set $\Unip_{\ckcO }(G)$ is empty. 
\end{thm}

By \Cref{thmB}, the classification and unitarity  of special unipotent representations of $G$ attached to $\check \CO$ are reduced to the case when $\check \CO$ has good parity. 

\subsection{Main results in the good parity case}\label{subsec:goodPR}

Except for Section \ref{sorth}, in the rest of this article we will only consider the good parity case. Without fully explaining all terminologies and notations (the readers are referred to \Cref{sec:bip}), we state the main results of this paper in the good parity case.   

We consider real classical groups of type $B,C,D$, namely we  assume that \[\star\in \{B,C,D,\widetilde {C}, C^*, D^*\}.\] 

The starting point is a certain combinatorially defined set $\PBPeg(\check \CO)$, called the extended set of painted bipartitions attached to $(G, \check \CO)$. According to \cite{BMSZa}, it counts the number of special unipotent representations of $G$ attached to $\check \CO$ (for this statement, we take $G$ to be the special orthogonal group rather than the orthogonal group if $\star\in \{B,D\}$). 

Our first key observation is that there is a combinatorially defined operation, called the descent that takes $\uptau \in \PBPeg(\check \CO)$ to a $\uptau '\in \mathrm{PBP}^{\mathrm{ext}}_{G'}(\check \CO')$, where $(G, G')$ is a reductive dual pair, and $\check \CO'$ is a $\check {G'}$-orbit in $\mathrm{Nil}(\check \g')$. 
This allows us to construct inductively a special unipotent representation $\pi_{\uptau}$ of $G$ via the method of theta lifting and in terms of the representation $\pi_{\uptau'}$ of $G'$.

Our main result on the classification and unitarity of special unipotent representations is then as follows. 

\begin{thm}[\Cref{thm100}]\label{thmC} Suppose that $\check \CO$ is a $\check G$-orbit in $\mathrm{Nil}(\check \g)$ that has good parity. 
\begin{enumerate}[wide=0em,label=(\alph*)]
\item For every $\uptau \in \PBPeg(\check \CO)$, the representation $\pi_{\uptau}$ of $G$  is irreducible, unitarizable,  and special unipotent attached to $\ckcO$.

\item If $\star\in \{B,D\}$ and $(\star, |\check \CO|)\neq (D, 0)$, then the map
\[
\begin{array}{rcl}
\PBPeg(\check \CO)\times \Z/2\Z&\rightarrow &\Unip_{\ckcO}(G),\\
  (\uptau, \epsilon)&\mapsto& \pi_{\uptau}\otimes \det^\epsilon
  \end{array}
\]
is bijective.
 In all other cases, the map
\[
\begin{array}{rcl}
\PBPeg(\check \CO)&\rightarrow &\Unip_{\ckcO}(G),\\
  \uptau &\mapsto& \pi_{\uptau}
  \end{array}
\]
is bijective.
\end{enumerate}
\end{thm}

Our second key observation is that the descent map carries with it a host of geometric information. 
In particular, it allows us to define a certain notion of geometric theta lift from $\CK(\CO')$ to $\CK(\CO)$, where
$\CO$ (resp., $\CO'$) is the Barbasch-Vogan dual of $\ckcO$ (resp., $\ckcO'$), and $\CK(\CO)$ (resp., $\CK(\CO')$) is a certain $\CK$ group associated to the pair $(G,\CO)$ (resp., $(G',\CO')$). (See \Cref{sec:dlift}.) By mimicking the inductive definition of $\pi_{\uptau}$ (by geometric theta lifting, in place of theta lifting), we construct an element $\AC(\uptau)$ of $\CK(\CO)$ (see \eqref{eq:AC}), called the associated cycle of $\uptau$. 

Our main result on the determination of the associated cycles of special unipotent representations then reads as follows. It plays a key role in establishing the exhaustion of special unipotent representations by our construction, namely Part (b) of \Cref{thmC}. 

\begin{thm}[\Cref{thmpitau}]\label{thmD} For every $\uptau\in \PBPeg(\ckcO)$, the following equality holds: 
\[
\AC_\CO(\pi_\uptau)=\AC(\uptau)\in \CK(\CO), 
\]
where $\CO$ denotes the Barbasch-Vogan dual of $\ckcO$, and $\AC_\CO(\pi_\uptau)$ denotes the associated cycle of $\pi_\uptau$. 
\end{thm}

\begin{remark} By the fundamental result of Schmid and Vilonen in \cite{SV}, the weak associated cycle agrees with the wavefront cycle under the Kostant-Sekiguchi
correspondence. Therefore \Cref{thmD} also determines the wavefront cycle of any special unipotent representation.     
\end{remark}

We highlight one additional result, from our determination of the associated cycles of special unipotent representations.  

\begin{thm}[\Cref{thmac0}] \label{thmE} If $\check \CO$ is quasi-distinguished (Definition \ref{defqd}), then the associated cycle map induces a bijection:
\[
\mathrm{AC}_{\CO}: \,\,\, \mathrm{Unip}_{\check \CO}(G)\rightarrow  \mathrm{AOD}(\CO),\]
where $\mathrm{AOD}(\CO)$ denotes the set of admissible orbit data in $\CO$. 
\end{thm}

\subsection{Organization}\label{subsec:outline}

We begin with some remarks on the dependence of this article on our earlier article \cite{BMSZa}. The main result of this article is the construction and the classification of special unipotent representations. (The unitarity is a consequence of the construction and the classification.) Our construction is based on the method of theta lifting and naturally is independent of the results of \cite{BMSZa}. Our classification (more precisely the exhaustion by our construction) depends on the counting of the special unipotent representations established in \cite{BMSZa}. But in fact we only need a counting inequality (see the remark after \Cref{thmcount}). While we prove a counting equality in \cite{BMSZa} (on irreducible representations with a given infinitesimal character and a given bound of the complex associated variety) from the theory of coherent continuation representations and its modern development, the corresponding inequality is relatively easy to establish and it follows from well-known results in Kazhdan-Lusztig theory (due to Lusztig, Joseph, Vogan, Barbasch-Vogan and Casian) which were in place in  the 1980s. 

We now give an outline on the organization of this article. In Section \ref{sec:bip}, we first review the combinatorics (including the counting) of special unipotent representations from our earlier paper \cite{BMSZa}. This includes the notion of painted bipartitions as well as the descent of a painted bipartition. The key combinatorial object is an  extended set $\mathrm{\PBPes}(\check \CO)$ of painted bipartitions attached to $\check \CO$. This will be our parameter set for special unipotent representations attached to $\check \CO$. Given 
$\uptau \in \mathrm{\PBPes}(\check \CO)$, we attach a real classical group $G_{\uptau}$ of type $B$, $C$, $D$, and we define inductively a special unipotent representation $\pi_{\uptau}$ of $G_{\uptau}$ via the method of theta lifting and in terms of $\pi_{\uptau'}$, where $\uptau'$ is the descent of $\uptau $. We state our main result (\Cref{thm100}), which says that $\pi_\uptau$ is irreducible, unitarizable and special unipotent attached to $\check \CO$, and special unipotent representations of a real classical group $G$ attached to $\check \CO$ 
are exhausted by such $\pi_\uptau$'s with $G_{\uptau}=G$ (up to the determinant twist, if $\star\in \{B,D\}$).  

Section \ref{sec:bipGeometry} is devoted to the geometry which we associate to $\uptau $. At its center is a certain double fibration of moment maps which connects the geometries associated to
$G_{\uptau}$ and $G_{\uptau'}$, and which permits us to geometrically lift (from $G_{\uptau'}$ to $G_{\uptau}$) nilpotent orbits as well as equivariant algebraic vector bundles over these orbits. We call this construction the geometric theta lifting. Given 
$\uptau \in \mathrm{\PBPes}(\check \CO)$, by mimicking the inductive construction of $\pi_\uptau$ (by geometric theta lifting, in place of theta lifting), we attach an object in a certain Grothendieck group, which we refer to as the associated cycle of $\uptau$. We state the key properties of associated cycles, whose proofs are delayed until  \Cref{sec:ACC}. 

In Section \ref{sec:main}, we provide a blueprint for the proof of \Cref{thm100}. There are two main ingredients for the most involved part, which is exhaustion of special unipotent representations by our construction. First we will determine the associated cycle of $\pi_{\uptau}$. To be more precise, we will show that the associated cycle of $\pi_{\uptau}$ matches with the associated cycle of $\uptau$ (this is in Theorem \ref{thmpitau}, proved in Section \ref{sec:equac}). 
Second, some remarkable properties of the associated cycles (in Section \ref{sec:bipGeometry}) allow us to distinguish various $\pi_{\uptau}$'s by an inductive argument and to conclude the exhaustion from the counting result of \cite{BMSZa}. Note that we only require a weaker form of the counting result (i.e. an inequality) to establish the exhaustion. 

In Sections \ref{sec:Integrals} to \ref{sec:AC}, we develop the key representation theory tools of the paper. Section \ref{sec:Integrals} is devoted to the model of theta lifting by matrix coefficient integrals, as well as a general criterion on unitarity preservation in this context. In Section \ref{sec:dtheta}, we investigate double theta lifting and its connection with degenerate principal series representations. In  Section \ref{sec:AC}, we prove ``upper bounds'' on the associated cycles, using largely geometric arguments. 

In Section \ref{sec:nilmis}, we explicitly describe geometric theta lifts and weak associated cycles of degenerate principal series representations combinatorially. Roughly speaking, a combination of the results of Section \ref{sec:dtheta} and Section \ref{sec:nilmis} has the power to force ``lower bounds'' on the associated cycles. In Section \ref{sec:equac}, we match the associated cycle of $\pi_{\uptau}$ with the associated cycle of $\uptau$, as a consequence of the results of Sections \ref{sec:dtheta} to  \ref{sec:nilmis}.  
  
In Section \ref{sec:ACC}, we prove key properties of the associated cycles stated in 
Section \ref{sec:bipGeometry}. All proofs are combinatorial in nature as they depend on the detailed combinatorial description of the descent map (in Section 
\ref{sec:bip}) as well as geometric theta lifts (in 
Section \ref{sec:nilmis}).  (We have organized our presentation in this way in the hope that the exposition is as clear and direct as possible.) Section \ref{modify} is devoted to some necessary modifications, due to our preference for orthogonal groups in the current paper over special orthogonal groups in our previous paper \cite{BMSZa}, and for the case of unitary groups (which is easier compared to classical groups of type $B,C,D$).

In  Appendix  \ref{sec:orbitinduce}, we compute inductions of certain nilpotent orbits, which are used in the proof of Proposition \ref{lem:ind00} (on the weak associated cycles of certain induced representations).  

\section{From combinatorial parameters to special unipotent representations}\label{sec:bip}

From this section until \Cref{sec:ACC}, we will consider real classical groups of type $B,C,D$, namely we  assume that $\star\in \{B,C,D,\widetilde {C}, C^*, D^*\}$. The nilpotent orbit $\check \CO$ is always assumed to have good parity. 

In \Cref{sec:PB,subsec:countbip,sec:desc}, we review the combinatorial aspect of special unipotent representations established in \cite{BMSZa}.
In \Cref{subsec:comTOrep}, we define special unipotent representations attached to the combinatorial parameters, and state the main result of the paper.

\subsection{Combinatorial construct: painted bipartitions}
\label{sec:PB}

For a Young diagram $\imath$, write
\[
 \mathbf r_1(\imath)\geq \mathbf r_2(\imath)\geq \mathbf r_3(\imath)\geq \cdots
\]
for its row lengths, and similarly,
write
\[
 \mathbf c_1(\imath)\geq \mathbf c_2(\imath)\geq \mathbf c_3(\imath)\geq \cdots
\]
for its column lengths.
We introduce the set $\BOX(\imath)$ of boxes of $\imath$ as the following subset
of $\bN^+\times \bN^+$ ($\bN^+$ denotes the set of positive integers):
\begin{equation}\label{eq:BOX}
\BOX(\imath):=\Set{(i,j)\in\bN^+\times \bN^+ \,:\, j\leq \bfrr_i(\imath)}.
\end{equation}
In our convention, $(i,j)\in \BOX(\imath)$ represents the box of $\imath$ at the position $(i,j)$, namely in the $i$th row counted from top to bottom and $j$th column counted from left to right. In particular, when $\imath$ is nonempty, $(1,1)\in \BOX(\imath)$ represents the top left corner box of $\imath$. 
A subset of $\bN^+\times \bN^+$ of the form \eqref{eq:BOX} is also said to constitute the Young diagram $\imath$. Put
\begin{equation}\label{eq:BOX22}
\abs{\imath}:=\#(\BOX(\imath))=\sum_{i=1}^\infty \mathbf r_i(\imath)=\sum_{i=1}^\infty \mathbf c_i(\imath)\qquad (\textrm{$\#$ indicates the cardinality of a finite set}).
\end{equation}

We also introduce five symbols $\bullet$, $s$, $r$, $c$, and $d$, and make the following definition.

\begin{defn}
A painting on a Young diagram $\imath$ is an assignment (we
    place a symbol in each box)
  \[
    \mathcal P: \mathrm{Box}(\imath) \rightarrow \{\bullet, s, r, c, d \}
  \]
  with the following properties:
  \begin{enumerate}
  \item if we remove the boxes painted with $\{d\}$,
$\{c,d\}$, $\{r,c,d\}$ or $\{s,r,c,d\}$, the remainder still constitutes a Young diagram; 
    \item every row of $\imath$ has at most one box
          painted with $s$ and at most one box painted with $r$;    
    \item every column of $\imath$ has at most one
          box painted with $c$ and at most one box painted with $d$.    
  \end{enumerate}
A painted Young diagram is a pair $(\imath, \CP)$ consisting of a Young diagram $\imath$ and a painting $\CP$ on $\imath$.
\end{defn}

\begin{eg} Suppose that $\imath=\tytb{\ \ ,\  }$, then there are $25+12+6+2=45$ paintings on $\imath$ in total as listed below.
\begin{equation*}\label{eq:sp-nsp.C}
\begin{array}{ll}
   \tytb{\bullet \alpha ,\beta } \quad \alpha, \beta\in \{\bullet, s,r,c,d\} \qquad \qquad  \qquad  & \tytb{s \alpha ,\beta } \quad \alpha\in \{r,c,d\}, \beta\in \{s,r,c,d\} \medskip \medskip \\
     \tytb{r \alpha ,\beta } \quad \alpha\in \{c,d\}, \beta\in \{r,c,d\} \qquad \qquad  \qquad
   &  \tytb{c \alpha , d } \quad  \alpha\in \{c,d\}
     \end{array}
  \end{equation*}

  \end{eg}

When no confusion is possible, we write $\alpha\times \beta\times \gamma$ for a triple $(\alpha, \beta, \gamma)$.   We introduce two more labels $B^+$ and $B^-$ (over the label $B$, to be  used when counting the signature; see \eqref{ptqt}), and make the following definition.
 \begin{defn}\label{defpbp0}
 A painted bipartition is a triple
  $\tau=(\imath, \CP)\times (\jmath, \cQ)\times \gamma$, where $(\imath, \CP)$
  and $ (\jmath, \mathcal Q)$ are painted Young diagrams, and
  $\gamma \in \{B^+,B^-, C,D,\widetilde {C}, C^*, D^*\}$, subject to the
  following conditions:
  \begin{enumerate}
 \item $\CP$ and $\mathcal Q$ have the identical set of boxes painted with $\bullet$;    
    \item the symbols of $\CP$ are in 
          \[
          \left\{
          \begin{array}{ll}
            \{\bullet, c\}, &\hbox{if $\gamma=B^+$ or $B^-$}; \smallskip\\
            \{\bullet,  r, c,d\}, &\hbox{if $\gamma=C$}; \smallskip\\
            \{\bullet, s, r, c,d\}, &\hbox{if $\gamma=D$}; \smallskip\\
            \{\bullet, s, c\}, &\hbox{if $\gamma=\widetilde{ C}$}; \smallskip \\
            \{\bullet\}, &\hbox{if $\gamma=C^*$}; \smallskip \\
            \{\bullet, s\}, &\hbox{if $\gamma=D^*$},\\
          \end{array}
          \right.
          \]
    \item the symbols of $\mathcal Q$ are in           \[
          \left\{
          \begin{array}{ll}
            \{\bullet, s, r, d\}, &\hbox{if $\gamma=B^+$ or $B^-$}; \smallskip\\
            \{\bullet, s\}, &\hbox{if $\gamma=C$}; \smallskip\\
            \{\bullet\}, &\hbox{if $\gamma=D$}; \smallskip\\
            \{\bullet, r, d\}, &\hbox{if $\gamma=\widetilde{ C}$}; \smallskip\\
            \{\bullet, s,r\}, &\hbox{if $\gamma=C^*$}; \smallskip \\
            \{\bullet, r\}, &\hbox{if $\gamma=D^*$}.
          \end{array}
          \right.
          \]

  \end{enumerate}
  \end{defn}

 For any painted bipartition $\tau$ as in Definition \ref{defpbp0}, we write
\[
  \imath_\tau:=\imath,\ \cP_\tau:=\cP,\  \jmath_\tau:=\jmath,\  \cQ_\tau:=\cQ,\ \gamma_\tau:=\gamma,
\]
and
\[
  \star_\tau:= \left\{
     \begin{array}{ll}
         B, &\hbox{if $\gamma=B^+$ or $B^-$}; \smallskip\\
            \gamma, & \hbox{otherwise}.           \end{array}
   \right.
  \]

We further attach some objects to $\tau$ in what follows:
 \[
  \abs{\tau}, \   (p_\tau, \ q_\tau), \ G_\tau, \ \dim \tau, \ \varepsilon_\tau .
 \]

\smallskip

\noindent $\abs{\tau}$ : This is the natural number \[
  \abs{\tau}:=\abs{\imath}+\abs{\jmath}.
\]

 \smallskip
 \smallskip

  \noindent $(p_{\tau}, q_{\tau})$ :
  This is a pair of natural numbers given by the following recipe.
 \begin{enumerate}[label=(\alph*)]
  \item
  If $\star_\tau\in \{B, D, C^*\}$, $(p_\tau, q_\tau)$ is given by counting  the various symbols appearing in $(\imath, \CP)$, $(\jmath, \CQ)$ and $\{\gamma\}$ :
  \begin{equation}\label{ptqt}
  \left\{
     \begin{array}{l}
    p_\tau :=( \# \bullet)+ 2 (\# r) +(\# c )+ (\# d) + (\# B^+);\smallskip\\
    q_\tau :=( \# \bullet)+ 2 (\# s) + (\# c) + (\# d) + (\# B^-).\\
    \end{array}
    \right.
\end{equation}
Here
\[
\#\bullet:=\#(\cP^{-1}(\bullet))+\#(\cQ^{-1}(\bullet)),
\]
and the other terms are similarly defined.
\item
If $\star_\tau\in \{C, \widetilde C, D^*\}$,  $p_\tau:=q_\tau:=\abs{\tau}$.
\end{enumerate}
\smallskip

 \smallskip

  \noindent $G_{\tau}$: This is a classical group given by
  \begin{equation}\label{def:Gt}
 G_\tau:= \left\{
     \begin{array}{ll}
         \oO(p_\tau, q_\tau), &\hbox{if $\star_\tau=B$ or $D$}; \smallskip\\
            \Sp_{2\abs{\tau}}(\R), &\hbox{if $\star_\tau=C$}; \smallskip\\
           \widetilde{\Sp}_{2\abs{\tau}}(\R), &\hbox{if $\star_\tau=\widetilde{ C}$}; \smallskip \\
        \Sp(\frac{p_\tau}{2}, \frac{q_\tau}{2}), &\hbox{if $\star_\tau=C^*$}; \smallskip \\
          \oO^*(2\abs{\tau}), &\hbox{if $\star_\tau=D^*$}.\\
            \end{array}
   \right.
\end{equation}

\smallskip

\smallskip
 
\noindent $\dim \tau$:
This is the dimension of the standard representation of the complexification of $G_\tau$, or equivalently,
 \begin{equation}\label{def:dimtau}
 \dim \tau:= \left\{
     \begin{array}{ll}
          2\abs{ \tau}+1, &\hbox{if $\star_\tau=B$}; \medskip\\
         2 \abs{\tau}, &\hbox{otherwise}.
            \end{array}
   \right.
 \end{equation}
 
 \smallskip
 
 \smallskip

 \noindent $\varepsilon_\tau$:
This is the element in $\Z/2\Z$ such that
\be\label{epsilontau}
  \varepsilon_\tau=0\Leftrightarrow  \textrm{the symbol $d$ occurs in the first column of $( \imath, \cP)$ or $(\jmath, \cQ)$}.
\ee

\smallskip

The triple $\mathsf s_{\tau}:=(\star_{\tau}, p_{\tau},q_{\tau})\in  \{B,C,D, \widetilde C, C^*, D^*\}\times \bN\times \bN$ will also be referred to as the classical signature attached to $\tau$. See \Cref{classical} for the notions of classical spaces and classical signatures.  
\smallskip

\begin{eg} Suppose that
\[
\tau= \tytb{\bullet c ,\bullet , c,\none } \times \tytb{\bullet s r ,\bullet ,r, r }\times B^+.
\]
Then
\[
\begin{cases}
\abs{\tau}=10;\\
 p_\tau=4+6+2+0+1=13; \quad\\
 q_\tau=4+2+2+0+0=8; \quad \\
 G_\tau=\oO(13,8);\quad \\
 \dim \tau=21;\\
 \varepsilon_\tau=1.
 \end{cases}
 \]

\end{eg}

\subsection{Counting special unipotent representations by painted bipartitions}\label{subsec:countbip}

\begin{defn}
 A $\star$-pair is a pair  $(i,i+1)$ of consecutive positive integers such that
\[
   \left\{
     \begin{array}{ll}
      i\textrm{ is odd}, \quad &\textrm{if $\star\in\{C, \widetilde{C}, C^*\}$};  \\
      i \textrm{ is even}, \quad &\textrm{if $\star\in\{B, D, D^*\}$}. \\
       \end{array}
   \right.
\]
A $\star$-pair   $(i,i+1)$ is said to be
\begin{itemize}
\item
vacant in $\check \CO$, if $\mathbf r_i(\check \CO)=\mathbf r_{i+1}(\check \CO)=0$;
\item
balanced in $\check \CO$,  if  $\mathbf r_i(\check \CO)=\mathbf r_{i+1}(\check \CO)>0$;
\item
tailed in $\check \CO$,  if  $\mathbf r_i(\check \CO)-\mathbf r_{i+1}(\check \CO)$ is positive and odd;
\item
primitive in $\check \CO$, if    $\mathbf r_i(\check \CO)-\mathbf r_{i+1}(\check \CO)$ is positive and even.
\end{itemize}
Denote by $\mathrm{PP}_\star(\check \CO)$ the  set of all $\star$-pairs that are primitive in $\check \CO$.
\end{defn}

\begin{remark} Due to the good parity assumption on $\ckcO$, the condition that   
$\mathbf r_i(\check \CO)-\mathbf r_{i+1}(\check \CO)$ is positive and odd, is equivalent to the condition that $\mathbf r_i(\check \CO)$ is odd and $\mathbf r_{i+1}(\check \CO)=0$. In particular the notion of tailed in $\check \CO$ is relevant only when $\star\in\{C,C^*,D,D^*\}$. 
\end{remark}

\begin{remark} When $\star\neq \wtC$, the power set of $\mathrm{PP}_\star(\check \CO)$ gives another description of Lusztig's canonical quotient attached to $\ckcO$. The set $\mathrm{PP}_{\star}(\check\CO)$ appears implicitly in
\cite{So}*{Section~5}.
\end{remark}

We attach to $\check \CO$ a pair of Young diagrams
\begin{equation}\label{ijO}
(\imath_{\check \CO}, \jmath_{\check \CO}):=(\imath_\star(\check \CO), \jmath_\star(\check \CO)),
\end{equation}
 as follows.

\medskip

\noindent {\bf The case when $\star=B$.} In this case, the nilpotent orbit $\check \CO$ is of type $C$, and has even sized rows
  only. Define
 \[
   \mathbf c_{1}(\jmath_{\check \CO})=\frac{\mathbf r_1(\check \CO)}{2},
\]
and for all $i\geq 1$,
\[
\left (\mathbf c_{i}(\imath_{\check \CO}), \mathbf c_{i+1}(\jmath_{\check \CO})\right )=
            \left (\frac{\mathbf r_{2i}(\check \CO)}{2},  \frac{\mathbf r_{2i+1}(\check \CO)}{2}\right ).
\]

\medskip

\noindent {\bf The case when $\star=\widetilde C$.} In
this case, the nilpotent
orbit $\check \CO$ is of type $C$, and has even sized rows only. Define, for all $i\geq 1$,
\[
(\mathbf c_{i}(\imath_{\check \CO}), \mathbf c_{i}(\jmath_{\check \CO}))=
           \left (\frac{\mathbf r_{2i-1}(\check \CO)}{2},  \frac{\mathbf r_{2i}(\check \CO)}{2}\right).
\]

\medskip

\noindent {\bf The case when $\star \in \{ C,C^*\}$.} In this case, the nilpotent orbit $\check \CO$ is of type $B$, and has odd sized  rows only. Define, for all
$i\geq 1$,
\[
(\mathbf c_{i}(\jmath_{\check \CO}), \mathbf c_{i}(\imath_{\check \CO}))=
   \left\{
     \begin{array}{ll}
        (0,  0), &\hbox{if $(2i-1, 2i)$ is vacant  in $\check \CO$};\smallskip\\
        (\frac{\mathbf r_{2i-1}(\check \CO)-1}{2},  0), & \hbox{if $(2i-1, 2i)$ is tailed in $\check \CO$};\smallskip\\
                  (\frac{\mathbf r_{2i-1}(\check \CO)-1}{2},  \frac{\mathbf r_{2i}(\check \CO)+1}{2}), &\hbox{otherwise}.\\
            \end{array}
   \right.
\]
\medskip

\noindent {\bf The case when $\star\in \{D,D^*\}$.} In this case, the nilpotent orbit $\check \CO$ is of type $D$, and has odd sized rows only. Define
 \[
   \mathbf c_{1}(\imath_{\check \CO})= \left\{
     \begin{array}{ll}
      0,  &\hbox{if $\mathbf r_1(\check \CO)=0$}; \smallskip\\
       \frac{\mathbf r_1(\check \CO)+1}{2},   &\hbox{if $\mathbf r_1(\check \CO)>0$},\\
            \end{array}
   \right.
 \]
and for all $i\geq 1$,
\[
(\mathbf c_{i}(\jmath_{\check \CO}), \mathbf c_{i+1}(\imath_{\check \CO}))=
   \left\{
     \begin{array}{ll}
        (0,  0), &\hbox{if $(2i, 2i+1)$ is vacant in $\check \CO$};\smallskip\\
      \left  (\frac{\mathbf r_{2i}(\check \CO)-1}{2},  0\right ), & \hbox{if $(2i, 2i+1)$ is tailed in $\check \CO$};\smallskip\\
                \left  (\frac{\mathbf r_{2i}(\check \CO)-1}{2},  \frac{\mathbf r_{2i+1}(\check \CO)+1}{2}\right ), &\hbox{otherwise}.\\
            \end{array}
   \right.
\]

\begin{remark} The pair $(\imath_{\check \CO}, \jmath_{\check \CO})$ corresponds to an irreducible representation of the  Weyl group of $G_\C$ (which is isomorphic to the Weyl group of type $B_n$ in all cases), as in \cite[Section 11.4]{Carter}.  When $\star\neq \wtC$, this irreducible representation is special in the sense of Lusztig (\cite{Lsp}) and it corresponds under the Springer correspondence to the Barbasch-Vogan dual $\CO$ of $\check \CO$. See \cite[Section 8]{BMSZa}.

\end{remark}

Define
\[
\mathrm{PBP}_\star(\check \CO):=\Set{
\tau\textrm{ is a painted bipartition}  \,:\,  \star_\tau = \star, \ (\imath_\tau,\jmath_\tau) = (\imath_\star(\check \CO), \jmath_\star(\check \CO))
}
\]
and
\[
\mathrm{PBP}_G(\check \CO)\\
:=\Set{
\tau\in \mathrm{PBP}_\star(\check \CO)  \,:\,  G_\tau=G
}.
\]
When $\star\in \{B, D, C^*\}$,  the condition ``$G_{\tau} = G$" amounts to saying that $(p_\tau, q_\tau)=(p,q)$; when  $\star\in \{C, \wtC, D^*\}$, this condition is automatic. 
We also define the following extended parameter set:
\begin{equation}\label{def:PBPes}
\mathrm{PBP}^{\mathrm{ext}}_G(\check \CO):=\begin{cases}
\mathrm{PBP}_G(\check \CO)\times \{\wp\subseteq \mathrm{PP}_\star(\check \CO)\},\quad& \textrm{if }\star\in \{B,C, D, \widetilde C\};\\
\mathrm{PBP}_G(\check \CO)\times \{\emptyset \},\quad& \textrm{if }\star\in \{C^*, D^*\}.
\end{cases}
\end{equation}
Here and in the sequel, $\emptyset$ denotes the empty set (in the usual way), as well as the empty Young diagram or the painted Young diagram whose underlying Young diagram is empty.

Similarly, write 
\[
\mathrm{PBP}^{\mathrm{ext}}_\star(\check \CO):=\begin{cases}
\mathrm{PBP}_\star(\check \CO)\times \{\wp\subseteq \mathrm{PP}_\star(\check \CO)\},\quad& \textrm{if }\star\in \{B,C, D, \widetilde C\};\\
\mathrm{PBP}_\star(\check \CO)\times \{\emptyset \},\quad& \textrm{if }\star\in \{C^*, D^*\}.
\end{cases}
\]
In what follows, we will use $\uptau=(\tau,\wp)$ to denote an element in $\PBPes(\ckcO)$.

The following result is a weak form of \cite[Theorem 2.27]{BMSZa}, which is sufficient for our purpose. Note that in \cite{BMSZa}, we consider special orthogonal groups rather than orthogonal groups. 

\begin{prop}\label{thmcount}
The inequality 
\be\label{count0}
 \#(\Unip_{\ckcO}(G))\leq \begin{cases}
2  \#(\mathrm{PBP}^{\mathrm{ext}}_G(\check \CO)),&\quad  \text{if  $\star\in \{B,D\}$ and $(\star, |\check \CO|)\neq (D, 0)$;}\\
 \#(\mathrm{PBP}^{\mathrm{ext}}_G(\check \CO)),  &\quad  \text{otherwise}\\
\end{cases}
\ee
holds. 
\end{prop}
In fact, \cite[Theorem 2.27]{BMSZa} immediately implies that the equality holds in \eqref{count0} unless (possibly when) $\star=D$. When $\star=D$, we will see from \Cref{thm100} that the  equality still holds in \eqref{count0}. 

\begin{remark} The counting inequality in \Cref{thmcount} is explicated from a counting inequality of irreducible representations of general nature (\cite[Corollary 5.4]{BMSZa}). Unlike the counting equality, it does not depend on knowledge of the intricate relationship of Harish-Chandra cells and double cells (\cite[Section 4.5]{BMSZa}). The results of this paper imply that the inequality in \eqref{count0} is an equality. See \Cref{thmac7} and its proof. 
\end{remark}

\subsection{Descent of a painted bipartition}
\label{sec:desc}
We recall the notion of the descent of a painted bipartition $\tau$ from \cite{BMSZa}*{Section~10}. 
For the current paper, we only need to consider the case when $\tau\in \mathrm{PBP}_\star(\check \CO)$, namely when the pair $(\imath_\tau, \jmath_\tau)$ is attached to a nilpotent orbit as in \eqref{ijO}. 
 We will thus only recall the notion in this special case, which is less involved than the general case.

Define a label
\[
\star':=\widetilde{C}, \ D, \  C, \ B, \ D^*,\  \textrm{ or } \ C^*
\]
respectively if
\[
\star=B,\  C, \ D, \ \widetilde{C}, \ C^*, \ \textrm{ or }\  D^*.
\]
We call $\star'$ the Howe dual of $\star$.

 Define the naive dual descent of $\check \CO$ to be
 \[
  \check \nabla_{\mathrm{naive}}(\check \CO):= \textrm{the Young diagram obtained from $\check \CO$ by removing the first row}.
  \]

   The dual descent of $\check \CO$ is defined to be
  \be\label{duald}
   \check \CO':=\check \nabla(\check \CO):=\check \nabla_\star( \check \CO):=\begin{cases}
   \tytb{\   }\, , \quad& \textrm{if $\star\in \{D, D^*\}$ and $|\check \CO|=0$};\smallskip\\
   \check \nabla_{\mathrm{naive}}(\check \CO), \quad & \textrm{otherwise},
    \end{cases}
  \ee
  where $\tytb{\   }$ denotes the Young diagram with exactly one box.

For a Young diagram $\imath$, its naive descent, which is denoted by $\nabla_\mathrm{naive}(\imath)$, is defined to be the Young diagram obtained from $\imath$ by removing the first column. 

\def\bipartl{\mathrm{bi\cP_L}}
\def\bipartr{\mathrm{bi\cP_R}}
\def\dsdiagl{\mathrm{DS_L}}
\def\dsdiagr{\mathrm{DS_R}}
\def\DDl{\eDD_\mathrm{L}}
\def\DDr{\eDD_\mathrm{R}}

We first define the naive descent of a painted bipartition. Let $\tau=(\imath,\cP)\times (\jmath,\cQ)\times \gamma $ be a  painted bipartition such that $\star_\tau=\star$. Put
  \begin{equation} \label{eq:def.alphap}
    \gamma '=\begin{cases} B^+,
  & \textrm{if $\gamma =\widetilde{C}$ and $c$ does not occur in the first column of $(\imath,\cP)$}; \smallskip \\
  B^-,
  & \textrm{if $\gamma =\widetilde{C}$ and  $c$ occurs in the first column of $(\imath,\cP)$}; \smallskip \\
  \star', & \textrm{if $\gamma \neq \widetilde C$}.
  \end{cases}
  \end{equation}

\begin{lem}(\cite[Lemmas 10.4 and 10.5]{BMSZa}) 
\label{lem:DDn}

\noindent (a) If $\star \in \set{B,C,C^*}$, then there is a unique painted bipartition of the form $\tau'_\mathrm{naive}= (\imath',\cP')\times (\jmath',\cQ')\times \gamma '$ with the following properties:
  \begin{itemize}
        \item $
   (\imath',\jmath')= (\imath,\DD_\mathrm{naive}(\jmath)); \smallskip
   $
   \item for all $(i,j)\in \BOX(\imath')$,
   \[
     \cP'(i,j)=\begin{cases}
    \bullet \textrm{ or } s,&\textrm{ if  $\ \cP(i,j)\in \{\bullet, s\}$;} \smallskip \\
  \cP(i,j),& \textrm{ if $\ \cP(i,j)\notin \{\bullet, s\}$};\end{cases}
   \]
   \item for all $(i,j)\in \BOX(\jmath')$,
   \[
     \cQ'(i,j)=\begin{cases}
    \bullet \textrm{ or } s,&\textrm{ if  $\ \cQ(i,j+1)\in \{\bullet, s\}$;} \smallskip \\
  \cQ(i,j+1), & \textrm{ if $\ \cQ(i,j+1)\notin \{\bullet, s\}$}.  \end{cases}
   \]
    \end{itemize}

\noindent (b) If $\star \in \set{ \widetilde C, D,D^*}$, 
 then there is a unique painted bipartition of the form 
 $\tau'_\mathrm{naive}= (\imath',\cP')\times (\jmath',\cQ')\times \gamma'$ with the following properties:
 \begin{itemize}
 \item $
 (\imath',\jmath')= (\DD_\mathrm{naive}(\imath),\jmath); \smallskip
 $
 \item for all $(i,j)\in \BOX(\imath')$,
 \[
 \cP'(i,j)=\begin{cases}
 \bullet \textrm{ or } s,&\textrm{ if  $\ \cP(i,j+1)\in \{\bullet, s\}$;} \smallskip \\
 \cP(i,j+1),& \textrm{ if $\ \cP(i,j+1)\notin \{\bullet, s\}$};\end{cases}
 \]
 \item for all $(i,j)\in \BOX(\jmath')$,
 \[
 \cQ'(i,j)=\begin{cases}
 \bullet \textrm{ or } s,&\textrm{ if  $\ \cQ(i,j)\in \{\bullet, s\}$;} \smallskip \\
 \cQ(i,j), & \textrm{ if $\ \cQ(i,j)\notin \{\bullet, s\}$}.  \end{cases}
 \]
 \end{itemize}
 \end{lem}

 \begin{defn}
 In the notation of  \Cref{lem:DDn}, 
 we call $\tau'_\mathrm{naive}$ the naive descent of $\tau$, to be denoted by $\DDn(\tau)$.
 \end{defn}

\begin{eg}  
 Suppose that the nonzero row lengths of $\check \CO$ are $8,6,6,6,4,4,2$. Then 
    \[
     \uptau = \ytb{\bullet\bullet\bullet {c},\bullet {s} {c},{s} {c},{c}}
    \times \ytb{\bullet\bullet\bullet ,\bullet {r} {d},{d}{d}, \none}
    \times \widetilde C \in \PBP_{\widetilde{\Sp}_{36}(\bR)}(\check \CO).     
    \]
  We have that  
   \[
    \nabla_{\mathrm{naive}}(\uptau) =\ytb{\bullet\bullet{c} ,\bullet{c},{c} }
    \times  \ytb{\bullet\bullet {s} ,\bullet {r} {d},{d}{d}}\times B^- \in \PBP_{\rO(14,15)}(\check \CO').
    \]
\end{eg}

Suppose that
$
\tau=(\imath,\cP)\times(\jmath,\cQ)\times \gamma \in  \mathrm{PBP}_\star(\check \CO)
$
and write
\[
  \tau'_{\mathrm{naive}}=(\imath', \cP'_{\mathrm{naive}})\times (\jmath', \cQ'_{\mathrm{naive}})\times \gamma '
\]
for the naive descent of $\tau$. It is clearly an element of $  \mathrm{PBP}_{\star'}(\check \CO')$, where $\check \CO'=\check \nabla(\check \CO)$.

The naive descent map falls short, though only slightly, of certain desirable properties as described in \Cref{prop:DD.BDinj}  and \Cref{lem:delta}. These  properties, in turn, help ensure that the associated cycle of a painted bipartition (to be defined in \eqref{eq:AC}) enjoys several remarkable features (see \Cref{subsec:kpac}). It turns out that there are only two corner cases where one must modify a single symbol after the naive descent, in order for the resulting descent map (to be defined in Definition \ref{defdescent}) to possess the aforementioned desirable properties. The corresponding two lemmas (Lemmas \ref{descb1} and \ref{descd2}) are straightforward to verify, so we omit their proofs; instead, we provide an example for each case.

Recall that for a Young diagram $\imath$, its column lengths are written as
\[
 \mathbf c_1(\imath)\geq \mathbf c_2(\imath)\geq \mathbf c_3(\imath)\geq \cdots .
\]

\begin{lem}\label{descb1}
  Suppose that
\[  \begin{cases}
 \gamma = B^+; & \\
 \bfrr_2(\ckcO)>0; & \\
 \cQ(\mathbf c_1(\imath),1)\in \set{r,d}.
\end{cases}
\]
 Then there is a unique element in $\mathrm{PBP}_{\star'}(\check \CO')$ of the form
  \[
      \tau'=(\imath', \cP')\times (\jmath', \cQ'_{\mathrm{naive}} )\times \gamma '
  \]
 such that
\[
\cP'(i,j) = \begin{cases}
  s, & \ \text{ if $(i,j) = (\mathbf c_1(\imath'),1)$;}\\
  \cP'_{\mathrm{naive}}(i,j), & \ \text{ otherwise}
  \end{cases}
  \]
 for all $(i,j)\in \BOX(\imath')$.
\end{lem}

\begin{eg} Suppose that  the nonzero row lengths of $\check \CO$ are $4,4,4,2$. Then 
 \[
 \tau= \ytb{\bullet c, c} \times \ytb{\bullet r, rd}\times
  B^+ \in \PBP_{\rO(10,5)}(\ckcO).
 \]
We have that
\[
 \tau'_{\mathrm{naive}}= \ytb{s c, c} \times \ytb{r, d}\times
  \widetilde C\qquad\textrm{and}\qquad \tau'= \ytb{s c, s} \times \ytb{r, d}\times
  \widetilde C,
 \]
both in $\PBP_{\widetilde \Sp_{10}(\bR)}(\check \CO')$. 
\end{eg}

\begin{lem}\label{descd2}
  Suppose that
  \[  \begin{cases}
 \gamma = D; & \smallskip \\
\mathbf r_2(\check \CO)=\mathbf r_3(\check \CO)>0;  &\smallskip\\
\left(\cP(\mathbf c_2(\imath),1), \textrm{\vrule width 0pt  height 0.9em  \relax} \cP(\mathbf c_2(\imath),2)\right)=(r,c); &\smallskip\\
 \cP(\mathbf c_1(\imath),1)\in \set{r,d}.
\end{cases}
\]
 Then there is a unique element in $\mathrm{PBP}_{\star'}(\check \CO')$ of the form
  \[
      \tau'=(\imath', \cP')\times (\jmath', \cQ'_{\mathrm{naive}} )\times \gamma '
  \]
  such that 
  \[
\cP'(i,j) = \begin{cases}
  r, & \ \text{ if } \ (i,j) = (\mathbf c_1(\imath'),1); \\
  \cP'_{\mathrm{naive}}(i,j), & \ \text{ otherwise}
\end{cases}
\]
 for all $(i,j)\in \BOX(\imath')$.
\end{lem}

\begin{eg} Suppose that the nonzero row lengths of $\check \CO$ are $7,7,7,3$. 
Then
 \[
 \tau= \ytb{\bullet\bullet, \bullet s, \bullet s, r c} \times \ytb{\bullet\bullet,\bullet,\bullet, \none }\times
  D \in \PBP_{\rO(11,13)}(\ckcO).
 \]
We have that  
\[
 \tau'_{\mathrm{naive}}=\ytb{\bullet, \bullet , \bullet ,  c} \times \ytb{\bullet s,\bullet,\bullet, \none } \times
  C\qquad\textrm{and}\qquad \tau'=\ytb{\bullet, \bullet , \bullet ,  r} \times \ytb{\bullet s,\bullet,\bullet, \none } \times
  C,
 \]
both in $\PBP_{\Sp_{16}(\bR)}(\check \CO')$.
\end{eg}

We now define the descent of a painted bipartition in $\mathrm{PBP}_\star(\check \CO)$. 

\begin{defn}\label{defdescent}
The descent of $\tau$ is defined to be
\[
  \nabla(\tau):= \begin{cases}
  \tau', & \ \text{ if either of the conditions of Lemma \ref{descb1}  or Lemma \ref{descd2} holds}; \\
  \tau'_\mathrm{naive}, 
  & \ \text{ otherwise},
\end{cases}
\]
which is an element of $  \mathrm{PBP}_{\star'}(\check \CO')$.
Here $\tau'$ is as in Lemma  \ref{descb1} or Lemma \ref{descd2}.
\end{defn}
To summarize, we have by now a well-defined descent map
\begin{equation}\label{def:dspbp}
\nabla: \mathrm{PBP}_{\star}(\check \CO)\rightarrow \mathrm{PBP}_{\star'}(\check \CO').
\end{equation}

A key property of the descent map is the following injectivity result. 

\begin{prop}\label{prop:DD.BDinj}{(\cite{BMSZa}*{Proposition 10.8 and Corollary~10.10})}

\noindent (a) 
    If $\star \in  \set{C,\wtC,D^*}$, then the map
\begin{equation}
  \begin{array}{rcl}
  \DD:  \PBP_\star(\ckcO)&\rightarrow&
   \PBP_{\star'}(\ckcOp)
   \end{array}
\end{equation}
is injective.

\noindent (b) 
If $\star \in \set{B, D,C^*}$, then the map
\begin{equation}
  \begin{array}{rcl}
   \PBP_\star(\ckcO)&\rightarrow&
   \PBP_{\star'}(\ckcOp)\times \BN\times \bN\times \Z/2\Z, \smallskip\\
   \tau& \mapsto & (\DD(\tau), p_\tau, q_\tau, \varepsilon_\tau)
   \end{array}
\end{equation}
is injective.

\end{prop}

\subsection{From combinatorial parameters to representations}
\label{subsec:comTOrep}
We retain the setting and notation of Section \ref{sec:desc}. Recall that $\star\in \{ B, C,  D, \widetilde{C},  C^*, D^*\}$, and $\check \CO$ is a $\check G$-orbit in $\mathrm{Nil}(\check \g)$ that has good parity.
Then the dual descent $\check \CO'$ of $\check \CO$  has good parity with respect to $\star'$, where $\star'$ is the Howe dual of $\star$.

Recall the descent map from \Cref{sec:desc}: 
 \[
   \nabla:  \mathrm{PBP}_\star(\check \CO)\rightarrow  \mathrm{PBP}_{\star'}(\check \CO').
 \]
Clearly we have
\[
 \mathrm{PP}_{\star'}(\check \CO')=\{(i,i+1)\,:\, i\in \bN^+, \, (i+1, i+2)\in \mathrm{PP}_{\star}(\check \CO) \}.
\]
Define the dual descent of a subset $\wp\subset \mathrm{PP}_{\star}(\check \CO)$ to be the subset
\begin{equation}\label{eq:DD.wp}
  \wp':=\ckDD(\wp):=\{(i,i+1)\,:\, i\in \bN^+, \, (i+1, i+2)\in \wp \}\subset \mathrm{PP}_{\star'}(\check \CO').
\end{equation}

Let  $\uptau = (\tau,\wp)\in \mathrm{\PBPes}(\check \CO)$.  We define its  descent to be the element
 \[
  \uptau' := (\tau',\wp'):=\nabla(\uptau):= (\DD(\tau), \ckDD(\wp))\in \mathrm{\PBP}^{\mathrm{ext}}_{\star'}(\check \CO').
 \]

We now come to the basic framework of theta lifting (\cite{Howe79, Howe89}). 

Let $(W_{\tau, \tau'}, \la \,\cdot\,,\,\cdot\,\ra_{\tau, \tau'})$ be a real symplectic space of  dimension $\dim \tau\cdot \dim \tau'$. As usual, there are continuous homomorphisms $G_\tau\rightarrow \Sp(W_{\tau, \tau'})$ and $G_{\tau'}\rightarrow \Sp(W_{\tau, \tau'})$ whose images form a reductive dual pair in $\Sp(W_{\tau, \tau'})$. We form the semidirect product
   \[
   J_{\tau, \tau'}:=(G_\tau\times G_{\tau'})\ltimes \oH(W_{\tau, \tau'}),
   \]
   where
  \[
  \oH(W_{\tau, \tau'}):=W_{\tau, \tau'}\times \R
  \]
  is the Heisenberg group with group multiplication
  \[
  (w,t)(w',t'):=(w+w', t+t'+\la w,w'\ra_{\tau, \tau'}), \quad
  w,w'\in W_{\tau, \tau'},\,\,t,t'\in \R.
 \]

 Let $\omega_{\tau, \tau'}$ be a suitably normalized smooth oscillator representation of $J_{\tau, \tau'}$ such that every $t\in \R\subset J_{\tau, \tau'}$ acts on it through the scalar multiplication by $e^{\sqrt{-1}\, t}$. See Section \ref{secoscil} for details. 
  For any Casselman-Wallach representation $\pi'$ of $G_{\tau'}$, write
 \[
   \check \Theta_{\tau'}^{\tau}(\pi'):=(\omega_{\tau, \tau'}\widehat \otimes \pi')_{G_{\tau'}} \qquad (\textrm{the Hausdorff coinvariant space}),
 \]
 where $\widehat \otimes$ indicates the complete projective tensor product. This representation is clearly Fr\'echet, smooth, and of moderate growth, and so a Casselman-Wallach representation by the fundamental result of Howe (\cite{Howe89}).

For notational convenience, we shall also write $G_{\uptau}:=G_{\tau}$ and $G_{\uptau'}:=G_{\tau'}$.

 In what follows we will attach a representation $\pi_{\uptau}$ of $G_{\uptau}$ by the method of theta lifting. For the initial case when $|\check \CO|=0$,
 define
 \[
 \pi_{\uptau}:= 
 \begin{cases}
 \textrm{the determinant character} , &\text{ if $\gamma _\tau=B^-$ so that $G_\tau=\rO(0,1)$;} \smallskip\\
 \textrm{the one-dimensional genuine representation} , &\text{ if $\star=\widetilde C$ so that $G_\tau=\widetilde{\Sp}_0(\R)$;} \smallskip\\
 \textrm{the one-dimensional trivial representation} , &\text{ otherwise.}           
 \end{cases}
 \]
 When $|\check \CO|\neq 0$, we define the representation $\pi_{\uptau}$ of $G_{\uptau}$ by induction on the number of nonempty rows of $\check \CO$:
 \begin{equation}\label{eq:def-pi}
    \pi_{\uptau}:=\left\{
     \begin{array}{ll}          
         \check \Theta_{\tau'}^{\tau}(\pi_{\uptau'})\otimes (1_{p_\tau, q_\tau}^{+,-})^{\varepsilon_{\tau}}, &\hbox{if  $\star=B$ or $D$}; \smallskip \\
         \check \Theta_{\tau'}^{\tau}(\pi_{\uptau'}\otimes \det^{\varepsilon_{\wp}}), &\hbox{if $\star=C$ or $\widetilde C$};  \smallskip\\
              \check \Theta_{\tau'}^{\tau}(\pi_{\uptau'}), &\hbox{if $\star=C^*$ or $D^*$}. \\
            \end{array}
   \right.
 \end{equation}
 Here $1_{p_\tau, q_\tau}^{+,-}$ denotes the character of $\oO(p_\tau, q_\tau)$ whose restriction to $\oO(p_\tau)\times \oO(q_\tau)$ equals $1\otimes \det$ ($1$ stands for the trivial character), $\varepsilon_{\tau}$ is defined in \eqref{epsilontau} and $\varepsilon_{\wp}$ denotes the element in $\Z/2\Z$ such that  \[
 \varepsilon_{\wp}=1\Leftrightarrow  (1,2)\in \wp.
\]
Similar inductive procedures will be adopted when we define the associated cycle and the combinatorial cycle attached to $\uptau$ in 
\eqref{eq:AC} and \eqref{indac}, respectively. The three induction formulas (for representations, associated cycles and combinatorial cycles) are of analogous nature. 

The following is the main result of this paper. 
\begin{thm}\label{thm100} Suppose that $\check \CO$ is a $\check G$-orbit in $\mathrm{Nil}(\check \g)$ that has good parity. 
\begin{enumerate}[wide=0em,label=(\alph*)]
\item For every $\uptau = (\tau,\wp)\in \PBPeg(\check \CO)$, the representation $\pi_{\uptau}$ of $G$  is irreducible, unitarizable,  and special unipotent attached to $\ckcO$.

\item If $\star\in \{B,D\}$ and $(\star, |\check \CO|)\neq (D, 0)$, then the map
\[
\begin{array}{rcl}
\PBPeg(\check \CO)\times \Z/2\Z&\rightarrow &\Unip_{\ckcO}(G),\\
  (\uptau, \epsilon)&\mapsto& \pi_{\uptau}\otimes \det^\epsilon
  \end{array}
\]
is bijective. In all other cases, the map
\[
\begin{array}{rcl}
\PBPeg(\check \CO)&\rightarrow &\Unip_{\ckcO}(G),\\
  \uptau &\mapsto& \pi_{\uptau}
  \end{array}
\]
is bijective.
\end{enumerate}
\end{thm}

As we will see in \Cref{sec:main}, the determination of the associated cycle of $\pi_\uptau$ will play a key role 
in the proof of \Cref{thm100}. The unitarity of $\pi_\uptau$ will follow from the method of matrix coefficient integrals (see \Cref{unitarity}). 

\section{From combinatorial parameters to associated cycles} \label{sec:bipGeometry}

In this section, we introduce the  notion of a classical space of classical signature $\mathsf s$, and define a notion of geometric theta lift associated to two classical spaces of 
classical signatures $\mathsf s$ and $\mathsf s'$ (which are dual in a certain precise sense; see \Cref{secmmap}). By means of geometric theta lift we attach, for each $\uptau$ in the parameter set $\mathrm{\PBPes}(\check \CO)$, 
an element in a certain Grothendieck group which we call the associated cycle of $\uptau$. We state the key properties of associated cycles which will be used in the proof of our main result in \Cref{sec:main} and whose proofs are delayed until \Cref{sec:ACC}.

\subsection{Classical spaces}\label{classical}
Let $\star\in \{B,C,D, \widetilde C, C^*, D^*\}$, as before. Put
\be\label{epsilond}
(\epsilon, \dot \epsilon):=(\epsilon_\star, \dot \epsilon_\star):=\begin{cases}
  (1,1),&\quad \textrm{ if  $\star\in \{B,D\}$;} \smallskip \\
 (-1,-1),&\quad \textrm{ if  $\star\in \{C,\widetilde C\}$;} \smallskip \\
 (-1,1),&\quad \textrm{ if  $\star=C^*$;} \smallskip \\
 (1,-1),&\quad \textrm{ if  $\star=D^*$.}
 \end{cases}
\ee
A classical signature is defined to be a triple  $\mathsf s=(\star, p,q)\in  \{B,C,D, \widetilde C, C^*, D^*\}\times \bN\times \bN$ such that
\[
\begin{cases}
  p+q\textrm{ is odd },&\quad \textrm{ if  $\star=B$;} \smallskip \\
   p+q\textrm{ is even },&\quad \textrm{ if  $\star=D$;} \smallskip \\
 p=q,&\quad \textrm{ if  $\star\in \{C,\widetilde C, D^*\}$;} \smallskip \\
\textrm{both $p$ and $q$ are even} ,&\quad \textrm{ if  $\star=C^*$.}  
 \end{cases}
\]
 Suppose that $\mathsf s=(\star, p,q)$ is a classical  signature in the rest of this section.
Set
\begin{equation}\label{eq:abss}
  \abs{\sfss}:=p+q.
\end{equation}

We omit the proof of the following lemma (\cf \cite[Section~1.3]{Ohta}).
\begin{lem}\label{lem:cartan}
  There is a quadruple  $(V, \la\,,\,\ra, J,L)$ satisfying the following conditions:
  \begin{itemize}
   \item $V$ is a complex vector space of dimension $\abs{\sfss}$;
   \item $\la\,,\,\ra$ is an $\epsilon$-symmetric non-degenerate bilinear form on $V$;
   \item $J: V\rightarrow V$ is a conjugate linear automorphism of $V$ such that $J^2=\epsilon\cdot \dot \epsilon$;
  \item $L: V\rightarrow V$ is a  linear automorphism of $V$ such that $L^2=\dot \epsilon$; 
\item   $\inn{J u}{Jv}=
  \overline{\inn{u}{v}},\quad$  for all $u,v\in V$;
\item  $ \inn{Lu}{Lv}=
  \inn{u}{v}$, $\quad$ for all  $ u,v\in V$;
  \item $LJ =  JL$;
  \item  the Hermitian form $(u,v)\mapsto \inn{Lu}{Jv}$ on
    $V$ is positive definite;
    \item if $\dot \epsilon=1$, then $\dim\{ v\in V \,:\, Lv=v\}=p$ and $\dim\{ v\in V\,:\, Lv=-v\}=q$.
  \end{itemize}
 Moreover, such a quadruple is  unique in the following sense: if  $(V', \la\,,\,\ra', J',L')$ is another quadruple satisfying the analogous conditions, then there is a linear isomorphism $\phi: V\rightarrow V'$ that respectively transforms $ \la\,,\,\ra$, $J$, and $L$ to  $ \la\,,\,\ra'$, $J'$, and $L'$.
 \end{lem}

In the notation of Lemma \ref{lem:cartan}, we  call $(V, \la\,,\,\ra, J,L)$ a classical space of classical signature $\mathsf s$, and denote it by
 $ (V_{\mathsf s}, \la\,,\,\ra_{\mathsf s}, J_{\mathsf s},L_{\mathsf s})$.

We now introduce some general notation for a classical space $(V, \la\,,\,\ra, J,L)$ of classical signature $\mathsf s$, which will be used throughout the article.

Denote by $G_{\C}$ the isometry group of $ (V, \la\,,\,\ra)$, which is a complex orthogonal group if $\epsilon=1$ and a complex symplectic group if $\epsilon=-1$. Respectively denote by $G_{\C}^{J}$ and $G_{\C}^{L}$ the centralizer of $J$ and $L$ in  $G_{\C}$. Then $G_{\C}^{J}$ is a real classical group isomorphic to
\be\label{realg}
 \left\{
     \begin{array}{ll}
         \oO(p, q), &\hbox{if $\star=B$ or $D$}; \smallskip\\
            \Sp_{2p}(\R), &\hbox{if $\star\in \{C,\widetilde C\}$}; \smallskip\\
                   \Sp(\frac{p}{2}, \frac{q}{2}), &\hbox{if $\star=C^*$}; \smallskip \\
          \oO^*(2p), &\hbox{if $\star=D^*$}.\\
            \end{array}
   \right.
\ee
Suppose that 
\be\label{defg}
G= \left\{
     \begin{array}{ll}
       \textrm{the metaplectic double cover of $G_{\C}^{J}$}, \quad &\hbox{if $\star=\widetilde C$}; \smallskip\\
            G_{\C}^{J},  \quad  &\hbox{otherwise}.\\
            \end{array}
   \right.
\ee
Denote by $K$ the inverse image of  $G_{\C}^{L}$  under the natural homomorphism $G\rightarrow G_{\C}$, which is a maximal compact subgroup of $G$. Write $K_{\C}$ for the complexification of  $K$, which is a reductive complex linear algebraic group.

Write
\be\label{comdet}
{\det} : K_{\C}\rightarrow \C^\times
\ee
for the composition of
\[
   K_{\C}\xrightarrow{\textrm{the natural homomorphism}} \GL(V)\xrightarrow{\textrm{the determinant character}}\C^\times.
\]
This is the trivial character unless $\star=B$ or $D$.

For every $\lambda\in \C$, write $V_{\lambda}\subset V$ for the eigenspace of $L$ with eigenvalue $\lambda$. Write
\be\label{comdetlam}
{\det}_\lambda : K_{\C}\rightarrow \C^\times
\ee
for the composition of
\begin{equation*}
   K_{\C}\xrightarrow{\textrm{the natural homomorphism}} \GL(V_{\lambda})\xrightarrow{\textrm{the determinant character}}\C^\times.
\end{equation*}
If $\star=\widetilde C$, then the natural homomorphism  $ K_{\C}\rightarrow \GL(V_{\sqrt{-1}})$ is a double cover, and there is a unique genuine algebraic character 
\be\label{dethalf}
{\det}_{\sqrt{-1}}^{\frac{1}{2}}:  K_{\C}\rightarrow \C^\times 
\ee
whose square equals $\det_{\sqrt{-1}}$.

Recall from the introductory section that $\g$ equals the Lie algebra of $G_{\C}$. 
We have  decompositions
\[
\g=\fkk \oplus \p
\quad \text{ and } \quad
\g^*=\fkk^*\oplus \p^*,
\]
where $\mathfrak k$ is the Lie algebra of $K_{\C}$, and $\p$ is the orthogonal complement of $\mathfrak k$ in $\g$ under the trace form.
We identify $\g^*$ with $\g$ and $\p^*$ with $\p$ by using the trace form. 
Denote by $\Nil(\g)=\Nil(\g^*)$ the set of nilpotent $G_{\C}$-orbits in $\fgg=\g^*$, and by
$\Nil(\p)=\Nil(\p^*)$ the set of nilpotent $K_{\C}$-orbits in $\fpp=\p^*$.

Given a $K_{\C}$-orbit $\sO$ in $\p^*$, write $\CK(\sO)$ for the Grothendieck group of the category of
$K_{\C}$-equivariant algebraic vector bundles  over $\sO$ (throughout the paper we use $\C$ for the coefficient ring in all Grothendieck groups). Denote by $\CK^+(\sO)\subset \CK(\sO)$ the submonoid generated by $K_{\C}$-equivariant algebraic vector bundles over $\sO$. 

Following Vogan (\cite[Definition 7.13]{Vo89}), we make the following definition.
\begin{defn}\label{defaod}
  Let $\sO$ be a $K_\C$-orbit in $\p^*$. An admissible orbit datum over
  $\sO$ is an irreducible $K_\C$-equivariant algebraic vector bundle $\CE$
  over $\sO$ such that
  \begin{itemize}
    \item $\CE_\mathbf e$ is isomorphic to a multiple of
    $(\bigwedge^{\mathrm{top}} \fkk_\mathbf e )^{\frac{1}{2}}$ as a representation of
    $\fkk_\mathbf e$;
    \item $\CE$ is genuine if $\star=\widetilde C$.
  \end{itemize}
  Here $\mathbf e\in \sO$, $\CE_\mathbf e$ is the fibre of $\CE$ at $\mathbf e$, $\fkk_\mathbf e$
  denotes the Lie algebra of the stabilizer of $\mathbf e$ in $K_\C$, and
  $(\bigwedge^{\mathrm{top}} \fkk_\mathbf e)^{\frac{1}{2}}$ is a one-dimensional
  representation of $\fkk_\mathbf e$ whose tensor square is the top degree exterior 
  product $\bigwedge^{\mathrm{top}} \fkk_\mathbf e$.
\end{defn}
Note that in the situation of the classical groups we consider in this article, all admissible
orbit data are line bundles.  
Write
\[
\mathrm{AOD}(\sO)\subset \CK^+(\sO)
\]
for the set of isomorphism classes of admissible orbit data over $\sO$.

Let $\CO$ be a $G_{\C}$-orbit in $\g^*$. 
It is well-known that $\CO\cap \p^*$ has only finitely many $K_{\C}$-orbits (\cite{KoR}). Put
\[
\CK(\CO):=\bigoplus_{\sO\textrm{ is a $K_{\C}$-orbit in $\CO\cap \p^*$ }}\CK(\sO),
\]
and
\[
\CK^+(\CO):=\sum_{\sO\textrm{ is a $K_{\C}$-orbit in $\CO\cap \p^*$ }}\CK^+(\sO).
\]
Put
\begin{equation}\label{eq:AODO}
\mathrm{AOD}(\CO):=\bigsqcup_{\sO\textrm{ is a $K_{\C}$-orbit in $\CO\cap \p^*$ }}\mathrm{AOD}(\sO)\subset \CK^+(\CO).
\end{equation}

Define a partial order $\preceq $ on $ \CK(\CO)$ such that
\be\label{pord}
  \CE_1\preceq \CE_2\Leftrightarrow \CE_2-\CE_1\in \CK^+(\CO) \qquad (\CE_1, \CE_2\in \CK(\CO)).
\ee
For every algebraic character $\chi$ of $K_{\C}$, the twisting map
\[
\CK(\CO)\rightarrow \CK(\CO), \qquad \CE\mapsto \CE\otimes \chi
\]
is obviously defined.

\begin{remark} For all notation defined above, we will put a subscript $\mathsf s$ if there is a need to indicate its dependence on the classical signature $\mathsf s$.
\end{remark}

\subsection{The moment maps and descent of a $K_{\C}$-orbit in $\p$}\label{secmmap}
Recall that $\star'$ is the Howe dual of $\star$. Suppose that $\mathsf s'=(\star', p',q')$ is another classical signature.  Put
\[
  W_{\mathsf s, \mathsf s'} := \Hom_\bC(V_\mathsf s,V_{\mathsf s'}).
\]
Then we have the adjoint map
\[
  W_{\mathsf s, \mathsf s'} \rightarrow W_{\mathsf s', \mathsf s},\qquad \phi\mapsto \phi^*
\]
that is specified by requiring
 \be\label{adjointmap}
    \inn{\phi v}{v'}_{\mathsf s'} = \inn{v}{ \phi^* v'}_{\mathsf s},  \qquad\textrm{for all }v\in
    V_\mathsf s,\, v'\in V_{\mathsf s'}, \, \phi\in   W_{\mathsf s, \mathsf s'}.
  \ee

Define three maps
\[
  \la\,,\,\ra_{\mathsf s, \mathsf s'}:  W_{\mathsf s, \mathsf s'}\times  W_{\mathsf s, \mathsf s'}\rightarrow \C, \qquad(\phi_1,\phi_2)\mapsto \tr(\phi_1^* \phi_2),
\]
\[
J_{\mathsf s, \mathsf s'}: W_{\mathsf s, \mathsf s'}\rightarrow W_{\mathsf s, \mathsf s'}, \qquad \phi\mapsto  J_{\mathsf s'}\circ \phi \circ J_{\mathsf s}^{-1};
\]
and
\[
L_{\mathsf s, \mathsf s'}: W_{\mathsf s, \mathsf s'}\rightarrow W_{\mathsf s, \mathsf s'}, \qquad \phi\mapsto  \dot \epsilon L_{\mathsf s'}\circ \phi \circ L_{\mathsf s}^{-1}.
\]
It is routine to check that
\[
( W_{\mathsf s, \mathsf s'},  \la\,,\,\ra_{\mathsf s, \mathsf s'}, J_{\mathsf s, \mathsf s'}, L_{\mathsf s, \mathsf s'})
\]
is a classical space of classical signature $(C, \frac{\abs{\sfss}\cdot \abs{\sfss'}}{2}, \frac{\abs{\sfss}\cdot \abs{\sfss'}}{2})$.

Write
\begin{equation}\label{def:Xss'}
   W_{\mathsf s, \mathsf s'}=\cX_{\mathsf s, \mathsf s'}\oplus  \cY_{\mathsf s, \mathsf s'},
\end{equation}
where $\cX_{\mathsf s, \mathsf s'}$ and $ \cY_{\mathsf s, \mathsf s'}$ are  the eigenspaces of $L_{\mathsf s, \mathsf s'}$ with eigenvalues $\sqrt{-1}$ and $-\sqrt{-1}$, respectively. Note that $\phi \in \cX_{\mathsf s, \mathsf s'}$ amounts to the condition 
  \[
 \phi L_{\mathsf s}=-\sqrt{-1}\dot \epsilon L_{\mathsf s'}
\phi,\]
or equivalently 
\begin{equation}\label{eq:eigen}
 L_{\mathsf s}\phi ^{*}=\sqrt{-1}\dot \epsilon \phi ^* L_{\mathsf s'}. 
 \end{equation}
Then we have the following two well-defined algebraic maps: (\cite[Section 1]{KP} and \cite[Section 2]{NOZ})
  \be\label{momentmap}
    \xymatrix@R=0em@C=4em{
      \fpp_\mathsf s &\ar[l]_{M_\mathsf s:=M_\mathsf s^{\sfss, \sfss'}} \cX_{\mathsf s, \mathsf s'}\ar[r]^{M_{\sfss'}:=M_{\mathsf s'}^{\sfss, \sfss'}}& \fpp_{\mathsf s'},\\
     \phi^* \phi & \ar@{|->}[l] \phi \ar@{|->}[r] & \phi \phi^*.
    }
  \ee

These two maps $M_\mathsf s$ and $M_{\mathsf s'}$ are called the moment maps. They are both $K_{\mathsf s,\C}\times K_{\mathsf s', \C}$-equivariant. Here  $K_{\mathsf s', \C}$ acts trivially on $\p_\mathsf s$,
 $K_{\mathsf s, \C}$ acts trivially on $\p_{\mathsf s'}$, and all the other actions are the obvious ones. 

Put
\[
  W_{\mathsf s, \mathsf s'}^\circ:=\{\phi\in W_{\mathsf s, \mathsf s'}\,:\, \textrm{the image of $\phi$ is non-degenerate with respect to $\la\,,\,\ra_{\mathsf s'}$}\}
\]
and
\[
  \cX_{\mathsf s, \mathsf s'}^\circ:=\cX_{\mathsf s, \mathsf s'}\cap W_{\mathsf s, \mathsf s'}^\circ.
\]

\def\Ms{M_{\sfss}}
\def\cXsso{\cX_{\sfss,\sfss'}^\circ}
\def\cXss{\cX_{\sfss,\sfss'}}
\def\Ksp{K_{\sfss',\bC}}
\def\Lsp{L_{\sfss'}}
\def\Ls{L_{\sfss}}
\def\Vsp{V_{\sfss'}}
\begin{lem}
  \label{descko}
Let $\sO$ be a $K_{\mathsf s, \C}$-orbit in $\p_{\mathsf s}$. Suppose that $\sO$ is contained in the image of the moment map $M_\mathsf s$. Then the set
\be\label{kkpo}
  M_{\mathsf s}^{-1}(\sO)\cap \cXsso
\ee
is a single $K_{\mathsf s,\C}\times K_{\mathsf s', \C}$-orbit. Moreover, for every element $\phi$ in $M_{\mathsf s}^{-1}(\sO)\cap \cX_{\mathsf s, \mathsf s'}^\circ$, there is an exact sequence of algebraic groups:
\[
  1\rightarrow (K_{\mathsf s',\C})_\phi \rightarrow (K_{\mathsf s,\C}\times K_{\mathsf s', \C})_\phi\xrightarrow{\textrm{the projection to the first factor}} (K_{\mathsf s,\C})_{\mathbf e}\rightarrow 1,
\]
where $\mathbf e:=M_\mathsf s(\phi)\in \sO$,  $(K_{\mathsf s',\C})_\phi$ is the stabilizer of $\phi$ in $K_{\mathsf s',\C}$, $(K_{\mathsf s,\C}\times K_{\mathsf s', \C})_\phi$ is the stabilizer of $\phi$ in $K_{\mathsf s,\C}\times K_{\mathsf s', \C}$, and $(K_{\mathsf s,\C})_{\mathbf e}$ is the stabilizer of $\mathbf e$ in $K_{\mathsf s,\C}$.
\end{lem}

\begin{remark} A version of Lemma \ref{descko} for a nilpotent $G$-orbit in $\g_{\R}$ is in \cite[Lemma 3.4]{Zh}.
\end{remark}

\begin{proof}
 When the rank of some (and hence all) elements in $\sO\subset \Hom_\bC(V_\sfss,V_\sfss)$
 is equal to $\dim V_{\sfss'}$, 
  all elements in $\Ms^{-1}(\sO)\subset \Hom_\bC(V_\sfss,V_{\sfss'})$ 
  are surjective and hence belong to  $\cXsso$.   
In this case, the lemma is proved in \cite[Lemma 13]{Ohta}.
  
 In general, take an element  $\phi_0 \in \Ms^{-1}(\sO)$. Then its image $\Im(\phi_0)$ is an $\Lsp$-stable subspace of $\Vsp$.
 Fix any $\Lsp$-stable decomposition
 \[
  \Im(\phi_0) = V' \oplus N'
 \]
 such that the form $\inn{}{}_{\Vsp}$ is non-degenerate on $V'$ and zero on
 $N'$. Let $\phi_1$ be the composition of $\phi_0$ with the projection to $V'$
 in the above decomposition. 
 Then 
 $\phi_1$ is an element in  $\Ms^{-1}(\sO)\cap \cXsso$. This proves that the set 
 $\Ms^{-1}(\sO)\cap \cXsso$ is nonempty.
 
 For an element $\phi\in \Ms^{-1}(\sO)\cap \cXsso$, write $\mathbf e :=\Ms (\phi)=\phi^* \phi $. 
We have that $\Ker (\phi) =\Ker (\mathbf{e})$, which is $\Ls$-stable,  and $\phi ^*|_{\Im(\phi)}$ is a linear bijection onto $\Im(\mathbf{e})$. 

Let $\Vsp(\sO)$ denote the set of all $\Lsp$-stable 
   non-degenerate subspaces $V''$ of $\Vsp$ such that 
  \[
     \dim (V'')_{\lambda} = \dim \left(\Im(\mathbf e)\right)_{\sqrt{-1} \dot \epsilon \lambda}
  \qquad \textrm{for all $\lambda \in \C$}.
\]
Here a subscript complex number   indicates the  corresponding eigenspace of $\Lsp$ (or $\Ls$). Note that these eigenspaces are zero unless $\lambda^2=-\dot \epsilon$.  

 In view of \eqref{eq:eigen}, $\Im(\phi)\in \Vsp(\sO)$ for every $\phi\in \Ms^{-1}(\sO)\cap \cXsso$. Since $\Vsp(\sO)$ is a single $\Ksp$-orbit, the lemma is reduced to the previously discussed case.
\end{proof}

In the notation of Lemma \ref{descko}, write
\begin{equation}\label{eq:descent}
  \nabla^{\mathsf s}_{\mathsf s'}(\sO):=\textrm{the image of the set \eqref{kkpo} under the moment map  $M_{\mathsf s'}$,}
\end{equation}
which is a $ K_{\mathsf s', \C}$-orbit in $\p_{\mathsf s'}$. This is called the descent of $\sO$. It is an element of $\Nil(\p_{\mathsf s'})$ if $\sO\in \Nil(\p_{\mathsf s})$.

Similar to \eqref{momentmap}, we have the following two well-defined algebraic maps:
   \be\label{momentmap2}
    \xymatrix@R=0em@C=4em{
      \g_\mathsf s &\ar[l]_{\tilde M_\mathsf s:=\tilde M_\mathsf s^{\sfss, \sfss'}} W_{\mathsf s, \mathsf s'}\ar[r]^{\tilde M_{\mathsf s'}:=\tilde M_{\mathsf s'}^{\sfss, \sfss'}}& \g_{\mathsf s'},\\
     \phi^* \phi & \ar@{|->}[l] \phi \ar@{|->}[r] & \phi \phi^*.
    }
  \ee
These two maps are also called the moment maps. Similar to the maps in \eqref{momentmap},   they  are both $G_{\mathsf s,\C}\times G_{\mathsf s', \C}$-equivariant.

Now we suppose that the $G_{\sfss, \C}$-orbit $\CO\subset\g_\sfss$
  is contained in the image of the moment map $\tilde M_{\mathsf s}$. Similar to
the first assertion of Lemma \ref{descko}, the set
\be\label{kkpo2}
  \tilde M_{\mathsf s}^{-1}(\CO)\cap W_{\mathsf s, \mathsf s'}^\circ
\ee
is a single $G_{\mathsf s,\C}\times G_{\mathsf s', \C}$-orbit.
Write
\[
 \CO':= \nabla^{\mathsf s}_{\mathsf s'}(\CO):=\textrm{the image of the set \eqref{kkpo2} under the moment map  $\tilde M_{\mathsf s'}$,}
\]
which is a $ G_{\mathsf s', \C}$-orbit in $\g_{\mathsf s'}$.
This is  called the descent of $\CO$. It is an element of  $\Nil(\g_{\mathsf s'})$ if $\CO\in \Nil(\g_{\mathsf s})$.

Recall the well-known identification of a nilpotent orbit in $\g_\sfss$ with its Young diagram (see \cite{CM}). 
Via the notion of the naive descent of a Young diagram in Section \ref{sec:desc}, we may define the naive descent $\DD_\mathrm{naive}(\CO)$ of $\CO\in \Nil(\g_\sfss)$ at the level of Young diagrams. We record the following lemma for later use.

\begin{lem}\label{imageofmm}
Suppose that $\CO\in \Nil(\g_\sfss)$. Then $\CO$ is contained in the image of the moment map $\tilde M_{\mathsf s}$ if and only if
\[
\delta:=  \abs{\sfss'}-\abs{\DD_\mathrm{naive}(\CO)}\geq 0.
\]
When this is the case, the Young diagram of $\DD_{\mathsf s'}^{\mathsf s}(\CO)\in \Nil(\g_{\sfss'})
$ is obtained from that of $\DD_\mathrm{naive}(\CO)$ by adding $\delta$ boxes in the first column.
\end{lem}
\begin{proof}
This is implied by \cite[Theorem 3.6]{DKPC}.
\end{proof}

\subsection{Geometric theta lift} \label{sec:dlift}
We will define a notion of geometric theta lift, which will be our main tool for computing the associated cycle of a representation constructed by the machinery of theta lifting. See Theorem \ref{thm:GDS.AC} in \Cref{sec:equac}.

Let $\zeta_{\mathsf s, \mathsf s'}$ denote the algebraic character on $K_{\mathsf s, \C}\times K_{\mathsf s', \C}$  such that
\[
 (\zeta_{\mathsf s, \mathsf s'})|_{K_{\mathsf s, \C}}=
   \begin{cases}
    1, \quad & \textrm{if $\star\in \{B, D, C^*\}$}; \smallskip\\
      ({\det}_{\sqrt{-1}})^{\frac{p'-q'}{2}}, \quad & \textrm{if $\star\in \{C, D^*\}$};\smallskip\\
     ({\det}_{\sqrt{-1}}^{\frac{1}{2}})^{p'-q'}, \quad & \textrm{if $\star=\widetilde C$,}\\
  \end{cases}
\]
and
\[
 (\zeta_{\mathsf s, \mathsf s'})|_{K_{\mathsf s', \C}}=
   \begin{cases}
    1, \quad & \textrm{if $\star\in \{C, \widetilde C, D^*\}$}; \smallskip\\
      ({\det}_{\sqrt{-1}})^{\frac{p-q}{2}}, \quad & \textrm{if $\star\in \{D, C^*\}$};\smallskip\\
     ({\det}_{\sqrt{-1}}^{\frac{1}{2}})^{p-q}, \quad & \textrm{if $\star=B$}.
 \end{cases}
\]
Here and henceforth, when no confusion is possible, we  use $1$ to indicate the trivial representation of a group (we also use $1$ to denote the identity element of a group). For other notation, see 
\eqref{comdetlam} and \eqref{dethalf}. This algebraic character arises naturally in the theory of theta lifting  (see \Cref{deforos}). 

Let  $\sO$ be a $K_{\mathsf s, \C}$-orbit in $\p_{\mathsf s}$ as before. Suppose that $\sO$ is contained in the image of the moment map $M_{\mathsf s}$, and write $\sO':=\nabla^{\mathsf s}_{\mathsf s'}(\sO)$.
Let $\phi,\mathbf e$ be as in Lemma  \ref{descko} and let $\mathbf e':=M_{\mathsf s'}(\phi)$. We have an exact sequence
\[
  1\rightarrow  (K_{\mathsf s',\C})_\phi\rightarrow (K_{\mathsf s,\C}\times K_{\mathsf s', \C})_\phi\rightarrow (K_{\mathsf s,\C})_{\mathbf e}\rightarrow 1
\]
as in Lemma  \ref{descko}.

 Let $\CE'$ be a $K_{\mathsf s',\C}$-equivariant algebraic vector bundle  over $\sO'$. Its fibre
$\CE'_{\bfee'}$ at $\bfee'$ is an algebraic representation of the stabilizer group $(K_{\mathsf s',\C})_{\mathbf e'}$. We also view it as a representation of the group
$(K_{\mathsf s,\C}\times K_{\mathsf s', \C})_\phi$ by the pull-back through the homomorphism
\[
  (K_{\mathsf s,\C}\times K_{\mathsf s', \C})_\phi\xrightarrow{\textrm{the projection to the second factor}} (K_{\mathsf s',\C})_{\mathbf e'}.
\]
Then $\CE'_{\mathbf e'} \otimes \zeta_{\mathsf s, \mathsf s'}$ is a representation of $ (K_{\mathsf s,\C}\times K_{\mathsf s', \C})_\phi$, and the coinvariant space
\begin{equation}\label{geotwist}
(\CE'_{\mathbf e'} \otimes \zeta_{\mathsf s, \mathsf s'})_{ (K_{\mathsf s',\C})_\phi}
\end{equation}
 is an algebraic representation of $(K_{\mathsf s,\C})_{\mathbf e}$. Write 
 \begin{equation}\label{liftab}
\CE:= \check \vartheta_{\sO'}^{\sO}(\mathcal E')
\end{equation} 
for the  $K_{\mathsf s,\C}$-equivariant algebraic vector bundle  over $\sO$ whose fibre at $\mathbf e$ equals the coinvariant space representation in \eqref{geotwist}. In this way, we get an exact functor $  \check \vartheta_{\sO'}^{\sO}$ from the category of
$K_{\mathsf s',\C}$-equivariant algebraic vector bundles  over $\sO'$ to the category of $K_{\mathsf s,\C}$-equivariant algebraic vector bundles  over $\sO$. This exact functor induces a  homomorphism of the  Grothendieck groups:
\[
   \check \vartheta_{\sO'}^{\sO}:  \CK_{\mathsf s'}(\sO')\rightarrow  \CK_{\mathsf s}(\sO).
\]
The above homomorphism is independent of the choice of $\phi$.

Recall that $\CO':= \nabla^{\mathsf s}_{\mathsf s'}(\CO)$, and finally we define the geometric theta lift to be the homomorphism
\begin{equation}\label{defn:geolift}
 \check \vartheta_{\CO'}^{\CO}: \CK_{\mathsf s'}(\CO')\rightarrow \CK_{\mathsf s}(\CO)
\end{equation}
such that
\[
 \check \vartheta_{\CO'}^{\CO}(\CE')= \sum_{\sO\textrm{ is a $K_{\mathsf s, \C} $-orbit in $\CO\cap \p_{\mathsf s}$,  $\, \nabla^{\mathsf s}_{\mathsf s'}(\sO)=\sO'$}}    \check \vartheta_{\sO'}^{\sO}(\CE'),
\]
for every $K_{\mathsf s', \C} $-orbit $\sO'$ in $\CO'\cap \p_{\mathsf s'}$, and $\CE'\in \CK_{\mathsf s'}(\sO')$.

\subsection{Associated cycles of painted bipartitions}\label{subsecass}

Note that $\g=\g_\sfss$ by the assumptions in \Cref{classical}. View $\check \CO\in\overline{\mathrm{Nil}}(\check \g)$ as a Young diagram (that has good parity with respect to $\star$). 

\begin{defn}[\cite{BVUni} and \cite{BMSZ1}]\label{bvdual0}
The  Barbasch-Vogan dual $\mathrm d_{\mathrm{BV}}^\star(\check \CO)$ of $\check \CO$ is the Young diagram satisfying
\begin{eqnarray*}
 &&\left (\mathbf c_i(\mathrm d_{\mathrm{BV}}^\star(\check \CO)), \mathbf c_{i+1}(\mathrm d_{\mathrm{BV}}^\star(\check \CO))\right )\\
  &=&\begin{cases}
      (\mathbf r_i(\check \CO), \mathbf r_{i+1}(\check \CO)),\quad& \textrm{if  $(i,i+1)$ is a $\star$-pair that is vacant or balanced in $\check \CO$};  \\
        (\mathbf r_i(\check \CO)-1, 0),\quad& \textrm{if  $(i,i+1)$ is a $\star$-pair that is  tailed in $\check \CO$};  \\
         (\mathbf r_i(\check \CO)-1, \mathbf r_{i+1}(\check \CO)+1),\quad& \textrm{if  $(i,i+1)$ is a $\star$-pair that is primitive in $\check \CO$},
  \end{cases}
  \end{eqnarray*}
and
\[
    \mathbf c_1(\mathrm d_{\mathrm{BV}}^\star(\check \CO))=\begin{cases}
      0,\qquad& \textrm{if  $\star\in \{D, D^*\}$ and $|\check \CO|=0$};  \\
       \mathbf r_1(\check \CO)+1, \qquad& \textrm{if  either $\star=B$, or $\star\in \{D, D^*\}$ and $|\check \CO|\neq 0$.}
          \end{cases}
\]

\end{defn}

Note that
\[
 \abs{\mathrm d_{\mathrm{BV}}^\star(\check \CO)}=\abs{\sfss},
\]
and $\mathrm d_{\mathrm{BV}}^\star(\check \CO)$ is identified with an element of $\Nil(\g_\sfss)$.
Recall  the ideal $I_{\check \CO}$ of $\mathcal U(\g_\sfss)$ from Section \ref{secsu}.

\begin{lem}\label{bvdual}
The  associated variety of $I_{\check \CO}$ equals   the Zariski closure of $\mathrm d_{\mathrm{BV}}^\star(\check \CO)$.
\end{lem}
\begin{proof}
This is implied by \cite[Corollary~A3]{BVUni} and \cite[Theorem B]{BMSZ1}.
\end{proof}

As before, let $\check \CO'$ be the dual descent of $\check \CO$, which has good parity (with respect to $\star'$). Assume that
\[
  \abs{\mathrm d_{\mathrm{BV}}^{\star'}(\check \CO')}=\abs{\sfss'}.
\]
Then  $\mathrm d_{\mathrm{BV}}^{\star'}(\check \CO')$ is identified with an element of $\Nil(\g_{\sfss'})$.

\begin{lem}\label{dualdesc}
The orbit $\mathrm d_{\mathrm{BV}}^{\star}(\check \CO)$ is contained in the image of the moment map $\tilde M_{\mathsf s}$, and
\[
  \nabla_{\mathsf s'}^{\mathsf s}(\mathrm d_{\mathrm{BV}}^{\star}(\check \CO))=\mathrm d_{\mathrm{BV}}^{\star'}(\check \CO').
\]
\end{lem}
\begin{proof}
This is elementary to check by using \Cref{imageofmm}.
\end{proof}

In view of Lemma \ref{dualdesc}, we suppose that
\[
\CO=\mathrm d_{\mathrm{BV}}^\star(\check \CO)\in \Nil(\g_\sfss)\quad \textrm{and}\quad \CO'=\mathrm d_{\mathrm{BV}}^{\star'}(\check \CO')\in \Nil(\g_{\sfss'}).
\]

Put
\begin{equation}\label{def:PBPs}
  \PBPe_\star(\check \CO,\mathsf s):=\{(\tau, \wp)\in  \PBPe_\star(\check \CO)\,:\, (p_\tau,q_\tau)=(p,q)\}.
\end{equation}
This is identical to the set $ \PBPe_G(\check \CO)$ because of the assumption \eqref{defg}. 

Let $\uptau=(\tau, \wp)\in \PBPe_\star(\check \CO,\mathsf s)$ so that $\mathsf s_{\tau}:=(\star, p_\tau, q_\tau)=\mathsf s$. 
In what follows we will define the associated cycle $\mathrm{AC}(\uptau)\in\CK_{\mathsf s}(\CO)$ of $\uptau$.

For the initial case when $|\check \CO|=0$
so that $\CO\cap \p_{\mathsf s}$ is a singleton, we define  $\mathrm{AC}(\uptau)\in\CK_{\mathsf s}(\CO)$ to be the element that corresponds to the following  representation of $K_{\mathsf s,\C}$:
 \[
\left\{
     \begin{array}{ll}
                    \textrm{the determinant character} , &\hbox{ if $\gamma_\tau=B^-$ so that $K_{\mathsf s,\C}=\rO(0,1)$;} \smallskip\\
                    \textrm{the one-dimensional genuine representation} , &\hbox{ if $\star=\widetilde C$ so that $K_{\mathsf s,\C}=\{\pm 1\}$;} \smallskip\\
 \textrm{the one-dimensional trivial representation} , &\hbox{ otherwise.}            \end{array}
   \right.
\]

 Write  $\uptau'\in \PBPe_{\star'}(\check \CO')$ for the descent of $\uptau$, and assume that $\uptau'\in \PBPe_{\star'}(\check \CO',\mathsf s')$.
When $|\check \CO|\neq 0$,  similar to the definition of $\pi_\uptau$ in \eqref{eq:def-pi} (of Section \ref{subsec:comTOrep}), we inductively define
$\mathrm{AC}(\uptau)\in \CK_{\mathsf s}(\CO) $ by
 \begin{equation}\label{eq:AC}
   \mathrm{AC}(\uptau):=\left\{
     \begin{array}{ll}
         \check \vartheta_{\CO'}^{\CO}(\mathrm{AC}(\uptau'))\otimes ({\det}_{-1})^{\varepsilon_{\tau}}, &\hbox{if  $\star=B$ or $D$}; \smallskip\\
         \check \vartheta_{\CO'}^{\CO}(\mathrm{AC}(\uptau')\otimes \det^{\varepsilon_{\wp}}), &\hbox{if $\star=C$ or $\widetilde C$};\smallskip \\
              \check \vartheta_{\CO'}^{\CO}(\mathrm{AC}(\uptau')), &\hbox{if $\star=C^*$ or $D^*$}. \\
            \end{array}
   \right.
    \end{equation}
Here ${\det}_{-1}: K_{\mathsf s,\C}\rightarrow \C^\times$ is the character of $K_{\mathsf s,\C}$ defined in \eqref{comdetlam}. 
We will refer to $\mathrm{AC}(\uptau)$ as the associated cycle of $\uptau$.

\subsection{Key properties of associated cycles}\label{subsec:kpac}

Recall that $\check \CO$ is assumed to have good parity. 
Recall also that a nilpotent orbit in $\check \g$ is said to be distinguished if it  has empty intersection with every proper Levi subalgebra of $\check \g$. Combinatorially, this is equivalent to saying that no pair of rows of the Young diagram has equal nonzero length (see \cite[Theorem 8.2.14]{CM}). Note that all distinguished nilpotent orbits in $\check \g$ have good parity.

\begin{defn}\label{defqd}
The orbit $\check \CO$ is said to be quasi-distinguished if there is no $\star$-pair that is balanced in $\check \CO$.

\end{defn}

When $\star\in \set{C^*, D^*}$, the situation is made simpler by the following non-existence result. 

\begin{lem}[{\cite{BMSZa}*{Proposition~10.1}}]\label{prop:CD*}
  Suppose that $\star\in \set{C^*, D^*}$. If the set $\PBP_\star(\ckcO)$ is nonempty, then $\check \CO$ is quasi-distinguished.
\end{lem}

We postpone the proof of the following propositions concerning associated cycles to  \Cref{sec:ACC}.

\begin{prop}\label{thmac1}
 For every $\uptau\in \PBPe_\star(\check \CO,\mathsf s)$, the associated cycle $\mathrm{AC}(\uptau)\in \CK_{\mathsf s}(\CO)$ is a nonzero  sum of  admissible orbit data
over pairwise distinct $K_{\sfss,\C}$-orbits in $\CO\cap \p_{\sfss}$.
\end{prop}

\begin{prop}\label{thmac3}
Suppose that  $\check \CO$ is quasi-distinguished.  If $\star\in \{B,D\}$ and $(\star, |\check \CO|)\neq (D, 0)$, then  the map
\[
\mathrm{AC}: \PBPe_\star(\check \CO,\mathsf s)\times \Z/2\Z \rightarrow  \mathrm{AOD}_{\sfss}(\CO),\quad (\uptau, \epsilon)\mapsto \mathrm{AC}(\uptau)\otimes {\det}^{\epsilon}
\]
is well-defined and bijective. In all other cases, the
map
\[
\mathrm{AC}: \PBPe_\star(\check \CO,\mathsf s)\rightarrow \mathrm{AOD}_{\sfss}(\CO)
\]
is well-defined and bijective.
Here the set $\mathrm{AOD}_{\sfss}(\cO)$ is defined as in
\eqref{eq:AODO}.
\end{prop}

\begin{prop}\label{thmac4}
Suppose that  $\star\in \{B,D\}$ and $(\star, |\check \CO|)\neq (D, 0)$. Let $\uptau_i=(\tau_i, \wp_i)\in \PBPe_\star(\check \CO,\mathsf s)$ and $\epsilon_i\in \Z/2\Z$ ($i=1,2$).   If
\[
  \mathrm{AC}(\uptau_1)\otimes {\det}^{\epsilon_1}= \mathrm{AC}(\uptau_2)\otimes {\det}^{\epsilon_2},
\]
then
\[
  \epsilon_1=\epsilon_2\qquad\textrm{and}\qquad \varepsilon_{\tau_1}=\varepsilon_{\tau_2}.
\]
 \end{prop}

\begin{prop}\label{thmac5}
Suppose that  $\star\in \{C,\widetilde C, D^*\}$. Let $\uptau_i=(\tau_i, \wp_i)\in \PBPe_\star(\check \CO,\mathsf s)$, and write $\uptau'_i=(\tau'_i, \wp'_i)$ for its descent ($i=1,2$).   If
\[
  \mathrm{AC}(\uptau_1)= \mathrm{AC}(\uptau_2),
\]
then
\[
 ( p_{\tau'_1}, q_{\tau'_1})=( p_{\tau'_2}, q_{\tau'_2}),
 \]
 and
 \[
  \varepsilon_{\wp_1}=\varepsilon_{\wp_2}\  \textrm{when $\ \star\in\{C,\wtC\}$}.
\]
 \end{prop}

\section{Associated cycles and structure of proof of \Cref{thm100}}\label{sec:main}

In this section,  we provide a blueprint for the proof of \Cref{thm100} and highlight the role of associated cycles in its proof.  We recall the general notation of Section \ref{classical}. Thus we have $G$,  $G_\C$, $K$, $K_\C$, associated to a classical space $(V, \la\,,\,\ra, J,L)$ of classical signature $\mathsf s$.
Recall that $\check \CO$ is a $\check G$-orbit in $\mathrm{Nil}(\check \g)$ that has good parity, and $\CO\in \Nil(\g^*)$ is its  Barbasch-Vogan dual. The algebraic variety $\CO\cap \p^*$
is a finite union of $K_\C$-orbits.  As in Section \ref{classical}, we set  
\[
\CK(\CO):=\bigoplus_{\sO\textrm{ is a $K_\C$-orbit in
      $\CO\cap \p^*$}} \CK(\sO).
\]

We say that a Casselman-Wallach representation of $G$ is  $\CO$-bounded  if
the associated variety  of its annihilator ideal
is contained in the Zariski closure $\overline \CO$ of $\CO$. Note that all representations in $\mathrm{Unip}_{\check \CO}(G)$ are $\CO$-bounded. Write $\CK_{\overline \CO}(G)$ for the  Grothendieck group of the category of all such representations.
From \cite[Theorem 2.13]{Vo89},  we have a canonical homomorphism
\[
\xymatrix{
  \mathrm{AC}_\cO\colon   \CK_{\overline \cO}(G) \ar[r]& \CK(\CO).
}
\]
For an $\CO$-bounded Casselman-Wallach representation $\pi$ of $G$, we call $ \mathrm{AC}_\cO(\pi)$ the associated cycle of $\pi$. This is a fundamental invariant attached to $\pi$.

The following theorem will be proved in \Cref{sec:comANDgeo}. 

\begin{thm}\label{thmpitau} Let $\uptau\in \PBPeg(\ckcO)$.  
\begin{enumerate}[label=(\alph*),wide=0em]
\item 
The representation $\pi_\uptau$ is irreducible, unitarizable and special unipotent attached to $\check \CO$. 
\item The following equality holds: 
\[
\AC_\CO(\pi_\uptau)=\AC(\uptau)\in \CK(\CO),
\]
where $\AC(\uptau)$ is defined in \eqref{eq:AC}. 
\end{enumerate}
\end{thm}

Based on \Cref{thmpitau}, and the tools introduced in \Cref{sec:bip} and \Cref{sec:bipGeometry}, we prove the following result, on the exhaustion of  special unipotent representations by our construction. 

\begin{thm}\label{thmac7}
If $\star\in \{B,D\}$ and $(\star, |\check \CO|)\neq (D, 0)$, then  the map
\be\label{bijthm1}
 \PBPe_G(\check \CO)\times \Z/2\Z \rightarrow \Unip_{\check \cO}(G),\quad (\uptau, \epsilon)\mapsto \pi_\uptau\otimes {\det}^{\epsilon}
\ee
is bijective. In all other cases, the
map
\begin{equation}\label{bijthm2}
\PBPe_G(\check \CO)\rightarrow \Unip_{\check \cO}(G), \quad \uptau\mapsto \pi_\uptau
\end{equation}
is bijective.
\end{thm}

\begin{proof}
In view of \Cref{thmcount}, we only need to show the injectivity of  the two maps in the statement of the theorem. We prove by induction on the number of nonempty rows of $\check \CO$. The theorem is trivially true in the case when  $|\check \CO|=0$. So we assume that $|\check \CO|\neq 0$  and the theorem has been proved for the dual descent $\check \CO'$.
Suppose that  $\uptau_i=(\tau_i, \wp_i)\in \PBPeg(\check \CO)$, $\epsilon_i\in \Z/2\Z$ ($i=1,2$). Write  $\uptau_i'=(\tau_i', \wp_i')$ for the descent of $\uptau_i$.

First assume that $\star\in \{B,D\}$, and
\[
  \pi_{\uptau_1}\otimes {\det}^{\epsilon_1}\cong \pi_{\uptau_2}\otimes {\det}^{\epsilon_2}.
\]
Theorem \ref{thmpitau} implies that
\[
\mathrm{AC}(\uptau_1)\otimes {\det}^{\epsilon_1}=\mathrm{AC}(\uptau_2)\otimes {\det}^{\epsilon_2}.
\]
By Proposition \ref{thmac4}, we know that
\[
  \epsilon_1=\epsilon_2\qquad\textrm{and}\qquad \varepsilon_{\tau_1}=\varepsilon_{\tau_2}.
\]
Then the definition of $ \pi_{\uptau_1}$ and $ \pi_{\uptau_2}$ implies that
\[
 \check  \Theta_{\tau_1'}^{\tau_1}(\pi_{\uptau_1'})\cong \check \Theta_{\tau_2'}^{\tau_2}(\pi_{\uptau_2'}).
\]
 Consequently,
\[
\pi_{\uptau_1'}\cong \pi_{\uptau_2'}
\]
by the injectivity property of the theta correspondence. Hence $\uptau'_1=\uptau'_2$ by the induction hypothesis, and \Cref{prop:DD.BDinj} then implies $\uptau_1=\uptau_2$.
This proves that the map \eqref{bijthm1} is injective.

A slightly simplified argument shows that the map \eqref{bijthm2} is injective when $\star=C^*$.

Now assume that $\star\in \{C,\widetilde C\}$. Suppose that
\[
  \pi_{\uptau_1}\cong \pi_{\uptau_2}.
\]
Theorem \ref{thmpitau} implies that
\[
\mathrm{AC}(\uptau_1)=\mathrm{AC}(\uptau_2).
\]
By Proposition \ref{thmac5}, we know that
\[
 ( p_{\tau'_1}, q_{\tau'_1})=( p_{\tau'_2}, q_{\tau'_2})\qquad\textrm{and}\qquad \varepsilon_{\wp_1}=\varepsilon_{\wp_2}.
\]
Then the definition of $ \pi_{\uptau_1}$ and $ \pi_{\uptau_2}$ implies that
\[
  \check \Theta_{\tau_1'}^{\tau_1}(\pi_{\uptau_1'}\otimes {\det}^{\varepsilon_{\wp_1}})\cong \check \Theta_{\tau_2'}^{\tau_2}(\pi_{\uptau_2'}\otimes {\det}^{\varepsilon_{\wp_2}}).
\]
The injectivity property of the theta correspondence implies that
\[
\pi_{\uptau_1'}\otimes {\det}^{\varepsilon_{\wp_1}}\cong \pi_{\uptau_2'}\otimes {\det}^{\varepsilon_{\wp_2}},
\]
which further implies that $\pi_{\uptau_1'}\cong \pi_{\uptau_2'}$.
Hence $\uptau'_1=\uptau'_2$ by the induction hypothesis, and \Cref{prop:DD.BDinj} then implies $\uptau_1=\uptau_2$. This proves that the map \eqref{bijthm2} is injective.

A slightly simplified argument shows that the map \eqref{bijthm2} is injective when $\star=D^*$. This completes the proof of the theorem.
\end{proof}

Part (a) of \Cref{thmpitau} and \Cref{thmac7} combine to give
\Cref{thm100}. 

In view of Propositions \ref{thmac1} and \ref{thmac3}, 
Part (b) of  \Cref{thmpitau} and \Cref{thmac7} also imply the following result on the associated cycles of special unipotent representations.

\begin{thm}\label{thmac0} 

\noindent (a) For every $\pi\in \mathrm{Unip}_{\check \CO}(G)$, the associated cycle $\mathrm{AC}_{\CO}(\pi)\in \CK(\CO) $ is a nonzero  sum of  admissible orbit data
over pairwise distinct $K_\C$-orbits in $\CO\cap \p^*$. 

\noindent  (b) If $\check \CO$ is quasi-distinguished, then the map 
\[
\mathrm{AC}_{\CO}: \mathrm{Unip}_{\check \CO}(G)\rightarrow  \mathrm{AOD}(\CO)
\]
is well-defined and bijective. 
\end{thm}

\begin{remark}
Suppose that $\star\in \{C^*, D^*\}$ so that $G$ is quaternionic.
If $\check \CO$ is quasi-distinguished, then there is precisely one admissible orbit datum over  $\sO$ for each $K_\C$-orbit $\sO\subset \CO\cap \p^*$.
Thus
\[
 \mathrm{AOD}(\CO)=K_\C\backslash  (\CO\cap \p^*).
\]
If $\check \CO$ is not quasi-distinguished, then
$\PBP_\star(\ckcO)$ is empty (see \Cref{prop:CD*}), and 
hence $\mathrm{Unip}_{\check \CO}(G)$ is empty. Note that in this case $\CO\cap \p^*$ is also empty (see \cite[Theorems 9.3.4 and 9.3.5]{CM}). 

\end{remark}

\section{Theta lifting via matrix coefficient integrals and unitarity preservation}\label{sec:Integrals}

In this section,
we investigate theta lifting via
matrix coefficient integrals. The main result is \Cref{positivity000}, which proves preservation of unitarity under a certain convergence condition.  

Let $\mathsf s=(\star, p,q)$ and $ \mathsf s'=(\star', p',q')$ be classical signatures such that $\star'$ is the Howe dual of $\star$.

\subsection{The oscillator representation}\label{secoscil}
We use the notation of Section \ref{secmmap}. Write
$
  W_{\mathsf s, \mathsf s'}^{J_{\mathsf s, \mathsf s'}}\subset W_{\mathsf s, \mathsf s'}
$
for the fixed point set of $J_{\mathsf s, \mathsf s'}$. It is a real symplectic space under the restriction of the form $\la\,,\,\ra_{\mathsf s, \mathsf s'}$. Let $H_{\mathsf s, \mathsf s'}:= W_{\mathsf s, \mathsf s'}^{J_{\mathsf s, \mathsf s'}}\times \R$
denote the Heisenberg group attached to $W_{\mathsf s, \mathsf s'}^{J_{\mathsf s, \mathsf s'}}$, with group multiplication
\[
  (\phi ,t)\cdot (\phi ', t'):=(\phi +\phi ', t+t'+\la \phi , \phi '\ra_{\mathsf s, \mathsf s'}), \qquad \phi ,\phi '\in  W_{\mathsf s, \mathsf s'}^{J_{\mathsf s, \mathsf s'}}, \quad t, t'\in \R.
\]
Denote by $\h_{\mathsf s, \mathsf s'}$ the complexified Lie algebra of $H_{\mathsf s, \mathsf s'}$. Then  $\cX_{\mathsf s, \mathsf s'}$ (as in \eqref{def:Xss'}) is  an abelian Lie subalgebra of $\h_{\mathsf s, \mathsf s'}$.

The following is the smooth version of the Stone-von Neumann theorem.

\begin{lem}[{\cf \cite[Theorem 5.1]{Cl89} and \cite[Section 4]{Ad07}}]\label{vn}
Up to isomorphism, there exists a unique irreducible smooth Fr\'echet representation $\omega_{\mathsf s, \mathsf s'}$ of $H_{\mathsf s, \mathsf s'}$ of moderate growth 
with central character
\[
\R\rightarrow \C^\times, \ t\mapsto e^{\sqrt{-1}\, t}.
\]
Moreover, the space
\[
  \omega_{\mathsf s, \mathsf s'}^{\cX_{\mathsf s, \mathsf s'}}:=\{v\in \omega_{\mathsf s, \mathsf s'}\,:\, x \cdot v=0\quad \textrm{for all $x\in \cX_{\mathsf s, \mathsf s'}$}\}
\]
is one-dimensional.
\end{lem}

The group $G_{\mathsf s}\times G_{\mathsf s'}$ acts on $H_{\mathsf s, \mathsf s'}$ as group automorphisms via the following natural action of $G_{\mathsf s}\times G_{\mathsf s'}$ on  $W_{\mathsf s, \mathsf s'}^{J_{\mathsf s, \mathsf s'}}$:
\[
  (g, g')\cdot \phi:=g'\circ \phi\circ g^{-1}, \qquad (g,g')\in G_{\mathsf s}\times G_{\mathsf s'},\ \phi\in W_{\mathsf s, \mathsf s'}^{J_{\mathsf s, \mathsf s'}}.
\]
From this action, we form the semidirect product $(G_{\mathsf s}\times G_{\mathsf s'})\ltimes H_{\mathsf s, \mathsf s'}$.

Recall the character $\zeta_{\mathsf s, \mathsf s'}$ from Section \ref{sec:dlift}.

\begin{lem}[{\cf \cite[Proposition 7.5]{Ad07}}]\label{deforos}
The representation $\omega_{\mathsf s, \mathsf s'}$ of $H_{\mathsf s, \mathsf s'}$ in 
\Cref{vn} uniquely extends to a  smooth representation of 
$(G_{\mathsf s}\times G_{\mathsf s'})\ltimes H_{\mathsf s, \mathsf s'}$
such  that $K_{\mathsf s}\times  K_{\mathsf s'}$ acts on  
$ \omega_{\mathsf s, \mathsf s'}^{\cX_{\mathsf s, \mathsf s'}}$ 
through the scalar multiplication by $\zeta_{\mathsf s, \mathsf s'}$.
\end{lem}

Let $\omega_{\mathsf s, \mathsf s'}$ denote the representation of $(G_{\mathsf s}\times G_{\mathsf s'})\ltimes H_{\mathsf s, \mathsf s'}$ as in \Cref{deforos},
henceforth called the smooth oscillator representation or simply the oscillator representation. As is well known, this representation is unitarizable (\cite{Weil}). Fix an invariant continuous Hermitian inner product $\inn{}{}$ on it, which is unique up to a positive scalar multiplication.
Denote by $\hat \omega_{\mathsf s, \mathsf s'}$ the completion of $\omega_{\mathsf s, \mathsf s'}$ with respect to this inner product, which is a unitary representation of  $(G_{\mathsf s}\times G_{\mathsf s'})\ltimes H_{\mathsf s, \mathsf s'}$.

 Let $\omega_{\mathsf s, \mathsf s'}^\vee$ denote the  contragredient of the  oscillator  representation $\omega_{\mathsf s, \mathsf s'}$. It is  the smooth Fr\'echet representation of $(G_{\mathsf s}\times G_{\mathsf s'})\ltimes H_{\mathsf s, \mathsf s'}$ of moderate growth specified by the following conditions:
 \begin{itemize}
 \item there is given a  $(G_{\mathsf s}\times G_{\mathsf s'})\ltimes H_{\mathsf s, \mathsf s'}$-invariant, non-degenerate, continuous bilinear form
 \[
   \la\,,\,\ra:  \omega_{\mathsf s, \mathsf s'}\times \omega_{\mathsf s, \mathsf s'}^\vee\rightarrow \C;
 \]
 \item  $\omega_{\mathsf s, \mathsf s'}^\vee$  is irreducible as a representation of $H_{\mathsf s, \mathsf s'}$.
 \end{itemize}
Since $\omega_{\mathsf s, \mathsf s'}$ is contained in the unitary representation $\hat \omega_{\mathsf s, \mathsf s'}$, $\omega_{\mathsf s, \mathsf s'}^\vee$ is identified with the complex conjugate of $\omega_{\mathsf s, \mathsf s'}$.

Put
\[
   {\sfss'}^-:=(\star', q', p'),
\]
and fix a linear isomorphism
\be\label{iota0}
 \iota_{\sfss'}:  V_{\sfss'}\rightarrow V_{\sfss'^-}, 
\ee
such that $(-\la\,,\,\ra_{\sfss'}, J_{\sfss'},-L_{\sfss'})$ corresponds to   $(\la\,,\,\ra_{\mathsf s'^-}, J_{\mathsf s'^-}, L_{\mathsf s'^-})$ under this isomorphism.
This induces an isomorphism
\be\label{isoggp}
  G_{\sfss'}\rightarrow G_{\sfss'^-}, \qquad g'\mapsto g'^-.
\ee
Then we have a group isomorphism
\be\label{identifyjg}
\begin{array}{rcl}
  (G_{\mathsf s}\times G_{\mathsf s'})\ltimes H_{\mathsf s, \mathsf s'}&\rightarrow& (G_{\mathsf s'^-}\times G_{\mathsf s})\ltimes H_{\mathsf s'^-, \mathsf s},\smallskip\\
   (g, g', (\phi, t))&\mapsto & (g'^-, g, (\phi^*  \circ \iota_{\sfss'}^{-1}, t)),
   \end{array}
\ee
where $\phi^*$ is defined as in \eqref{adjointmap}. 
It is easy to see that the irreducible representation $\omega_{\sfss, \sfss'}$
corresponds to the irreducible representation $\omega_{\sfss'^-, \sfss}$ under this isomorphism.

 For every Casselman-Wallach representation $\pi'$ of $G_{\mathsf s'}$,
put
\[
   \check \Theta_{\mathsf s'}^{\mathsf s}(\pi'):=(\omega_{\mathsf s, \mathsf s'}\widehat \otimes \pi')_{G_{\mathsf s'}} \qquad (\textrm{the  Hausdorff coinvariant space}).
\]
This is a Casselman-Wallach representation of $G_{\mathsf s}$.

\subsection{Growth of Casselman-Wallach representations}

Write $\p_\mathsf s^{J_\mathsf s}$ for the centralizer of $J_\mathsf s$ in $\p_\mathsf s$, which is a real form of $\p_\mathsf s$. The Cartan decomposition asserts that
\[
  G_\mathsf s=K_\mathsf s\cdot \exp(\p_\mathsf s^{J_\mathsf s}).
\]
Denote by  $\Psi_\mathsf s$ the function on $G_\mathsf s$ satisfying the following conditions:
\begin{itemize}
\item it is both left and right $K_\mathsf s$-invariant;
\item for all $g\in \exp(\p_\mathsf s^{J_\mathsf s})$,
\[
  \Psi_\mathsf s(g)=\prod_{a} \left(\frac{1+a}{2}\right)^{-\frac{1}{2}},
\]
 where $a$ runs over all eigenvalues of $g: V_\mathsf s\rightarrow V_\mathsf s$, counted with multiplicities.
\end{itemize}
Note that all the  eigenvalues of $g \in \exp(\p_\mathsf s^{J_\mathsf s})$ are positive real numbers, and  $0<\Psi_\mathsf s(g)\leq 1$ for all $g\in G_\mathsf s$. 

Set
\be\label{def:nus}
  \nu_\mathsf s:=\begin{cases}
    \abs{\mathsf s},\quad &\textrm{if }\star\in \{C, \widetilde C\};\\
     \abs{\mathsf s}-1,\quad & \textrm{if }\star= C^*;\\
      \abs{\mathsf s}-2,\quad & \textrm{if }\star\in \{B,D\};\\
       \abs{\mathsf s}-3,\quad & \textrm{if }\star= D^*.
  \end{cases}
\ee

Denote by $\Xi_\sfss$ the bi-$K_\sfss$-invariant Harish-Chandra's $\Xi$ function on $G_\sfss$.

\begin{lem}\label{boundpsi}
There exists a real number $C_\sfss>0$ such that
\[
  \Psi_\sfss^{\nu_\sfss}(g)\leq C_\sfss\cdot \Xi_\sfss (g)\quad\textrm{ for all $g\in G_\sfss$}.
\]

\end{lem}
\begin{proof}
This is implied by the well-known  estimate of Harish-Chandra's $\Xi$ function (\cite[Theorem 4.5.3]{Wa1}).
\end{proof}

\begin{lem}\label{int}
Let $f$ be a bi-$K_\sfss$-invariant positive function  on $G_\sfss$  such that
\[
  \p_\mathsf s^{J_\mathsf s}\rightarrow \R, \qquad x\mapsto f(\exp(x))
\]
is a polynomial function. Then for every real number $\nu>0$, the function $f\cdot \Psi_\mathsf s^\nu\cdot \Xi_\sfss^2$ is integrable with respect to a Haar measure on $G_\mathsf s$.
\end{lem}
\begin{proof}
  This follows from the integral formula for $G_\mathsf s$ under the Cartan decomposition (see
 \cite[Lemma 2.4.2]{Wa1}), as well as the  estimate of Harish-Chandra's $\Xi$ function (\cite[Theorem 4.5.3]{Wa1}).
\end{proof}

For every Casselman-Wallach representation $\pi$ of $G_\mathsf s$, write $\pi^\vee$ for its contragredient representation, which is a  Casselman-Wallach representation  of $G_\mathsf s$ equipped with a $G_{\mathsf s}$-invariant, non-degenerate, continuous bilinear map
 \[
   \la\,,\,\ra: \pi \times\pi^\vee\rightarrow \C.
 \]

\begin{defn}
Let $\nu\in \R$. 
A  positive function $\Psi$ on $G_\mathsf s$ is $\nu$-bounded if there is
a  function $f$ as in \Cref{int} 
such that
\[
  \Psi(g)\leq f(g)\cdot \Psi_{\sfss}^\nu(g)\cdot \Xi_\sfss(g)\qquad \textrm{for all }g\in G_\sfss.
\]
A Casselman-Wallach representation $\pi$ of $G_\mathsf s$ is said to be $\nu$-bounded if there
exist a $\nu$-bounded positive function $\Psi$ on $G_\sfss$, and continuous seminorms $\abs{\,\cdot\,}_{\pi}$ and $\abs{\,\cdot\,}_{\pi^\vee}$ on  $ \pi$ and $\pi^\vee$ respectively such that
\[
 \abs{ \la g \cdot u, v\ra}\leq \Psi(g)\cdot \abs{u}_{\pi}\cdot \abs{v}_{\pi^\vee}
\]
for all $u\in \pi$, $v\in \pi^\vee$, and $g\in G_{\mathsf s}$.

\end{defn}

Let $X_\sfss$ be a maximal $J_\sfss$-stable totally isotropic subspace of $V_\sfss$, which is unique  up to the action of $K_\sfss$. Put $Y_\sfss:=L_\sfss(X_\sfss)$. Then $X_\sfss\cap Y_\sfss=\{0\}$.
Write
\[
P_\sfss=R_\sfss\ltimes N_\sfss
\]
 for the parabolic subgroup of $G_\sfss$ stabilizing $X_\sfss$, where $R_\sfss$ is the Levi subgroup stabilizing both $X_\sfss$ and $Y_\sfss$, and $N_\sfss$ is the unipotent radical. For every character $\chi: R_\sfss\rightarrow \C^\times$, view it as a character of $P_\sfss$ that is trivial on $N_\sfss$. Write
\[
  I(\chi):=I_\sfss(\chi):=\Ind_{P_\sfss}^{G_\sfss}\chi \qquad(\textrm{normalized smooth induction}),
\]
which is a Casselman-Wallach representation of $G_\sfss$ under the right translations. Note that the representations   $I(\chi)$ and $ I(\chi^{-1})$ are contragredients of each other with the $G_\sfss$-invariant pairing
 \[
  \la\,,\,\ra:  I(\chi)\times I(\chi^{-1})\rightarrow \C, \quad (f, f')\mapsto \int_{K_\sfss} f(g) \cdot f'(g) \od\!g.
 \]
 Unless otherwise specified, all the measures on Lie groups occurring  in this article are  right-invariant Haar measures.

Let $\nu_\chi$ be a real number such that $\abs{\chi}$ equals the composition of
\be\label{nuchi}
   R_\sfss\xrightarrow{\textrm{the natural homomorphism} }\GL(X_\sfss)\xrightarrow{\abs{\det}^{\nu_\chi}} \C^\times.
\ee

\begin{defn}
A classical signature $\sfss$ is split if $V_\sfss=X_\sfss\oplus Y_\sfss$.
\end{defn}

When $\sfss$ is split, $P_\sfss$ is called a Siegel parabolic subgroup of $G_\sfss$.

\begin{lem}\label{growthdp}
Suppose that $\sfss$ is split and let $\chi: P_\sfss\rightarrow \C^\times$ be a character. Then there is a positive function $\Psi$ on $G_\sfss$ with the following properties:
\begin{itemize}
\item
 $\Psi$ is $(1-2\abs{\nu_\chi}-\frac{\abs{\sfss}}{2})$-bounded if $\star\in \{B,C,D, \widetilde C\}$, and  $(2-2\abs{\nu_\chi}-\frac{\abs{\sfss}}{2})$-bounded if $\star\in \{C^*,D^*\}$;
 \item
for all $f\in I(\chi)$, $f'\in I(\chi^{-1})$ and $g\in G_\sfss$,
\be\label{psi123}
\abs{ \la g\cdot f, f'\ra}\leq \Psi(g) \cdot \abs{f|_{K_\sfss}}_\infty \cdot \abs{f'|_{K_\sfss}}_\infty\qquad(\textrm{$\abs{\,\,}_\infty$ stands for the supremum norm}).
\ee
\end{itemize}
\end{lem}
\begin{proof}

 Let $f_0$ denote the element in $I(\abs{\chi})$ such that $(f_0)|_{K_\sfss}=1$, and likewise let $f'_0$ denote the element in $I(\abs{\chi^{-1}})$ such that $(f'_0)|_{K_\sfss}=1$. Put
 \[
   \Psi(g):=\la g \cdot f_0, f_0'\ra, \qquad g\in G_\sfss.
 \]
 Then it is easy to see that \eqref{psi123} holds. Note that $\Psi$ is an elementary spherical function, and the lemma then follows by using the well-known estimate of the elementary spherical functions (\cite[Lemma 3.6.7]{Wa1}).
\end{proof}

\subsection{Matrix coefficient integrals against the oscillator representation}

We begin with the following lemma.

\begin{lem}\label{matrico}
 There exist continuous seminorms $\abs{\,\cdot\,}_{\mathsf s, \mathsf s'}$ and $\abs{\,\cdot\,}_{\mathsf s, \mathsf s'}^\vee$ on  $ \omega_{\mathsf s, \mathsf s'}$ and $\omega_{\mathsf s, \mathsf s'}^\vee$ respectively such that
\[
 \abs{ \la (g,g')\cdot u, v\ra}\leq \Psi_{\mathsf s}^{\abs{\mathsf s'}}(g)\cdot \Psi_{\mathsf s'}^{\abs{\mathsf s}}(g')\cdot \abs{u}_{\mathsf s, \mathsf s'}\cdot \abs{v}_{\mathsf s, \mathsf s'}^\vee
\]
for all $u\in \omega_{\mathsf s, \mathsf s'}$, $v\in \omega_{\mathsf s, \mathsf s'}^\vee$, and $(g,g')\in G_{\mathsf s}\times G_{\mathsf s'}$.
\end{lem}
\begin{proof}
  This follows from the  proof of \cite[Theorem 3.2]{Li89}.
\end{proof}

 \begin{defn}\label{defn:CRcov}
A Casselman-Wallach representation of $G_{\mathsf s'}$ is convergent for $\check \Theta_{\mathsf s'}^{\mathsf s}$
if it is $\nu$-bounded for some $\nu>\nu_{\mathsf s'}-\abs{\mathsf s}$.
\end{defn}

\begin{eg}
Suppose that $\star'\neq \widetilde C$. Then the trivial representation of $G_{\mathsf s'}$ is convergent for $\check \Theta_{\mathsf s'}^{\mathsf s}$ if $\abs{\sfss}>2\nu_{\sfss'}$.
\end{eg}

Let $\pi'$ be a  Casselman-Wallach representation of $G_{\mathsf s'}$ that is convergent for $\check \Theta_{\mathsf s'}^{\mathsf s}$.
Consider the integrals
\be\label{convint00}
\begin{array}{rcl}
 (\pi' \times \omega_{\mathsf s, \mathsf s'})\times (\pi'^\vee \times \omega_{\mathsf s, \mathsf s'}^\vee )&\rightarrow &\C, \smallskip \\
   ((u,v),(u',v')) &\mapsto &\int_{G_{\mathsf s'}} \la g\cdot u, u'\ra\cdot \la g\cdot v, v'\ra \od\! g.
   \end{array}
 \ee

\begin{lem}\label{lemconv}
The integrals in \eqref{convint00} are absolutely convergent and the map \eqref{convint00} is   continuous and multi-linear.
\end{lem}
\begin{proof}
This is a direct consequence of Lemmas \ref{boundpsi}, \ref{int} and \ref{matrico}.
\end{proof}

By Lemma \ref{lemconv}, the integrals in \eqref{convint00} yield a continuous bilinear form
\be\label{convint01}
 (\pi' \widehat \otimes \omega_{\mathsf s, \mathsf s'})\times (\pi'^\vee \widehat \otimes \omega_{\mathsf s, \mathsf s'}^\vee )\rightarrow \C.
 \ee
Put
\begin{equation}\label{thetab0}
  \Thetab_{\mathbf s'}^{\mathbf s}(\pi'):=\frac{\pi' \widehat \otimes \omega_{\mathsf s, \mathsf s'}}{\textrm{the left kernel of \eqref{convint01}}}.
\end{equation}

\begin{prop}\label{boundm}
The representation $\Thetab_{\mathbf s'}^{\mathbf s}(\pi')$ of $G_{\mathsf s}$ is a quotient of  $\check \Theta_{\mathbf s'}^{\mathbf s}(\pi')$, and is  $(\abs{\mathsf s'}-\nu_\sfss)$-bounded.
\end{prop}
\begin{proof}
Note that the bilinear form \eqref{convint01} is  $(G_{\mathsf s'}\times G_{\mathsf s'})$-invariant, as well as $G_{\mathsf s}$-invariant. Thus $\Thetab_{\mathbf s'}^{\mathbf s}(\pi')$ is a quotient of  $\check \Theta_{\mathbf s'}^{\mathbf s}(\pi')$, and is therefore a Casselman-Wallach representation. Its contragredient representation
is identified with
\[
(\Thetab_{\mathbf s'}^{\mathbf s}(\pi'))^\vee:=\frac{\pi'^\vee \widehat \otimes \omega^\vee_{\mathsf s, \mathsf s'}}{\textrm{the right kernel of \eqref{convint01}}}.
\]
 Lemmas \ref{int} and \ref{matrico} imply that there are continuous seminorms $\abs{\,\cdot\,}_{\pi', \mathsf s, \mathsf s'}$ and $\abs{\,\cdot\,}_{\pi'^\vee, \mathsf s, \mathsf s'}$ on  $\pi' \widehat \otimes \omega_{\mathsf s, \mathsf s'}$ and $\pi'^\vee \widehat \otimes \omega^\vee_{\mathsf s, \mathsf s'}$ respectively such that
\[
 \abs{ \la g\cdot u, v\ra}\leq \Psi_{\mathsf s}^{\abs{\mathsf s'}}(g)\cdot \abs{u}_{\pi', \mathsf s, \mathsf s'}\cdot \abs{v}_{\pi'^\vee, \mathsf s, \mathsf s'}
\]
for all $u\in \pi' \widehat \otimes \omega_{\mathsf s, \mathsf s'}$, $v\in \pi'^\vee \widehat \otimes \omega^\vee_{\mathsf s, \mathsf s'}$, and $g\in G_{\mathsf s}$.
The proposition  then easily follows in view of  Lemma \ref{boundpsi}.
 \end{proof}

\subsection{Unitarity}\label{unitarity}

For the notion of weak containment of unitary representations, see \cite{CHH} for example.

\begin{lem}\label{weaklycont}
Suppose that $\abs{\mathsf s}\geq  \nu_{\mathsf s'}$. Then as a unitary representation of $G_{\mathsf s'}$, $\hat \omega_{\mathsf s, \mathsf s'}$ is weakly  contained
in the regular representation.
\end{lem}
\begin{proof}
This has been known to experts (see   \cite[Theorem 3.2]{Li89}).  Lemmas \ref{boundpsi} and \ref{matrico} imply that for a dense subspace of $\hat \omega_{\mathsf s, \mathsf s'}|_{G_{\mathsf s'}}$, the diagonal  matrix coefficients are almost square-integrable.
 Thus the lemma follows from \cite[Theorem 1]{CHH}.
\end{proof}

However, if $\star'\in\{B, C^*, D^*\}$, $\abs{\mathsf s}\neq\nu_{\mathsf s'}$ due to  parity. For this reason, we introduce
\begin{equation}\label{nuplus}
   \nu_{\mathsf s'}^+:=\begin{cases}
    \nu_{\mathsf s'},\quad &\textrm{if }\star'\in \{C, D, \widetilde C\};\\
     \nu_{\mathsf s'}+1,\quad &\textrm{if }\star'\in \{B, C^*, D^*\}.\\
  \end{cases}
\end{equation}

The following definition is a slight variation of \Cref{defn:CRcov}.
 \begin{defn}\label{defn:CR33}
A Casselman-Wallach representation  of $G_{\mathsf s'}$ is overconvergent  for $\check \Theta_{\mathsf s'}^{\mathsf s}$ if
 it is  $\nu$-bounded  for some $\nu>\nu_{\mathsf s'}^+ -\abs{\mathsf s}$.
\end{defn}

We will prove the  following unitarity result in the rest of this subsection.
\begin{thm}\label{positivity000}
Assume that $\abs{\mathsf s}\geq \nu_{\mathsf s'}^+$. Let $\pi'$ be a Casselman-Wallach representation of $G_{\mathsf s'}$ that is overconvergent  for $\check \Theta_{\mathsf s'}^{\mathsf s}$. If $\pi'$ is unitarizable, then so is $\Thetab_{\mathbf s'}^{\mathbf s}(\pi')$.
\end{thm}

\begin{remark} Another general approach to unitarity in the context of theta lifting is due to He, using the notion of strongly semistable range (\cite[Chapter 5]{He} or \cite{He2}).
\end{remark}

Recall the following positivity result of matrix coefficient integrals, which is a special case of \cite[Theorem~A.5]{HLS}.

\begin{lem}\label{positivity}
Let $H$ be a real reductive group with a maximal compact subgroup $K_H$. Let $\pi_1$ and $\pi_2$ be two unitary representations of $H$ such that $\pi_2$ is weakly
contained in the regular representation. Let $u_1, u_2, \cdots, u_r$ ($r\in \bN$) be vectors in $\pi_1$ such that for all $i,j=1,2, \cdots, r$,
the integral
\[
  \int_H \la h\cdot u_i, u_j\ra\,\Xi_H (h) \od\!h 
\]
is absolutely convergent, where  $\Xi_H$ is the bi-$K_H$-invariant Harish-Chandra's $\Xi$ function on $H$, and $\od\!h$ is a Haar measure on $H$.   Let $v_1,v_2,\cdots, v_r$ be  $K_H$-finite vectors in $\pi_2$.
Put
\[
u:=\sum_{i=1}^r u_i\otimes v_i\in \pi_1\otimes \pi_2.
\]
Then the integral
\[
\int_H \langle h \cdot u,u \rangle\,\od\! h
\]
absolutely converges to a nonnegative real number.
\end{lem}

Now we come to the proof of Theorem \ref{positivity000}.

\begin{proof}[Proof of Theorem \ref{positivity000}]
Fix an invariant continuous Hermitian inner product on $\pi'$, and  write $\hat \pi'$ for the completion of $\pi'$ with respect to this Hermitian inner product.
The space $\pi' \widehat \otimes \omega_{\mathsf s, \mathsf s'}$ is equipped with the  inner product $\la\,,\,\ra$  that is the tensor product of the ones on $\pi'$ and $\omega_{\mathsf s, \mathsf s'} $. It suffices to show that
\[
  \int_{G_{\mathsf s'}}\la g \cdot u,u\ra\od\! g\geq 0
\]
for all $u$ in a dense subspace of $\pi' \widehat \otimes \omega_{\mathsf s, \mathsf s'}$.

If $\nu^+_{\mathsf s'}<0$, then $\star'\in \{D, D^*\}$ and $\abs{\mathsf s'}=0$. The theorem is trivial in this case.  Thus we assume that $\nu^+_{\mathsf s'}\geq 0$.
We also assume that $\star=B$. The proof in the other cases is similar and is omitted.

Note that $\nu^+_{\mathsf s'}=\nu_{\mathsf s'}=\abs{\mathsf s'}$ is even.
 Let $\mathsf s_1=(B, p_1, q_1)$ and  $\mathsf s_2=(D, p_2, q_2)$ be two classical signatures such that
\[
  (p_1, q_1)+(p_2, q_2)=(p,q)\quad \textrm{and}\quad \abs{\mathsf s_2}=\nu^+_{\mathsf s'}.
\]
View $\hat \omega_{\mathsf s_2, \mathsf s'}$ as a unitary representation of the symplectic group $G_{(C, p',q')}$.
Define $\pi_2$ to be its  pull-back through the covering homomorphism $G_{\mathsf s'}\rightarrow G_{(C, p',q')}$. By Lemma \ref{weaklycont}, the representation $\pi_2$ of $G_{\mathsf s'}$ is weakly contained in the regular representation.

Put
\[
  \pi_1:=\hat \pi'\widehat \otimes_{\mathrm h} (\hat \omega_{\mathsf s_1, \mathsf s'}|_{G_{\mathsf s'}})\qquad (\textrm{$\widehat \otimes_{\mathrm h}$ indicates the Hilbert space tensor product}).
\]
 Lemmas  \ref{int} and \ref{matrico} imply that the integral
\[
 \int_{G_{\mathsf s'}} \la g\cdot u,v\ra \cdot \Xi_{G_{\mathsf s'}}(g)\od\!g
\]
is absolutely convergent for all $u,v\in \pi' \otimes \omega_{\mathsf s_1, \mathsf s'}$. The theorem then follows by Lemma \ref{positivity}.
\end{proof}

\section{Double theta lifts and degenerate principal series}\label{sec:dtheta}

\def\GLE{\GL(\bfE)^{J_{\bfU}}}
\def\GLEz{\GL_{\bfE_0}}
\def\GLE{{\GL_{\bfE}}}
\def\wtGLE{\widetilde{\GLE}}
\def\wtGLEz{\widetilde{\GLEz}}
\def\wtPE{\widetilde{P_\bfE}}
\def\JU{{J_{\bfU}}}
\def\LU{{L_{\bfU}}}
\def\wtGU{\widetilde{G}_\bfU}
\def\dsfss{{\dot{\mathsf s}}}

The purpose of this section is to prove \Cref{doubtt}, which relates double theta lifts with degenerate principal series representations. The underlying idea (due to He; see \cite[Chapter 4]{He}) is that one can extract information on a theta lift from the knowledge of double theta lifts, and thus it should be viewed as a variant of the doubling method (\cite[Section II]{Ra}).  \Cref{doubtt} will be used in \Cref{sec:equac} to determine the associated cycle of a special unipotent representation.   

Let
\[
\dot{ \mathsf s}=(\dot \star, \dot p,\dot q),\qquad  \mathsf s'=(\star', p',q'),\qquad \textrm{and}\qquad \sfss''=(\star'', p'', q'')
\]
be classical signatures such that
\begin{itemize}
  \item $ (q', p')+(p'', q'')=(\dot p, \dot q)$;\smallskip
  \item if $\dot \star\in \{B,D\}$,  then $\star', \star''\in \{B,D\}$; \smallskip
  \item if $\dot \star\notin \{B,D\}$,  then $\star'=\star''=\dot \star$.
\end{itemize}

As in Section \ref{secoscil}, put $\sfss'^-:=(\star', q', p')$. We view $V_{\sfss'^-}$ and $V_{\sfss''}$ as subspaces of $V_\dsfss$ such that $(\la\,,\,\ra_{\dsfss}, J_{\dsfss},L_{\dsfss})$  extends both $(\la\,,\,\ra_{\sfss'^-}, J_{\sfss'^-}, L_{\sfss'^-})$ and
$(\la\,,\,\ra_{\mathsf s''}, J_{\mathsf s''}, L_{\mathsf s''})$, and $V_{\sfss'^-}$ and $V_{\sfss''}$ are perpendicular to each other under the form $\la\,,\,\ra_{\dsfss}$. Then we have an orthogonal decomposition
\[
  V_\dsfss=V_{\sfss'^-}\oplus V_{\sfss''},
\]
and both $G_{\sfss'^-}$ and $G_{\sfss''}$ are identified with subgroups of $G_{\dsfss}$.

Fix a linear isomorphism
\be\label{iota1}
 \iota_{\sfss'}:  V_{\sfss'}\rightarrow V_{\sfss'^-}
\ee
as in \eqref{iota0}, which induces an isomorphism
\be\label{isosfss1}
 G_{\sfss'}\rightarrow G_{\sfss'^-}, \quad g\mapsto g^-.
\ee
Let $\pi'$ be a Casselman-Wallach representation of $G_{\sfss'}$. Write $\pi'^-$ for the representation of $G_{\sfss'^-}$ that corresponds to $\pi'$ under the isomorphism \eqref{isosfss1}.

\subsection{Matrix coefficient integrals and double theta lifts}

We begin with the following lemma.

\begin{lem}\label{boundxx}
The function $(\Xi_\dsfss)|_{G_{\sfss'^-}}$ on $G_{\sfss'^-}$  is   $\abs{\sfss''}$-bounded.
\end{lem}
\begin{proof}
This follows from the  estimate of Harish-Chandra's $\Xi$ function (\cite[Theorem 4.5.3]{Wa1}).
\end{proof}

 \begin{defn}\label{defn:CR}
The Casselman-Wallach representation  $\pi'^-$ of $G_{\mathsf s'^-}$ is convergent  for a Casselman-Wallach representation  $\dot \pi$ of $G_{\dsfss}$ if there are real numbers $\nu'$ and $\dot \nu$ such that $\pi'$ is $\nu'$-bounded, $\dot \pi$ is $\dot \nu$-bounded, and
\[
  \nu'+\dot \nu>-\abs{\sfss''}.
\]
\end{defn}

Let $\dot \pi$  be a Casselman-Wallach representation of $G_{\dsfss}$ such that $\pi'^-$  is convergent  for  $\dot \pi$.
Consider the integrals
\be\label{convint0004}
\begin{array}{rcl}
 (\pi'^- \times \dot \pi)\times ((\pi'^-)^\vee \times \dot \pi^\vee )&\rightarrow &\C, \smallskip \\
   ((u,v),(u',v')) &\mapsto &\int_{G_{\mathsf s'}} \la g^-\cdot u, u'\ra\cdot \la g^-\cdot v, v'\ra \od\! g.
   \end{array}
 \ee

\begin{lem}\label{intpi0004}
The integrals in \eqref{convint0004} are absolutely convergent and the map \eqref{convint0004} is   continuous and multi-linear.
\end{lem}
\begin{proof}
Note that
$
  (\Psi_{\dsfss})|_{G_{\sfss'^-}}=\Psi_{\sfss'^-}.
$
Thus the lemma follows from Lemmas \ref{int} and \ref{boundxx}.
\end{proof}

By Lemma \ref{intpi0004}, the integrals in \eqref{convint0004} yield a continuous bilinear form
\be\label{convint0011}
 \la\,,\,\ra: (\pi'^- \widehat \otimes{ \dot \pi})\times ((\pi'^-)^\vee \widehat \otimes {\dot \pi}^\vee )\rightarrow \C.
 \ee
Put
\begin{equation}\label{theta.gen}
 \pi'^-*\dot \pi:=\frac{\pi'^-\widehat \otimes \dot \pi}{\textrm{the left kernel of \eqref{convint0011}}}.
\end{equation}
This is a smooth Fr\'echet representation of $G_{\sfss''}$ of moderate growth.

For every  classical signature $\sfss=(\star, p,q)$, put
\be\label{bsfss}
  [\sfss]:=\begin{cases}
    (C, p,q), \quad & \textrm{if $\star=\widetilde C$};\\
    \sfss, \quad & \textrm{if $\star\neq \widetilde C$}.
  \end{cases}
\ee

\begin{prop}\label{doublelift}
Suppose that $\sfss=(\star, p,q)$ is a classical signature such that both  $\star'$ and $\star''$ equal the Howe dual of $\star$,  and $2\nu_\sfss<\abs{\dsfss}$.
Assume that $\pi'$  is $\nu'$-bounded for some $\nu'>\nu_{\sfss'}-\abs{\sfss}$.
Then
\begin{itemize}
\item $\pi'$ is convergent for $\check \Theta^{\sfss}_{\sfss'}$;
 \item $\Thetab^{\sfss}_{\sfss'}(\pi')$  is convergent for $\check \Theta^{\sfss''}_{\sfss}$;
 \item
  the trivial representation
$1$ of $G_{[\sfss]}$ is convergent for $\check \Theta^{\dsfss}_{[\sfss]}$;
\item
  $\pi'^-$ is convergent for  $\Thetab_{[\sfss]}^\dsfss(1)$;
 \item as representations of $G_{\sfss''}$,
\begin{equation}
\label{thetabv00}
  \Thetab^{\sfss''}_{\sfss}(\Thetab^{\sfss}_{\sfss'}(\pi'))\cong \pi'^-* \Thetab_{[\sfss]}^\dsfss(1).
\end{equation}
\end{itemize}
\end{prop}
\begin{proof}
The first four claims in the proposition are obvious. We only need to prove the last one.
Note that the integrals in
\begin{equation}\label{intt00}
\begin{array}{rcl}
   (\pi'\totimes \omega_{\sfss,\sfss'}\totimes\omega_{\sfss'', \sfss })\times
    ((\pi')^\vee \totimes\omega_{\sfss,\sfss'}^\vee\totimes \omega_{\sfss'', \sfss }^\vee)&\rightarrow &\C,\\
    (u,v)&\mapsto &\int_{G_{\sfss'}\times G_\sfss} \inn{(g',g)\cdot u}{v}\od\! g'\, \od\! g
    \end{array}\end{equation}
are absolutely convergent and define a continuous bilinear map. (Here $G_{\sfss'}$ acts trivially on $\omega_{\sfss'', \sfss }$.) Also note that
\[
  \omega_{\sfss,\sfss'}\totimes\omega_{\sfss'', \sfss }\cong \omega_{\dsfss,\sfss}.
\]
In view of Fubini's theorem, the proposition follows as both sides of \eqref{thetabv00} are isomorphic to the quotient of $\pi'\totimes \omega_{\sfss,\sfss'}\totimes\omega_{\sfss'', \sfss }$ by the left kernel of the pairing \eqref{intt00}.
\end{proof}

\subsection{Matrix coefficient integrals against degenerate principal series}
\label{sec:DP}

We are particularly interested in the case when $\dot \pi$ is a degenerate principal series representation.
Suppose that $\dsfss$ is split,  and
\[
  p'\leq p''\qquad\textrm{and}\qquad q'\leq q''.
\]
Then $\star'=\star''$, and there is a  split classical signature $\sfss_0=(\star_0, \kappa, \kappa)$ ($\kappa\in \BN$) such that
\begin{itemize}
\item $\star_0=\dot \star$;
\item $(p', q')+(\kappa, \kappa )=(p'', q'')$.
\end{itemize}
We view $V_{\sfss'}$ and $V_{\sfss_0}$ as subspaces of $V_{\sfss''}$ 
such that $(\inn{}{}_{\sfss''}, J_{\sfss''}, L_{\sfss''})$  
extends  both $(\inn{}{}_{\mathsf s'}, J_{\mathsf s'}, L_{\mathsf s'})$ 
and $(\inn{}{}_{\sfss_0}, J_{\sfss_0},L_{\sfss_0})$, 
and  $V_{\sfss'}$  and $V_{\sfss_0}$ are perpendicular 
to each other under the form $\inn{}{}_{\sfss''}$.

Put
\[
   V_{\sfss'}^\triangle:=\{\iota_{\sfss'}(v)+v\in V_\dsfss \,:\, v\in V_{\sfss'}\}\quad\textrm{and}\quad V_{\sfss'}^\nabla:=\{\iota_{\sfss'}(v)-v\in V_\dsfss \,:\, v\in V_{\sfss'}\}.
 \]
As before, we have that
\[
  V_{\sfss_0}=X_{\sfss_0}\oplus Y_{\sfss_0},
\]
where $X_{\sfss_0}$ is a maximal $J_{\sfss_0}$-stable totally isotropic subspace of $V_{\sfss_0}$, and $Y_{\sfss_0}:=L_{\sfss_0}(X_{\sfss_0})$.
Suppose that
 \[
 X_{\dsfss}= V_{\sfss'}^\triangle \oplus X_{\sfss_0}\quad\textrm{and}\quad Y_{\dsfss}=V_{\sfss'}^\nabla \oplus Y_{\sfss_0}.
 \]
In summary, we have decompositions
\be\label{decomv0}
  V_{\dsfss}=V_{\sfss'^-}\oplus V_{\sfss''}= V_{\sfss'^-}\oplus V_{\sfss'}\oplus  V_{\sfss_0}= (V_{\sfss'}^\triangle \oplus X_{\sfss_0})\oplus (V_{\sfss'}^\nabla \oplus Y_{\sfss_0}).
\ee

As before, $X_\dsfss$ and $X_{\sfss_0}$ yield
 the Siegel parabolic subgroups
 \[
  P_{\dsfss}=R_{\dsfss}\ltimes N_\dsfss\subset G_\dsfss\qquad \textrm{and}\qquad P_{\sfss_0}=R_{\sfss_0}\ltimes N_{\sfss_0}\subset G_{\sfss_0}.
 \]
Write
\be\label{parabolic}
  P_{\sfss'',\sfss_0}=R_{\sfss'',\sfss_0}\ltimes N_{\sfss'',\sfss_0}
\ee
for the parabolic subgroup of $G_{\sfss''}$ stabilizing $X_{\sfss_0}$, where $R_{\sfss'',\sfss_0}$ is the Levi subgroup stabilizing both $X_{\sfss_0}$ and $Y_{\sfss_0}$, and $N_{\sfss'',\sfss_0}$ is the unipotent radical. We have an obvious homomorphism
\be  \label{parabolic2}
  G_{\sfss'}\times R_{\sfss_0}\rightarrow R_{\sfss'',\sfss_0},
\ee
which is a two-fold covering map when $\dot \star=\widetilde C$, and an isomorphism in the other cases.
For every $g\in G_{\sfss'}$, write $g^\triangle:= g^- g$ (see \eqref{isosfss1}), which is an element of $R_{\dsfss}$.

Let $\dot \chi: R_{\dsfss}\rightarrow \C^\times $ be a character. It yields a degenerate principal series representation
\[
I(\dot \chi):=I_{\dsfss}(\dot \chi):=\Ind_{P_\dsfss}^{G_\dsfss} \dot \chi
\]
 of $G_\dsfss$.
Write
\be\label{chi0}
  \chi_0:=\dot \chi|_{R_{\sfss_0}},
\ee
and define a character
\be\label{chip}
 \chi':  G_{\sfss'}\rightarrow \C^\times, \quad g\mapsto \dot \chi(g^\triangle) \qquad(\textrm{this is a quadratic character}).
\ee

Recall the representations  $\pi'$  of $G_{\sfss'}$ and
$\pi'^-$ of $G_{\sfss'^-}$. If $\dot \star =\widetilde C$, we assume that both $\pi'$ and $\dot \chi$ are genuine.  Then $(\pi'\otimes \chi') \otimes \chi_0$ descends to a Casselman-Wallach representation of $R_{\sfss'',\sfss_0}$. View it as a representation of $P_{\sfss'',\sfss_0}$ via the trivial action of $N_{\sfss'',\sfss_0}$, and form the representation  $\Ind_{P_{\sfss'',\sfss_0}}^{G_{\sfss''}} ((\pi'\otimes \chi') \otimes \chi_0)$.

Let $\nu_{\dot \chi}\in \R$ be  as in \eqref{nuchi}. The rest of this subsection is devoted to a proof of the following theorem.

\begin{thm}\label{midp}
Assume that the Casselman-Wallach representation $\pi'$ of $G_{\sfss'}$ is $\nu'$-bounded for some
\[
 \nu'>\begin{cases}
  2\abs{\nu_{\dot \chi}}- \frac{\abs{\sfss_0}}{2}-1,\quad &\textrm{if $\dot \star\in \{B,C,D, \widetilde C\}$};\\
   2\abs{\nu_{\dot \chi}}- \frac{\abs{\sfss_0}}{2}-2,\quad & \textrm{if $\dot \star\in \{C^*,D^*\}$}.
 \end{cases}
\]
 Then $\pi'^-$ is convergent for $I_{\dsfss}(\dot \chi)$, and
\[
 \pi'^-*I(\dot \chi)\cong \Ind_{P_{\sfss'',\sfss_0}}^{G_{\sfss''}} ((\pi'\otimes \chi')\otimes \chi_0).
\]
\end{thm}

Let the notation and assumptions be as in Theorem \ref{midp}.
Recall that the representations   $I(\dot \chi)$ and $ I(\dot \chi^{-1})$ are contragredients of each other with the $G_\dsfss$-invariant pairing
 \[
  \la\,,\,\ra:  I(\dot \chi)\times I(\dot \chi^{-1})\rightarrow \C, \quad (f, f')\mapsto \int_{K_\dsfss} f(g) \cdot f'(g) \od\!g.
 \]

The first assertion of Theorem \ref{midp} is a direct consequence of Lemma \ref{growthdp}. As in \eqref{convint0011},  we have a continuous bilinear form
\be\label{convint00112}
\begin{array}{rcl}
(\pi'^- \widehat \otimes{ I(\dot \chi)})\times ((\pi'^-)^\vee \widehat \otimes I(\dot \chi^{-1}) )&\rightarrow &\C,\\
 (u,v)&\mapsto& \int_{G_{\sfss'}} \la g^-\cdot u,v\ra \od\! g
 \end{array}
 \ee
so that
\begin{equation*}
 \pi'^-*I(\dot \chi):=\frac{\pi'^-\widehat \otimes I(\dot \chi)}{\textrm{the left kernel of \eqref{convint00112}}}.
\end{equation*}

Note that
\[
\G_\dsfss^\circ:= P_{\dsfss}\cdot G_{\sfss''}
\] is open and dense in $G_\dsfss$, and its complement has measure zero in $G_\dsfss$. Moreover,
\be\label{opencell}
P_{\dsfss}\backslash \G_\dsfss^\circ= (R_{\sfss_0}\ltimes N_{\sfss'', \sfss_0})\backslash G_{\sfss''}.
\ee
Form the normalized Schwartz induction
\[
   I^\circ(\dot \chi):=\ind_{R_{\sfss_0}\ltimes N_{\sfss'', \sfss_0}}^{G_{\sfss''}} \chi_0.
\]
The reader is referred to \cite[Section 6.2]{CS21} for the general notion of Schwartz inductions (in a slightly different unnormalized setting). Similarly, put
\[
   I^\circ(\dot \chi^{-1}):=\ind_{R_{\sfss_0}\ltimes N_{\sfss'', \sfss_0}}^{G_{\sfss''}}\chi_0^{-1}.
\]
Note that
\be\label{modulus}
\textrm{the modulus character of $P_\dsfss$ restricts to the modulus character of  $R_{\sfss_0}\ltimes N_{\sfss'', \sfss_0}$.}
\ee
In view of \eqref{opencell} and \eqref{modulus}, by extension by zero, $I^\circ(\dot \chi)$ is viewed as a closed subspace of $ I(\dot \chi)$, and $ I^\circ(\dot \chi^{-1})$  is viewed as a closed subspace of $ I(\dot \chi^{-1})$. These two closed subspaces are  $(G_{\sfss'}\times G_{\sfss''})$-stable.

Similar to \eqref{convint00112}, we have a continuous bilinear form
\be\label{convint001123}
\begin{array}{rcl}
  (\pi'^- \widehat \otimes{ I^\circ (\dot \chi)})\times ((\pi'^-)^\vee \widehat \otimes I^\circ (\dot \chi^{-1}) )&\rightarrow &\C,\\
 (u,v)&\mapsto& \int_{G_{\sfss'}} \la g^-\cdot u,v\ra \od\! g.
 \end{array}
 \ee
Put
\begin{equation*}
 \pi'^-*I^\circ (\dot \chi):=\frac{\pi'^- \widehat \otimes I^\circ (\dot \chi)}{\textrm{the left kernel of \eqref{convint001123}}},
\end{equation*}
which is still a smooth Fr\'echet representation of $G_{\sfss''}$ of moderate growth.

\begin{lem}\label{isoipi}
 As representations of $G_{\sfss''}$,
 \[
   \pi'^-*I^\circ (\dot \chi)\cong \Ind_{P_{\sfss'',\sfss_0}}^{G_{\sfss''}} ((\pi'\otimes  \chi')\otimes \chi_0).
   \]
\end{lem}
\begin{proof}
As Fr\'echet spaces, $\pi'^-$ is obviously identified with $\pi'$.
Note that the natural map $G_{\sfss'}\rightarrow (R_{\sfss_0}\ltimes N_{\sfss'', \sfss_0})\backslash G_{\sfss''}$ is proper and hence the following integrals are absolutely convergent and yield a $G_{\sfss''}$-equivariant continuous linear map:
\[
\begin{array}{rcl}
\xi:   \pi'^-\widehat \otimes I^\circ (\dot \chi)&\rightarrow & \Ind_{P_{\sfss'',\sfss_0}}^{G_{\sfss''}} ((\pi'\otimes \chi')\otimes \chi_0),\\
   v\otimes f &\mapsto  & \left(h\mapsto\int_{G_{\sfss'}}\chi'(g)\cdot  f(g^{-1}h)(g\cdot v)\od\! g\right).
   \end{array}
\]
Moreover, this map is surjective (\cf \cite[Section 6.2]{CS21}). It is thus open by the open mapping theorem. Similarly, we have an open surjective
$G_{\sfss''}$-equivariant  continuous linear map
\[
\begin{array}{rcl}
\xi':   (\pi'^-)^\vee\widehat \otimes I^\circ (\dot \chi^{-1})&\rightarrow & \Ind_{P_{\sfss'',\sfss_0}}^{G_{\sfss''}} ((\pi'^\vee\otimes \chi')\otimes \chi_0^{-1}),\\
   v'\otimes f' &\mapsto  & \left(h\mapsto\int_{G_{\sfss'}} \chi'(g)\cdot  f'(g^{-1}h)(g\cdot v')\od\! g\right).
   \end{array}
\]

For all $v\otimes f \in \pi'^-\widehat \otimes I^\circ (\dot \chi)$ and $v'\otimes f' \in (\pi'^-)^\vee \widehat \otimes I^\circ (\dot \chi^{-1})$, we have that
\begin{eqnarray*}
 && \int_{G_{\sfss'}}\la g^-\cdot (v\otimes f), v'\otimes f'\ra \od\! g\\
 &=& \int_{G_{\sfss'}}\la g^-\cdot v, v'\ra \cdot \la g^-\cdot f, f'\ra \od\! g\\
 &=& \int_{G_{\sfss'}}\la g^-\cdot v, v'\ra \cdot \int_{K_{\sfss''}} \int_{G_{\sfss'}}   (g^-\cdot f)(hx) \cdot  f'(hx)  \od\! h \od \! x \od \! g\\
 &=& \int_{G_{\sfss'}}\la g^-\cdot v, v'\ra \cdot \int_{K_{\sfss''}} \int_{G_{\sfss'}}  \chi'(g) \cdot f(g^{-1}h x) \cdot  f'(hx)  \od\! h \od \! x \od \! g\\
  &=& \int_{K_{\sfss''}}  \int_{G_{\sfss'}}  \int_{G_{\sfss'}} \la (h^- (g^-)^{-1})\cdot v, v'\ra \cdot \chi'(hg^{-1}) \cdot f(gx) \cdot  f'(hx)  \od\! g \od \! h \od \! x \\
 &=&  \int_{K_{\sfss''}}\la  \xi(v\otimes f)(x),\xi'(v'\otimes f')(x)\ra  \od \! x\\
 &=& \la  \xi(v\otimes f),\xi'(v'\otimes f')\ra.
\end{eqnarray*}
This implies the lemma.
\end{proof}

\begin{lem}\label{imb}
Let $u\in \pi'^- \widehat \otimes{ I(\dot \chi)}$. Assume that
\begin{equation}\label{intguv0}
  \int_{G_{\sfss'}} \la g^-\cdot u, v\ra\od\! g =0
\end{equation}
for all $v\in (\pi'^-)^\vee \widehat \otimes I^\circ (\dot \chi^{-1})$. Then \eqref{intguv0} also holds for all $v\in (\pi'^-)^\vee \widehat \otimes I(\dot \chi^{-1})$.
\end{lem}

\begin{proof}
Take a sequence $(\eta_1, \eta_2, \eta_3, \cdots)$ of real-valued smooth functions on $P_{\dsfss}\backslash G_{\dsfss}$ such that
\begin{itemize}
\item
for all $i\geq 1$,  the support of $\eta_i$ is contained in $P_\dsfss \backslash  G_\dsfss^\circ$;
\item for all $i\geq 1$ and  $x\in  P_{\dsfss}\backslash G_{\dsfss}$, $ 0\leq  \eta_i(x)\leq \eta_{i+1}(x)\leq 1$; \smallskip
 \item
 $\bigcup_{i=1}^\infty \eta_i^{-1}(1)= P_{\dsfss}\backslash G_{\dsfss}^\circ$.
\end{itemize}
Let $v\in  (\pi'^-)^\vee \widehat \otimes I (\dot \chi^{-1})$. Note that $\eta_i  I (\dot \chi^{-1})\subset  I^\circ (\dot \chi^{-1})$. Thus $\eta_i v\in (\pi'^-)^\vee \widehat \otimes  I^\circ (\dot \chi^{-1})$.

Lemma \ref{growthdp} and Lebesgue's dominated convergence theorem imply that
\[
  \int_{G_{\sfss'}} \la g^-\cdot u, v\ra\od\! g
  =\lim_{i\rightarrow +\infty}  \int_{G_{\sfss'}} \la g^-\cdot u, \eta_i v\ra \od\! g =0.
\]
This proves the lemma.
\end{proof}

Lemma \ref{imb} implies that we have a natural continuous linear map
\[
 \Ind_{P_{\sfss'',\sfss_0}}^{G_{\sfss''}} ((\pi'\otimes  \chi')\otimes \chi_0)= \pi'^-* I^\circ (\dot \chi)\rightarrow \pi'^-* I(\dot \chi).
\]
This map is clearly injective and $G_{\sfss''}$-equivariant. Similarly to Lemma \ref{imb}, we know that the natural pairing
\[
 ( \pi'^-* I(\dot \chi))\times( (\pi'^-)^\vee* I^\circ(\dot \chi^{-1}))\rightarrow \C
\]
is well-defined and non-degenerate. Thus Theorem \ref{midp} follows by the following lemma.

\begin{lem}\label{imb2}
Let $H$ be a real reductive group. Let $\pi$ be a Casselman-Wallach representation of $H$, and let $\tilde \pi$ be a smooth Fr\'echet representation of $H$ with an $H$-equivariant injective continuous linear map
\[
  \phi: \pi\rightarrow \tilde \pi.
\]
Assume that  there is a non-degenerate $H$-invariant continuous bilinear map
\[
\la\,,\,\ra: \tilde \pi\times \pi^\vee\rightarrow \C,
\]
such that the composition of
\[
 \pi\times \pi^\vee\xrightarrow{(u,v)\mapsto (\phi(u),v)} \tilde \pi\times \pi^\vee\xrightarrow{\la\,,\,\ra} \C
\]
is the natural pairing. Then $\phi$ is a topological isomorphism.
\end{lem}
\begin{proof}
 Let $u\in \tilde \pi$.
 Using the theorem of Dixmier-Malliavin (\cite[Theorem 3.3]{DM}), we write
 \[
   u=\sum_{i=1}^{s} \int_H \varphi_i(h) ( h\cdot u_i)\od\! h \qquad (s\in \bN, \ u_i\in\tilde \pi),
 \]
 where $\varphi_i$'s are compactly supported smooth functions on $H$.  As a continuous linear functional on $\pi^\vee$, we have that
  \[
   \la u,\, \cdot\, \ra=\sum_{i=1}^{s} \int_H \varphi_i(h)  \cdot \la h\cdot u_i, \,\cdot\,\ra \od\! h.
 \]
By \cite[Lemma 3.5]{SZ1},  the right-hand side  functional  equals  $\la u_0, \,\cdot\,\ra$ for a unique $u_0\in \pi$. Thus $u=\phi(u_0)$, and the lemma follows by the open mapping theorem.
\end{proof}

\subsection{Double theta lifts and parabolic induction}\label{doublep}
Let $\star$ be the Howe dual of $\star'=\star''$. Let $k\in \bN$ and we assume that the set
\be\label{sstark}
 S_{\star, k}:=\{\textrm{classical signatures of the form $(\star, p_1,q_1)$ with $p_1+q_1=k$}\}
\ee
is nonempty. Thus $k$ is odd if $\star=B$, and $k$ is even in all other cases.
In this subsection, we further assume that $k\geq 2$ and
\be\label{dotp}
 \dot p=\dot q= \begin{cases}
 k-1,  &\quad \textrm{if } \star\in \{B,D\};\\
k+1,  &\quad  \textrm{if $\star\in \{C, \widetilde C\}$;}\\
k, & \quad  \textrm{if $ \star\in\{C^*,D^*\}$.}
\end{cases}
 \ee

We also assume that the character $\dot \chi: R_{\dsfss}\rightarrow \C^\times $
satisfies  the following conditions:
\be\label{chid}
\begin{cases}
 \dot \chi \textrm{ is genuine and $\dot \chi^4=1$,}& \textrm{if $ \star=B$};\\
 \dot \chi^2=1,  & \textrm{if $\star=D$;}\\
 \dot \chi=1,\quad   &\textrm{if } \star\in \{C,\widetilde C\};\\
    \dot \chi^2\textrm{ equals the composition of $R_\dsfss\xrightarrow{\textrm{natural map}} \GL(X_\dsfss)\xrightarrow{\det}\C^\times$},  & \textrm{if $ \star=C^*$;}\\
 \dot \chi^2\textrm{ equals the composition of $R_\dsfss\xrightarrow{\textrm{natural map}} \GL(Y_\dsfss)\xrightarrow{\det}\C^\times$},  & \textrm{if $ \star=D^*$.}
\end{cases}
\ee
If  $\star\notin \{B,D\}$, the character   $\dot{\chi}$ is uniquely determined by \eqref{chid}.
If  $\star\in \{B,D\}$, there are two characters satisfying \eqref{chid}, and we let $\dot \chi'$ denote the one other than $\dot \chi$.

The relationship between the degenerate principal series representation $I(\dot{\chi})$ and
Rallis quotients is summarized in the following lemma. (The corresponding result for unitary groups is in \cite[Introduction]{LZ1}, which is relevant for \Cref{subsec:unitary}.) 
The idea originated in \cite[Section II]{Ra} (see also \cite{KR}).

\begin{lem}\label{degens}
\noindent
(a) If $\star\in \{B,D\}$,  then
\[
  I(\dot{\chi}) \oplus I(\dot \chi')\cong \bigoplus_{\sfss_1\in S_{\star, k}}
   \check \Theta_{\sfss_1}^{\dsfss}(1).
\]

\noindent
(b) If $\star\in \{C,\widetilde C\}$, then
\[
  I(\dot{\chi})\cong \check \Theta_{\sfss_1}^{\dsfss}(1)\oplus  (\check \Theta_{\sfss_1}^{\dsfss}(1)\otimes \det),\qquad \textrm{where } \ \mathsf s_1=(C, \frac{k}{2}, \frac{k}{2}).
\]

\noindent
(c) If $\star=D^*$, then
\[
  I(\dot{\chi})\cong \check \Theta_{\sfss_1}^{\dsfss}(1),\qquad \textrm{where }\ \mathsf s_1=(D^*,  \frac{k}{2}, \frac{k}{2}).
\]

\noindent
(d) If $\star=C^*$, then there is an exact sequence of representations of $G_{\dsfss}$:
\[
0\rightarrow \bigoplus_{\mathsf s_1\in S_{\star, k}} \check \Theta_{\sfss_1}^{\dsfss}(1) \rightarrow
 I(\dot{\chi})\rightarrow  \bigoplus_{\mathsf s_2\in S_{\star, k-2}} \check \Theta_{\sfss_2}^{\dsfss}(1) \rightarrow 0.
\]

\noindent
(e) All the representations $ \check \Theta_{\sfss_1}^{\dsfss}(1) $ and $ \check \Theta_{\sfss_2}^{\dsfss}(1)$ appearing in (a), (b), (c) and (d) are irreducible and unitarizable.

\end{lem}
\begin{proof}
 See \cite[Theorem 2.4]{Ku}, \cite[Introduction]{LZ2}, \cite[Theorem 6.1]{LZ3} and  \cite[Sections 9 and 10]{Ya}.
\end{proof}

\begin{lem}\label{lem:coinv}
For all $\mathsf s_1$ appearing in Lemma \ref{degens}, the trivial representation $1$ of $G_{\mathsf s_1}$ is overconvergent for $\check \Theta_{\mathsf s_1}^{\dsfss}$, and
\[
   \Thetab_{\mathsf s_1}^{\dsfss}(1)=\check \Theta_{\mathsf s_1}^{\dsfss}(1).
\]

\end{lem}
\begin{proof}
Note that the trivial representation $1$ of $G_{\mathsf s_1}$ is $(-\nu_{\mathsf s_1})$-bounded, and $-\nu_{\mathsf s_1}>\nu_{\mathsf s_1}^+ -\abs{\dsfss}$. This implies the first assertion.
Since $\check \Theta_{\mathsf s_1}^{\dsfss}(1)$ is irreducible and $\Thetab_{\mathsf s_1}^{\dsfss}(1)$ is  a quotient of $\check \Theta_{\mathsf s_1}^{\dsfss}(1)$, for the proof of the second assertion, it suffices to show that $\Thetab_{\mathsf s_1}^{\dsfss}(1)$ is nonzero.

Put
\[
  V_{\sfss_1}(\R):=\begin{cases}
    \textrm{the fixed point set  of  $J_{\sfss_1}$ in $V_{\sfss_1}$},\quad &\textrm{if $\star\in \{B,C, D, \widetilde C\}$};\\
    V_{\sfss_1},\quad & \textrm{if $\star\in \{C^*, D^*\}$}.
  \end{cases}
\]
As usual, realize $\hat \omega_{\dsfss, \sfss_1}$ on the space of square-integrable functions on $(V_{\sfss_1}(\R))^{\dot p}$ so that $ \omega_{\dsfss, \sfss_1}$ is identified with the space of the Schwartz functions, and $G_{\sfss_1}$ acts on it through the obvious transformation.

Take a positive-valued Schwartz function $\phi$ on
$(V_{\sfss_1}(\R))^{\dot p}$. Then
\[
  \inn{ g\cdot \phi}{ \phi}=\int_{(V_{\sfss_1}(\R))^{\dot p}} \phi(g^{-1}\cdot x) \cdot \phi(x)
  \od \! x>0, \quad \textrm{for all }g\in G_{\sfss_1}.
\]
Thus
\[
  \int_{G_{\sfss_1}} \inn{ g \cdot \phi}{ \phi} \od\! g \neq 0,
\]
and the lemma follows.
\end{proof}

\begin{lem}\label{doublelift5}
Assume that $\pi'$ is $\nu'$-bounded for some
\[
  \nu'>
  \begin{cases}
 -\frac{\abs{\sfss_0}}{2}-3,\quad& \textrm{if $\star =C^*$};\\
   -\frac{\abs{\sfss_0}}{2}-1,\quad& \textrm{otherwise}.
   \end{cases}
   \]
 Then for all $\sfss_1$ appearing in Lemma \ref{degens} and all unitary characters $\gamma'$ of $G_{\sfss'}$,  the representation $\pi'\otimes \gamma'$ is convergent for $\check \Theta^{\sfss_1}_{\sfss'}$, and
 $\Thetab^{\sfss_1}_{\sfss'}(\pi'\otimes \gamma')$ is overconvergent for $\check \Theta^{\sfss''}_{\sfss_1}$.
 \end{lem}
\begin{proof}
The first assertion follows by noting that
\be\label{p123}
\nu_{\sfss'}-\abs{\sfss_1}=
  \begin{cases}
 -\frac{\abs{\sfss_0}}{2}-3,\quad& \textrm{if $\star =C^*$};\\
   -\frac{\abs{\sfss_0}}{2}-1,\quad& \textrm{otherwise}.
   \end{cases}
   \ee
The proof of the second assertion is similar to the proof of the first assertion of Lemma \ref{lem:coinv}.
\end{proof}

Recall the character
$ \chi_0:=\dot \chi|_{R_{\sfss_0}}$ from \eqref{chi0}. Similarly we define $\chi'_0:=\dot \chi'|_{R_{\sfss_0}}$.
Recall that $\pi'$ is assumed to be genuine when $\dot \star=\widetilde C$.

\begin{thm}\label{doubtt}
Assume that $\pi'$  is $\nu'$-bounded for some
$
  \nu'>
 -\frac{\abs{\sfss_0}}{2}-1.
$

\noindent
(a) If $\star\in \{B,D\}$,  then
\[
    \bigoplus_{\sfss_1\in S_{\star, k}}   \Thetab^{\sfss''}_{\sfss_1}(\Thetab^{\sfss_1}_{\sfss'}(\pi'))\cong \Ind_{P_{\sfss'',\sfss_0}}^{G_{\sfss''}} (\pi'\otimes \chi_0)
    \oplus  \Ind_{P_{\sfss'',\sfss_0}}^{G_{\sfss''}} (\pi' \otimes \chi'_0).
    \]

\smallskip

\noindent
(b) If $\star\in \{C,\widetilde C\}$, then
\[
 \Thetab^{\sfss''}_{\sfss_1}(\Thetab^{\sfss_1}_{\sfss'}(\pi'))\oplus\left(  (\Thetab^{\sfss''}_{\sfss_1}(\Thetab^{\sfss_1}_{\sfss'}(\pi'\otimes \det)))\otimes \det \right)\cong \Ind_{P_{\sfss'',\sfss_0}}^{G_{\sfss''}} (\pi'\otimes \chi_0),
\]
where $\sfss_1=(\star, \frac{k}{2},\frac{k}{2})$.
  \smallskip

\noindent
(c) If $\star=D^*$, then
\[
 \Thetab^{\sfss''}_{\sfss_1}(\Thetab^{\sfss_1}_{\sfss'}(\pi'))\cong \Ind_{P_{\sfss'',\sfss_0}}^{G_{\sfss''}} (\pi'\otimes \chi_0),
 \]
where $ \mathsf s_1=(D^*, \frac{k}{2}, \frac{k}{2})$.

\smallskip

\noindent
(d) If $\star=C^*$, then there is an exact sequence
\[
 0\rightarrow   \bigoplus_{\mathsf s_1\in S_{\star, k}}  \Thetab^{\sfss''}_{\sfss_1}(\Thetab^{\sfss_1}_{\sfss'}(\pi'))\rightarrow \Ind_{P_{\sfss'',\sfss_0}}^{G_{\sfss''}} (\pi'\otimes \chi_0)\rightarrow \cJ\rightarrow 0
\]
 of representations of $G_{\sfss''}$ such that $\cJ$ is a quotient of
\[
\bigoplus_{\mathsf s_2\in S_{\star, k-2}} \check \Theta^{\sfss''}_{\sfss_2}(\check \Theta^{\sfss_2}_{\sfss'}(\pi')).
\]
\end{thm}
\begin{proof}
Using Lemma \ref{growthdp}, we know that $\pi'^-$ is convergent for $I(\dot \chi)$ (and convergent for $I(\dot \chi')$ when $\star\in \{B,D\}$). Also note that the character $\chi'$ (see \eqref{chip}) is trivial.

If we are in the situation (a), (b) or (c), by using Theorem \ref{midp}, Lemma \ref{lem:coinv} and Proposition \ref{doublelift}, the theorem follows by applying the operation $\pi'^-*(\,\cdot\,)$ to the isomorphism in Lemma \ref{degens}.

Now assume that we are in the situation (d).
For simplicity, write $0\rightarrow I_1\rightarrow I_2\rightarrow I_3\rightarrow
0$ for the exact sequence in part (d) of Lemma \ref{degens}. Since $\pi'^-$ is nuclear as a Fr\'echet space, the sequence
\[
0\rightarrow \pi'^-\totimes I_1\rightarrow \pi'^-\totimes I_2\rightarrow \pi'^-\totimes I_3\rightarrow
0
\]
is also topologically exact.
Note that the natural map
\[
\pi'^- * I_1\rightarrow \pi'^- * I_2
\]
is injective, and the natural map
\[
  \pi'^-\totimes I_3\cong \frac{\pi'^-\totimes I_2}{\pi'^- \totimes I_1}\longrightarrow\frac{ \pi'^-*I_2}{\pi'^-*I_1}
  \]
descends to a surjective map
\[
  (\pi'^-\totimes  I_3)_{G_{\sfss'^-}}\rightarrow \frac{ \pi'^-*I_2}{\pi'^-*I_1}.
\]
Note that
\[
   (\pi'^-\totimes  I_3)_{G_{\sfss'^-}}\cong \bigoplus_{\mathsf s_2\in S_{\star, k-2}} \check \Theta^{\sfss''}_{\sfss_2}(\check \Theta^{\sfss_2}_{\sfss'}(\pi')).
\]
Theorem \ref{midp} implies that
\[
\pi'^-*I_2\cong \Ind_{P_{\sfss'',\sfss_0}}^{G_{\sfss''}} (\pi'\otimes \chi_0).
\]
Lemma \ref{lem:coinv} and Proposition \ref{doublelift} imply that
\[
\pi'^-*I_1\cong  \bigoplus_{\mathsf s_1\in S_{\star, k} } \Thetab^{\sfss''}_{\sfss_1}(\Thetab^{\sfss_1}_{\sfss'}(\pi')).
\]
Therefore the theorem follows.
\end{proof}

\section{Regular descents and bounding of the associated cycles}\label{sec:AC}

In this section, we investigate the bounding of associated cycles in the context of theta lifting. The main result is in \Cref{thm:upper.AC}. Earlier results with a similar theme were proved in \cites{NOT, NZ, LM}. 

Let $\mathsf s=(\star, p,q)$ and $ \mathsf s'=(\star', p',q')$ be classical signatures such that $\star'$ is the Howe dual of $\star$.
Let $\CO\in \Nil(\g_\mathsf s)$, whose Zariski closure in $\g_\sfss$ is denoted by $\overline \CO$.

\subsection{Associated cycles}

In this subsection, we will recall the definition of associated cycles from \cite{Vo89}.
\begin{defn}\label{spkmodule}
 Let $H_\C$ be a complex linear algebraic group. Let $\CA$ be a commutative $\C$-algebra carrying a locally algebraic linear action of $H_\C$ by algebra automorphisms.
  An $(\CA, H_\C)$-module is an  $\CA$-module
$
  \sigma
$
equipped with a locally algebraic linear action of $H_\C$ such that
the module structure map
$
  \CA\otimes \sigma\rightarrow \sigma
$
is $H_{\C}$-equivariant. An $(\CA, H_\C)$-module is said to be finitely generated if it is so as an $\CA$-module.

\end{defn}

For every affine complex algebraic variety $Z$, write $\C[Z]$ for the algebra of regular functions on $Z$. By convention, all elements of complex algebraic varieties are closed.
Given a finitely generated $(\C[\overline \CO\cap \p_\sfss], K_{\sfss, \C})$-module $\sigma^\circ$, for each $K_{\sfss, \C}$-orbit $\sO\subset \CO\cap \p_\sfss$, the set
\[
  \bigsqcup_{\mathbf e\in \sO} \C_\mathbf e\otimes_{\C[\overline \CO\cap \p_\sfss]}\sigma^\circ
\]
is naturally a $K_{\sfss, \C}$-equivariant algebraic vector bundle over $\sO$, where $\C_\mathbf e$  denotes the complex number field $\C$ viewed as  a $\C[\overline \CO\cap \p_\sfss]$-algebra via the evaluation map at $\mathbf e$. We define $\mathrm{AC}_\sO(\sigma^\circ)\in \CK_\sfss(\sO)$ to be the Grothendieck group element associated to this bundle. Define the  associated cycle of $\sigma^\circ$ to be
\[
  \mathrm{AC}_\CO(\sigma^\circ):=\sum_{\sO\textrm{ is a   $K_{\sfss, \C}$-orbit in $ \CO\cap \p_\sfss$}} \mathrm{AC}_\sO(\sigma^\circ)\in \CK_\sfss(\CO).
\]

Write $I_{\overline \CO\cap \p_\sfss}$ for the radical ideal of $\C[\p_\sfss]$ corresponding to the closed subvariety $\overline \CO\cap \p_\sfss$.
Every $(\C[\overline \CO\cap \p_\sfss], K_{\sfss, \C})$-module is clearly an $(\C[\p_\sfss], K_{\sfss,\C})$-module.
On the other hand, for each  $(\C[\p_{\sfss}], K_{\sfss, \C})$-module $\sigma$,
\[
  \frac{(I_{\overline \CO\cap \p_\sfss})^i\cdot \sigma}{(I_{\overline \CO\cap \p_\sfss})^{i+1}\cdot \sigma}\qquad (i\in \bN)
\]
is naturally a $(\C[\overline \CO\cap \p_\sfss], K_{\sfss, \C})$-module.
We say that $\sigma$ is $\CO$-bounded if
\[
(I_{\overline \CO\cap \p_\sfss})^i \cdot\sigma=0\qquad
\textrm{
for some $i\in \bN$.}
\]
When $\sigma$ is finitely generated, this is equivalent to saying that  the support of $\sigma$ is contained in $\overline{\CO}\cap \p_\sfss$.

 When $\sigma$ is finitely generated and $\CO$-bounded, we define its associated cycle to be
\[
  \mathrm{AC}_\CO(\sigma):=\sum_{i\in \bN}\mathrm{AC}_\CO\left(\frac{(I_{\overline \CO\cap \p_\sfss})^i\cdot \sigma}{(I_{\overline \CO\cap \p_\sfss})^{i+1}\cdot \sigma}\right)\in  \CK_\sfss(\CO).
\]
The assignment $\mathrm{AC}_\CO$ is additive in the following sense: the equality
\[
 \mathrm{AC}_\CO(\sigma)= \mathrm{AC}_\CO(\sigma')+ \mathrm{AC}_\CO(\sigma'')
\]
holds for every exact sequence
$0\rightarrow \sigma'\rightarrow \sigma\rightarrow \sigma''\rightarrow 0$ of finitely generated $\CO$-bounded $(\C[\p_{\sfss}], K_{\sfss, \C})$-modules.

Given a $(\g_\sfss, K_{\sfss})$-module $\rho$ of finite length, pick a filtration
\be\label{goodf}
\CF:\quad  \cdots\subset \rho_{-1}\subset \rho_0\subset \rho_1\subset \rho_2\subset \cdots
\ee
of $\rho$ that is good in the following sense:
\begin{itemize}
\item for each $i\in \Z$, $\rho_i$ is a finite-dimensional  $K_{\sfss}$-stable subspace of $\rho$;
\item $\g_\sfss \cdot \rho_i\subset \rho_{i+1}$ for all $i\in \Z$, and $\g_\sfss \cdot \rho_i=\rho_{i+1}$ when  $i\in \Z$ is sufficiently large;
\item $\bigcup_{i\in \Z} \rho_i=\rho$, and $\rho_i=0$ for some $i\in \Z$.
\end{itemize}
Then the grading
\[
  \mathrm{Gr}(\rho):= \mathrm{Gr}(\rho,\CF):=\bigoplus_{i\in \Z} \rho_i/\rho_{i+1}
\]
is naturally a finitely generated  $(\oS(\p_{\sfss}), K_{\sfss, \C})$-module, where $\oS(\p_{\sfss})$ denotes the symmetric algebra of $\p_\sfss$. Recall that we have identified $\p_\sfss$ with its dual space $\p_\sfss^*$ by using the trace form. Hence $\oS(\p_{\sfss})$ is identified with $\C[\p_\sfss]$, and $\mathrm{Gr}(\rho)$ is a $(\C[\p_{\sfss}], K_{\sfss, \C})$-module.

Recall that $\rho$ is said to be $\CO$-bounded if the associated variety of its annihilator ideal in $\CU(\g_\mathsf s)$ is contained in  $\overline \CO$.

\begin{lem}\label{l62}
Let $\rho$ be a $(\g_\sfss, K_{\sfss, \C})$-module  of finite length,  and let $\CF$ be a  good filtration of $\rho$. Then $\rho$ is $\CO$-bounded if and only if $\mathrm{Gr}(\rho, \CF)$ is $\CO$-bounded as a $(\C[\p_{\sfss}], K_{\sfss, \C})$-module.
\end{lem}
\begin{proof}
This is a direct consequence of   \cite[Theorem 8.4]{Vo89}.
\end{proof}

When $\rho$ is  $\CO$-bounded, its associated cycle is defined to be
\[
   \mathrm{AC}_\CO(\rho):= \mathrm{AC}_\CO(\mathrm{Gr}(\rho))\in  \CK_\sfss(\CO).
\]
This is independent of the good filtration \eqref{goodf}. Moreover, taking the associated cycles is additive in the following sense: the equality
\[
 \mathrm{AC}_\CO(\rho)= \mathrm{AC}_\CO(\rho')+ \mathrm{AC}_\CO(\rho'')
\]
holds for every exact sequence
$0\rightarrow \rho'\rightarrow \rho\rightarrow \rho''\rightarrow 0$ of  $\CO$-bounded  $(\g_\sfss, K_{\sfss})$-modules of finite length.

For every Casselman-Wallach representation $\pi$ of $G_\sfss$, write $\pi^{\mathrm{alg}}$ for the underlying $(\g_\sfss, K_\sfss)$-module of $\pi$. When $\pi$ is $\CO$-bounded, we define the associated cycle $\mathrm{AC}_\CO(\pi):=\mathrm{AC}_\CO(\pi^{\mathrm{alg}})$.

\subsection{Algebraic theta lifts and commutative theta lifts}

We retain the notation of Section \ref{sec:bipGeometry} and Section \ref{sec:Integrals}.
Recall from Section \ref{sec:Integrals} the one-dimensional space $ \omega_{\mathsf s, \mathsf s'}^{\cX_{\mathsf s, \mathsf s'}}$, on which $K_{\mathsf s}\times  K_{\mathsf s'}$ acts by the character $\zeta_{\mathsf s, \mathsf s'}$. 
Denote by $\omega^{\mathrm{alg}}_{\mathsf s, \mathsf s'}$ the $\h_{\mathsf s, \mathsf s'}$-submodule of $\omega_{\mathsf s, \mathsf s'}$ generated by  $ \omega_{\mathsf s, \mathsf s'}^{\cX_{\mathsf s, \mathsf s'}}$. This is an $(\g_\mathsf s\times \g_{\mathsf s'}, K_\mathsf s\times K_{\mathsf s'})$-module. For every $ (\g_{\mathsf s'}, K_{\mathsf s'})$-module $\rho'$, define its algebraic (full) theta lift to be the $(\g_{\mathsf s}, K_{\mathsf s})$-module
\[
   \check \Theta_{\mathsf s'}^{\mathsf s}(\rho'):=(\omega^{\mathrm{alg}}_{\mathsf s, \mathsf s'}\otimes \rho')_{\g_{\mathsf s'}, K_{\mathsf s'}} \qquad (\textrm{the  coinvariant space}).
\]
It has finite length whenever $\rho'$ has finite length (\cite[Section 4]{Howe89}).

Recall from \eqref{momentmap} the moment maps
  \[
    \xymatrix@R=0em@C=3.0em{
     \p_\sfss^*=\fpp_\mathsf s &\ar[l]_{M_\mathsf s} \cX_{\mathsf s, \mathsf s'}\ar[r]^{M_{\mathsf s'}}& \p_{\sfss'}^*=\fpp_{\mathsf s'},\\
    \qquad \phi^* \phi  &\ar@{|->}[l] \phi \ar@{|->}[r] & \phi \phi^*.\qquad
    }
  \]
We view $\C[\cX_{\mathsf s, \mathsf s'}]$ as a $\C[\p_\sfss]\otimes \C[\p_{\sfss'}]$-algebra by using the moment maps. The natural action of $K_{\sfss,\C}\times K_{\sfss',\C}$ on $\cX_{\mathsf s, \mathsf s'}$ yields a locally algebraic linear action of  $K_{\sfss,\C}\times K_{\sfss',\C}$ on $\C[\cX_{\mathsf s, \mathsf s'}]$. Thus $\C[\cX_{\mathsf s, \mathsf s'}]$  is naturally a  $(\C[\p_{\sfss}]\otimes \C[\p_{\sfss'}], K_{\sfss, \C}\times K_{\sfss', \C})$-module.

For every  $(\C[\p_{\sfss'}], K_{\sfss', \C})$-module $\sigma'$, define its commutative theta lift to be
\[
   \check \Theta_{\mathsf s'}^{\mathsf s}(\sigma'):=(\C[\CX_{\sfss, \sfss'}]\otimes_{\C[\p_{\sfss'}]} \sigma'\otimes \zeta_{\sfss, \sfss'})_{K_{\mathsf s', \C}} \qquad (\textrm{the  coinvariant space}),
\]
which is naturally a  $(\C[\p_{\sfss}], K_{\sfss, \C})$-module.

\begin{lem}
Suppose that $\sigma'$ is a finitely generated  $(\C[\p_{\sfss'}], K_{\sfss', \C})$-module. Then the  $(\C[\p_{\sfss}], K_{\sfss, \C})$-module   $\check \Theta_{\mathsf s'}^{\mathsf s}(\sigma')$ is also finitely generated.
\end{lem}
\begin{proof} This follows from the structural analysis of the space of joint harmonics, as in \cite[Section 4]{Howe89}.
\end{proof}

\begin{lem}\label{lm}
Suppose that $\rho'$ is a  $(\g_{\sfss'}, K_{\sfss'})$-module of finite length. Then there exist a good filtration $\CF'$ on $\rho'$, a good filtration $\CF$ on $\check \Theta_{\mathsf s'}^{\mathsf s}(\rho')$, and a surjective $(\C[\p_{\sfss}], K_{\sfss, \C})$-module homomorphism
\[
  \check \Theta_{\mathsf s'}^{\mathsf s}(\mathrm{Gr}(\rho',\CF')) \rightarrow \mathrm{Gr}(\check \Theta_{\mathsf s'}^{\mathsf s}(\rho'),\CF).
\]
\end{lem}
\begin{proof} This follows from the discussions in \cite[Section 3.2]{LM}.
\end{proof}

Recall from \eqref{momentmap2}  the moment maps
\[
    \xymatrix@R=0em@C=3em{
      \g_\mathsf s &\ar[l]_{\tilde M_\mathsf s} W_{\mathsf s, \mathsf s'}\ar[r]^{\tilde M_{\mathsf s'}}& \g_{\mathsf s'},\\
     \phi^* \phi & \ar@{|->}[l] \phi \ar@{|->}[r] & \phi \phi^*.
    }
  \]

\begin{lem}\label{comobound}
Suppose that $\CO'\in \Nil(\g_{\sfss'})$ and  $\tilde M_{\mathsf s'}^{-1}(\overline{\CO'})\subset \tilde M_{\sfss}^{-1}(\overline{\CO})$, where  $\overline{\CO'}$ denotes the Zariski closure of $\CO'$ in $\g_{\sfss'}$.  Then
\begin{itemize}
\item
$\check \Theta_{\mathsf s'}^{\mathsf s}(\sigma')$ is $\CO$-bounded for every $\CO'$-bounded $(\C[\p_{\sfss'}], K_{\sfss', \C})$-module $\sigma'$;
\item
 $\check \Theta_{\mathsf s'}^{\mathsf s}(\rho')$ is $\CO$-bounded for every $\CO'$-bounded $(\g_{\mathsf s'}, K_{\sfss'})$-module $\rho'$ of finite length;
 \item $\check \Theta_{\mathsf s'}^{\mathsf s}(\pi')$ is $\CO$-bounded for every $\CO'$-bounded Casselman-Wallach representation $\pi'$ of $G_{\mathsf s'}$.
 \end{itemize}
\end{lem}
\begin{proof} The assumption of the lemma implies  that
\[
 ( I_{\overline{\CO}\cap \p_\sfss})^i\cdot \C[\CX_{\sfss, \sfss'}] \subset I_{\overline{\CO'}\cap \p_{\sfss'}}\cdot  \C[\CX_{\sfss, \sfss'}]\qquad \textrm{for some } i\in \bN,
\]
where $I_{\overline{\CO'}\cap \p_{\sfss'}}$ denotes  the radical ideal of $\C[\p_{\sfss'}]$ corresponding to the closed subvariety $\overline{\CO'}\cap \p_{\sfss'}$.
This implies the first assertion of the lemma. In view of Lemma \ref{l62}, the second assertion is implied by the first one and Lemma \ref{lm}. The third assertion is implied by the second one since there is a surjective
$(\g_{\sfss}, K_{\sfss})$-module homomorphism
\[
\check \Theta_{\mathsf s'}^{\mathsf s}(\pi'^{\mathrm{alg}})\rightarrow \left (\check \Theta_{\mathsf s'}^{\mathsf s}(\pi')\right )^{\mathrm{alg}}.
 \]
\end{proof}

\subsection{Regular descent of a nilpotent orbit}\label{regud}

\begin{defn}\label{defn:regular}
The orbit $\CO\in \Nil(\g_\mathsf s)$ is regular for $\DD_{\mathsf s'}^{\mathsf s}$ if
either
\[
\abs{\sfss'}=\abs{\DD_\mathrm{naive}(\CO)},
\]
 or
 \[
 \abs{\sfss'}>\abs{\DD_\mathrm{naive}(\CO)}\qquad\textrm{and}\qquad \mathbf c_1(\CO)=\mathbf c_2(\CO).
\]

\end{defn}

In the rest of this section we assume that $\CO$ is regular for $\DD_{\mathsf s'}^{\mathsf s}$. Then Lemma \ref{imageofmm} implies that  $\CO$ is contained in the image of the moment map $\tilde M_{\mathsf s}$. Put  $\CO':=\DD_{\mathsf s'}^{\mathsf s}(\CO)\in  \Nil(\g_{\mathsf s'})$, and let  $\overline{\CO'}$ denote the Zariski closure of $\CO'$ in $\g_{\sfss'}$.

\begin{lem}\label{liftop000}
As subsets of $\g_\sfss$,
\[
  \tilde M_{\mathsf s}(\tilde M_{\mathsf s'}^{-1}(\overline{\CO'}))=\overline{\CO}.
\]
\end{lem}
\begin{proof}
This is implied by \cite[Theorems 5.2 and 5.6]{DKPC}.
\end{proof}

\begin{lem}\label{comobound2}
Suppose that $\sigma'$ is an $\CO'$-bounded $(\C[\p_{\sfss'}], K_{\sfss', \C})$-module.  Then the  $(\C[\p_{\sfss}], K_{\sfss, \C})$-module   $\check \Theta_{\mathsf s'}^{\mathsf s}(\sigma')$ is $\CO$-bounded.
\end{lem}
\begin{proof}
This follows from Lemmas \ref{liftop000} and \ref{comobound}.
\end{proof}

Put
\[
\partial :=\overline \CO\setminus \CO\qquad\textrm{and}\qquad  \bar \partial :=\g_\sfss\setminus \partial.
\]
Then $\bar \partial$ is an open subvariety of $\g_\sfss$ and $\CO$ is a closed subvariety of $\bar \partial$.
Put
\[
W_{\sfss, \sfss'}^{\bar \partial}:=\tilde M_{\sfss}^{-1}(\bar \partial )
\qquad\textrm{and}\qquad
  W_{\sfss, \sfss'}^{\CO, \CO'}:=\tilde M_{\sfss}^{-1}(\CO)\cap  \tilde M_{\sfss'}^{-1}(\CO').
\]
Then $W_{\sfss, \sfss'}^{\bar \partial}$ is an open subvariety of $W_{\sfss, \sfss'}$, and the following lemma implies that $W_{\sfss, \sfss'}^{\CO, \CO'}$ is a closed subvariety of $W_{\sfss, \sfss'}^{\bar \partial}$.

\begin{lem}\label{liftop0}
The set $W_{\sfss, \sfss'}^{\CO, \CO'}$ is  a single $G_{\sfss, \C}\times G_{\sfss',\C}$-orbit. Moreover, it is  contained in $W_{\sfss,\sfss'}^\circ$ and equals
\[
   W_{\sfss, \sfss'}^{\bar \partial} \cap \tilde M_{\mathsf s'}^{-1}(\overline{\CO'}).
\]
\end{lem}
\begin{proof}
The first assertion is implied by \cite[Theorem 3.6]{DKPC}. Recall from \eqref{kkpo2} that $W_{\sfss,\sfss'}^\circ\cap \tilde M_{\sfss}^{-1}(\CO)$ is a single  $G_{\sfss, \C}\times G_{\sfss',\C}$-orbit  whose image under the moment map $\tilde M_{\sfss'}$ equals $\CO'$. Thus
\[
W_{\sfss,\sfss'}^\circ\cap W_{\sfss, \sfss'}^{\CO, \CO'}=W_{\sfss,\sfss'}^\circ\cap \tilde M_{\sfss}^{-1}(\CO),
\]
which  is also a   single  $G_{\sfss, \C}\times G_{\sfss',\C}$-orbit. Hence $W_{\sfss, \sfss'}^{\CO, \CO'}\subset W_{\sfss,\sfss'}^\circ$.

In fact, \cite[Theorem 3.6]{DKPC} also implies that
\[
W_{\sfss, \sfss'}^{\CO, \CO'}=\tilde M_{\sfss}^{-1}(\CO)\cap  \tilde M_{\sfss'}^{-1}(\overline{\CO'}).
\]
Thus the last assertion is a direct consequence of Lemma \ref{liftop000}.
\end{proof}

Write
\[
    \CX_{\sfss, \sfss'}^{\bar \partial}:=\CX_{\sfss, \sfss'} \cap W_{\sfss, \sfss'}^{\bar \partial}\qquad\textrm{and}\qquad \CX_{\sfss, \sfss'}^{\CO, \CO'}:=\CX_{\sfss, \sfss'} \cap W_{\sfss, \sfss'}^{\CO, \CO'}.
   \]
Then  $ \CX_{\sfss, \sfss'}^{\bar \partial}$ is an open subvariety of $ \CX_{\sfss, \sfss'}$ and  $\CX_{\sfss, \sfss'}^{\CO, \CO'}$ is a  closed subvariety of  $\CX_{\sfss, \sfss'}^{\bar \partial}$. We have a decomposition
 \[
   \CX_{\sfss, \sfss'}^{\CO, \CO'}=\bigsqcup_{\sO\textrm{ is a   $K_{\sfss, \C}$-orbit in $ \CO\cap \p_\sfss$ that is contained in the image of $M_\sfss$}}  \CX_{\sfss, \sfss'}^{\sO, \CO'},
 \]
 where
 \[
   \CX_{\sfss, \sfss'}^{\sO, \CO'}:=M_\sfss^{-1}(\sO)\cap \CX_{\sfss, \sfss'}^{\CO, \CO'},
 \]
 which is a Zariski open and closed subset of $\CX_{\sfss, \sfss'}^{\CO, \CO'}$. Lemmas \ref{descko} and \ref{liftop0} imply that $ \CX_{\sfss, \sfss'}^{\sO, \CO'}$ is a single $K_{\sfss, \C}\times K_{\sfss',\C}$-orbit.

\medskip

We defer the proof of the following proposition to Sections \ref{sec:GG} and \ref{secpp}.
\begin{prop}\label{propreduced}
 The scheme theoretic  fibre product
\[
\CX_{\sfss, \sfss'}^{\bar \partial}\times_{\p_{\sfss'}} ({\CO'}\cap \p_{\sfss'})
\]
is reduced, where $\CX_{\sfss, \sfss'}^{\bar \partial}$ is viewed as a $\p_{\sfss'}$-scheme via the moment map $M_{\sfss'}$.
\end{prop}

By Lemma \ref{liftop0} and Proposition \ref{propreduced}, we know that
\be\label{cxee}
 \CX_{\sfss, \sfss'}^{\CO, \CO'}=\CX_{\sfss, \sfss'}^{\bar \partial}\times_{\p_{\sfss'}} ({\CO'}\cap \p_{\sfss'})=\CX_{\sfss, \sfss'}^{\bar \partial}\times_{\p_{\sfss'}} ({\overline{\CO'}}\cap \p_{\sfss'}),
 \ee
 which  is a smooth closed subvariety of $\CX_{\sfss, \sfss'}^{\bar \partial}$.

For every element $\mathbf e\in \CO\cap \p_\sfss$, write $K_\mathbf e$ for its  stabilizer in $K_{\sfss, \C}$. Let $\C_\mathbf e$ be the field $\C$ viewed as  a  $\C[\overline \CO\cap \p_\sfss]$-algebra via the evaluation map at $\mathbf e$. Put
\[
  \CX_{\sfss, \sfss'}^{\mathbf e, \CO'}:=M_\sfss^{-1}(\mathbf e)\cap M_{\sfss'}^{-1}(\CO'\cap\p_{\sfss'})=M_\sfss^{-1}(\mathbf e)\cap M_{\sfss'}^{-1}(\overline{\CO'}\cap\p_{\sfss'}),
\]
which is a closed subvariety of $\CX_{\sfss, \sfss'}$. If $\mathbf e$ is not contained in the image of $M_\sfss$, then $\CX_{\sfss, \sfss'}^{\mathbf e, \CO'}=\emptyset$. Otherwise, it is a single
$(K_\mathbf e\times K_{\sfss',\C})$-orbit, and Lemma \ref{descko} implies that it is also a single $K_{\sfss',\C}$-orbit.
\begin{lem}\label{fiber111}
For every element $\mathbf e\in \CO\cap \p_\sfss$,
\[
  \C_\mathbf e \otimes_{\C[\p_{\sfss}]} \C[\CX_{\sfss, \sfss'}]\otimes_{\C[\p_{\sfss'}]} \C[\overline{\CO'}\cap \p_{\sfss'}]=\C[\CX_{\sfss, \sfss'}^{\mathbf e, \CO'}].
\]
\end{lem}
\begin{proof}
As schemes, we have that
\begin{eqnarray*}
  &&\{\mathbf e\}\times_{\p_\sfss}{\CX_{\sfss, \sfss'}}\times_{\p_{\sfss'}} (\overline{\CO'}\cap \p_{\sfss'})\\
  &=& \{\mathbf e\}\times_{\p_\sfss}{\CX^{\bar \partial}_{\sfss, \sfss'}}\times_{\p_{\sfss'}} (\overline{\CO'}\cap \p_{\sfss'})\\
    &=& \{\mathbf e\}\times_{\p_\sfss}   \CX_{\sfss, \sfss'}^{\CO, \CO'}\\
    &=& \{\mathbf e\} \times_{\CO\cap \p_\sfss} ( (\CO\cap \p_\sfss) \times_{\p_\sfss} \CX_{\sfss, \sfss'}^{\CO, \CO'})\\
     &=& \{\mathbf e\} \times_{\CO\cap \p_\sfss} \CX_{\sfss, \sfss'}^{\CO, \CO'}\\
   &=& \CX_{\sfss, \sfss'}^{\mathbf e, \CO'}.
\end{eqnarray*}
The last equality holds because the morphism $M_\sfss: \CX_{\sfss, \sfss'}^{\CO, \CO'}\rightarrow \CO\cap \p_\sfss$  is  $K_{\sfss, \C}\times K_{\sfss',\C}$-equivariant and hence smooth in the sense of algebraic geometry. This proves the lemma.
\end{proof}

\subsection{Commutative theta lifts and  geometric theta lifts}

The purpose of this subsection is to prove the following proposition.
\begin{prop}\label{propthetag}
Suppose that $\sigma'$ is a finitely generated  $(\C[\overline{\CO'}\cap \p_{\sfss'}], K_{\sfss', \C})$-module. Then
\[
\mathrm{AC}_{\CO}( \check \Theta_{\mathsf s'}^{\mathsf s}(\sigma'))= \check \vartheta_{\CO'}^{\CO}(\mathrm{AC}_{\CO'}(\sigma'))
\]
as elements of $\CK_\sfss(\CO)$.

\end{prop}

By a quasi-coherent module over a scheme $Z$, we mean a quasi-coherent module over the structure sheaf of $Z$. We say that a quasi-coherent module $\CM$ over a scheme $Z$ descends to a closed subscheme $Z_1$ of $Z$ if $\CM$ is isomorphic to the push-forward of a
 quasi-coherent module over $Z_1$ via the closed embedding $Z_1\rightarrow Z$. This is equivalent to saying that the ideal sheaf defining $Z_1$ annihilates $\CM$.

When $Z$ is an affine complex algebraic variety and $\sigma$ is a $\C[Z]$-module,
 we write $\CM_\sigma$ for  the quasi-coherent module over $Z$ corresponding to $\sigma$.

\begin{lem}\label{geotheta1}
Let $\sigma$ be a finitely generated $\CO$-bounded $(\C[\p_{\sfss}],K_{\sfss,\C})$-module.  Assume that  $(\CM_\sigma)|_{\bar \partial\cap \p_\sfss}$ descends to $\CO\cap \p_\sfss$. Then
\[
  \mathrm{AC}_{\CO}(\sigma)=\mathrm{AC}_{\CO}(\C[\overline \CO\cap \p_{\sfss}]\otimes_{\C[\p_{\sfss}]} \sigma).
\]
\end{lem}
\begin{proof}
Let $\mathbf e\in \CO\cap \p_\sfss$. Note that for every ideal $I$ of $\C[\p_{\sfss}]$,
\[
  (I\cdot \sigma)_\mathbf e=(I_\mathbf e)\cdot \sigma_\mathbf e\subset  \sigma_\mathbf e,
\]
where a subscript $\mathbf e$ indicates the localization at $\mathbf e$. The assumption of the lemma implies that
\[
  ((I_{\overline{\CO}\cap \p_\sfss})^i)_\mathbf e \cdot \sigma_\mathbf e=0\qquad \textrm{for all } i\in \bN^+.
\]
Thus
\[
  ((I_{\overline{\CO}\cap \p_\sfss})^i \cdot \sigma)_\mathbf e=0,
\]
which implies the  lemma.
\end{proof}

\begin{lem}\label{geotheta2}
Suppose that  $Z$ is an affine complex algebraic variety with a transitive algebraic action of $K_{\sfss, \C}$. Let $\sigma$ be a finitely generated $(\C[Z], K_{\sfss, \C})$-module.
Let $z\in Z$ and write $K_z$ for the stabilizer of $z$ in $K_{\sfss, \C}$. Then  the natural map
\[
  \sigma^{K_{\sfss, \C}}\rightarrow (\C_z\otimes_{\C[Z]} \sigma)^{K_{z}}
\]
is a linear isomorphism. Here $\C_z$ is the field $\C$ viewed as  a $\C[Z]$-algebra via the evaluation map at $z$, and a superscript group indicates the space of  invariant vectors under the group action.
\end{lem}
\begin{proof}
Note that the set
\[
 \bigsqcup_{x\in Z} \C_x\otimes_{\C[Z]} \sigma
\]
is naturally a $K_{\sfss, \C}$-equivariant algebraic vector bundle over $Z$, and  $\sigma$ is identified with the space of algebraic sections of this bundle, see \cite[Theorem~2.7]{CPS}. Thus
\[
  \sigma\cong \,^{\mathrm{alg}}\Ind_{K_z}^{K_{\sfss, \C}}(\C_z\otimes_{\C[Z]}\sigma)\qquad(\textrm{algebraically induced representation}).
\]
The lemma then follows by the algebraic version of the Frobenius reciprocity.
\end{proof}

\begin{lem}\label{geotheta3}
Suppose that $\sigma'$ is a finitely generated  $(\C[\overline{\CO'}\cap \p_{\sfss'}], K_{\sfss', \C})$-module, and write $\sigma:=\check \Theta_{\mathsf s'}^{\mathsf s}(\sigma')$. Then the  quasi-coherent module $(\CM_\sigma)|_{\bar \partial \cap \p_\sfss}$ descends  to $\CO\cap \p_\sfss$.
\end{lem}
\begin{proof}
Write
\[
 \tilde \sigma:=\C[\CX_{\sfss, \sfss'}]\otimes_{\C[\p_{\sfss'}]} \sigma'\otimes \zeta_{\sfss, \sfss'}=(\C[\CX_{\sfss, \sfss'}]\otimes_{\C[\p_{\sfss'}]} {\C[\overline{\CO'}\cap \p_{\sfss'}]})\otimes_{\C[\overline{\CO'}\cap \p_{\sfss'}]}(\sigma'\otimes \zeta_{\sfss, \sfss'}),
\]
to be viewed as a $\C[\CX_{\sfss, \sfss'}]$-module. Write $\tilde \sigma_0:=\tilde \sigma$, viewed as  a $\C[\p_\sfss]$-module via the moment map $M_\sfss$.

Proposition \ref{propreduced} and the equalities  in \eqref{cxee} imply that
$(\CM_{ \tilde \sigma})|_{\CX_{\sfss, \sfss'}^{\bar \partial}}$ descends to $ \CX_{\sfss, \sfss'}^{\CO, \CO'}$. This implies that $(\CM_{ \tilde \sigma_0})|_{\bar \partial\cap \p_\sfss}$ descends to $ \CO\cap \p_\sfss$. The lemma then follows since $\sigma$ is a direct summand of $\tilde \sigma_0$.
\end{proof}

We are now ready to prove Proposition \ref{propthetag}.

\begin{proof}[Proof of Proposition \ref{propthetag}]
Let  $\sigma'$ be a finitely generated  $(\C[\overline{\CO'}\cap \p_{\sfss'}], K_{\sfss', \C})$-module, and write  $\sigma:=\check \Theta_{\mathsf s'}^{\mathsf s}(\sigma')$.
 Lemmas \ref{geotheta3} and \ref{geotheta1} imply that
\[
  \mathrm{AC}_{\CO}(\sigma)=\mathrm{AC}_{\CO}(\C[\overline \CO\cap \p_{\sfss}]\otimes_{\C[\p_{\sfss}]} \sigma).
\]

Suppose that $\mathbf e\in \CO\cap \p_\sfss$. Then
\begin{eqnarray*}
  && \C_\mathbf e \otimes_{\C[\overline \CO\cap \p_{\sfss}]} (\C[\overline \CO\cap \p_{\sfss}]\otimes_{\C[\p_{\sfss}]} \sigma)\\
  &=&  \C_\mathbf e \otimes_{\C[\p_{\sfss}]} \sigma\\
  &=&  \C_\mathbf e \otimes_{\C[\p_{\sfss}]} (\C[\CX_{\sfss, \sfss'}]\otimes_{\C[\p_{\sfss'}]} \sigma'\otimes \zeta_{\sfss, \sfss'})_{K_{\mathsf s', \C}} \\
   &=&\left( (\C_\mathbf e \otimes_{\C[\p_{\sfss}]} \C[\CX_{\sfss, \sfss'}]\otimes_{\C[\p_{\sfss'}]} \C[\overline{\CO'}\cap \p_{\sfss'}]) \otimes_{\C[\overline{\CO'}\cap \p_{\sfss'}]}\sigma'\otimes \zeta_{\sfss, \sfss'}\right)_{K_{\mathsf s', \C}} \\
      &=&\left( \C[\CX_{\sfss, \sfss'}^{\mathbf e, \CO'}] \otimes_{\C[\overline{\CO'}\cap \p_{\sfss'}]}\sigma'\otimes \zeta_{\sfss, \sfss'}\right)_{K_{\mathsf s', \C}} \qquad (\textrm{by Lemma \ref{fiber111}}).
\end{eqnarray*}

If $\mathbf e$ is not in the image of $M_\sfss$, then the above module is zero. Now we assume that $\mathbf e$ is in the image of $M_\sfss$ so that $\CX_{\sfss, \sfss'}^{\mathbf e, \CO'}$ is a single $K_{\sfss',\C}$-orbit. Pick an arbitrary element $\phi\in \CX_{\sfss, \sfss'}^{\mathbf e, \CO'}$, and write $\mathbf e':=M_{\sfss'}(\phi)$. Let
\[
  1\rightarrow  (K_{\mathsf s',\C})_\phi \rightarrow (K_{\mathsf s,\C}\times K_{\mathsf s', \C})_\phi\xrightarrow{\textrm{the projection to the first factor}} (K_{\mathsf s,\C})_{\mathbf e}\rightarrow 1
\]
be the exact  sequence as in Lemma \ref{descko}.

 Let $\C_\phi$ denote the field $\C$ viewed as  a $\C[\CX_{\sfss, \sfss'}^{\mathbf e, \CO'}]$-algebra via the evaluation map at $\phi$. Then we have that
\begin{eqnarray*}
      &&\left( \C[\CX_{\sfss, \sfss'}^{\mathbf e, \CO'}] \otimes_{\C[\overline{\CO'}\cap \p_{\sfss'}]}\sigma'\otimes \zeta_{\sfss, \sfss'}\right)_{K_{\mathsf s', \C}}\\
             &=&\left( \C[\CX_{\sfss, \sfss'}^{\mathbf e, \CO'}] \otimes_{\C[\overline{\CO'}\cap \p_{\sfss'}]}\sigma'\otimes \zeta_{\sfss, \sfss'}\right)^{K_{\mathsf s', \C}}\quad \qquad (\textrm{because $K_{\sfss',\C}$ is reductive})\\
             &=&\left( \C_\phi\otimes_{\C[\CX_{\sfss, \sfss'}^{\mathbf e, \CO'}]} \C[\CX_{\sfss, \sfss'}^{\mathbf e, \CO'}] \otimes_{\C[\overline{\CO'}\cap \p_{\sfss'}]}\sigma'\otimes \zeta_{\sfss, \sfss'}\right)^{ (K_{\mathsf s',\C})_\phi} \qquad (\textrm{by Lemma \ref{geotheta2}})\\
               &=&\left( \C_{\mathbf e'}  \otimes_{\C[\overline{\CO'}\cap \p_{\sfss'}]}\sigma'\otimes \zeta_{\sfss, \sfss'}\right)^{ (K_{\mathsf s',\C})_\phi}\\
                 &=&\left( \C_{\mathbf e'}  \otimes_{\C[\overline{\CO'}\cap \p_{\sfss'}]}\sigma'\otimes \zeta_{\sfss, \sfss'}\right)_{ (K_{\mathsf s',\C})_\phi} \quad \qquad (\textrm{because $ (K_{\mathsf s',\C})_\phi$ is reductive}).
\end{eqnarray*}
Putting all things together, we have an identification
\[
\C_\mathbf e \otimes_{\C[\overline \CO\cap \p_{\sfss}]} (\C[\overline \CO\cap \p_{\sfss}]\otimes_{\C[\p_{\sfss}]} \sigma)=\left( \C_{\mathbf e'}  \otimes_{\C[\overline{\CO'}\cap \p_{\sfss'}]}\sigma'\otimes \zeta_{\sfss, \sfss'}\right)_{ (K_{\mathsf s',\C})_\phi}.
\]
It is routine to check that this identification respects the $(K_{\mathsf s,\C}\times K_{\mathsf s', \C})_\phi$-actions, and $ (K_{\mathsf s',\C})_\phi$ acts trivially on both sides. Thus the identification respects the $(K_{\mathsf s,\C})_{\mathbf e}$-actions. In view of Lemmas \ref{geotheta1} and \ref{geotheta3}, this proves the proposition.
\end{proof}

\subsection{Algebraic theta lifts and  geometric theta lifts}

Proposition \ref{propthetag} has the following consequence.
\begin{prop}\label{propthetag2}
Suppose that $\sigma'$ is a finitely generated $\CO'$-bounded $(\C[\p_{\sfss'}], K_{\sfss', \C})$-module. Then
\[
\mathrm{AC}_{\CO}( \check \Theta_{\mathsf s'}^{\mathsf s}(\sigma'))\preceq  \check \vartheta_{\CO'}^{\CO}(\mathrm{AC}_{\CO'}(\sigma'))
\]
in $\CK_\sfss(\CO)$.

\end{prop}
\begin{proof}
Suppose that  $0\rightarrow \sigma'_1\rightarrow \sigma'_2\rightarrow \sigma'_3\rightarrow 0$ is an exact sequence of  $\CO'$-bounded finitely generated $(\C[\p_{\sfss'}], K_{\sfss', \C})$-modules. Then
\[
 \check \Theta_{\mathsf s'}^{\mathsf s}(\sigma'_1)\rightarrow \check \Theta_{\mathsf s'}^{\mathsf s}(\sigma'_2)\rightarrow \check \Theta_{\mathsf s'}^{\mathsf s}(\sigma'_3)\rightarrow 0
 \]
 is an exact sequence of  $\CO$-bounded finitely generated $(\C[\p_{\sfss}], K_{\sfss, \C})$-modules.
Thus
 \[
   \mathrm{AC}_{\CO}( \check \Theta_{\mathsf s'}^{\mathsf s}(\sigma'_2))\preceq \mathrm{AC}_{\CO}( \check \Theta_{\mathsf s'}^{\mathsf s}(\sigma'_1))+\mathrm{AC}_{\CO}( \check \Theta_{\mathsf s'}^{\mathsf s}(\sigma'_3)).
 \]
 On the other hand, it is clear that
  \[
  \check \vartheta_{\CO'}^{\CO}(\mathrm{AC}_{\CO'}(\sigma'_2))= \check \vartheta_{\CO'}^{\CO}(\mathrm{AC}_{\CO'}(\sigma'_1))+ \check \vartheta_{\CO'}^{\CO}(\mathrm{AC}_{\CO'}(\sigma'_3)).
  \]
The proposition then easily follows by Proposition \ref{propthetag}.
\end{proof}

We are now ready to prove the main result of this section. Recall that  $\cO$ is assumed to be regular for $\DD_{\mathsf s'}^{\mathsf s}$.

\begin{thm}\label{thm:upper.AC}
Let $\rho'$ be an $\CO'$-bounded $(\g_{\mathsf s'}, K_{\mathsf s'})$-module of finite length. Then  $\check \Theta_{\mathsf s'}^{\mathsf s}(\rho')$ is $\CO$-bounded, and
    \[
    \mathrm{AC}_{\cO}(\check \Theta_{\mathsf s'}^{\mathsf s}(\rho'))\preceq \check \vartheta_{\cO'}^\cO(\mathrm{AC}_{\cO'}(\rho')).
  \]
\end{thm}
\begin{proof}
Let $\CF'$ and $\CF$ be  good filtrations  on $\rho'$ and  $\check \Theta_{\mathsf s'}^{\mathsf s}(\rho')$ respectively as in Lemma \ref{lm} so that there exists a
 surjective $(\C[\p_{\sfss}], K_{\sfss, \C})$-module homomorphism
\be\label{surpkm}
  \check \Theta_{\mathsf s'}^{\mathsf s}(\mathrm{Gr}(\rho',\CF')) \rightarrow \mathrm{Gr}(\check \Theta_{\mathsf s'}^{\mathsf s}(\rho'),\CF).
\ee
The first assertion then follows from Lemmas \ref{l62} and \ref{comobound2}.

Put $\sigma':=\mathrm{Gr}(\rho',\CF')$.
Then
\begin{eqnarray*}
      && \mathrm{AC}_{\cO}(\check \Theta_{\mathsf s'}^{\mathsf s}(\rho'))\\
      &=&  \mathrm{AC}_{\cO}(\mathrm{Gr}(\check \Theta_{\mathsf s'}^{\mathsf s}(\rho'),\CF))\\
             &\preceq& \mathrm{AC}_{\cO}(\check \Theta_{\mathsf s'}^{\mathsf s}(\sigma')) \quad\  \qquad (\textrm{by \eqref{surpkm}})\\
          &  \preceq & \check \vartheta_{\CO'}^{\CO}(\mathrm{AC}_{\CO'}(\sigma'))\quad \qquad (\textrm{by Proposition \ref{propthetag2}})\\
               &=&\check \vartheta_{\CO'}^{\CO}(\mathrm{AC}_{\CO'}(\rho')).
               \end{eqnarray*}
This proves the theorem.
\end{proof}

\subsection{Geometries of regular descent}
\label{sec:GG}

\def\UU{{\bar \partial}}
\def\dbM{\breve{M}}
\def\dbMM{\breve{MM}}
\def\dbX{\breve{X}}
\def\dbfpp{\breve{\fpp}}
\def\ZdbX{\cZ_{\dbX}}
\def\aV{\acute{V}}
\def\fggs{\fgg_{\sfss}}
\def\fggsp{\fgg_{\sfss'}}
\def\fggspp{\fgg_{\sfss''}}
\def\fggspo{\fgg_{\sfss'_0}}
\def\fggspt{\fgg_{\sfss'_1}}
\def\fggspi{\fgg_{\sfss'_i}}
\def\fkks{\fkk_{\sfss}}
\def\fkksp{\fkk_{\sfss'}}
\def\fkkspo{\fkk_{\sfss'_0}}
\def\fkkspt{\fkk_{\sfss'_1}}
\def\fkkspi{\fkk_{\sfss'_i}}
\def\fpps{\fpp_{\sfss}}
\def\fppsp{\fpp_{\sfss'}}
\def\fppspo{\fpp_{\sfss'_0}}
\def\fppspt{\fpp_{\sfss'_1}}
\def\fppspi{\fpp_{\sfss'_i}}
\def\DDss{\DD_{\sfss'}^{\sfss}}
\def\DDsso{\DD_{\sfss'_0}^{\sfss}}
\def\cOpo{\cOp_{0}}
\def\Mss{M_{\sfss,\sfss'}}
\def\Ms{M_{\sfss}}
\def\Msp{M_{\sfss'}}

\def\Ks{{K}_{\sfss,\bC}}
\def\Ksp{{K}_{\sfss',\bC}}
\def\Kspo{{K}_{\sfss'_0,\bC}}
\def\Kspt{{K}_{\sfss'_1,\bC}}

\def\Gs{{G}_{\sfss}}
\def\Gsp{{G}_{\sfss'}}
\def\Gspo{{G}_{\sfss'_0}}
\def\Gspt{{G}_{\sfss'_1}}
\def\CMs{{\tilde M}_{\sfss}}
\def\CMsp{{\tilde M}_{\sfss'}}
\def\CMss{{\tilde M}_{\sfss,\sfss'}}
\def\Wss{{W_{\sfss,\sfss'}}}
\def\Woss{{W^{\circ}_{\sfss,\sfss'}}}
\def\CXss{\cX_{\sfss,\sfss'}}
\def\X{\bfee}
\def\Xp{{\bfee'}}
\def\Xpo{{\bfee'_0}}
\def\sOp{\sO'}
\def\fggpsb{{\fgg^{\boxslash}_{\sfss'}}}
\def\fpppsb{{\fpp^{\boxslash}_{\sfss'}}}
\def\fpppso{\fpp_{\sfss'_0}}
\def\fpppst{\fpp_{\sfss'_1}}

\def\Vs{V_\sfss}
\def\Vsp{V_{\sfss'}}
\def\Vspp{V_{\sfss''}}
\def\Vspo{V_{\sfss'_0}}
\def\Vspt{V_{\sfss'_1}}
\def\Vspi{V_{\sfss'_i}}
\def\Wssi{W_{\sfss,\sfss'_i}}
\def\Wsso{W_{\sfss,\sfss'_0}}
\def\Wsst{W_{\sfss,\sfss'_1}}

\def\CXUO{\cX_{\sfss,\sfss'}^{\bar\partial ,\cOp}}
\def\CXOO{\cX_{\sfss,\sfss'}^{\cO ,\cOp}}

The following lemma is a form of  the Jacobian criterion for
regularity (see \cite[Theorem~2.19]{LiuAG}).
\begin{lem}\label{jacobic}
Let $Z$ and $Z'$ be smooth complex algebraic varieties, and let $z'\in Z'$. Let $f: Z\rightarrow Z'$ be a morphism of algebraic varieties such that $f^{-1}(z')$ is a smooth subvariety of $Z$. Then the scheme theoretic fibre product $Z\times_{Z'} \{z'\}$ is reduced if and only if the sequence
\be\label{jc0}
  \mathrm T_z(f^{-1}(z'))\xrightarrow{\textrm{ inclusion }} \mathrm T_z(Z)\xrightarrow{\textrm{the differential of $f$}} \mathrm T_{z'}(Z')
\ee
is exact for all $z\in f^{-1}(z')$.
\end{lem}

Here and as usual $\mathrm T$ indicates the tangent space.

Recall that  $\CO\in \Nil(\g_\sfss)$ is regular for $\DD_{\mathsf s'}^{\mathsf s}$.
Let $\mathbf e'\in \CO'=\DD_{\mathsf s'}^{\mathsf s}(\CO)$. Denote by $G_{\mathbf e'}$ the stabilizer of $\mathbf e'$ in $G_{\sfss',\C}$, and by $\g_{\mathbf e'}$ the Lie algebra of $G_{\mathbf e'}$. Consider the map
\[
 \tilde M':=  (\tilde M_{\sfss'})|_{W_{\sfss, \sfss'}^{\bar \partial}} : W_{\sfss, \sfss'}^{\bar \partial}\rightarrow \g_{\sfss'}.
\]
By Lemma \ref{liftop0}, its fibre at $\mathbf e'$ equals the set
\[
W_{\sfss, \sfss'}^{\CO, \mathbf e'}:=\tilde M_{\sfss}^{-1}(\CO)\cap  \tilde M_{\sfss'}^{-1}(\mathbf e'),
\]
and this set is a single $G_{\sfss,\C}\times G_{\mathbf e'}$-orbit. Thus the set $W_{\sfss, \sfss'}^{\CO, \mathbf e'}$ is a smooth closed subvariety of $W_{\sfss, \sfss'}^{\bar \partial}$.

Let $\phi\in W_{\sfss, \sfss'}^{\CO, \mathbf e'}$. Then we have a sequence
\be\label{jc1}
    \mathrm T_\phi (W_{\sfss, \sfss'}^{\CO, \mathbf e'})\xrightarrow{\textrm{ inclusion}} \mathrm T_\phi(W_{\sfss, \sfss'}^{\bar \partial})\xrightarrow{\textrm{the differential of $\tilde M'$}} \mathrm T_{\mathbf e'}(\g_{\sfss'})
\ee
as in \eqref{jc0}.
The tangent spaces $\mathrm T_\phi(W_{\sfss, \sfss'}^{\bar \partial})$ and $\mathrm T_{\mathbf e'}(\g_{\sfss'})$ are respectively identified with $W_{\sfss, \sfss'}$ and $\g_{\sfss'}$ as usual. Then the second map in \eqref{jc1} equals the map
\be\label{differ1}
  W_{\sfss, \sfss'}\rightarrow \g_{\sfss'}, \qquad b\mapsto b\phi^{*}+\phi b^{*} .
\ee
Note that the  image of the first map in \eqref{jc1} is identified with the image of the following map
\be\label{differ2}
 \fggs\times  \g_{\mathbf e'} \rightarrow W_{\sfss, \sfss'}, \qquad (a,a')\mapsto a'\phi - \phi a.
\ee
Thus the sequence \eqref{jc1} is exact if and only if the following sequence is exact:
\begin{equation}\label{eq:CGG}
  \fggs\times  \g_{\mathbf e'} \xrightarrow{\eqref{differ2} } W_{\sfss, \sfss'}
  \xrightarrow{\eqref{differ1}} \fggsp.
\end{equation}

\begin{lem}\label{lemexact1}
The sequence \eqref{jc1} is exact when $\abs{\sfss'}=\abs{\nabla_{\mathrm{naive}}(\CO)}$.
\end{lem}
\begin{proof}
The map \eqref{differ1} equals the composition of
\be\label{differ3}
  W_{\sfss, \sfss'}\xrightarrow{b\mapsto b\phi^*} \g\mathfrak l(V_{\sfss'})\xrightarrow{x\mapsto x-x^*} \g_{\sfss'}.
\ee
Here $ \g\mathfrak l(V_{\sfss'})$ denotes the algebra of linear endomorphisms of $V_{\sfss'}$, and $x^*\in  \g\mathfrak l(V_{\sfss'})$ is specified by requiring that
\be\label{ajoint4}
  \la x u, v\ra_{\sfss'}=\la u, x^* v\ra_{\sfss'},\qquad \textrm{for all }u,v\in V_{\sfss'}.
\ee
The second map in \eqref{differ3} is always surjective.

Note that $\abs{\sfss'}=\abs{\nabla_{\mathrm{naive}}(\CO)}$ if and only if  $\phi: V_{\sfss}\rightarrow V_{\sfss'}$ is surjective. Assume this is the case. Then $\phi^*$ is injective, and consequently the first map in  \eqref{differ3}  is  surjective. Thus the map \eqref{differ1} is also surjective. Since $\phi\in W_{\sfss, \sfss'}^{\CO, \mathbf e'}$ is arbitrary, this implies that the scheme theoretic fibre product
$W_{\sfss, \sfss'}^{\bar \partial}\times_{\g_{\sfss'}}\{\mathbf e'\}$ is reduced (see \cite[Proposition~10.4]{HS}). Lemma \ref{jacobic} then implies that the sequence \eqref{jc1} is exact.
\end{proof}

\begin{lem}\label{lemexact2}
The sequence \eqref{jc1} is exact when $\mathbf c_1(\CO)=\mathbf c_2(\CO)$.
\end{lem}
\begin{proof}
Write
\be\label{v12}
V_{\sfss'}=V_1'\oplus V_2',
\ee
where $V_1':=\phi(V_{\sfss})$, which is a non-degenerate subspace of $V_{\sfss'}$, and $V_2'$ is the orthogonal complement of $V_1'$ in $V_{\sfss'}$. Then
\[
   W_{\sfss, \sfss'}=\left\{\begin{bmatrix}
        b_1\\
        b_2
      \end{bmatrix} \,:\, b_1\in  \Hom_\C(V_\sfss, V_1'), \ b_2\in  \Hom_\C(V_\sfss, V_2')\right\} \qquad\textrm{and}\qquad \phi=\begin{bmatrix}
        \phi_1\\
        0
      \end{bmatrix},
\]
where $\phi_1\in \Hom_\C(V_\sfss, V_1')$ is a surjective linear map.

Using the decomposition \eqref{v12}, we also have that
\[
 \g_{\sfss'}:= \Set{\begin{bmatrix}
        a_1 & -\psi^*\\
        \psi & a_2
      \end{bmatrix}\,:\, a_1\in \g_1', \, a_2\in \g_2', \,\psi \in \Hom_\C(V_1', V_2')},
\]
where $\g_1'$ and $\g_2'$ are respectively the Lie algebras of the isometry groups of $V_1'$ and $V_2'$. Here and henceforth, the adjoint operation $\psi\mapsto \psi^*$ is defined as in
\eqref{ajoint4} and \eqref{adjointmap}.

Note that
\[
  \mathbf e'=\begin{bmatrix}
        \mathbf e_1' & 0\\
        0 & 0
      \end{bmatrix}, \qquad \textrm{where  } \, \mathbf e_1':=\phi_1 \phi_1^*\in \g_1',
\]
and
\[
 \g_{\mathbf e'}:= \Set{\begin{bmatrix}
        a_1 & -\psi^*\\
        \psi & a_2
      \end{bmatrix}\,:\, a_1\in \g_{\mathbf e_1'}, \, a_2\in \g_2', \,\psi \in \Hom_\C(V_1', V_2')^{\mathbf e_1'}},
\]
where $\g_{\mathbf e_1'}$ denotes the centralizer of $\mathbf e_1'$ in $\g_1'$, and
\[
   \Hom_\C(V_1', V_2')^{\mathbf e_1'}:=\{\psi\in  \Hom_\C(V_1', V_2')\,:\, \psi  {\mathbf e_1'}=0\}.
\]

The two maps in \eqref{eq:CGG} are respectively given by
\[
  \left(a, \begin{bmatrix}
        a_1 & -\psi^*\\
        \psi & a_2
      \end{bmatrix}\right)\mapsto   \begin{bmatrix}
        a_1 \phi_1-\phi_1 a\\
        \psi \phi_1
      \end{bmatrix}
\]
and
\[
  \begin{bmatrix}
       b_1\\
       b_2
      \end{bmatrix}\mapsto   \begin{bmatrix}
       b_1 \phi_1^*+\phi_1 b_1^*  & \phi_1 b_2^*\\
        b_2 \phi_1^* & 0
      \end{bmatrix}.
\]
In order to prove the lemma, it suffices to show that the sequences
 \[
  \fggs\times  \g_{\mathbf e_1'} \xrightarrow{(a,a_1)\mapsto  a_1 \phi_1-\phi_1 a } \Hom_\C(V_\sfss, V_1')
  \xrightarrow{b_1\mapsto  b_1\phi_1^{*}+\phi_1 b_1^{*} } \g_1'
  \]
and
\be\label{exacthom}
  \Hom_\C(V_1', V_2')^{\mathbf e_1'} \xrightarrow{\psi \mapsto  \psi \phi_1 } \Hom_\C(V_\sfss, V_2')
  \xrightarrow{b_2\mapsto  b_2\phi_1^{*}}  \Hom_\C(V_1', V_2')
\ee
are both exact. Lemma \ref{lemexact1} implies that the first sequence is exact.

Note that the Young diagram of the nilpotent orbit in $\g_1'$ containing $\mathbf e_1'$ equals $\nabla_{\mathrm{naive}}(\CO)$. Thus
\[
  \dim (V_1'/\mathbf e_1'(V_1'))=\mathbf c_1(\nabla_{\mathrm{naive}}(\CO))=\mathbf c_2(\CO).
\]
On the other hand,
\begin{eqnarray*}
  \dim (V_{\sfss}/\phi_1^*(V_1'))&=&\dim (V_{\sfss})-\dim (\phi(V_{\sfss}))\\
  &=&\dim(\Ker(\phi))\\
  &=&\dim(\Ker(\phi^*\phi))\\
  &=&\mathbf c_1(\CO).
\end{eqnarray*}
The first map of \eqref{exacthom} is injective since $\phi_1$ is surjective. Then we have that
\begin{eqnarray*}
     && \dim ( \Hom_\C(V_1', V_2')^{\mathbf e_1'} ) \\
      &=& \dim V_2' \cdot \dim (V_1'/\mathbf e_1'(V_1'))\\
      & = &\dim V_2' \cdot  \bfcc_{2}(\cO)\\
      & = &\dim V_2' \cdot  \bfcc_{1}(\cO)\\
        & = &\dim V_2' \cdot  \dim (V_{\sfss}/\phi_1^*(V_1'))\\
        &=& \textrm{dimension of the kernel of the second map of  \eqref{exacthom}}.
    \end{eqnarray*}
Therefore the sequence \eqref{exacthom} is exact since the composition of  \eqref{exacthom} is the zero map. This finishes the proof of the lemma.
\end{proof}

\subsection{Proof of Proposition \ref{propreduced}} \label{secpp}

Now we suppose that $\mathbf e'\in \CO'\cap \p_{\sfss'}$. Denote by $K_{\mathbf e'}$ the stabilizer of $\mathbf e'$ in $K_{\sfss',\C}$. Consider the map
\[
 M':=  M_{\sfss'}|_{\CX_{\sfss, \sfss'}^{\bar \partial}} : \CX_{\sfss, \sfss'}^{\bar \partial}\rightarrow \p_{\sfss'}.
\]
As before, its fibre at $\mathbf e'$ equals the variety
\[
\CX_{\sfss, \sfss'}^{\CO, \mathbf e'}:=M_{\sfss}^{-1}(\CO\cap \p_\sfss)\cap  M_{\sfss'}^{-1}(\mathbf e').
\]
This variety is a finite union of open and closed
$K_{\sfss,\C}\times K_{\mathbf e'}$-orbits:
\[
   \CX_{\sfss, \sfss'}^{\CO, \mathbf e'}=\bigsqcup_{\sO\textrm{ is a   $K_{\sfss, \C}$-orbit in $ \CO\cap \p_\sfss$ such that $\sO\subset M_\sfss(\CX_{\sfss, \sfss'})$ and $\mathbf e'\in \nabla^\sfss_{\sfss'}(\sO)$}}  \CX_{\sfss, \sfss'}^{\sO, \mathbf e'},
 \]
 where
 \[
   \CX_{\sfss, \sfss'}^{\sO, \mathbf e'}:=M_\sfss^{-1}(\sO)\cap \CX_{\sfss, \sfss'}^{\CO, \mathbf e'}.
 \]
Thus  $\CX_{\sfss, \sfss'}^{\CO, \mathbf e'}$ is a smooth closed subvariety of $\CX_{\sfss, \sfss'}^{\bar \partial}$.

\begin{lem}\label{jacobic22}
Suppose that  $\phi\in \CX_{\sfss, \sfss'}^{\CO, \mathbf e'}$. Then the sequence
\be\label{jc111}
    \mathrm T_\phi (\CX_{\sfss, \sfss'}^{\CO, \mathbf e'})\xrightarrow{\textrm{ inclusion}} \mathrm T_\phi(\CX_{\sfss, \sfss'}^{\bar \partial})\xrightarrow{\textrm{the differential of $M'$}} \mathrm T_{\mathbf e'}(\p_{\sfss'})
\ee
is exact.
\end{lem}
\begin{proof}
Recall $\dot \epsilon\in \{\pm 1\}$ from \eqref{epsilond}. We have commutative diagrams
\[
 \begin{CD}
 \fggs \times  \g_{\mathbf e'} @>  \qquad  \textrm{\eqref{differ2} }\qquad  >>  W_{\sfss, \sfss'}
  @> \qquad  \textrm{\eqref{differ1}} \qquad  >>\fggsp\\
  @V(a,\,a') \mapsto  (L_\sfss\circ a \circ L_{\sfss}^{-1},\,  L_{\sfss'} \circ a' \circ L_{\sfss'}^{-1}) VV  @V b\mapsto - \dot \epsilon \sqrt{-1} L_{\sfss'}\circ b \circ L_\sfss^{-1} VV
  @V a' \mapsto - L_{\sfss'} \circ a' \circ L_{\sfss'}^{-1} VV \\
\fggs \times  \g_{\mathbf e'} @> \qquad   \textrm{\eqref{differ2} }\qquad  >>  W_{\sfss, \sfss'}
  @>\qquad  \textrm{\eqref{differ1}} \qquad  >> \fggsp.
\end{CD}
\]
Lemmas \ref{lemexact1} and \ref{lemexact2} imply that the horizontal sequences are exact. The lemma then follows by taking the fixed points of the vertical arrows.
\end{proof}

Lemmas \ref{jacobic} and \ref{jacobic22} imply that the scheme theoretic fibre product $\CX_{\sfss, \sfss'}^{\bar \partial}\times_{\p_{\sfss'}} \{\mathbf e'\}$ is reduced.
In other words,
\[
\left(\CX_{\sfss, \sfss'}^{\bar \partial}\times_{\p_{\sfss'}} ({\CO'}\cap \p_{\sfss'})\right)\times_{\CO'\cap \p_{\sfss'}}\{\mathbf e'\}
\]
is reduced.
Since $\mathbf e'\in \CO'\cap \p_{\sfss'}$ is arbitrary, this further implies that  the scheme theoretic fibre product
\[
\CX_{\sfss, \sfss'}^{\bar \partial}\times_{\p_{\sfss'}} ({\CO'}\cap \p_{\sfss'})
\]
is reduced (see \cite[Proposition~11.3.13]{EGAIV3} and \cite[Th\'eor\`eme~6.9.1]{EGAIV2}). Proposition  \ref{propreduced} is now proved.

\subsection{Geometric theta lifts of admissible orbit data}

\def\wedgetop{{\bigwedge}^{\mathrm{top}}}

\medskip

\DeclareDocumentCommand{\dliftv}{O{\sfss'} O{\sfss}}{
  {{\check \vartheta}_{#1}^{#2}}}
\def\dlift{{\check\vartheta}}
\def\AOD{\mathrm{AOD}}

Recall that  $\CO$ is regular for $\DDss$. Suppose that $\phi\in \CXOO$, $\mathbf e=M_\sfss(\phi)$ and $\mathbf e'=M_{\sfss'}(\phi)$. As before, $K_\mathbf e$ and $K_{\mathbf e'}$ are respectively the  stabilizers of $\mathbf e$ and $\mathbf e'$ in $K_{\sfss, \C}$ and $K_{\sfss', \C}$.  Write $\mathfrak k_{\mathbf e}$ and $\mathfrak k_{\mathbf e'}$ for the Lie algebras of $K_{\mathbf e}$ and $K_{\mathbf e'}$, respectively. As before, $(K_{\sfss, \C}\times K_{\sfss', \C})_{\phi}$ denotes the stabilizer of $\phi$ in $K_{\sfss, \C}\times K_{\sfss', \C}$.
Write $(\mathfrak k_\sfss\times \mathfrak k_{\sfss'})_\phi$ for the Lie algebra of  $(\Ks\times \Ksp)_{\phi}$.

\begin{lem}\label{lem:tan}
 As representations of  $(\mathfrak k_\sfss\times \mathfrak k_{\sfss'})_\phi$,
  \begin{equation}\label{eq:aod}
    \wedgetop \mathfrak k_{\X} \otimes \wedgetop \CXss  \cong \wedgetop \mathfrak k_{\Xp}.
  \end{equation}
Here $\wedgetop$ indicates the top degree exterior product of a complex vector space. 
\end{lem}
\begin{proof}
  We retain the notation in the proof of \Cref{lemexact2}. Since $\phi\in \CX_{\sfss, \sfss'}$, both $V_1'$ and $V_2'$ are $L_{\sfss'}$-stable.  Write
  \[
    \g_i'=\mathfrak k_i'\oplus \p_i'\qquad (i=1,2),
  \]
  where $\mathfrak k_i'$ is the centralizer of $L_{\sfss'}$ in $\g_i'$, and $\p_i'$ is its orthogonal complement in $\g_i'$ under the trace form.

   The sequence \eqref{jc111} leads to a short exact sequence
  \[
    0\longrightarrow (\fkks\oplus \mathfrak k_{\mathbf e'})/ (\mathfrak k_\sfss\times \mathfrak k_{\sfss'})_\phi
    \longrightarrow \CXss \longrightarrow \p_{\sfss'}/\p_2'  \longrightarrow 0
  \]
  of representations of $(\Ks\times \Ksp)_{\phi}$.
By
Lemma \ref{descko}, we have an exact sequence
\[
  0\rightarrow \mathfrak k_2'\rightarrow (\mathfrak k_\sfss\times \mathfrak k_{\sfss'})_\phi\rightarrow \mathfrak k_{\mathbf e}\rightarrow 0
\]
of representations of $(\Ks\times \Ksp)_{\phi}$. Therefore, as representations of $(\Ks\times \Ksp)_{\phi}$, we have that
\begin{eqnarray*}
      &&\wedgetop \mathfrak k_{\X} \otimes \wedgetop \CXss\\
      & \cong & \wedgetop (\p_{\sfss'}/\p_2')  \otimes \wedgetop (\fkks\oplus \mathfrak k_{\mathbf e'}) \otimes
      \left(\wedgetop \mathfrak k_2'\right)^{-1}\\
       & \cong & \wedgetop (\p_{\sfss'}/\p_2')  \otimes  \wedgetop \fkks \otimes
      \left(\wedgetop \mathfrak k_2'\right)^{-1}\otimes \wedgetop  \mathfrak k_{\mathbf e'}. \\
\end{eqnarray*}
The lemma then follows since  $(\mathfrak k_\sfss\times \mathfrak k_{\sfss'})_\phi$ acts trivially on $ \wedgetop (\p_{\sfss'}/\p_2')  \otimes  \wedgetop \fkks \otimes
      \left(\wedgetop \mathfrak k_2'\right)^{-1}$.
      \end{proof}

\begin{lem}\label{lem:aod}
  Suppose that $\sO$ is a $K_{\sfss,\C}$-orbit in $\cO\cap \fpps$ that is contained in the image of $M_\sfss$,  and
  $\sOp = \DDss(\sO)\subset \cOp\cap \fppsp$.
  Then for every $\CE' \in \AOD_{\sfss'}(\sOp)$,
   the geometric theta lift $\dlift_{\sO'}^{\sO}(\CE')$ defined in \eqref{liftab}    
   is either zero or an element of  $\AOD_{\sfss}(\sO)$.
\end{lem}
\begin{proof}
  Note that $(\zeta_{\sfss,\sfss'})^{2} \cong (\wedgetop \CXss)^{-1}$ as representations of $\mathfrak k_\sfss\times \mathfrak k_{\sfss'}$, and $\CE'$ is represented by a line bundle. Thus
   the lemma is a direct consequence of  \Cref{lem:tan}.
\end{proof}

The following is an obvious consequence of \Cref{lem:aod}. 
\begin{cor}
\label{lem:actau} 
Let $\uptau=(\tau,\wp)\in  \PBPe_\star (\ckcO, \mathsf s)$. Then the associated cycle $\AC(\uptau)$ defined in \eqref{eq:AC}
is a nonnegative integral linear combination of elements of
$\AOD_{\sfss }(\cO)$. 
\end{cor}

\section{Geometric theta lifts and induced nilpotent orbits: combinatorial description}\label{sec:nilmis}

In this section, we first review the combinatorial parametrization  of nilpotent
orbits and admissible orbit data. In terms of the combinatorial parameters, we will 
  describe geometric theta lifts and calculate weak associated cycles of some induced representations. 
  
Recall that $\star\in \{B,C,D,\widetilde {C}, C^*, D^*\}$. As before,  $\sfss=(\star, p,q)$ is a classical signature. Recall that $\abs{\sfss}:=p+q$. We call $\star$ the type of $\sfss$, and write $\star_\sfss:=\star$. Define the signature of $\sfss$ to be $\sign{\sfss}:=(p,q)$. 
Let $\sfss'=(\star', p',q')$ be another classical signature, where  $\star'$ is the Howe dual of $\star$.

\subsection{$\mathfrak{sl}_2$-modules}
As usual, $\mathfrak{sl}_2(\C)$ denotes the complex Lie algebra consisting of $2\times 2$ complex matrices
of trace zero. 
 For each $i\in \bN^+$, write $\St{i}:=\Sym^{i\!-\!1}(\bC^{2})$ (the $(i-1)$-th symmetric power), which is an irreducible
 $\slt$-module of dimension $i$. We equip $\St{i}$ with a triple 
 $(\inn{}{}_{\St{i}},J_{\St{i}},L_{\St{i}})$ that has the following properties:  
\def\Lsti{L_{\St{i}}}
\def\Jsti{J_{\St{i}}}
\def\innsti#1#2{\inn{#1}{#2}_{\St{i}}}
\begin{itemize}
\item $\inn{}{}_{\St{i}}$ is a bilinear form on $\St{i}$ that is $\slt$-invariant in the sense that
\[
\innsti{x\cdot u}{v}+ \innsti{u}{x\cdot v} = 0\qquad\textrm{for all  $u,v\in \St{i}$,  $\ x\in \slt$};
\]
  \item $J_{\St{i}}$ is the complex conjugation on $\St{i}=\Sym^{i-1}(\bC^{2})$ corresponding to the real form $\Sym^{i-1}(\bR^{2})$;
  \item $L_{\St{i}} = (-1)^{\floor{\frac{i-1}{2}}}\, \Sym^{i-1}(\text{\tiny$\begin{pmatrix}\phantom{-}0 & 1 \\-1 & 0 \end{pmatrix}$})$, which is a linear automorphism of $\St{i}$ such that $L_{\St{i}}^2=(-1)^{i-1}$;
  \item  the Hermitian form  \be\label{slth}
  \St{i}\times \St{i} \rightarrow \C, \quad (u,v)\mapsto \innsti{\Lsti(u)}{\Jsti(v)}
  \ee
 is positive definite.
\end{itemize}
Note that the first three conditions imply that \eqref{slth} is a Hermitian form. 

It is easy to see that the form $\innsti{}{}$ with the above conditions exists and is unique up to  positive scalar multiplications. Moreover,  the quadruple  $(\St{i}, \inn{}{}_{\St{i}},J_{\St{i}},L_{\St{i}})$
is a classical space of classical signature
\[
\left\{
  \begin{array}{ll}
    (B,(i+1)/2,(i-1)/2),\quad & \textrm{if $i$ is odd;}\\
    (C,i/2,i/2), \quad & \textrm{if $i$ is even.}
    \end{array}
    \right.
    \]

Put
\be\label{Sslt}
  \Xslt := \begin{bmatrix}1 & \phantom{-} \sqii \\
    \phantom{-} \sqii & -1 \end{bmatrix}
  \qquad \text{and} \qquad
\eslt := \begin{bmatrix}0 & \sqii\\ 0 & 0 \end{bmatrix}.
\ee
For each $\CO\in\Nil(\g_\sfss)$, 
by the Jacobson-Morozov theorem,  there is a
Lie algebra homomorphism
\begin{equation}\label{eq:slt00}
\varphi_\CO \colon \slt \rightarrow \fggs \subset \gl(\Vs)
 \end{equation}
such that  $\varphi_\CO(\Xslt), \varphi_\CO(\eslt)\in \CO$. Moreover, such a homomorphism is unique up to conjugation by $G_{\sfss, \C}$. Write \be\label{vco}
V_{\CO}:=V_\sfss,
\ee
to be viewed as an $\slt$-module via the homomorphism \eqref{eq:slt00}. 

Similarly, for each $\sO\in\Nil(\p_\sfss)$,  there is a
Lie algebra homomorphism
\begin{equation}\label{eq:slt000}
\varphi_\sO \colon \slt \rightarrow \fggs \subset \gl(\Vs)
 \end{equation}
such that 
\begin{itemize}
\item the diagram
\[
 \begin{CD}
          \slt
                  @>   \varphi_\sO  >>  \g_\sfss\\
            @V   \textrm{complex conjugation}  VV         @ VV x\mapsto J_\sfss x J_\sfss^{-1} V \\
      \slt
                  @>   \varphi_\sO  >>  \g_\sfss
  \end{CD}
\]
commutes;
\item the diagram
\[
 \begin{CD}
          \slt
                  @>   \varphi_\sO  >>  \g_\sfss\\
            @V   \textrm{negative transpose}  VV         @ VV x\mapsto L_\sfss x L_\sfss^{-1} V \\
      \slt
                  @>   \varphi_\sO  >>  \g_\sfss
  \end{CD}
\]
commutes;
    \item $\varphi_\sO(\Xslt)\in \sO$.
\end{itemize}
Moreover, such a homomorphism is unique up to conjugation by $K_{\sfss}$.  The map
\[
  \Nil(\p_\sfss)\rightarrow \{\textrm{nilpotent $G_\sfss$-orbits in $\sqii \g_{\sfss, \R}$ }\},\quad \sO\mapsto G_\sfss\cdot  (\varphi_\sO(\eslt))
\]
is a well-defined bijection, 
where $\g_{\sfss, \R}$ denotes the Lie algebra of $G_\sfss$. This is called the Kostant-Sekiguchi
correspondence.  See \cite{Se},  \cite[Section~6]{Vo89},  or \cite[Section 6]{SV}.

Write $V_{\sO}:=V_\sfss$, to be viewed as an $\slt$-module via the homomorphism \eqref{eq:slt000}. Then we have a decomposition
\begin{equation}\label{eq:Vl.100}
V_{\sO} = \bigoplus_{i\in \bN^{+}} \,^i V_{\sO}\otimes_\bC \St{i},
\end{equation}
where
\[
\,^i V_{
\sO} := \Hom_{\slt}( \St{i},V_{\sO})
\]
is the multiplicity space. Note that the decomposition \eqref{eq:Vl.100} is $J_\sfss$-stable as well as $L_\sfss$-stable. Moreover, for each $i\in \bN^+$, the subspace $\,^i V_{\sO}\otimes_\bC \St{i}$ of $V_\sfss$ is non-degenerate (with respect to $\la\,,\,\ra_\sfss$). 
It is easy to see that 
 there is a unique triple $(\inn{}{}_{\sO,i},J_{\sO,i},L_{\sO,i})$ with the following properties:
\begin{itemize}
\item $\inn{}{}_{\sO,i}$ is a bilinear form on $\,^i V_{\sO}$ such that 
\[
\inn{}{}_{\sO,i}\otimes \inn{}{}_{\St{i}}=\inn{}{}_\sfss;
\]
  \item $J_{\sO,i}$ is a conjugate-linear automorphism of $\,^i V_{\sO}$ such that 
\[
J_{\sO,i}\otimes J_{\St{i}}=J_\sfss;
\]
  \item   $L_{\sO,i}$ is a linear automorphism of $\,^i V_{\sO}$ such that 
\[
L_{\sO,i}\otimes L_{\St{i}}=L_\sfss.
\]
\end{itemize}

\def\CCSS{{\mathsf{CS}}}
\def\LLCC{{\mathsf{LC}}}

Define an equivalence relation $\sim$ on the label set $\{B,C,D, \wtC, C^*, D^*\}$ corresponding to the partition \be\label{partlabel}
\{\{B,D\}, \{C,\wtC\}, \{C^*\}, \{D^*\}\}.
\ee
Then there is a
unique classical signature 
\be\label{sso1}
\sfss_{\sO,i}=(\star_{\sO,i}, p_{\sO,i}, q_{\sO,i})
\ee
such that 
\begin{itemize}
    \item $\star_{\sO,i}\in \{B,C,D,C^*,D^*\}$;
    \item
     $\star_{\sO,i}\sim \star$ if $i$ is odd, and $\star_{\sO,i}\sim \star'$ if $i$ is even; 
    \item
      the quadruple 
$(\,^i V_{\sO},\inn{}{}_{\sO,i},J_{\sO,i},L_{\sO,i})$
is a classical space of classical signature $(\star_{\sO,i}, p_{\sO,i}, q_{\sO,i})$.
\end{itemize}

It is routine to check that 
\be\label{pqso}
  \begin{split}
    (p,q) = &
    \sum_{i\in \bN^{+}} \big(i \cdot p_{\sO, 2i} + i \cdot q_{\sO, 2i},
    i \cdot p_{\sO, 2i}+i\cdot q_{\sO, 2i}\big)\\
    &  +
    \sum_{i\in \bN^{+}} \big(i \cdot p_{\sO, 2i-1} + (i-1)\cdot q_{\sO, 2i-1},
    (i-1) \cdot p_{\sO, 2i-1}+i \cdot q_{\sO, 2i-1}\big).
  \end{split}
\ee

\subsection{Signed Young diagrams and the descent map}

For every Young diagram $\imath$ and every $i\in \bN^+$, write $\imath(i)$ for the number of rows of length $i$ in $\imath$.  We identify $\imath$ with the map
\[
  \bN^+\rightarrow \bN, \quad i\mapsto \imath(i).
\]
This map has  finite support, namely $\imath(i)=0$ for all but finitely many $i$. Via the above identification, the set of Young diagrams is then identified with the set of all maps $\bN^+\rightarrow \bN$
that have finite support, to be denoted by $\YD$.  
We say that a Young diagram $\imath$ has type $\star$ if 
\begin{itemize}
    \item $\imath(2)$, $\imath(4)$, $\imath(6)$, $\dots$, are even when $\star\in \{B,D, D^*\}$; 
    \item $\imath(1)$, $\imath(3)$, $\imath(5)$, $\dots$, are even when $\star\in \{C,\wtC, C^*\}$.
\end{itemize}
Write $\YD_\star$ for the subset of $\YD$ consisting of the Young diagrams that have type $\star$. Then the map
\begin{equation}\label{YDO}
  \begin{array}{rcl}
  \Nil(\g_{\sfss})&\rightarrow& \Set{\imath\in \YD_\star \,:\,  \abs{\imath}=\abs{\sfss}}, \smallskip\\
  \CO&\mapsto& \left(i\mapsto \dim \Hom_{\slt}(\St{i}, V_\CO)\right)
  \end{array}
\end{equation}
is a well-defined bijection, where $V_\CO$ is as in \eqref{vco}. Using this bijection, the set $\Nil(\g_{\sfss})$ is  identified with a subset of $\YD_\star$.

\begin{defn}
 A signed Young diagram  is a map
 \[
 \bN^+\rightarrow \bN\times \bN,\quad i\mapsto (p_i,q_i)
 \]
 with finite support (namely $p_i=q_i=0$ for all but finitely many $i$). It has  type $\star$ if the following condition is satisfied: 
 \begin{itemize}
     \item if $\star\in\{B,D\}$, then $p_i=q_i$ when $i$ is even;
     \item if $\star\in\{C,\wtC\}$, then $p_i=q_i$ when $i$ is odd;
     \item if $\star=C^*$, then $p_i, q_i$ are both even when $i$ is odd, and $p_i= q_i$  when $i$ is even;
     \item if $\star=D^*$, then  $p_i= q_i$  when $i$ is odd, and $p_i, q_i$ are both even when $i$ is even.
 \end{itemize}
\end{defn}

Let $\SYD_{\star}$ denote the set of all signed Young diagrams that have type $\star$. 
In view of the formula \eqref{pqso}, for every $\sO\in \SYD_{\star}$ we define its signature to be
\be\label{signso}
  \begin{split}
    \Sign(\sO) :=(p_\sO,q_\sO):= &
    \sum_{i\in \bN^{+}} \big(i \cdot p_{2i} + i \cdot q_{2i},
    i \cdot p_{2i}+i\cdot q_{2i}\big)\\
    &  +
    \sum_{i\in \bN^{+}} \big(i \cdot p_{2i-1} + (i-1)\cdot q_{2i-1},
    (i-1) \cdot p_{2i-1}+i \cdot q_{2i-1}\big),
  \end{split}
\ee
where we write $\sO(j)=( p_{j},q_{j})$ for every $j\in \bN^+$.
Then the map
\begin{equation}\label{SYDO}
  \begin{array}{rcl}
  \Nil(\p_{\sfss})&\rightarrow& \{\sO\in \SYD_\star\,:\, (p_\sO, q_\sO)=(p,q)\}, \smallskip\\
  \sO&\mapsto& \left(i\mapsto (p_{\sO,i}, q_{\sO,i})\right)
  \end{array}
\end{equation}
is a well-defined bijection, where $(p_{\sO,i}, q_{\sO,i})$ is as in \eqref{sso1}. 
 Using this bijection, the set $\Nil(\p_{\sfss})$ is  identified with a subset of $\SYD_\star$. 

For each signed Young diagram 
$\sO$, define its naive descent $\DDn(\sO)$ to be the signed Young diagram given by
\be\label{naivedso}
\big(\DDn(\sO)\big)(i) = \sO(i+1) \quad \text{for } i \in \bN^+.
\ee
It is clear that $\DDn(\sO)\in
\SYD_{\star'}$ when $\sO\in
\SYD_{\star}$.

Recall from \Cref{secmmap} 
the other classical signature $\sfss'=(\star', p',q')$, where  $\star'$ is the Howe dual of $\star$. 
The following lemma is an analogue of \Cref{imageofmm} and is proved by direct verification. The case of real nilpotent orbits is treated in \cite[Section 3.2]{Zh}.  

\begin{lem}\label{prop:DDss}
Let  $\sO\in \Nil(\fpps)$ and put 
\begin{equation}\label{eq:DDss1}
  (p_0,q_0):= (p',q') - \Sign(\DDn(\sO)). 
\end{equation}
Then $\sO$
is in the
image of the moment map $M_\sfss^{\sfss, \sfss'}$ (see \eqref{momentmap}) if and only if $p_0,q_0\geq 0$. 
When this is the case,   the descent $\sO' := \DDss(\sO)\in \Nil(\p_{\sfss'})$
is given  by
\[
\sO'(i)=
\left\{\begin{array}{ll}
  \sO(2)+(p_0,q_0),\quad& \textrm{if $i=1$};\\
  \sO(i+1),\quad& \textrm{if $i\geq 2$}.
  \end{array}
  \right.
  \]
\end{lem}

\subsection{Admissible orbit data and marked Young diagrams}
\label{sec:MYD}
\def\KCs{K_{[\sfss],\bC}}
\def\KCsX{(K_{[\sfss],\bC})_{\X}}
\def\KCsoX{(K_{[\sfss],\bC})_{\X}^{\circ}}

Recall the classical signature $[\sfss]$ defined in \eqref{bsfss}. Then we have a natural homomorphism $K_{\sfss,\C}\rightarrow K_{[\sfss],\C}$ that is either an isomorphism or a double covering map. 

Let $\sO\in\Nil(\p_\sfss)$. Fix a homomorphism  $\varphi_{\sO}: \slt\rightarrow \g_\sfss$  as in \eqref{eq:slt000}. Let $(K_{\sfss, \C})_\sO$ denote the centralizer of the image of $\varphi_{\sO}$ in $K_{\sfss, \C}$, and similarly let $(K_{[\sfss], \C})_\sO$ denote the centralizer of the image of $\varphi_{\sO}$ in $K_{[\sfss], \C}$. Then we have an identification
\[
(K_{[\sfss], \C})_\sO= \prod_{i\in \bN^{+}} K_{\sfss_{\sO,i},\C},
\]
where $\sfss_{\sO,i}=(\star_{\sO,i}, p_{\sO,i}, q_{\sO,i})$ is as in \eqref{sso1}.
Put $\mathbf e:=\varphi_\sO(\Xslt)$ (see \eqref{Sslt}), and write $(K_{\sfss,\C})_\mathbf e$ and $(K_{[\sfss],\C})_\mathbf e$ for its centralizers in $K_{\sfss,\C}$ and $K_{[\sfss],\C}$ respectively. Then $(K_{[\sfss], \C})_\sO$ is a Levi factor of $(K_{[\sfss], \C})_\mathbf e$, namely $(K_{[\sfss], \C})_\mathbf e$ is the semidirect product of $(K_{[\sfss], \C})_\sO$ with the unipotent radical. Consequently  $(K_{\sfss, \C})_\sO$ is a Levi factor of $(K_{\sfss, \C})_\mathbf e$.

For each $\CE\in \AOD_\sfss(\sO)$ and $i\in \bN^+$, we define a pair   $\CE(i)\in \Z\times \Z$ that is determined by the following conditions:
\begin{itemize}
    \item if $\star_{\sO,i}\notin\{B,D\}$, then $\CE(i)=(p_{\sO,i},q_{\sO,i})$;
    \item if $\star_{\sO,i}\in\{B,D\}$, then
    \[
    \CE(i)=((-1)^{\epsilon_i^+} \cdot p_{\sO,i}, (-1)^{\epsilon_i^-}\cdot  q_{\sO,i}),
    \]
    where $\epsilon_i^+$ and $\epsilon_i^-$ are elements of $\Z/2\Z$ such that
   \[
   \begin{cases}
       (\CE_{\mathbf e})|_{K_{\sfss_{\sO,i},\C}}\cong \det_1^{\epsilon_i^+} \otimes \det_{-1}^{\epsilon_i^-},\quad &\textrm{if $\star\neq \wtC$};\\
       \left(\left(\CE\otimes \det_{\sqrt{-1}}^{-\frac{1}{2}}\right)_{\mathbf e}\right)|_{K_{\sfss_{\sO,i},\C}}\cong \det_1^{\epsilon_i^+} \otimes \det_{-1}^{\epsilon_i^-},\quad &\textrm{if $\star= \wtC$}.
    \end{cases}
   \]
\end{itemize}
 Here $\det_{1}$ and $\det_{-1}$ are characters of $K_{\sfss_{\sO,i},\C}$ as in \eqref{comdetlam}, 
 $\det_{\sqrt{-1}}^{-\frac{1}{2}}:=\left(\det_{\sqrt{-1}}^{\frac{1}{2}}\right)^{-1}$ is a character of $K_{\sfss, \C}$ as in \eqref{dethalf}, a subscript $\mathbf e$ indicates the fibre at $\mathbf e$ of a vector bundle, and  $\left(\CE\otimes \det_{\sqrt{-1}}^{-\frac{1}{2}}\right)_{\mathbf e}$ (which is originally a representation of $(K_{\sfss,\C})_\mathbf e$) is viewed as a representation of $(K_{[\sfss],\C})_\mathbf e$ by descent.

\begin{defn}
A marked Young diagram is a map
 \[
 \bN^+\rightarrow \Z\times \Z,\quad i\mapsto (p_i,q_i)
 \]
 with finite support (namely $p_i=q_i=0$ for all but finitely many $i$). It has  type $\star$ if the following condition is satisfied: 
 \begin{itemize}
     \item if $\star\in\{B,D\}$, then $p_i=q_i\in \bN$ when $i$ is even;
     \item if $\star\in\{C,\wtC\}$, then $p_i=q_i\in \bN$ when $i$ is odd;
     \item if $\star=C^*$, then $p_i, q_i\in 2\bN$ when $i$ is odd, and $p_i= q_i\in \bN$  when $i$ is even;
     \item if $\star=D^*$, then  $p_i= q_i\in \bN$  when $i$ is odd, and $p_i, q_i\in 2\bN$  when $i$ is even.
 \end{itemize}
\end{defn}

Let $\MYD_{\star}$ denote the set of all marked Young diagrams that have type $\star$. Define a map
\begin{equation}\label{eq:F1}
\sV\colon \Z\times \Z\rightarrow \bN\times \bN,\qquad (a,b)\mapsto (\abs{a},\abs{b}).
\end{equation}
Put
\[
  \MYD_\star(\sO) := \Set{\CE\in \MYD_\star  \,:\, \sV\circ \CE=\sO }.
\]
Then we have a well-defined map
\be\label{identmy}
\begin{array}{rcl}
\AOD_{\sfss}(\sO) & \rightarrow & \MYD_\star(\sO),\\
    \CE&\mapsto&(i\mapsto \CE(i)). 
    \end{array}
\ee
It is easy to see that the map \eqref{identmy} is independent of the choice of $\varphi_\sO$, and is bijective whenever the set $\AOD_{\sfss}(\sO)$ is nonempty. By using the map \eqref{identmy}, we identify every admissible orbit datum in $\AOD_{\sfss}(\sO)$ with the corresponding marked Young diagram in $\MYD_\star(\sO)$.

\subsection{Four operations on marked Young diagrams}\label{fouro}

In this subsection, we will define four operations on marked Young diagrams. The first one, corresponding to sign twist for orthogonal groups, will be used to define the combinatorial cycle in \eqref{indac}. The next three ones will be used to describe theta lifts of admissible orbit data in \eqref{eq:LCD} and \eqref{eq:LDC}.

\smallskip

(i) {\bf  Sign twist.}
Suppose that $\star\in \set{B,D}$. Let $(\epsilon^{+},\epsilon^{-})\in \bZ/2\bZ\times \bZ/2\bZ$. For every $\CE\in \MYD_{\star}$  and $i\in \bN^+$, write
\[
(\CE \lotimes(\epsilon^{+},\epsilon^{-}))(i) := \begin{cases}
  ( (-1)^{\frac{(i+1)}{2}\cdot \epsilon^{+} +\frac{(i-1)}{2}\cdot \epsilon^{-}}p_{i},(-1)^{\frac{(i-1)}{2}\cdot \epsilon^{+} +\frac{(i+1)}{2}\cdot \epsilon^{-}} q_{i}),
  & \text{if $i$  is odd;}\\
  ( p_{i},q_{i}), & \text{if $i$  is even,}
\end{cases}
\]
where $( p_{i},q_{i}):=\CE(i)$. Then $\CE \lotimes(\epsilon^{+},\epsilon^{-})\in \MYD_{\star}$ and we get a  map 
\be\label{stwist}
\begin{array}{rcl}
  \MYD_{\star}&\rightarrow&\MYD_{\star}, \\
   \CE&\mapsto &\CE \lotimes(\epsilon^{+},\epsilon^{-}).
   \end{array}
\ee
This is called the sign twist operation. 

It is easily verified that the diagram 
\[
 \begin{CD}
       \AOD_{\sfss}(\sO)
                  @>  (\,\cdot\,) \otimes((\det_1)^{\epsilon^+}\otimes (\det_{-1})^{\epsilon^-}) >>  \AOD_{\sfss}(\sO)\\
            @V  \eqref{identmy}  VV         @ VV  \eqref{identmy} V \\
     \MYD_\star (\sO)
                  @>\qquad \ (\,\cdot\,)  \otimes(\epsilon^{+},\epsilon^{-})\ \qquad  >>  \MYD_\star (\sO)
  \end{CD}
\]
commutes for all $(\epsilon^{+},\epsilon^{-})\in \bZ/2\bZ\times \bZ/2\bZ$, where $\det_{1}$ and $\det_{-1}$ are characters of $K_{\sfss,\C}$ as in \eqref{comdetlam}.

\medskip
(ii) {\bf  Involution.}
Suppose that $\star\in \set{C,\wtC}$. For every $\CE\in \MYD_{\star}$, and $i\in \bN^+$, write
\begin{equation}\label{eq:inv}
(\maltese\CE)(i) := \begin{cases}
  -\CE(i),& \text{if $i\equiv 2 \pmod{4}$;}\\
  \CE(i), & \text{otherwise.}
\end{cases}
\end{equation}
Then $\maltese\CE\in \MYD_{\star}$, and the map
\be\label{tinv}
  \mathrm T: \MYD_{\star}\rightarrow \MYD_{\star}
  \ee
  is an involution, namely $\mathrm T^2$ is the identity map.
  
\medskip

(iii) {\bf  Augmentation.}
Let $(p_{0},q_{0})\in \bN\times \bN$. When $\star\in \{B,D\}$ or  
$ (\star,p_{0},q_{0})$ is a classical signature, we define a map
\be\label{aug}
  \begin{array}{rcl}
  \MYD_{\star'}&\rightarrow & \MYD_{\star},\\
  \CE'&\mapsto & \CE'\circledast (p_0,q_0):=\left(i\mapsto \begin{cases}
    (p_0,q_0), & \text{if } i=1;\\
    \CE'(i-1), & \text{if } i>1
  \end{cases}\right).
  \end{array}
\ee
This is called the augmentation operation, which changes type of the marked Young diagram from $\star'$ to $\star$. 

\medskip

(iv) {\bf  Truncation.}
Let $(p_{0},q_{0})\in \bZ\times \bZ$. For  every $\CE\in \MYD_{\star}$  
we write
\begin{equation}
  \label{eq:sub}
  \CE\sqsupseteq (p_{0},q_{0}) \quad \textrm{if and only if } \quad 
  \begin{cases}
  p_{1}\geq p_0 \geq 0  \ \text{ or }\ 
  p_{1}\leq p_0 \leq 0; $\ $ and \\
  q_{1}\geq q_0 \geq 0 \  \text{ or } \ 
  q_{1}\leq q_0 \leq 0, \\
  \end{cases} 
\end{equation}
where $ (p_{1},q_{1}):=\CE(1)$.

 When $\star\in \{B,D\}$ or  
$ (\star,p_{0},q_{0})$ is a classical signature, we define a map
\be\label{tru}
 \begin{array}{rcl}
 \Lambda_{(p_0,q_0)}: \{\CE\in \MYD_{\star}\,:\,\CE\sqsupseteq (p_{0},q_{0})\} &\rightarrow & \MYD_{\star},\\
  \CE&\mapsto & \Lambda_{(p_0,q_0)}\CE, 
  \end{array}
\ee
where $\Lambda_{(p_0,q_0)}\CE$ is the marked Young diagram 
\[
i\mapsto \begin{cases}
    \CE(1)-(p_{0},q_{0}), & \text{if } i=1;\\
    \CE(i), & \text{if } i>1. 
  \end{cases}
  \]
This is called the truncation operation.

For every set $X$, write $\Z[X]$ for the free abelian group with free generators $X$. Then the sign twisting operation \eqref{stwist}
extends to a homomorphism 
\be\label{homom1}
\begin{array}{rcl}
  \Z[\MYD_{\star}]&\rightarrow&\Z[\MYD_{\star}], \\
   \CE&\mapsto &\CE \lotimes(\epsilon^{+},\epsilon^{-}).
   \end{array}
\ee
Similarly, the involution 
\eqref{tinv} extends to a homomorphism 
\be\label{homom2}
\begin{array}{rcl}
\mathrm T:   \Z[\MYD_{\star}]&\rightarrow&\Z[\MYD_{\star}]. \\
   \end{array}
\ee
For every pair $(p_0,q_0)\in \frac{1}{2}\Z\times \frac{1}{2}\Z$, we define the augmentation operation  
\be\label{homom3}
\begin{array}{rcl}
  \Z[\MYD_{\star'}]&\rightarrow&\Z[\MYD_{\star}], \\
   \CE'&\mapsto &\CE' \circledast (p_0,q_0)
   \end{array}
\ee
to be the homomorphism determined by the following conditions: 
\begin{itemize}
    \item if $(p_{0},q_{0})\in \bN\times \bN$, and $\star\in \{B,D\}$ or  
$ (\star,p_{0},q_{0})$ is a classical signature, then the map \eqref{homom3} extends the map \eqref{aug};
\item
otherwise  \eqref{homom3} is the zero map. 
\end{itemize}
The truncation operation \eqref{tru} also 
extends to a  homomorphism
\be\label{homom4}
 \begin{array}{rcl}
 \Lambda_{(p_0,q_0)}: \Z[\MYD_{\star}] &\rightarrow & \Z[\MYD_{\star}],\\
  \CE&\mapsto & \Lambda_{(p_0,q_0)}\CE
  \end{array}
\ee
that  vanishes on $\MYD_\star\setminus \set{\CE\in \MYD_{\star}\,:\,\CE\sqsupseteq (p_{0},q_{0})}$. 

The map 
\[
  \MYD_{\star}\rightarrow \SYD_\star, \quad \CE\mapsto \sV\circ \CE \qquad (\text{$\sV$ is defined in \eqref{eq:F1}})
\]
is surjective  and extends to a surjective homomorphism
\be\label{homom5}
\Z[\MYD_{\star}]\rightarrow \Z[\SYD_\star].
\ee
The homomorphism \eqref{homom3} (uniquely) descends to a homomorphism 
\be\label{homom6}
\begin{array}{rcl}
  \Z[\SYD_{\star'}]&\rightarrow&\Z[\SYD_{\star}], \\
   \sO'&\mapsto &\sO' \circledast (p_0,q_0)
   \end{array}
\ee
in the sense that the diagram 
\[
 \begin{CD}
       \Z[\MYD_{\star'}]
                  @>(\,\,) \circledast (p_0,q_0) >>  \Z[\MYD_{\star}]\\
            @V \eqref{homom5}   VV         @ VV  \eqref{homom5}  V \\
      \Z[\SYD_{\star'}]
                  @> (\,\,) \circledast (p_0,q_0) >>  \Z[\SYD_{\star}]\\
  \end{CD}
\]
commutes. Similarly the homomorphism \eqref{homom4} descends to a homomorphism 
\be\label{homom7}
\begin{array}{rcl}
  \Lambda_{(p_0,q_0)}: \Z[\SYD_{\star}]&\rightarrow&\Z[\SYD_{\star}], \\
   \sO&\mapsto & \Lambda_{(p_0,q_0)} \sO.
   \end{array}
\ee
Obviously the homomorphisms \eqref{homom1} and \eqref{homom2} descend to the identity map on  $\Z[\SYD_{\star}]$.

\subsection{Theta lift of admissible orbit data: combinatorial description}

\def\CEp{\CE'}
\def\dliftso{{\check \vartheta}_{\sfss',\cO'}^{\sfss\,,\,\cO}}

Let $\CO\in \Nil(\g_\sfss)$, and put 
\be\label{MYDs}
  \MYD_\sfss(\CO):=\bigsqcup_{\sO\in \Nil(\p_\sfss), \sO\subset \CO}  \MYD_\star(\sO). 
  \ee
  Suppose that $\CO$ is contained in the image of the moment map $\tilde M_{\sfss}$ (see \eqref{momentmap2}), and put $\CO':=\nabla_{\sfss'}^\sfss(\CO)\in \Nil(\g_{\sfss'})$. 
In what follows we will define the theta lift of a marked Young diagram, namely a homomorphism
\[
  \dliftso \colon \bZ[\MYD_{\sfss'}(\CO')]\rightarrow \bZ[\MYD_\sfss(\CO)].
\]  
This is the combinatorial analog of the geometric theta lift and will be used to provide a combinatorial description of the associated cycles of the relevant theta lifts (see Theorem \ref{thm:GDS.AC} in \Cref{sec:equac}). 

 Suppose that $\CE'\in\MYD_{\star '}(\sO')$, where $\sO'\in \Nil(\p_{\sfss'})$ and $\sO'\subset \CO'$.
 We only need to define $\dliftso (\CE')$.  
 As in \Cref{imageofmm}, put
\[
 \delta:=  \abs{\sfss'}-\abs{\DD_\mathrm{naive}(\CO)}.
\]
 \Cref{imageofmm} implies that
\[
\delta= \bfcc_{1}(\cO')-\bfcc_{2}(\cO)\geq 0.
\]
 As in \eqref{signso}, we write $\sO'(i)=( p'_i, q'_i)$ ($i\in \bN^+$) and put
\[
  \lsign(\sO') :=
    \sum_{i\in \bN^{+}} \big(p'_{2i},q'_{2i} \big)
      + \sum_{i\in \bN^{+}} \big(q'_{2i-1},p'_{2i-1} \big).
\]

{\bf The case when $\star \in \set{B,D, C^{*}}$.} Note that $\delta$ is even when $\star\in \set{B,D}$. Set
\[
(p_{0}, q_{0})  := (p,q) -(p',q')-\lsign(\sO') + 
(\frac{\delta}{2},\frac{\delta}{2}).
\]
We define
\begin{equation}\label{eq:LCD}
  \dliftso(\CEp) :=
  \begin{cases}
  \left(\maltese^{\frac{p-q+1}{2}}(\Lambda_{(\frac{\delta}{2},\frac{\delta}{2})}( \CEp))\right)\circledast (p_{0},q_{0}), &
    \text{if } \star=B; \\
    \left(\maltese^{\frac{p-q}{2}}(\Lambda_{(\frac{\delta}{2},\frac{\delta}{2})}( \CEp))\right)\circledast (p_{0},q_{0}), &
    \text{if }  \star=D; \\
    \big(\Lambda_{(\frac{\delta}{2},\frac{\delta}{2})}( \CEp)\big)\circledast (p_{0},q_{0}), &
    \text{if } \star=C^{*}. \\
  \end{cases}
\end{equation}

{\bf The case when  $\star \in \set{C,\wtC,D^{*}}$.}
Set \[
n_{0} := \frac{\bfcc_{1}(\cO)-\bfcc_{2}(\cO)}{2}.
\]
Then $n_0\in \bN$ when $\star\in \set{C,\wtC}$. Also note that $\delta$ is even when $\star=D^*$. We define
\begin{equation}\label{eq:LDC}
  \dliftso(\CEp) :=
  \begin{cases}
    \maltese^{\frac{p'-q'}{2}}\big(\sum_{j=0}^{\delta} \Lambda_{(j,\delta-j)}( \CEp)\circledast (n_{0},n_{0})\big), & \text{if
    } \star = C;\\
    \maltese^{\frac{p'-q'-1}{2}}\big(\sum_{j=0}^{\delta} \Lambda_{(j,\delta-j)}( \CEp)\circledast (n_{0},n_{0})\big), & \text{if
    } \star = \wtC;\\
    \sum_{j=0}^{\frac{\delta}{2}} \Lambda_{(2j,\delta-2j)}( \CEp)\circledast (n_{0},n_{0}),\quad & \text{if
    $\star = D^{*}$.}
  \end{cases}
\end{equation}

Note that when $\star \in \set{B,D}$, we have $\star' \in \set{C,\wtC}$. Conversely when $\star \in \set{C,\wtC}$, we have $\star' \in \set{B,D}$. It is then easily seen that in both cases $\dliftso(\CEp)$ is well-defined for $\CEp \in \MYD_{\sfss'}(\CO')$, and that $\dliftso(\CE')\in \MYD_\sfss(\CO)$.

If $\cO$ is regular for $\DDss$, then admissible orbit data are
lifted to admissible orbit data, in view of \Cref{lem:aod}. The following proposition is now routine to check, which we skip.

\begin{prop}\label{lem:clift}
Suppose that $\cO$ is regular for $\DDss$. Then the diagram 
\[
  \begin{tikzcd}
    \bZ[\AOD_{\sfss'}(\cO')]\ar[r,"\dliftso"]\ar[d] & \bZ[\AOD_{\sfss}(\cO)]
     \ar[d] \\
    \bZ[\MYD_{\sfss'}(\cO')] \ar[r,"\dliftso"] & \bZ[\MYD_{\sfss}(\cO)]
  \end{tikzcd}
\]
commutes, where the vertical arrows are the homomorphisms induced by the map \eqref{identmy}, and the top horizontal map is defined in  \Cref{sec:dlift}. 
\end{prop}

\subsection{Induction of signed Young diagrams}
\label{subsec:induced}
\def\indsss{\Ind_{R_{\sfss'',\sfss_0}}^{G_{\sfss''}}}
\def\indss{\Ind_{R_{\sfss'',\sfss_0,\bC}}^{G_{\sfss'',\bC}}}
\def\WFw{{\mathrm{WF^{weak}}}}
\def\Gor{G_1}
\def\Ror{R_1}
\def\fggor{\fgg_{1,\bR}}
\def\frror{\frr_{1,\bR}}
\def\fnnor{\fnn_{1,\bR}}
\def\sOpr{\sO'_\bR}
\def\sOr{\sO_\bR}
\def\Indrg{\Ind_{R_{1}}^{G_{1}}}
\def\Indrgc{\Ind_{R_{1,\bC}}^{G_{1,\bC}}}

Recall that $\sfss'=(\star', p',q')$. Let $\CO'\in\Nil(\g_{\sfss'})$, as before. Let $\sfss''$ be a classical signature of the form $(\star',p'+\kappa,q'+\kappa)$, where 
\begin{itemize}
    \item $\kappa\in \bN^+$ when $\star'\in \{B,C,\wtC, D\}$, and 
    \item $\kappa\in 2 \bN^+$ when $\star'\in \{C^*,D^*\}$. 
\end{itemize}
Then $\g_{\sfss''}$ has a Levi subalgebra of the form $\gl_{\kappa}(\C)\times \g_{\sfss'}$. Identify $\CO'$ with the nilpotent orbit $\{0\}\times \CO'$ in $\gl_{\kappa}(\C)\times \g_{\sfss'}$,
and put
\be\label{induo0}
\CO'':=\Ind_{\sfss'}^{\sfss''}(\CO'):=\Ind_{\gl_{\kappa}(\C)\times \g_{\sfss'}}^{\g_{\sfss''}}(\CO')\in \Nil(\g_{\sfss''})
\ee
for the induced nilpotent orbit, namely the unique nilpotent orbit in $\Nil(\g_{\sfss''})$ that contains a nonempty Zariski open subset of  $
  \CO'+\mathfrak n_0$. Here  $\mathfrak n_0$ is the nilpotent radical of a parabolic subalgebra of $\g_{\sfss''}$ that contains $\gl_{\kappa}(\C)\times \g_{\sfss'}$ as a Levi factor. See \cite{LS} or \cite[Section 7.1]{CM}.

We assume that 
\be \label{eq:dimcon}
\kappa \geq \mathbf c_2(\CO'),
\ee
and put 
\be \label{eq:tzero}
  \delta_0:=\kappa -\mathbf c_1(\CO'). 
\ee
By \cite[Theorem 7.3.3]{CM}, the Young diagram of $\CO''$ is determined by the following conditions:
\begin{itemize}
    \item $\bfcc_i(\cOpp)=\bfcc_{i-2}(\cOp)$ for $i=4,5,6,\dots$; 
    \item if $\star'\in \set{B,D, C^*}$,
    then
    \[
      (\bfcc_1(\cOpp), \bfcc_2(\cOpp),\bfcc_3(\cOpp))
      =\begin{cases}
        (\kappa,\kappa,\bfcc_1(\cOp)),     & \text{if $\delta_0\geq 0$ is even}; \\
          (\kappa+1,\kappa-1,\bfcc_1(\cOp)), & \text{if $\delta_0\geq 0$ is odd};  \\
          (\bfcc_1(\cOp),\kappa,\kappa),     & \text{if $\delta_0<0$};             \\
      \end{cases}
    \]
    \item  if $\star'\in \set{C,\wtC, D^*}$,
    then
    \[
      (\bfcc_1(\cOpp), \bfcc_2(\cOpp),\bfcc_3(\cOpp))
      =\begin{cases}
        (\kappa,\kappa,\bfcc_1(\cOp)),     & \text{if $\delta_0\geq 0$}; \\
        (\bfcc_1(\cOp),\kappa,\kappa), & \text{if $\delta_0 < 0$ is even};  \\
        (\bfcc_1(\cOp),\kappa +1,\kappa-1),     & \text{if $\delta_0<0$ is odd}.             \\
      \end{cases}
    \]
\end{itemize}

Recall that $\Nil(\p_{\sfss'})$ is identified with a subset of $\SYD_{\star'}$.  Put 
\[
\SYD_{\sfss'}(\cO'):=\Nil_{\sfss'}(\CO'):=\{\sO'\in \Nil(\p_{\sfss'})\,:\, \sO'\subset \CO'\}\subset \SYD_{\star'}.
\]
Similarly we have a set  $\SYD_{\sfss''}(\cO'')$. In what follows we will define a homomorphism  
\[
  \Ind_{\sfss'}^{\sfss''}: \Z[\SYD_{\sfss'}(\cO')] \rightarrow \Z[\SYD_{\sfss''}(\cO'')].
\]
Let $\sOp\in
\SYD_{\sfss'}(\cOp)$. We only need to define $\Ind_{\sfss'}^{\sfss''}(\sO')\in \Z[\SYD_{\sfss''}(\cO'')]$. Note that $\delta_0$ is even when $\star'\in \{C^*, D^*\}$, as the set $\SYD_{\sfss'}(\cOp)$ is nonempty. 

\medskip 

{\bf The case when  $\star' \in \set{B,D, C^{*}}$.}
In this case, we define 
\begin{equation}\label{eq:IndSYD1}
      \Ind_{\sfss'}^{\sfss''}(\sOp): =
      \begin{cases}
        \sOp\circledast (\frac{\delta_0}{2},\frac{\delta_0}{2})\circledast (0,0),
         & \text{if $\delta_0 \geq 0$ is even}; \\
        2  \left(\sOp\circledast (\frac{\delta_0-1}{2},\frac{\delta_0-1}{2})\circledast (1,1) \right),
         & \text{if $\delta_0\geq 0$ is odd};   \\
        \sum_{a,b\in \bN, a+b=-\delta_0}
        (\Lambda_{(a,b)}\sOp)\circledast(0,0)\circledast (a,b),
         & \text{if $\star'\in \{B,D\}$ and $\delta_0<0$};
         \\
         \sum_{a,b\in 2\bN, a+b=-\delta_0}
        (\Lambda_{(a,b)}\sOp)\circledast(0,0)\circledast (a,b),
         & \text{if $\star'=C^*$ and $\delta_0<0$}.
      \end{cases}
\end{equation}

{\bf The case when  $\star' \in \set{C,\wtC, D^{*}}$.}
In this case, we define 
\begin{equation}\label{eq:IndSYD2}
      \Ind_{\sfss'}^{\sfss''}(\sOp) :=
      \begin{cases}
        \sum_{a, b\in \bN, a+b=\delta_0}\sOp\circledast (a,b)\circledast (0,0),
           & \text{if $\star'\in \{C,\wtC\}$ and $\delta_0\geq 0$}; \\
            \sum_{a, b\in 2 \bN, a+b=\delta_0}\sOp\circledast (a,b)\circledast (0,0),
           & \text{if $\star'=D^*$ and $\delta_0\geq 0$}; \\
        (\Lambda_{(\frac{-\delta_0}{2},\frac{-\delta_0}{2})}\sOp) \circledast (0,0)
        \circledast (\frac{-\delta_0}{2},\frac{-\delta_0}{2}),
         & \text{if $\delta_0 < 0$ is even};  \\
        2\left((\Lambda_{(\frac{1-\delta_0}{2},\frac{1-\delta_0}{2})}\sOp)
        \circledast (1,1) \circledast (\frac{1-\delta_0}{2},\frac{1-\delta_0}{2})\right),
            & \text{if $\delta_0<0$ is odd}.             
      \end{cases}
\end{equation}
   
In both cases, it is easily verified that  $\Ind_{\sfss'}^{\sfss''}(\sOp)\in \Z[\SYD_{\sfss''}(\CO'')]$.

\smallskip
\smallskip

For every   $K_{\sfss',\C}$-equivariant algebraic vector bundle $\CE'$ over  $\sO'\in \Nil(\p_{\sfss'})$, define 
\[
   (\CE')^{\mathrm{weak}}:=\textrm{(rank of $\CE'$)}\cdot \sO'\in \Z[\SYD_{\sfss'}(\CO')].
\]
This yields a homomorphism
\[
  \CK_{\sfss'}(\CO')\rightarrow \Z[\SYD_{\sfss'}(\CO')], \quad \CE'\mapsto (\CE')^{\mathrm{weak}}.
\]
For an $\CO'$-bounded Casselman-Wallach representation $\pi '$ of $G_{\sfss'}$, 
we call $(\AC_{\cO'}(\pi'))^{\mathrm{weak}}$ the weak associated cycle of $\pi'$. Likewise we have a homomorphism 
\[
  \CK_{\sfss''}(\CO'')\rightarrow \Z[\SYD_{\sfss''}(\CO'')], \qquad \CE''\mapsto (\CE'')^{\mathrm{weak}}.
\]

Put  
\[
  \sfss_0 :=
  \begin{cases}
     (D, \kappa, \kappa ),& \textrm{if $\star'=B$};\\
     (\star',\kappa, \kappa),\quad & \textrm{otherwise}.
  \end{cases}
\]
As in   
\eqref{parabolic}, 
  $P_{\sfss'',\sfss_0}$ is a parabolic subgroup of $G_{\sfss''}$ with a Levi component $G_{\sfss'}\cdot R_{\sfss_0} $, where
\[
  R_{\sfss_0} \cong
  \begin{cases}
     \GL_{\kappa}(\R),\quad & \textrm{if $\star'\in \{B,D,C\}$};\\
     \widetilde \GL_{\kappa}(\R),\quad & \textrm{if $\star'=\wtC$};\\
     \GL_{\frac{\kappa}{2}}(\bH),\quad & \textrm{if $\star'\in \{C^*,D^*\}$}.
  \end{cases}
\]

\begin{prop}\label{lem:ind00}
Let $\pi'$ be an $\CO'$-bounded Casselman-Wallach representation of $G_{\sfss'}$ that extends to a representation of the Levi subgroup $G_{\sfss'}\cdot R_{\sfss_0}$ (still denoted by $\pi'$). View it  as a representation of $P_{\sfss'',\sfss_0}$ through the trivial action of the unipotent radical. Then the (normalized smooth) parabolically induced representation $\Ind_{P_{\sfss'',\sfss_0}}^{G_{\sfss''}} (\pi')$ is $\CO''$-bounded, and 
\be\label{wac0}
 \left( \AC_{\cO''}(\Ind_{P_{\sfss'',\sfss_0}}^{G_{\sfss''}} (\pi'))\right)^{\mathrm{weak}} = \Ind_{\sfss'}^{\sfss''}\left( (\AC_{\cO'}(\pi'))^{\mathrm{weak}}\right).
\ee
\end{prop}

\begin{proof} By the result of Schmid-Vilonen (\cite{SV}), the weak associated cycle agrees with the wavefront cycle under the Kostant-Sekiguchi correspondence.  On the other hand, the wavefront cycle for a parabolically induced representation is described by Barbasch in \cite[Corollary~5.0.10]{B.Orbit} via induction of nilpotent orbits. In Propositions \ref{prop:A6} and \ref{prop:A7} of \Cref{sec:orbitinduce}, we explicitly determine the induced nilpotent orbits and their multiplicities that are relevant to the current proposition. One may then read off the signed Young diagrams of the nilpotent $K_{\sfss'',\C}$-orbits occurring in the left-hand side of \eqref{wac0}.  Together with the computed multiplicities, the proposition follows. 
\end{proof}

\begin{remark} We thank Professor K. Vilonen for confirming that the results of \cite{SV} remain valid for all real reductive groups that may or may not be linear.
\end{remark}

\section{Matching the associated  cycles and proof of Theorem \ref{thmpitau}}\label{sec:equac}

The main purpose of this section is to prove \Cref{thmpitau}, and in particular to determine the associated cycle of the constructed representation $\pi _{\uptau}$, where 
$\uptau \in \mathrm{\PBP}^{\mathrm{ext}}_{G}(\check \CO)$. 

As before, $\mathsf s=(\star, p,q)$ and $ \mathsf s'=(\star', p',q')$ are classical signatures such that $\star'$ is the Howe dual of $\star$.

\subsection{Good descents and the doubling method}\label{subsec:gdescent}

\begin{defn}
A nilpotent orbit $\CO\in \Nil(\g_\mathsf s)$ is good for $\DD_{\mathsf s'}^{\mathsf s}$ if it is regular for $\DD_{\mathsf s'}^{\mathsf s}$ (Definition \ref{defn:regular}) and satisfies the following additional condition:
\[
\begin{cases}
   \mathbf c_1(\CO)>\mathbf c_2(\CO), \qquad  &\textrm{if $\star \in\{B,D\}$}; \\
      \abs{\sfss'}-\abs{\DD_\mathrm{naive}(\CO)}\in\{0,1\},\qquad  &\textrm{if $\star \in\{C, \widetilde C\}$}; \\
\abs{\sfss'}=\abs{\DD_\mathrm{naive}(\CO)}, \qquad &\textrm{if $\star \in \{C^*, D^*\}$}.
  \end{cases}
\]

\end{defn}

In the rest of this section we suppose that $\CO\in \Nil(\g_\sfss)$ is good for $\DD_{\mathsf s'}^{\mathsf s}$. Put $\CO':=\DD_{\mathsf s'}^{\mathsf s}(\CO)$ as before. Write
 \[
  \kappa :=\begin{cases}
    \abs{\sfss}-\abs{\sfss'}-1, \qquad  &\textrm{if $\star \in\{B,D\}$}; \\
       \abs{\sfss}-\abs{\sfss'}+1, \qquad  &\textrm{if $\star \in\{C, \widetilde C\}$}; \\
 \abs{\sfss}-\abs{\sfss'},\qquad  &\textrm{if $\star \in \{C^*, D^*\}$},
  \end{cases}
\]
which is a nonnegative integer, and is even when $\star \in \{C^*, D^*\}$.
Similar to \eqref{p123}, we have that
\be\label{p1234}
-(\kappa+1)=
  \begin{cases}
 \nu_{\sfss'}-\abs{\sfss} +2,\quad& \textrm{if $\star =C^*$};\\
  \nu_{\sfss'}-\abs{\sfss},\quad& \textrm{otherwise}.
   \end{cases}
   \ee
Here $\nu_{\sfss'}$ is defined in \eqref{def:nus}. Thus $\pi'$ is convergent for $\check \Theta_{\mathsf s'}^{\mathsf s}$ whenever $\pi'$ is a Casselman-Wallach representation of $G_{\sfss'}$  that is  $\nu'$-bounded for some
$\nu'> -(\kappa +1)$. 

We will prove the following theorem in this section.
As alluded to earlier, $\cO$ is assumed to be  good for $\DD_{\mathsf s'}^{\mathsf s}$.

\begin{thm}\label{thm:GDS.AC}
Let $\pi'$ be an $\CO'$-bounded Casselman-Wallach representation of $G_{\sfss'}$  that is  $\nu'$-bounded for some
$
  \nu'>
 -(\kappa +1).
$
Then $\Thetab_{\mathsf s'}^{\mathsf s}(\pi')$, $\check \Theta_{\mathsf s'}^{\mathsf s}(\pi')$  and $\check \Theta_{\mathsf s'}^{\mathsf s}(\pi'^{\mathrm{alg}})$ are all $\CO$-bounded, and the following equalities in $\CK_\sfss(\CO)$ hold:  
    \[
  \mathrm{AC}_{\cO}(\Thetab_{\mathsf s'}^{\mathsf s}(\pi'))=    \mathrm{AC}_{\cO}(\check \Theta_{\mathsf s'}^{\mathsf s}(\pi'))= \mathrm{AC}_{\cO}(\check \Theta_{\mathsf s'}^{\mathsf s}(\pi'^{\mathrm{alg}}))=   \check \vartheta_{\cO'}^\cO(\mathrm{AC}_{\cO'}(\pi')).
  \]
\end{thm}

Define three classical signatures
\[
  \dsfss:=(\dot \star, \dot p, \dot q):=\begin{cases}
    (\star', \abs{\sfss'}+\kappa,  \abs{\sfss'}+\kappa), \qquad  &\textrm{if $\star'\neq B$}; \\
        (D, \abs{\sfss'}+\kappa,  \abs{\sfss'}+\kappa),\qquad  &\textrm{if $\star'=B$},
  \end{cases}
\]
\[
  \sfss'':=(\star', p'+\kappa, q'+\kappa)\qquad\textrm{and}\qquad \sfss_0:=(\dot \star, \kappa, \kappa). 
  \]
Then we have decompositions
\[
  V_{\dsfss}=V_{\sfss'^-}\oplus V_{\sfss''}= V_{\sfss'^-}\oplus V_{\sfss'}\oplus  V_{\sfss_0}= (V_{\sfss'}^\triangle \oplus X_{\sfss_0})\oplus (V_{\sfss'}^\nabla \oplus Y_{\sfss_0}).
\]
as in \eqref{decomv0}.

Theorem \ref{thm:GDS.AC} is obviously true when $\abs{\sfss}\leq 1$. Thus, for the remainder of this subsection and the next, we assume that $\abs{\sfss}\geq 2$.  Consequently $\kappa>0$. As in \eqref{dotp}, we have that
\[
 \dot p=\dot q= \begin{cases}
 \abs{\sfss}-1,\quad   &\textrm{if } \star\in \{B,D\};\\
 \abs{\sfss}+1,  & \textrm{if $\star\in \{C, \widetilde C\}$;}\\
 \abs{\sfss}, & \textrm{if $ \star\in\{C^*,D^*\}$.}
\end{cases}
\]
Thus we are in the situation of
Section \ref{doublep} (with $\abs{\sfss}$ equal to the integer $k$ of Section \ref{doublep}).

Let
\[
  P_{\sfss'',\sfss_0}=R_{\sfss'',\sfss_0}\ltimes N_{\sfss'',\sfss_0}=(G_{\sfss'}\cdot R_{\sfss_0})\ltimes N_{\sfss'',\sfss_0}
\]
be the parabolic subgroup of $G_{\sfss''}$ as in \eqref{parabolic} and \eqref{parabolic2}. Write $\mathfrak r_{\sfss_0}$ and $\mathfrak n_{\sfss'',\sfss_0}$ for the complexified Lie algebras of $R_{\sfss_0}$ and $N_{\sfss'',\sfss_0}$ respectively so that  the complexified Lie algebra of $P_{\sfss'', \sfss_0}$ equals $(\g_{\sfss'}\times \mathfrak r_{\sfss_0})\ltimes \mathfrak n_{\sfss'',\sfss_0}$.
As in \eqref{induo0}, write 
\[
  \CO'':= \Ind_{\sfss'}^{\sfss''}  (\CO')\in \Nil(\g_{\sfss''}) 
\]
for the induced orbit.

\begin{lem}
The nilpotent orbit $\CO''$ is good for $\nabla_\mathsf s^{\sfss''}$ and $\nabla_\mathsf s^{\sfss''}(\CO'')=\CO$.
\end{lem}
\begin{proof}
This follows from the description of induced orbits in \cite[Section 7.3]{CM}, which is reproduced in 
\Cref{subsec:induced}. 
\end{proof}

Recall the partial order  $\preceq $ on $ \CK_\sfss(\CO)$ defined in  \eqref{pord}.

\begin{lem}\label{lem74}
Suppose that $\CE\in \CK_{\sfss}(\CO)$ and $0\preceq \CE\preceq \check \vartheta_{\sfss', \CO'}^{\sfss, \CO}(\CE')$ for some $\CE'\in \CK^+_{\sfss'}(\CO')$. If $\CE\neq 0$, then  $\check \vartheta_{\sfss, \CO}^{\sfss'', \CO''}(\CE)\neq 0$.
\end{lem}
\begin{proof}
  Without loss of generality we assume that $\CE'$ is represented by an irreducible equivariant algebraic vector bundle over a nilpotent orbit $\sO'\in \Nil_{\sfss'}(\cO')$, and $\CE$ is represented by an irreducible equivariant algebraic vector bundle over a nilpotent orbit $\sO\in \Nil_{\sfss}(\cO)$.
  The definition of geometric theta lift implies that $\sO'=\DDss(\sO)$. Using \Cref{prop:DDss},
  one checks that there is a nilpotent orbit $\sO''\in \Nil_{\sfss''}(\cO'')$ such that
  $\sO = \DD_{\sfss}^{\sfss''}(\sO'')$.
  When $\cO = \DDn(\cO'')$, the non-vanishing of $\check \vartheta_{\sfss, \CO}^{\sfss'', \CO''}(\CE)$
  follows from the injectivity of $\dlift_{\sO}^{\sO''}$.
  When $\abs{\cO} > \abs{\DDn(\cO'')}$, we have that $\star \in \set{B,D}$, $\sO'=\DDn(\sO)$, 
  and  $\CE = \dlift_{\sO'}^{\sO}(\CE')$.
  Then it is straightforward to check that $\check \vartheta_{\sfss,\sO}^{\sfss'',\sO''}(\CE) \neq 0$ and hence $\check \vartheta_{\sfss, \CO}^{\sfss'', \CO''}(\CE)\neq 0$.
\end{proof}

For each $\sfss_1\in S_{\star, \abs{\sfss}}$ (see \eqref{sstark}),  $\g_{\sfss_1}$ is identified with $\g_{\sfss}$ and as such we also consider $\CO$ as an element of $\Nil(\g_{\sfss_1})$.

\begin{prop}\label{prpaseq}
 Let $\pi'$ be an $\CO'$-bounded Casselman-Wallach representation of $G_{\sfss'}$ that extends to a representation of  the Levi subgroup $R_{\sfss_0}\cdot G_{\sfss'}$. View it  as a representation of $P_{\sfss'',\sfss_0}$ through the trivial action of the unipotent radical.  
Assume that $\pi'$ is genuine when $\star'=\widetilde C$.

\noindent
(a) If $ \star\in \{B,D\}$,  then
\[
   \sum_{\sfss_1\in S_{\star, \abs{\sfss}}}  \left(\check \vartheta^{\sfss'', \CO''}_{\sfss_1, \CO}(\check \vartheta^{\sfss_1, \CO}_{\sfss',\CO'}(\mathrm{AC}_{\CO'}(\pi')))\right)^{\mathrm{weak}}=2\cdot \left( \mathrm{AC}_{\CO''}(\Ind_{P_{\sfss'',\sfss_0}}^{G_{\sfss''}} (\pi'))\right)^{\mathrm{weak}}.
        \]

  \smallskip

    \noindent
(b) If $\star \in \{C, \widetilde C\}$, then
\begin{eqnarray*}
   && \left(\check \vartheta^{\sfss'', \CO''}_{\sfss, \CO}(\check \vartheta^{\sfss, \CO}_{\sfss',\CO'}(\mathrm{AC}_{\CO'}(\pi')))\right)^{\mathrm{weak}} +  \left(\check \vartheta^{\sfss'', \CO''}_{\sfss, \CO}(\check \vartheta^{\sfss, \CO}_{\sfss',\CO'}(\mathrm{AC}_{\CO'}(\pi'\otimes \det )))\right)^{\mathrm{weak}}\\
   & =&\left( \mathrm{AC}_{\CO''}(\Ind_{P_{\sfss'',\sfss_0}}^{G_{\sfss''}} (\pi'))\right)^{\mathrm{weak}}.
\end{eqnarray*}
  \smallskip

\noindent
(c) If $\star=D^*$, then
\[
 \left(\check \vartheta^{\sfss'', \CO''}_{\sfss, \CO}(\check \vartheta^{\sfss, \CO}_{\sfss',\CO'}(\mathrm{AC}_{\CO'}(\pi')))\right)^{\mathrm{weak}} = \left( \mathrm{AC}_{\CO''}(\Ind_{P_{\sfss'',\sfss_0}}^{G_{\sfss''}} (\pi'))\right)^{\mathrm{weak}}.
 \]

\smallskip

\noindent
(d) If $\star=C^*$, then
\[
  \sum _{\sfss_1\in S_{\star, \abs{\sfss}} } \left(\check \vartheta^{\sfss'', \CO''}_{\sfss_1, \CO}(\check \vartheta^{\sfss_1, \CO}_{\sfss',\CO'}(\mathrm{AC}_{\CO'}(\pi')))\right)^{\mathrm{weak}}=\left( \mathrm{AC}_{\CO''}(\Ind_{P_{\sfss'',\sfss_0}}^{G_{\sfss''}} (\pi'))\right)^{\mathrm{weak}}.
\]
\end{prop}
\begin{proof}
 By direct computation, the proposition follows from \Cref{lem:clift} and \Cref{lem:ind00}.
\end{proof}

\begin{lem}\label{lem:GDS.AC2}
  Let $\pi'$ be an $\CO'$-bounded Casselman-Wallach representation of $G_{\sfss'}$  that is  $\nu'$-bounded for some
$
  \nu'>
 -(\kappa+1).
$
Then
 \begin{eqnarray*}
  && \left( \sum_{\sfss_1\in S_{\star, \abs{\sfss}} } \mathrm{AC}_{\cO''}( \Thetab^{\mathsf s''}_{\mathsf s_1}(\Thetab_{\mathsf s'}^{\mathsf s_1}(\pi')))\right)^{\mathrm{weak}}
\\
&=& \left( \sum_{\sfss_1\in S_{\star, \abs{\sfss}}}   \ckvartheta^{\sfss'', \CO''}_{\sfss_1, \CO}(\ckvartheta^{\sfss_1, \CO}_{\sfss',\CO'}(\mathrm{AC}_{\CO'}(\pi')))\right)^{\mathrm{weak}}.
\end{eqnarray*}

\end{lem}

\begin{proof}
There is no loss of generality to assume that $\pi'$ is genuine when $\star'=\widetilde C$. When $\star\neq C^*$, the lemma is a direct consequence of Theorem \ref{doubtt} and Proposition \ref{prpaseq}.

Now assume that $\star=C^*$.  Let $\CJ$ be as in Part (d) of Theorem \ref{doubtt}. In view of Lemma \ref{comobound}, and using formulas in  \cite[Theorems 5.2 and 5.6]{DKPC}, one checks that $\CJ$ is $\CO_2$-bounded for a certain orbit $\CO_2\in \Nil(\g_{\sfss''})$ that is contained in the boundary $\overline{\CO''}\setminus \CO''$ of  $\CO''$. Thus
$\CJ$ is $\cO''$-bounded, and $ \mathrm{AC}_{\cO''}(\CJ)=0$. Hence  the lemma when $\star=C^*$  is also a consequence of Theorem \ref{doubtt} and Proposition \ref{prpaseq}.
\end{proof}

\subsection{Proof of \Cref{thm:GDS.AC}}
Let $\pi'$ be as in \Cref{thm:GDS.AC}.
For every $\sfss_1\in S_{\star, \abs{\sfss}}$ we have
surjective $(\g_{\sfss_1}, K_{\sfss_1})$-module homomorphisms
\[
\check \Theta_{\mathsf s'}^{\mathsf s_1}(\pi'^{\mathrm{alg}})\rightarrow \left (\check \Theta_{\mathsf s'}^{\mathsf s_1}(\pi')\right )^{\mathrm{alg}}\rightarrow
 \left (\Thetab_{\mathsf s'}^{\mathsf s_1}(\pi')\right )^{\mathrm{alg}}.
 \]
 Thus \Cref{thm:upper.AC} implies that  $\check \Theta_{\mathsf s'}^{\mathsf s_1}(\pi'^{\mathrm{alg}})$, $\check \Theta_{\mathsf s'}^{\mathsf s_1}(\pi')$ and $\Thetab_{\mathsf s'}^{\mathsf s_1}(\pi')$ are all $\CO$-bounded,  and
    \be\label{ineq000}
   \mathrm{AC}_{\cO}( \Thetab_{\mathsf s'}^{\mathsf s_1}(\pi'))\preceq  \mathrm{AC}_{\cO}(\check \Theta_{\mathsf s'}^{\mathsf s_1}(\pi'))\preceq  \mathrm{AC}_{\cO}(\check \Theta_{\mathsf s'}^{\mathsf s_1}(\pi'^{\mathrm{alg}}))\preceq \check \vartheta_{\sfss', \cO'}^{\sfss_1, \cO}(\mathrm{AC}_{\cO'}(\pi')).
  \ee
  Consequently,
   \be\label{acco}
   \mathrm{AC}_{\cO}( \Thetab_{\mathsf s'}^{\mathsf s_1}(\pi'))\preceq   \check \vartheta_{\sfss', \cO'}^{\sfss_1, \cO}(\mathrm{AC}_{\cO'}(\pi')).
  \ee

By Lemma \ref{doublelift5},  $ \Thetab_{\mathsf s'}^{\mathsf s_1}(\pi')$ is over-convergent for $ \check \Theta_{\mathsf s_1}^{\mathsf s''}$ so that $\Thetab_{\mathsf s_1}^{\mathsf s''}( \Thetab_{\mathsf s'}^{\mathsf s_1}(\pi'))$ is defined.  Similar to \eqref{acco}, we have that
   \be\label{acco2}
   \mathrm{AC}_{\cO''}( \Thetab_{\mathsf s_1}^{\mathsf s''}(\Thetab_{\mathsf s'}^{\mathsf s_1}(\pi')))\preceq   \check \vartheta_{\sfss_1, \cO}^{\sfss'', \cO''}(\mathrm{AC}_{\cO}( \Thetab_{\mathsf s'}^{\mathsf s_1}(\pi'))).
  \ee
Combining \eqref{acco} and \eqref{acco2}, we obtain that
\be\label{acineq}
   \mathrm{AC}_{\cO''}( \Thetab_{\mathsf s_1}^{\mathsf s''}(\Thetab_{\mathsf s'}^{\mathsf s_1}(\pi')))\preceq   \check \vartheta_{\sfss_1, \cO}^{\sfss'', \cO''}(\check \vartheta_{\sfss', \cO'}^{\sfss_1, \cO}(\mathrm{AC}_{\cO'}(\pi'))).
\ee
Then Lemma \ref{lem:GDS.AC2} implies that the equality in \eqref{acineq} holds. In particular,
\be\label{acineq2}
   \mathrm{AC}_{\cO''}( \Thetab_{\mathsf s}^{\mathsf s''}(\Thetab_{\mathsf s'}^{\mathsf s}(\pi')))=  \check \vartheta_{\sfss, \cO}^{\sfss'', \cO''}(\check \vartheta_{\sfss', \cO'}^{\sfss, \cO}(\mathrm{AC}_{\cO'}(\pi'))).
\ee

In view of \eqref{acco}, write
\[
  \check \vartheta_{\sfss', \cO'}^{\sfss, \cO}(\mathrm{AC}_{\cO'}(\pi'))=\mathrm{AC}_{\cO}( \Thetab_{\mathsf s'}^{\mathsf s}(\pi'))+\CE,\qquad\textrm{where $\ \CE\in \CK^+_\sfss(\CO)$}.
\]
Then
\begin{eqnarray}
\nonumber \check \vartheta_{\sfss, \cO}^{\sfss'', \cO''}(\check \vartheta_{\sfss', \cO'}^{\sfss, \cO}(\mathrm{AC}_{\cO'}(\pi')))
\nonumber &=&\mathrm{AC}_{\cO''}( \Thetab_{\mathsf s}^{\mathsf s''}(\Thetab_{\mathsf s'}^{\mathsf s}(\pi')))\\
\nonumber &\preceq &  \check \vartheta_{\sfss, \cO}^{\sfss'', \cO''}(\mathrm{AC}_{\cO}( \Thetab_{\mathsf s'}^{\mathsf s}(\pi')))\\
\label{prece}  &\preceq &  \check \vartheta_{\sfss, \cO}^{\sfss'', \cO''}(\mathrm{AC}_{\cO}( \Thetab_{\mathsf s'}^{\mathsf s}(\pi'))+\CE)\\
\nonumber &=&\check \vartheta_{\sfss, \cO}^{\sfss'', \cO''}(\check \vartheta_{\sfss', \cO'}^{\sfss, \cO}(\mathrm{AC}_{\cO'}(\pi'))).
\end{eqnarray}
In particular, the equality holds in \eqref{prece}. Therefore  $\check \vartheta_{\sfss, \cO}^{\sfss'', \cO''}(\CE)=0$, and
Lemma \ref{lem74} implies that $\CE=0$. This completes the proof of Theorem \ref{thm:GDS.AC}.

 \subsection{Proof of Theorem \ref{thmpitau} }\label{sec:comANDgeo}
  As in \Cref{sec:bipGeometry} and \Cref{sec:main}, 
  $\check \CO$ is a Young diagram that has good parity. Suppose that
$\uptau=(\tau, \wp)\in \PBPe_\star(\check \CO,\mathsf s)$. Then $\mathsf s =\mathsf s_{\tau}:=(\star, p_\tau, q_\tau)$. We have the dual descent
\[
  \check \CO':=\check \nabla(\check \CO),
\]
and the descent
 \[
  \uptau' := (\tau',\wp'):=\nabla(\uptau):= (\DD(\tau), \ckDD(\wp))\in \mathrm{\PBPe}_{\star'}(\check \CO').
 \]
 Set
 \[
  \mathrm d'_\tau:=
    \dim \tau' \qquad (\text{see \eqref{def:dimtau}}).
    \]

Suppose that $\CO\in \Nil(\g_\sfss)$ equals the Barbasch-Vogan dual of $\check \CO$.   Recall that $\pi_\uptau$ is the Casselman-Wallach representation of $G_{\mathsf s}$, defined in Section \ref{subsec:comTOrep}.
Also recall the element $\mathrm{AC}(\uptau)\in \CK_{\mathsf s}(\CO)$ from Section \ref{subsecass}.

Recall the character $\chi_{\check \CO}: \CU(\g_\sfss)^{G_{\sfss, \C}}\rightarrow \C$  defined in Section \ref{sec:intro}.
\begin{lem}\label{inf1}
The algebra $\CU(\g_\sfss)^{G_{\sfss, \C}}$ acts on $\pi_\uptau$ through the character $\chi_{\check \CO}$.
\end{lem}
\begin{proof}
By induction, this is implied by \cite[Theorem
1.19]{PrzInf}.
\end{proof}

Recall the ideal $I_{\check \CO}\subset \mathcal U(\g_\sfss)$ from Section \ref{sec:intro}.
\begin{lem}\label{inf2}
Let $\pi$ be an $\CO$-bounded irreducible Casselman-Wallach representation of $G_{\sfss}$ such that $\mathcal U(\g_\sfss)^{G_{\sfss, \C}}$ acts on it through the character $\chi_{\check \CO}$. Then $I_{\check \CO}$ annihilates $\pi$ and $\mathrm{AC}_{\CO}(\pi)\neq 0$.   
\end{lem}

\begin{proof}
By \cite[Korollar 3.6]{BK}, the annihilator ideal of $\pi$ (or rather $\pi^{\mathrm{alg}}$, the  underlying $(\g_\sfss, K_\sfss)$-module of $\pi$) must be $I_{\check \CO}$, and so the complex associated variety of $\pi$ equals $\overline \CO$. By \cite[Theorem 8.4]{Vo89}, this implies $\mathrm{AC}_{\CO}(\pi)\neq 0$. 
\end{proof}

In the rest of this subsection we will prove the following theorem, which, together with Lemmas \ref{inf1} and \ref{inf2},  implies  Theorem \ref{thmpitau}.
\begin{thm}\label{thmpitau222}
The representation $\pi_\uptau$ is irreducible, unitarizable, $\CO$-bounded and $( \mathrm d'_\tau-\nu_{\sfss})$-bounded.     Moreover,
\[
\mathrm{AC}_\CO(\pi_\uptau)=\mathrm{AC}(\uptau)\in \CK_{\mathsf s}(\CO).
\]
\end{thm}

We prove Theorem \ref{thmpitau222} by induction on $\mathbf c_1(\check \CO)$ (which equals the number of nonempty rows of $\check \CO$).  Theorem \ref{thmpitau222} is obvious when $\mathbf c_1(\check \CO)=0$. Now we assume $\mathbf c_1(\check \CO)>0$ and Theorem \ref{thmpitau222} holds for  smaller $\mathbf c_1(\check \CO)$.

 \begin{lem}\label{lem78}
 The orbit $\CO$ is good for $\nabla^\sfss_{\sfss_{\tau'}}$, and $\nabla^\sfss_{\sfss_{\tau'}}(\CO)$ equals the Barbasch-Vogan dual  of $\check \CO'$.
 \end{lem}
\begin{proof}
The first assertion is routine to check, and the second one is Lemma \ref{dualdesc}.
\end{proof}

In view of Lemma \ref{lem78}, we may take $\sfss'=\sfss_{\tau'}$ and results of \Cref{subsec:gdescent} will apply. Then
\[
\CO':= \nabla^\sfss_{\sfss'}(\CO)=\textrm{the Barbasch-Vogan dual  of $\check \CO'$}.
\]

By the induction hypothesis, the representation $\pi_{\uptau'}$ is irreducible, unitarizable, $\CO'$-bounded and $( \mathrm d'_{\tau'}-\nu_{\sfss'})$-bounded,     and
\[
\mathrm{AC}_{\CO'}(\pi_{\uptau'})=\mathrm{AC}(\uptau')\in \CK_{\mathsf s'}(\CO').
\]

It is routine to check case by case that
\[
\mathrm d'_{\tau'}-\nu_{\sfss'}>\max \left \{\nu^+_{\sfss'}-\abs{\sfss}, -(\kappa +1)\right \}=
  \begin{cases}
\nu^+_{\sfss'}-\abs{\sfss} +1,\quad& \textrm{if $\star =C^*$};\\
 \nu^+_{\sfss'}-\abs{\sfss},\quad& \textrm{otherwise}.
   \end{cases}
\]
Thus $\pi_{\uptau'}\otimes \chi$ is over-convergent for $\check \Theta_{\sfss'}^{\sfss}$ for every unitary character $\chi$ of $G_{\sfss'}$.

Recall from \eqref{eq:def-pi} that
\[
    \pi_{\uptau}:=\left\{
     \begin{array}{ll}
         \check \Theta_{\sfss'}^{\sfss}(\pi_{\uptau'})\otimes (1_{p_\tau, q_\tau}^{+,-})^{\varepsilon_{\tau}}, &\hbox{if  $\star=B$ or $D$}; \\
         \check \Theta_{\sfss'}^{\sfss}(\pi_{\uptau'}\otimes \det^{\varepsilon_{\wp}}), &\hbox{if $\star=C$ or $\widetilde C$}; \\
              \check \Theta_{\sfss'}^{\sfss}(\pi_{\uptau'}), &\hbox{if $\star=C^*$ or $D^*$}. \\
            \end{array}
   \right.
\]
Similarly define
\[
   \bar \pi_{\uptau}:=\left\{
     \begin{array}{ll}
         \Thetab_{\sfss'}^{\sfss}(\pi_{\uptau'})\otimes (1_{p_\tau, q_\tau}^{+,-})^{\varepsilon_{\tau}}, &\hbox{if  $\star=B$ or $D$}; \\
          \Thetab_{\sfss'}^{\sfss}(\pi_{\uptau'}\otimes \det^{\varepsilon_{\wp}}), &\hbox{if $\star=C$ or $\widetilde C$}; \\
               \Thetab_{\sfss'}^{\sfss}(\pi_{\uptau'}), &\hbox{if $\star=C^*$ or $D^*$}. \\
            \end{array}
   \right.
\]
Theorem \ref{positivity000} implies that
$\bar \pi_{\uptau}$ is unitarizable. Theorem \ref{thm:GDS.AC} and Proposition \ref{thmac1} imply that $\bar \pi_{\uptau}$ and $\pi_{\uptau}$ are $\CO$-bounded, and
\[
    \mathrm{AC}_{\cO}( \bar \pi_{\uptau})=\mathrm{AC}_{\cO}(\pi_{\uptau})=\mathrm{AC}(\uptau)\neq 0.
\]
Since $\bar \pi_{\uptau}$ is a completely reducible quotient of $\pi_{\uptau}$, and $\pi_{\uptau}$ has at most one  irreducible quotient by the fundamental result of Howe (\cite[Theorem 1A]{Howe89}), we conclude that $\bar \pi_{\uptau}$ is irreducible.

Write
$\CJ_\uptau$ for the kernel of the natural surjective homomorphism $\pi_{\uptau}\rightarrow \bar \pi_{\uptau}$. By Lemma \ref{inf1},  $\CU(\g_\sfss)^{G_{\sfss, \C}}$ acts on $\CJ_{\uptau}$ through the character $\chi_{\check \CO}$. Moreover, $\CJ_\uptau$  is $\CO$-bounded and  $\mathrm{AC}_{\CO}(\CJ_{\uptau})=0$. This implies that $\CJ_{\uptau}=0$ by Lemma \ref{inf2}.
Thus $\pi_{\uptau}=\bar \pi_{\uptau}$. Finally, Proposition \ref{boundm} implies that  $\bar \pi_{\uptau}$ is $(\mathrm d'_\tau-\nu_{\sfss})$-bounded. This finishes the proof of Theorem
\ref{thmpitau222}.

\section{Associated cycles: combinatorial aspect}
\label{sec:ACC}

\def\dsign{{}^d\mathrm{Sign}}

\def\acm{\cL}
\def\acme{\tilde{\cL}}
\def\dlifttso{{\check \vartheta}_{\sfss_{\tau'},\cO'}^{\sfss_{\tau},\;\cO}}
\def\DDtss{\DD_{\sfss_{\tau'}}^{\sfss_{\tau}}}
\def\taut{\tau_{\bftt}}

The goal of this section is to prove Propositions \ref{thmac1}-\ref{thmac5}, on the associated cycles. We will prove all results combinatorially, by introducing their combinatorial counterparts.
We treat the cases of $\star\in \{B,D\}$ in Sections \ref{sec:ppb}-\ref{secfurth}. The cases of  $\star\in \{C,\wtC\}$ are reduced to the cases of  $\star\in \{B,D\}$ in \Cref{seccwtc}. The quaternionic case is much simpler and will be sketched in \Cref{quat}. 

As before, $\star'$ is the Howe dual of $\star$, $\check \CO$ has good parity and  $\CO$ is the Barbasch-Vogan dual of $\check \CO$ (see \Cref{bvdual0}). The dual descent of $\check \CO$ is denoted by $\check \CO'$, and the Barbasch-Vogan dual of $\check \CO'$ is denoted by $\CO'$.  Recall from \Cref{sec:MYD} the notion of marked Young diagram and the related notation. Put
\begin{eqnarray*}
\MYD_\star(\CO)&:=&\{\CE\in \MYD_\star\,:\, \textrm{the composition of $\bN^+\xrightarrow{\CE} \Z\times \Z\xrightarrow{(a,b)\mapsto \abs{a}+\abs{b}}\bN$}\\
&&\qquad \qquad \qquad \textrm{represents the Young diagram $\CO$}\}. 
\end{eqnarray*}

\subsection{Combinatorial cycles}
We introduce the following notation. For each $(a,b)\in \Z\times \Z$, there is at most one element $\CE\in \MYD_\star$ such that 
\[
  \CE(i):=\begin{cases}
    (a,b),\quad\textrm{if $i=1$;}
    \\
    (0,0),\quad\textrm{if $i=2,3,4, \dots$.}
  \end{cases}
\] When such an element exists, we denote it by $(a,b)_\star\in \MYD_\star\subset\Z[\MYD_\star]$.

Recall the map $\sV: \Z\times \Z  \rightarrow \bN\times \bN$ from \eqref{eq:F1}. This induces a map
\be\label{sfcirc0}
 \MYD_\star\rightarrow \SYD_\star,\qquad \CE\mapsto \sV \circ \CE.
 \ee
We say that an element $\cA \in \bZ[\MYD_\star]$ is multiplicity free if it is a sum of
pairwise distinct elements in $\MYD_\star$ whose images under \eqref{sfcirc0} are also pairwise distinct.

Let  $\uptau=(\tau,\wp)\in \PBPes(\ckcO)$. Recall from \Cref{subsecass} 
the associated cycle $\AC(\uptau)$ of $\uptau$, which is an element in
$\Z[\AOD_{\sfss_\tau }(\cO)]$
by \Cref{lem:actau}. Let $\ac_{\uptau}$ be the element in $\Z[\MYD_{\star}(\CO)]$ corresponding to $\AC(\uptau)$ via the map in \eqref{identmy}, to be called the combinatorial cycle of $\uptau$. 

Suppose that $\uptau'=(\tau',\wp')=\nabla(\uptau)$ is the descent of $\uptau$. By \Cref{lem:clift}, $\ac_{\uptau}$  is
described combinatorially as follows.
\begin{itemize}
\item
When $\abs{\ckcO}=0$,
\[
  \ac_{\uptau} :=
  \begin{cases}
   (1,0)_{\star}, & \text{if } \gamma_\tau=B^+;\\
   (0,-1)_{\star}, & \text{if } \gamma_\tau=B^-;\\
   (0,0)_{\star}, & \text{otherwise}.
  \end{cases}
\]
\item
When $\abs{\ckcO} >0$,
\be\label{indac}
\ac_{\uptau} :=
\begin{cases}
  \dlifttso(\ac_{\uptaup})\lotimes (0,\varepsilon_{\tau}),
  & \text{if } \star \in \set{B,D}; \\
  \dlifttso(\ac_{\uptaup}\lotimes (\varepsilon_{\wp},\varepsilon_{\wp})),
  & \text{if } \star\in \set{C,\wtC};\\
  \dlifttso(\ac_{\uptaup}),
  & \text{if } \star \in \set{C^*,D^{*}}. \\
\end{cases}
\ee
Here $\dlifttso$ is defined by \eqref{eq:LCD} and \eqref{eq:LDC}.
\end{itemize}

Note that (see \eqref{MYDs} for the notation)
\be\label{signactau}
\ac_\uptau\in \Z[\MYD_{\sfss_{\tau}}(\CO)].
\ee

Using the induction formula \eqref{indac}, the following two lemmas are routine to verify. 
They will serve as the initial cases for the properties of combinatorial cycles we will establish later in this section, for $\star \in \set{B,D,C^*}$.

\begin{lem}\label{eqtail2}
Suppose that $\star \in \set{B,D}$ and $\bfrr_1(\CO)\leq 1$. Then $\ac_{\uptau}\in \MYD_\star(\CO)$ and 
\[
  \ac_\uptau=
    (p_\tau, (-1)^{\varepsilon_\tau} q_\tau)_\star.
\]
Moreover if $\bfrr_1(\CO) =1$, then the map 
\[
\begin{array}{rcl} 
\PBPes(\check \CO)\times \Z/2\Z &\rightarrow &
\{(a,b)\in\Z\times \Z\,:\, \abs{a}+\abs{b}=\abs{\CO}\}, \\
(\uptau,\epsilon)&\mapsto & (-1)^{\epsilon}(\ac_\uptau(1))
\end{array}
\]
is bijective. 
\end{lem}

\begin{lem}\label{eqtail22}
Suppose that $\star =C^*$ and $\bfrr_1(\CO)\leq 1$. Then $\ac_{\uptau}\in \MYD_\star(\CO)$ and 
\[
  \ac_\uptau=
    (p_\tau,  q_\tau)_\star.
\]
Moreover, the map 
\[
\begin{array}{rcl} 
\PBPes(\check \CO) &\rightarrow &
\{(a,b)\in 2\bN\times 2\bN\,:\, \abs{a}+\abs{b}=\abs{\CO}\}, \\
\uptau&\mapsto & \ac_\uptau(1)
\end{array}
\]
is well-defined and bijective. 

\end{lem}

\subsection{Tail of a painted bipartition and the descent map: a review}\label{sec:ppb}
\label{sec:tail}

We first recall from \cite[Section 10.5]{BMSZa} the definition of the tail of a painted bipartition. For this we suppose that $\star\in\{B,D\}$ and $(\star, \abs{\check \CO})\neq(D, 0)$.  
Put
\[
k := 
  \frac{\bfrr_{1}(\ckcO)-\bfrr_{2}(\ckcO)}{2} + 1,
\]
which is a positive integer. 
Let $\ckcO_{\bftt}$ denote the Young diagram  that consists of two rows with lengths $2k-1$ and $1$. Then it has type $D$ and has good parity (with respect to $D$). 

Every element in $\PBP_{D}(\ckcO_\bftt)$ has the form
 \begin{equation}
 \label{tail000}
  \ytb{{x_1} , {x_2} , {\enon\vdots},{\enon{\vdots}},{x_k}  } \times \emptyset \times
  D.
\end{equation}

\medskip

Let $ \tau=(\imath,\cP)\times(\jmath,\cQ)\times \gamma \in
\mathrm{PBP}_\star(\check \CO) $. The tail  $\tau_\bftt$ of $\tau$ will be a painted bipartition in
$\PBP_{D}(\ckcO_\bftt)$, which is defined below case by case.

{\bf The case when $\star = B$.}
In this case, we define the tail $\tau_\bftt$ to be the painted bipartition  \eqref{tail000} such that the multiset $\{x_1, x_2, \cdots, x_k\}$ is the
union of the multiset
\[
  \set{\cQ(j,1)\,:\, \bfcc_{1}(\imath)+1 \leq j \leq  \bfcc_{1}(\jmath) }
\]
with the set
\[
  \begin{cases}
 \set{c}, &
 \qquad
  \text{if $\gamma= B^+$, and either $\bfcc_{1}(\imath)=0$ or $\cQ(\bfcc_{1}(\imath),1)\in \set{\bullet,s}$};  \\
 \set{s},&
  \qquad \text{if $\gamma= B^-$, and either $\bfcc_{1}(\imath)=0$ or $\cQ(\bfcc_{1}(\imath),1)\in \set{\bullet,s}$}; \\
\set{\cQ(\bfcc_{1}(\imath),1)},&
\qquad
\text{otherwise.}
\end{cases}
\]
{\bf The case when $\star = D$ and $\abs{\check \CO}\neq 0$.}
In this case, $1\leq \bfcc_1(\jmath)+1\leq \bfcc_1(\imath)$.  We define the tail $\tau_\bftt$ to be the  painted
bipartition  \eqref{tail000} such that the multiset $\{x_1, x_2, \cdots, x_k\}$
is the union of the multiset
\[
\set{\cP(j,1)\,:\, \bfcc_{1}(\jmath)+2 \leq j \leq \bfcc_{1}(\imath)}
\]
with the set
\[
\begin{cases}
    \set{c},                          &
    \ \text{if $\bfrr_2(\ckcO)=\bfrr_3(\ckcO)$, $\cP(\bfcc_1(\imath),1)\in \set{r,d}$}, \  \text{ and }                                   \\
     & \quad   (\cP(\bfcc_{2}(\imath),1) ,\cP(\bfcc_{2}(\imath),2)) = (r,c);  \\
    \set{\cP(\bfcc_{1}(\jmath)+1,1)}, &
    \    \text{otherwise.}
  \end{cases}
\]
Note that the condition $\bfrr_2(\ckcO)=\bfrr_3(\ckcO)$ implies that $ \bfcc_2(\imath)=\bfcc_{1}(\jmath)+1\geq 1 $ and hence $\cP(\bfcc_{2}(\imath),1)$ and $\cP(\bfcc_{2}(\imath),2)$ are defined.

In both cases it is clear that $(\tau_\bftt)_\bftt=\tau_\bftt$.

Let $x_{\tau}$ be the symbol in the last box of the tail $\tau_\bftt$, namely 
\be\label{xtau}
x_\tau := \cP_{\tau_\bftt}(k,1).
\ee
It is clear that $x_\tau=x_{\tau_\bftt}$.

Recall the parity $\varepsilon_\tau\in \Z/2\Z$ defined in \eqref{epsilontau}. The following lemma is easily verified. 

\begin{lem}\label{lem:eptail}
Suppose that $\star \in \set{B,D}$, $(\star, \abs{\check \CO})\neq (D, 0)$, and  $\tau\in
\mathrm{PBP}_\star(\check \CO)$. Then \begin{itemize}
    \item $\varepsilon_\tau=0$ if and only if $x_\tau=d$;
    \item $p_{\tau_\bftt}=0$ if and only if $x_\tau=s$;
    \item $q_{\tau_\bftt}=0$ if and only if every symbol in $\tau_\bftt$ is $r$.
\end{itemize}
\end{lem}
 Note that \Cref{lem:eptail} implies that  \[
 \varepsilon_\tau=\varepsilon_{\tau_\bftt}.
 \]

Recall that $\check \CO':=\check \nabla(\check \CO)$ is the dual descent of $\check \CO$ (see \eqref{duald}).  The key properties of the descent map 
are summarized in the following proposition.

\begin{prop}[{\cite{BMSZa}*{Proposition~10.9}}]
\label{lem:delta}
Suppose that $\star \in \set{B,D}$ and $(\star, \abs{\check \CO})\neq (D, 0)$. Write $\ckcOpp := \ckDD(\ckcO')$ and consider the map
\begin{equation}\label{eq:delta}
   \PBP_\star(\ckcO)\longrightarrow
    \PBP_\star(\ckcOpp)\times \PBP_{D}(\ckcO_\bftt),
    \qquad \tau \mapsto (\nabla(\nabla(\tau)),\tau_\bftt).
\end{equation}

\noindent (a) Suppose that
$\bfrr_2(\ckcO)>\bfrr_3(\ckcO)$. Then the map \eqref{eq:delta} is bijective, and for every $\tau\in  \PBP_\star(\ckcO) $,
\begin{equation}\label{eq:sign.D}
\ssign(\tau)
=(\bfcc_2(\cO),\bfcc_2(\cO))+\ssign(\nabla(\nabla(\tau)))+\ssign(\tau_\bftt).
\end{equation}

\noindent (b)  Suppose that  $\bfrr_2(\ckcO)=\bfrr_3(\ckcO)>0$.
Then the map \eqref{eq:delta} is injective with image 
\begin{equation}\label{eq:delta.I}
    \Set{ (\tau'',\tau_0)  \in \PBP_\star(\ckcOpp)\times \PBP_D(\ckcO_\bftt)  \,:\,
    \begin{array}{l}
        \text{either
    $x_{\tau''} = d$, or} \\
    \text{$x_{\tau''}\in \set{r,c}$  and
    $\cP_{\tau_0}^{-1}(\set{s,c})\neq \emptyset$}
    \end{array}
}.
\end{equation}
Moreover,  for every $\tau\in  \PBP_\star(\ckcO) $,
\begin{equation}\label{eq:sign.GD}
\ssign(\tau)
=(\bfcc_2(\cO)-1,\bfcc_2(\cO)-1)+\ssign(\nabla(\nabla(\tau)))+\ssign(\tau_\bftt).
\end{equation}
\end{prop}

\subsection{Basic properties of combinatorial cycles}

We continue to assume that $\star\in \set{B,D}$. 
Let $\uptau:=(\tau, \wp)\in \PBPes(\ckcO)$. Its  descent is denoted by $\uptau':=(\tau', \wp')\in \PBP_{\star'}^{\textrm{ext}}(\ckcO')$ ($\star'$ is the Howe dual of $\star$), and the descent of $\uptau'$ is denoted by $\uptau'':=(\tau'', \wp'')\in \PBPes(\ckcO'')$, where $\ckcO''$ is the dual descent of $\ckcO'$. We call $\uptau''$ the double descent of $\uptau$. 

For every  $\CA\in \bZ[\MYD_{\star}]$, we define the one-box truncations of $\CA$ to be 
\[
\CA^{+} := \Lambda_{(1,0)}(\CA) \AND \CA^{-} := \Lambda_{(0,1)}(\CA).
\]

Applying the induction formula \eqref{indac} twice and using the signature formulas \eqref{eq:sign.D} and \eqref{eq:sign.GD}, we get the following lemma. 
\begin{lem}\label{lem:dlift00}
 Suppose that $\star\in\{B,D\}$ and $(\star, \abs{\check \CO})\neq (D, 0)$. Put 
\be\label{gammatau}
   \delta_\tau:=\begin{cases}
   \frac{p_{\tau_\bftt}-q_{\tau_\bftt}}{2}+1, \qquad & \textrm{if $\star=B$};\\
    \frac{p_{\tau_\bftt}-q_{\tau_\bftt}}{2}, \qquad &\textrm{if $\star=D$}.
   \end{cases}
    \ee
\noindent (a) 
 If  $\bfrr_{2}(\ckcO)>\bfrr_{3}(\ckcO)$, then
  \begin{equation}\label{eq:BD}
    \ac_{\uptau}= \maltese^{\delta_\tau}\big(\ac_{\uptau''}\otimes (\varepsilon_{\wp'}, \varepsilon_{\wp'})\circledast (n_0,n_0)\big)\circledast (p_{\tau_{\bftt}},q_{\tau_{\bftt}})
    \lotimes (0,\varepsilon_{\tau}),
  \end{equation}
 where $n_0:=\frac{\bfcc_{1}(\cO')-\bfcc_{2}(\cO')}{2}\ (\in \bN)$ and $\cO'$ is the Barbasch-Vogan dual of $\check \CO'$.
  
\noindent (b) 
 If  $\bfrr_{2}(\ckcO)=\bfrr_{3}(\ckcO)$,
  then
\be\label{decac}
    \ac_{\uptau} =
    \ac_{\uptau,+}+\ac_{\uptau,-},
\ee
    where 
  \begin{equation}\label{eq:BD201}
    \ac_{\uptau,+}:=\big(\maltese^{\delta_\tau} (\pac{\uptaupp}\circledast (0,0))\big) \circledast (p_{\tau_{\bftt}},q_{\tau_{\bftt}}-1)\lotimes(0,\varepsilon_{\tau})
    \end{equation}
    and
      \begin{equation}\label{eq:BD202}
    \ac_{\uptau,-}:= \big(\maltese^{\delta_\tau}(\nac{\uptaupp}\circledast (0,0)) \big) \circledast (p_{\tau_{\bftt}}-1,q_{\tau_{\bftt}})
    \lotimes(0,\varepsilon_{\tau}).
  \end{equation} 
  \end{lem}
  
  Let $\ac_{\uptau,+}$ and $\ac_{\uptau,-}$ be as in \eqref{eq:BD201} and \eqref{eq:BD202} respectively (when the assumptions of 
  Lemma \ref{lem:dlift00} (b) hold). \Cref{lem:dlift00} has the following obvious consequence.
  
  \begin{lem}\label{taile0}
  Suppose that $\star \in \set{B,D}$, $(\star, \abs{\check \CO})\neq (D, 0)$.
  
   \noindent (a) Suppose that $\bfrr_{2}(\ckcO)>\bfrr_{3}(\ckcO)$, and there is a marked Young diagram   $\CE\in \MYD_\star$ that has a nonzero coefficient in $\ac_\uptau$. Then  
   \[
     \CE(1)=(p_{\tau_\bftt}, (-1)^{\varepsilon_\tau} q_{\tau_\bftt}).
   \]
   
   \noindent (b) Suppose that $\bfrr_{2}(\ckcO)=\bfrr_{3}(\ckcO)$, and there is a marked Young diagram   $\CE_+\in \MYD_\star$ that has a nonzero coefficient in $\ac_{\uptau,+}$. Then  $q_{\tau_\bftt}\geq 1$ and
   \[
     \CE_+(1)=( p_{\tau_\bftt}, (-1)^{\varepsilon_\tau}(q_{\tau_\bftt}-1)).
   \]
   
   \noindent (c) Suppose that $\bfrr_{2}(\ckcO)=\bfrr_{3}(\ckcO)$, and there is a marked Young diagram   $\CE_-\in \MYD_\star$ that has a nonzero coefficient in $\ac_{\uptau,-}$. Then  $p_{\tau_\bftt}\geq 1$ and
   \[
     \CE_-(1)=( p_{\tau_\bftt}-1, (-1)^{\varepsilon_\tau} q_{\tau_\bftt}).
   \]
  \end{lem}
  
 As the first application of  \Cref{taile0}, we get the following result. 
 
  \begin{prop}\label{lem:dlift001}
 Suppose that $\star\in\{B,D\}$. Then $\ac_\uptau$ is multiplicity free. Moreover, when  $(\star, \abs{\check \CO})\neq(D,0)$ and $\bfrr_{2}(\ckcO)=\bfrr_{3}(\ckcO)$, both $\ac_{\uptau,+}$ and  $\ac_{\uptau,-}$ are multiplicity free.
 \end{prop}
 \begin{proof}
 Note that the four operations in \Cref{fouro} (sign twist, involution, augmentation, truncation) send multiplicity free elements to multiplicity free elements. Using \Cref{lem:dlift00}, we prove by induction on the number of rows of $\check \CO$ that
 \begin{itemize}
     \item if $\bfrr_{2}(\ckcO)>\bfrr_{3}(\ckcO)$, then $\ac_\uptau$ is multiplicity free; and 
     \item
     if $\bfrr_{2}(\ckcO)=\bfrr_{3}(\ckcO)$, then both $\ac_{\uptau,+}$ and  $\ac_{\uptau,-}$ are multiplicity free. 
 \end{itemize}
 In the second case, \Cref{taile0} implies that $\ac_\uptau=\ac_{\uptau,+}+\ac_{\uptau,-}$ is also multiplicity free. 
 \end{proof}
 
\begin{prop}\label{lem:BD204}
Suppose that $\star \in \set{B,D}$. 

\noindent (a) The inequality $\ac_{\uptau} \neq 0$ holds. 

\noindent (b)  If $(\star, \abs{\check \CO})\neq (D, 0)$ and $x_{\tau}=s$, then $\ac_{\uptau}^{+} =0$ and $\ac_{\uptau}^{-}=0$.

\noindent (c) If $(\star, \abs{\check \CO})\neq (D, 0)$ and $x_{\tau}\in \set{r,c}$, then $\ac_{\uptau}^{+} \neq 0$ and $\ac_{\uptau}^{-}=0$.

\noindent (d) If $(\star, \abs{\check \CO})\neq (D, 0)$ and $x_{\tau}=d$, then $\ac_{\uptau}^{+} \neq 0$ and
  $\ac_{\uptau}^{-}\neq 0$.

\end{prop}
\begin{proof}
The proposition clearly holds when $\abs{\check \CO}=0$. We prove by induction on the number of rows of $\check \CO$. Now we assume that $\abs{\check \CO}>0$ and the proposition holds when the number of rows of $\check \CO$ is strictly smaller. In particular, $\ac_{\uptau''}\neq 0$. 

First assume that $\bfrr_{2}(\ckcO)>\bfrr_{3}(\ckcO)$. Then \eqref{eq:BD} implies that $\ac_\uptau\neq 0$. By \Cref{lem:eptail}, we have that 
\begin{itemize}
    \item if $x_\tau=s$, then $p_{\tau_\bftt}=0$ and $(-1)^{\varepsilon_\tau} q_{\tau_\bftt}<0$;
      \item 
      if $x_\tau\in \{r,c\}$, then $p_{\tau_\bftt}>0$ and $(-1)^{\varepsilon_\tau} q_{\tau_\bftt}\leq 0$;
      \item 
      if $x_\tau=d$, then $p_{\tau_\bftt}>0$ and $(-1)^{\varepsilon_\tau} q_{\tau_\bftt}> 0$.
\end{itemize}
 Then the proposition follows from \Cref{taile0} (a).
      
Now we assume that $\bfrr_{2}(\ckcO)=\bfrr_{3}(\ckcO)$. Then $(\star, \abs{\check \CO''})\neq (D,0)$. By \Cref{lem:dlift001}, we know that $\ac_\uptau $ is nonzero  if and only if $\ac_{\uptau,+}$ or $\ac_{\uptau,-}$ is nonzero. Similarly, $(\ac_\uptau)^+ $ is nonzero  if and only if $(\ac_{\uptau,+})^+$ or $(\ac_{\uptau,-})^+$ is nonzero, and $(\ac_\uptau)^- $ is nonzero  if and only if $(\ac_{\uptau,+})^-$ or $(\ac_{\uptau,-})^-$ is nonzero.

{\bf The case when  $x_{\tau''}=d$.} In this case,  the induction hypothesis implies that \[
\ac_{\uptau''}^+\neq 0\quad\textrm{and} \quad \ac_{\uptau''}^-\neq 0.
\]

If  $x_\tau=s$, then as before, we have that 
\[
p_{\tau_\bftt}=0,\  q_{\tau_\bftt}\geq 2,\quad \textrm{and}\quad  \varepsilon_\tau=1.
\]
This implies that $\ac_{\uptau,-}=0$ and $\ac_{\uptau,+}\neq 0$. \Cref{taile0} (b) implies that 
$(\ac_{\uptau,+})^+=(\ac_{\uptau,+})^-=0$. 

If  $x_\tau=r$, then we have that
\[
p_{\tau_\bftt}\geq 2\  \quad \textrm{and}\quad  \varepsilon_\tau=1.
\]
This implies that $\ac_{\uptau,-}\neq 0$, and \Cref{taile0} (c) further implies that $(\ac_{\uptau,-})^+\neq 0$. 
Therefore  $(\ac_{\uptau})^+\neq 0$. By considering two cases when $q_{\tau_\bftt}$ is zero or nonzero, it is also easy to see that $(\ac_{\uptau})^-= 0$. 

If  $x_\tau=c$, then we have that
\[
p_{\tau_\bftt}, q_{\tau_\bftt}\geq 1\  \quad \textrm{and}\quad  \varepsilon_\tau=1.
\]
This implies that $\ac_{\uptau,+},\ac_{\uptau,-}\neq 0$, and \Cref{taile0} (b) and (c) further imply that $(\ac_{\uptau})^+\neq 0$ and $(\ac_{\uptau})^-= 0$.  

If  $x_\tau=d$, then we have that
\[
p_{\tau_\bftt}, q_{\tau_\bftt}\geq 1\  \quad \textrm{and}\quad  \varepsilon_\tau=0.
\]
This implies that $\ac_{\uptau,+},\ac_{\uptau,-}\neq 0$, and \Cref{taile0} (b) and (c) further imply that $(\ac_{\uptau})^+\neq 0$ and $(\ac_{\uptau})^-\neq 0$.  

In conclusion, we have proved that  the proposition holds when $x_{\tau''}=d$.

{\bf The case when  $x_{\tau''}\neq d$.} In this case,  \Cref{lem:delta} implies 
\[
x_{\tau''}\in \set{r,c}\quad\textrm{and}\quad 
    \cP_{\tau_\bftt}^{-1}(\set{s,c})\neq \emptyset.
\]
    By the induction hypothesis, we have that $(\ac_{\uptau''})^+\neq 0$ and $(\ac_{\uptau''})^-= 0$. 
     In particular, $\ac_{\uptau,-}=0$ and $\ac_\uptau=\ac_{\uptau,+}$. 
    
    As in the case when $x_{\tau''}=d$, the proposition is then  easily verified by considering the four possible values of $x_\tau$. 
\end{proof}

For all $\CA\in \Z[\MYD_\star]$ and $(p_0,q_0)\in \Z\times \Z$, we write $\cA \sqsupseteq (p_{0},q_{0})$ if there is a marked Young diagram $\CE$ with
a non-zero coefficient in $\cA$ such that $\CE\sqsupseteq (p_{0},q_{0})$ (see \eqref{eq:sub}). We write $\cA \nsqsupseteq (p_{0},q_{0})$ if $\cA \sqsupseteq (p_{0},q_{0})$ does not hold. 
\begin{lem}\label{lempn0}
Suppose that $\star\in \{B,D\}$ and   $\bfrr_{1}(\ckcO)> \bfrr_{3}(\ckcO)$. Then there exists no $\uptau_1, \uptau_2\in \PBPes(\ckcO)$ such that 
    \begin{equation}\label{eq:pnac}
    \maltese (\nac{\uptau_1}\circledast (0,0)) =\pac{\uptau_2} \circledast (0,0) \neq 0.
     \end{equation}
\end{lem}
\begin{proof}
Assume by contradiction that such $\uptau_1, \uptau_2\in \PBPes(\ckcO)$ exist. Suppose that $\uptau_1=\uptau:=(\tau,\wp)$. \Cref{lem:BD204} implies that $x_\tau=d$. Thus  
\[
  p_{\tau_\bftt}, q_{\tau_\bftt}\geq 1\qquad\textrm{and}\qquad \varepsilon_\tau=0.
\]
We claim that \be\label{conta11}
 \ac_{\uptau}\sqsupseteq (1,1).
 \ee
 
If $\bfrr_2(\ckcO)>\bfrr_3(\ckcO)$, then \Cref{lem:dlift00} (a) and 
\Cref{lem:BD204} (a) imply the claim. 

Now we assume  that $\bfrr_2(\ckcO)=\bfrr_3(\ckcO)$. Then we have that 
\be\label{pplusq00}
p_{\tau_\bftt}+q_{\tau_\bftt}=\bfrr_1(\ckcO)-\bfrr_2(\ckcO)+2\geq 4. 
\ee 

If $x_{\tau''}=d$, then by \Cref{lem:BD204} (d) we have that $\ac_{\uptau''}^+\neq 0$ and $\ac_{\uptau''}^-\neq 0$. Thus $\ac_{\uptau,+}\neq 0$ and $\ac_{\uptau,-}\neq 0$. The claim then follows from \Cref{taile0} (b) and (c). 

If $x_{\tau''}\neq d$, then \Cref{lem:delta} implies that
\[
x_{\tau''}\in \set{r,c}\quad\textrm{and}\quad 
    \cP_{\tau_\bftt}^{-1}(\set{s,c})\neq \emptyset.
\] 
The second condition above implies that $q_{\tau_\bftt}\geq 2$ (since $x_\tau=d$).  By \Cref{lem:BD204} (c) we have that $\ac_{\uptau''}^+\neq 0$ and $\ac_{\uptau''}^-= 0$. Thus $\ac_{\uptau,+}\neq 0$ and $\ac_{\uptau,-}= 0$. The claim then follows from \Cref{taile0} (b). 
 
 In conclusion, we have proved \eqref{conta11} in all cases. 
 Consequently, there exists a marked Young diagram $\CE\in \MYD_\star$ that has a nonzero coefficient in $\maltese (\nac{\uptau}\circledast (0,0))$ such that
 \[
   \CE(2)\in (-\bN^+)\times (-\bN). 
 \]

The lemma then follows by noting that 
\[
   \CE'(2)\in \bN\times \Z
 \]
 for every marked Young diagram $\CE'\in \MYD_\star$ that has a nonzero coefficient in $\pac{\uptau_2} \circledast (0,0)$. 
 \end{proof}

\begin{lem}\label{lem:signbd}
Suppose that $\star \in \set{B,D}$ and $(\star, \abs{\check \CO})\neq (D, 0)$. 
Let $\uptau_i=(\tau_i, \wp_i)\in \PBPe_\star(\ckcO)$ ($i=1,2$). If
  \[
    \ac_{\uptau_1}^+= \ac_{\uptau_2}^+\quad \textrm{and}\quad \varepsilon_{\tau_1}=\varepsilon_{\tau_2},
  \]
  or 
  \be\label{acmnot0}
  \ac_{\uptau_1}^-= \ac_{\uptau_2}^-\neq 0,
  \ee
  then
  \[
      (p_{(\tau_1)_\bftt}, q_{(\tau_1)_\bftt}) =(p_{(\tau_2)_\bftt}, q_{(\tau_2)_\bftt}).
  \]
\end{lem}
\begin{proof}
By \Cref{lem:BD204}, the condition \eqref{acmnot0} implies that
\[
\varepsilon_{\tau_1}=\varepsilon_{\tau_2}=0.
\]
Thus $\varepsilon_{\tau_1}=\varepsilon_{\tau_2}$ always holds under our assumptions. 

If $\ac_{\uptau_1}^+= \ac_{\uptau_2}^+=0$, then \Cref{lem:BD204} implies that $x_{\tau_i}=s$ ($i=1,2$). Thus 
\[
      (p_{(\uptau_i)_\bftt}, q_{(\uptau_i)_\bftt}) =(0,  \bfrr_1(\check \CO)-\bfrr_2(\check \CO) +2),
  \]
  and the lemma holds.
  Now we assume that 
 \[  \ac_{\uptau_1}^+= \ac_{\uptau_2}^+\neq 0\quad \textrm{or}\quad \ac_{\uptau_1}^-= \ac_{\uptau_2}^-\neq 0.
 \]
 
 If $\bfrr_2(\ckcO)>\bfrr_3(\ckcO)$, then \Cref{taile0} (a) implies that 
 \[
  p_{(\tau_1)_\bftt},  p_{(\tau_2)_\bftt}\geq 1\quad \textrm{and}\quad 
  (p_{{(\tau_1)}_{\bftt}}-1, (-1)^{\varepsilon_{\tau_1}}q_{(\tau_1)_{\bftt}})=
   (p_{{(\tau_2)}_{\bftt}}-1, (-1)^{\varepsilon_{\tau_2}}q_{(\tau_2)_{\bftt}}),
\]
or 
\[
  (-1)^{\varepsilon_{\tau_1}}q_{(\tau_1)_{\bftt}},  (-1)^{\varepsilon_{\tau_2}}q_{(\tau_2)_{\bftt}}\geq 1\quad \textrm{and}\quad 
  (p_{{(\tau_1)}_{\bftt}}, (-1)^{\varepsilon_{\tau_1}}q_{(\tau_1)_{\bftt}}-1)=
   (p_{{(\tau_2)}_{\bftt}}, (-1)^{\varepsilon_{\tau_2}}q_{(\tau_2)_{\bftt}}-1).
\]
Thus the lemma holds.

Now we further assume that $\bfrr_2(\ckcO)=\bfrr_3(\ckcO)$. Without loss of generality we assume that
$
  p_{(\tau_1)_{\bftt}}\geq p_{(\tau_2)_{\bftt}}.
$
Then \Cref{taile0} easily implies that
\[
   p_{(\tau_1)_{\bftt}}= p_{(\tau_2)_{\bftt}}\ \textrm{ or }\ p_{(\tau_2)_{\bftt}}+1. 
\]

Assume by contradiction that $p_{(\tau_1)_{\bftt}}= p_{(\tau_2)_{\bftt}}+1$. 
Write 
\[
  \ac_{\uptau_i}=\ac_{\uptau_i, +}+\ac_{\uptau_i,-}\qquad (i=1,2)
\]
as in \eqref{decac}. 
Then \Cref{taile0} implies that
\[
  \ac_{\uptau_1, -}^+=\ac_{\uptau_2, +}^+\neq 0\quad \textrm{or}\quad \ac_{\uptau_1, -}^-=\ac_{\uptau_2, +}^-\neq 0.
\]
This further implies that 
\[
\maltese(\nac{\uptau_1''}\circledast (0,0) ) =
    \pac{\uptau_2''}\circledast (0,0) \neq 0,
    \]
    which contradicts \Cref{lempn0}. Here $\uptau_i''\in \PBPes{(\ckDD^2(\ckcO))}$ is the double descent of $\uptau_i$ ($i=1,2$). This proves the lemma. 
\end{proof}

 \subsection{Further properties of combinatorial  cycles}\label{secfurth}
 
 We begin with the following lemma.
 \begin{lem}\label{lem:BD000}
Suppose that $\star \in \set{B,D}$ and $(\star, \abs{\check \CO})\neq (D, 0)$. Then there exists no  $\uptau_1:=(\tau_1,\wp_1)$, $\uptau_2:=(\tau_2,\wp_2)\in \PBPe_\star(\ckcO)$ such that
  \be\label{ac12}
    \ac_{\uptau_1}\lotimes (1,1)= \ac_{\uptau_2}.
  \ee
\end{lem}
\begin{proof}
Suppose by contradiction that \eqref{ac12} holds. \Cref{taile0} implies that
\[
(\ac_{\uptau_1}\lotimes (1,1))^+=0.
\]
Hence $\ac_{\uptau_2}^+=0$. \Cref{lem:BD204} then implies that $x_{\tau_2}=s$. Since \eqref{ac12} implies that $\ac_{\uptau_1}= \ac_{\uptau_2}\lotimes (1,1)$, we also have that $x_{\tau_1}=s$ by a similar argument.  

Now we have that $q_{(\tau_1)_\bftt}\geq 2$ and $\varepsilon_{\tau_1}=1$.  Together with  \Cref{lem:BD204} (a) and \Cref{taile0}, this implies that 
\[
(\ac_{\uptau_1}\lotimes (1,1))^- \neq 0.
\]
However, \Cref{lem:BD204} (b) implies that $\ac_{\uptau_2}^-=0$. This contradicts \eqref{ac12}. 
\end{proof}
 The following is the first main result of this subsection. 
 
\begin{prop}\label{lem:BD1111}
Suppose that $\star \in \set{B,D}$ and $(\star, \abs{\check \CO})\neq (D, 0)$. Let $\uptau_i=(\tau_i, \wp_i)\in \PBPe_\star(\ckcO)$ and
  $\epsilon_i\in \Z/2\Z$ ($i=1,2$). If
  \[
    \ac_{\uptau_1}\lotimes (\epsilon_{1},\epsilon_{1})= \ac_{\uptau_2}\lotimes( \epsilon_{2},\epsilon_{2} ),
  \]
  then
  \[
    \epsilon_1=\epsilon_2,\quad \varepsilon_{\tau_1}=\varepsilon_{\tau_2}, \AND \ac_{\uptau_{1}}=\ac_{\uptau_{2}}.
  \]
\end{prop}

\begin{proof}
\Cref{lem:BD000} implies that $\epsilon_{1}=\epsilon_{2}$. Thus 
\be\label{aceq12}
  \ac_{\uptau_{1}}=\ac_{\uptau_{2}}. 
\ee

Recall that $\varepsilon_{\tau_i}=0$ if and only if $x_{\tau_i}=d$ ($i=1,2$). By \Cref{lem:BD204}, this is equivalent to the condition that   
\[
  \ac_{\uptau_{i}}^+\neq 0\quad \textrm{and}\quad \ac_{\uptau_{i}}^-\neq 0.
\]
Therefore \eqref{aceq12} implies that  $\varepsilon_{\tau_1}=\varepsilon_{\tau_2}$. 
\end{proof}

\begin{lem}\label{lem:BD2222}
Suppose that $\star \in \set{B,D}$ and 
$\ckcO$ is quasi-distinguished. Let $\uptau_i=(\tau_i, \wp_i)\in \PBPe_\star(\ckcO)$ 
  ($i=1,2$). If
  \be\label{acleq4}
    \ac_{\uptau_1}= \ac_{\uptau_2},
  \ee
  then
  $
  \uptau_{1}=\uptau_{2}.
  $
\end{lem}
\begin{proof}
The lemma obviously holds when $\check \CO$ has at most one row (see \Cref{eqtail2}).  Now we assume that $\check \CO$ has at least two rows and the lemma holds when the number of rows of $\check \CO$ is strictly smaller. 
By \Cref{lem:BD1111} and  \Cref{lem:signbd},  we have that 
  \[
   \varepsilon_{\tau_1}=\varepsilon_{\tau_2} \AND  (p_{(\tau_1)_\bftt}, q_{(\tau_1)_\bftt}) =(p_{(\tau_2)_\bftt}, q_{(\tau_2)_\bftt}).
  \]
As before, let  $\uptau_i''\in \PBPes(\check \CO'')$ denote the double descent of $\uptau_i$ ($i=1,2$), where $\check \CO''$ is the dual descent of $\check \CO'$.   

Note that $\bfrr_2(\ckcO)>\bfrr_3(\ckcO)$. Thus the equality \eqref{acleq4} implies that
$\ac_{\uptau_1''}=\ac_{\uptau_2''}$. 
The induction hypothesis implies that $\uptau_1''=\uptau_2''$. Then by \Cref{lem:delta} and \Cref{eqtail2}, we conclude that $\uptau_1=\uptau_2$.
\end{proof}

\begin{lem}\label{lem:BD22223}
Suppose that $\star \in \set{B,D}$ and 
$\ckcO$ is quasi-distinguished. Then for all $\uptau:=(\tau, \wp)\in \PBPe_\star(\ckcO)$, 
$ \ac_{\uptau}\in \MYD_\star(\ckcO)$. Moreover, for all $\CE\in \MYD_\star(\ckcO)$, there exist  $\uptau:=(\tau, \wp)\in \PBPe_\star(\ckcO)$ and $\epsilon\in \Z/2\Z$ such that 
$\ac_\uptau\otimes (\epsilon, \epsilon)=\CE$. 
\end{lem}
\begin{proof}
The first assertion is clear in view of the inductive formula \eqref{eq:BD}. 
The second assertion obviously holds when $\abs{\ckcO}=0$. By \Cref{eqtail2}, it holds when $\check \CO$ has precisely one row.  
Now we assume that $\check \CO$ has at least two rows and the lemma holds when the number of rows of $\check \CO$ is strictly smaller.  Note that $\bfrr_2(\check \CO)>\bfrr_3(\check \CO)$.

Let  $\CE\in \MYD_\star(\ckcO)$. By \Cref{eqtail2}, there is a unique pair $(\tau_0, \epsilon)\in \PBP_D(\check \CO_\bftt)\times \Z/2\Z$ such that 
\[
  \CE(1)=((-1)^\epsilon p_{\tau_0}, (-1)^{\varepsilon_{\tau_0}+\epsilon} q_{\tau_0}).
\]
Let $\CE''$ denote the marked Young diagram given by
\[
i\mapsto \begin{cases}
     (-1)^{\delta_{\tau_0}} \big((\CE\otimes (\epsilon, \varepsilon_{\tau_0}+\epsilon))(i+2)\big),\quad &\textrm{if $i\equiv 1 \pmod{4}$};\\
     (\CE\otimes (\epsilon, \varepsilon_{\tau_0}+\epsilon))(i+2),\quad &\textrm{otherwise},
  \end{cases}
\]
where $\delta_{\tau_0}$ is defined as in \eqref{gammatau}. 
It is easily checked that $\CE''\in \MYD_\star(\ckcO'')$ and
\[
  \mathrm T^{\delta_{\tau_0}}(\CE''\circledast (n_0,n_0))\circledast (p_{\tau_0},q_{\tau_0}) \otimes (0, \varepsilon_{\tau_0})=\CE\otimes (\epsilon, \epsilon). 
\] 

By the induction hypothesis, there exist $\uptau'':=(\tau'', \wp'')\in \PBPe_\star(\ckcO'')$ and $\epsilon'\in \Z/2\Z$  such that 
\[
\ac_{\uptau''}\otimes (\epsilon', \epsilon')=\CE''.
\]
By \Cref{lem:delta}, there is a unique $\tau\in \PBP_\star(\check \CO)$ such that
\[
  \nabla(\nabla(\tau))=\tau''\qquad\textrm{and}\qquad \tau_\bftt=\tau_0.
\]

Note that there is a unique $\wp\in \PP_\star(\check \CO)$ such that
\[
 \check \nabla( \check \nabla(\wp))=\wp''\qquad \textrm{and}\qquad 
 \varepsilon_{\check \nabla(\wp)}=\epsilon'.
\]
Now it is clear that $\uptau:=(\tau, \wp)$ and $\epsilon$ have the required property. 
\end{proof}
\begin{prop}\label{lem:BDinj0}
Suppose that $\star \in \set{B,D}$, $(\star, \abs{\check \CO})\neq (D, 0)$, and 
$\ckcO$ is quasi-distinguished. Then the map 
 \be\label{acme0}
    \PBPes( \ckcO )\times \bZ/2\bZ \longrightarrow \bZ[\MYD_{\star}(\cO)],\quad
    (\uptau,\epsilon) \mapsto \ac_{\uptau}\otimes (\epsilon, \epsilon)
  \ee
is injective and its image equals $\MYD_{\star}(\cO)$.
\end{prop}

\begin{proof}
 The first assertion follows from   \Cref{lem:BD1111} and \Cref{lem:BD2222}. 
 The second assertion follows from \Cref{lem:BD22223}. 
 \end{proof}

\subsection{Combinatorial cycles for  $\star\in \set{C,\wtC}$}\label{seccwtc}

Similar to \Cref{lem:delta}, the descent map has  the following property in the case when $\star\in \{C,\wtC\}$.

\begin{prop}[{\cite{BMSZa}*{Proposition~10.8}}]\label{prop:CC.bij}
Suppose that $\star \in \set{C,\wtC}$ and
consider the descent map
\begin{equation}\label{eq:DD.CC}
\nabla: \PBP_\star(\ckcO)\longrightarrow  \PBP_{\star'}(\ckcOp).
\end{equation}
\begin{enuma}
\item If
 $\bfrr_1(\ckcO)>\bfrr_2(\ckcO)$, then
the map \eqref{eq:DD.CC}  is bijective.

\item If   $\bfrr_1(\ckcO)=\bfrr_2(\ckcO)$,
then the  map \eqref{eq:DD.CC} is injective with image 
\[
\Set{\tau'\in \PBP_{\star'}(\ckcOp)\,:\, x_{\tau'}\neq s}.
\]
\end{enuma}
\end{prop}

 The main purpose of this subsection is to prove the following proposition. 
\begin{prop}\label{lem:C}
  Suppose that $\star \in \set{C,\wtC}$. 
  
\noindent (a) For all $\uptau\in \PBPes(\ckcO)$, $\ac_\uptau$ is nonzero and multiplicity free. 

\noindent (b) Let $\uptau_{i} = (\tau_{i},\wp_{i})\in \PBPes(\ckcO)$ ($i=1,2$). If  $\ac_{\uptau_{1}} =  \ac_{\uptau_{2}}$, then
    \[
      \Sign(\tau'_{1})=\Sign(\tau'_{2}) \quad \text{and} \quad
    \varepsilon_{\wp_{1}}=\varepsilon_{\wp_{2}}, \]
    where $\tau'_{i} :=
    \DD(\tau_{i})$ ($i=1,2$).
    
    \noindent (c) If $\ckcO$ is quasi-distinguished, then the map \[
    \PBPes(\ckcO)\rightarrow \Z[\MYD_\star(\CO)],\quad \uptau\mapsto \ac_\uptau
    \]
    is injective and its image equals $\MYD_\star(\CO)$. 
  \end{prop}

Suppose that $\star\in \set{C,\wtC}$ in this subsection. 
If $(1,2)$ is vacant or tailed in $\check \CO$, then both $\PBPes(\ckcO)$ and $\MYD_\star(\CO)$ are singletons, and the proposition is clear. So we only need to treat the case when  $(1,2)$ is primitive  or balanced in $\check \CO$.

Define a map
\be\label{thetaoop}
\begin{array}{rcl}
\check{\vartheta}_{\check \CO'}^{\check \CO}:=\sum_{\sfss'} \check{\vartheta}_{\sfss',\check \CO'}^{\sfss, \check \CO} & : & \Z[\MYD_{\star'}(\CO')]=
  \bigoplus_{\sfss'}
  \Z[\MYD_{\sfss'}(\CO')] \smallskip\\
  &\rightarrow& \Z[\MYD_{\star}(\CO)]=\Z[\MYD_{\sfss}(\CO)],
  \end{array}
  \ee
  where $\sfss'$ runs over all  classical signatures of type $\star'$ such that $\abs{\sfss'}=\abs{\CO'}$, and $\sfss:=(\star, \frac{|\CO|}{2}, \frac{|\CO|}{2})$. 
  
For every $\CA\in \bZ[\MYD_{\star}] $, define its  descent signature to be 
\begin{eqnarray*}
  && \dsign(\CA) \\
  &:=& \set{\Sign(\nabla_{\textrm{naive}}(\sV\circ \CE))\,:\, \CE\in \MYD_\star, \,\textrm{$\CE$ has a nonzero coefficient in $\CA$} }\subset \bN\times \bN.
\end{eqnarray*}
Here $\sV\circ \CE$ is the signed Young diagram as in \eqref{sfcirc0}, and $\nabla_{\textrm{naive}}$ indicates the naive descent of a signed Young diagram as in \eqref{naivedso}. Similarly, for every $\CA'\in \bZ[\MYD_{\star'}] $, define its   signature to be 
\begin{eqnarray*}
  && \sign{\CA'} \\
  &:=& \set{\Sign(\sV\circ \CE') \,:\, \CE'\in \MYD_{\star'}, \,\textrm{$\CE'$ has a nonzero coefficient in $\CA'$} }\subset \bN\times \bN.
\end{eqnarray*}

Let 
$ \uptau_i:=(\tau_i, \wp_i)\in \PBPes(\ckcO),
$
whose descent is denoted 
$ \uptau'_i:=(\tau'_i, \wp'_i)\in \PBPesp(\ckcOp)$ ($i=1,2$).

\medskip 

 {\bf The case when $(1,2)$ is primitive in $\check \CO$.} 
 By the definition of ``primitive", $\bfrr_{1}(\ckcO)-\bfrr_{2}(\ckcO)$ is  positive and even.  In this case, for every classical signature $\sfss'=(\star', p',q')$ with $p'+q'=\abs{\CO'}$, and every  $\CA'\in \Z[\MYD_{\sfss'}(\CO')]$, we have that 
\[ 
 \check{\vartheta}_{\check \CO'}^{\check \CO}(\CA')=\maltese^{\delta'}\big( \CA' 
      \circledast (n_{0},n_{0})\big),
\] 
where $n_0:=\frac{\bfcc_{1}(\cO)-\bfcc_{2}(\cO)}{2}$ (which belongs to $\bN$), and 
   \be\label{y}
   \delta':=\begin{cases}
   \frac{p'-q'}{2}, \qquad & \textrm{if $\star=C$};\\
    \frac{p'-q'-1}{2}, \qquad &\textrm{if $\star=\wtC$}.
   \end{cases}
    \ee
 It is clear that
 \be\label{dsign1}
 \dsign(\check{\vartheta}_{\check \CO'}^{\check \CO}(\CA'))=\sign{\CA'}=
 \begin{cases}
   \{(p',q')\},\quad & \textrm{if $\CA'\neq 0$};\\
   \emptyset, \quad & \textrm{if $\CA'=0$}.
 \end{cases}
 \ee
    Moreover, the map \eqref{thetaoop}  restricts to a bijection
    \be\label{bi222}
    \check{\vartheta}_{\check \CO'}^{\check \CO}  :  \MYD_{\star'}(\CO')\rightarrow \MYD_{\star}(\CO),
    \ee
which further descends to a bijection
    \be\label{bi22222}
    \check{\vartheta}_{\check \CO'}^{\check \CO}  :  \SYD_{\star'}(\CO')\rightarrow \SYD_{\star}(\CO).
    \ee
    Here 
    \[
    \SYD_{\star}(\CO):=
    \{\sO\in \SYD_{\star}\,:\,\textrm{the composition of $\BN^+\xrightarrow{\sO}\BN\times \BN\xrightarrow{(a,b)\mapsto a+b}\BN$ equals $\CO$}\},
    \]
    and $\SYD_{\star'}(\CO')$ is similarly defined. 

\subsubsection*{Proof of \Cref{lem:C}~(a)} Let $\uptau\in \PBPes(\ckcO)$ whose descent is denoted by $\uptau'\in \PBPesp(\ckcO')$. By \Cref{lem:dlift001}, $\ac_{\uptau'}$ is nonzero  and 
 multiplicity free. Then the injectivity of \eqref{bi22222} implies that $\ac_{\uptau}$ is also nonzero and multiplicity free.
 \qed

\smallskip

\subsubsection*{Proof of \Cref{lem:C}~(b)} Assume 
that $\ac_{\uptau_{1}}=\ac_{\uptau_{2}}$. Then the injectivity of \eqref{bi222} implies that  
\[\ac_{\uptau'_{1}}\otimes
(\varepsilon_{\wp_{1}},\varepsilon_{\wp_{1}})=\ac_{\uptau'_{2}}\otimes
(\varepsilon_{\wp_{2}}, \varepsilon_{\wp_{2}}).
\]
Here $\uptau_i':=(\tau_i',\wp_i')$ is the descent of $\uptau_i$ ($i=1,2$).

Now \Cref{lem:BD1111}  implies that 
\[
\varepsilon_{\wp_{1}}=\varepsilon_{\wp_{2}}\qquad \textrm{and}\qquad 
\ac_{\uptau'_{1}}=\ac_{\uptau'_{2}}.
\]
By \Cref{lem:BD204} (a),  both $\ac_{\uptau'_{1}}$ and $\ac_{\uptau'_{2}}$
are nonzero. Thus   
 \eqref{dsign1} implies that
 \[
\{\sign{\tau_1'}\}= \sign{\ac_{\uptau'_{1}}}= \sign{\ac_{\uptau'_{2}}}=\{\sign{\tau_2'}\}.
 \] 
Therefore 
$\Sign(\tau'_{1})=\Sign(\tau'_{2})$. \qed

\subsubsection*{Proof of \Cref{lem:C}~(c)}
 Suppose that $\ckcO$ is quasi-distinguished. Then $\ckcO'$ is also quasi-distinguished.  Assume 
that 
$
\ac_{\uptau_{1}}=\ac_{\uptau_{2}}$. By the previous argument, we have that 
\[
\ac_{\uptau'_{1}}=\ac_{\uptau'_{2}}\qquad \textrm{and}\qquad 
\Sign(\tau'_{1})=\Sign(\tau'_{2}).
\]
By \Cref{lem:BD2222}, we have that  $\uptau'_{1}=\uptau'_{2}$. Hence $\uptau_{1}=\uptau_{2}$ by \Cref{prop:CC.bij}.  \qed

\newcommand\mapsfrom{\mathrel{\reflectbox{\ensuremath{\mapsto}}}}

Consider the commutative diagram
\be\label{cd00}
 \begin{CD}
 \PBPesp( \ckcO' )\times \bZ/2\bZ
                  @<  (\nabla(\uptau), \varepsilon_\wp)                   \,\mapsfrom\,
                  \uptau:=(\tau, \wp)
                  <<  \PBPes( \ckcO )\\
            @VV  V         @ VV  \uptau\,\mapsto\, \ac_\uptau V \\
    \Z[\MYD_{\star'}(\cO')]  @>  \check{\vartheta}_{\check \CO'}^{\check \CO}  >> \Z[\MYD_{\star}(\cO)],
  \end{CD}
\ee
where the left vertical arrow is the map \eqref{acme0} for $\check \CO'$. \Cref{prop:CC.bij} implies that the top horizontal arrow is bijective. Thus \Cref{lem:BDinj0} and the bijectivity of \eqref{bi222} imply that the image of the right vertical arrow is $\MYD_{\star}(\cO)$. \qed

\subsection*{The case when $(1,2)$ is balanced in $\check \CO$} By the definition of ``balanced", $\bfrr_{1}(\ckcO)=\bfrr_{2}(\ckcO)>0$.
For every classical signature $\sfss'=(\star', p',q')$ with $p'+q'=\abs{\CO'}$, and every  $\CA'\in \Z[\MYD_{\sfss'}(\CO')]$, we have that 
\[ 
 \check{\vartheta}_{\check \CO'}^{\check \CO}(\CA')= \maltese^{\delta'} ((\CA'^+ + \CA'^-)\circledast (0,0)  ),
\] 
where $\delta'$ is as in \eqref{y}. For simplicity write 
\[
\check \vartheta^+(\CA'):=\maltese^{\delta'} (\CA'^+\circledast (0,0)  )\qquad\textrm{and}\qquad \check \vartheta^-(\CA'):=\maltese^{\delta'} (\CA'^-\circledast (0,0)  )
\]
so that 
   \[
  \check{\vartheta}_{\check \CO'}^{\check \CO}(\CA')=\check \vartheta^+(\CA')+\check \vartheta^-(\CA').
    \]
Note that
\[
  \dsign(\check{\vartheta}^+(\CA'))\subset\{(p'-1,q')\}\qquad \textrm{and}\qquad
  \dsign(\check{\vartheta}^-(\CA'))\subset\{(p',q'-1)\}
\]
and hence 
\be\label{intemp}
\dsign(\check{\vartheta}^+(\CA'))\cap
  \dsign(\check{\vartheta}^-(\CA'))=\emptyset. 
\ee
Using \eqref{intemp}, we get the following two conclusions:
\begin{eqnarray}
\label{conc1}
  &&\textrm{If $\check{\vartheta}^+(\CA')\neq 0\ $ or  $\ \check{\vartheta}^-(\CA')\neq 0$, $\ \ $ then $\check{\vartheta}_{\check \CO'}^{\check \CO}(\CA')\neq 0$.}\smallskip \\
  \label{conc2}
  &&\textrm{If $\CA'$ is multiplicity free, then so is   $\check{\vartheta}_{\check \CO'}^{\check \CO}(\CA')$.}
\end{eqnarray}

As before, let $\uptau\in \PBPes(\ckcO)$ whose descent is denoted by $\uptau':=(\tau',\wp')\in \PBPesp(\ckcO')$.
\subsubsection*{Proof of \Cref{lem:C}~(a)}
According to \Cref{prop:CC.bij},  $x_{\taup}\neq s$.
Then \Cref{lem:BD204} implies that $\pac{\uptaup} \neq 0$, which further implies that $\check{\vartheta}^+ (\ac_{\uptaup})\neq 0$. 
Therefore, by \eqref{conc1}, we have that 
\[
\ac_{\uptau}=\check{\vartheta}_{\check \CO'}^{\check \CO}(\ac_{\uptau'})=\check{\vartheta}^+ (\mathcal L_{\uptaup})+\check{\vartheta}^- (\mathcal L_{\uptaup})\neq 0.
\]By \Cref{lem:dlift001}, $\ac_{\uptau'}$ is multiplicity free. Thus $\ac_{\uptau}=\check{\vartheta}_{\check \CO'}^{\check \CO}(\ac_{\uptau'})$ is also multiplicity free by \eqref{conc2}. \qed

\smallskip

The argument above shows that \be\label{dsign0}
 \dsign(\ac_{\uptau})=\{(p'-1,q')\}\ \ \textrm{or}\ \ \{(p'-1,q'), (p', q'-1)\},
\ee
where $(p',q'):=\sign{\tau'}:=(p_{\tau'}, q_{\tau'})$. 

\subsubsection*{Proof of \Cref{lem:C}~(b)}
Note that 
$\varepsilon_{\wp_{1}}=\varepsilon_{\wp_{2}}=0$. If $\ac_{\uptau_{1}}=\ac_{\uptau_{2}}$, then  \eqref{dsign0} implies that $\Sign(\tau'_{1})=\Sign(\tau'_{2})$. 
\qed

\smallskip

Note that $\ckcO$ is not quasi-distinguished, so \Cref{lem:C}~(c) is vacuous.

\subsection{Combinatorial cycles for  $\star\in \set{C^*,D^*}$}\label{quat} 

In this subsection, we sketch  proofs of  the following two propositions. 
\begin{prop}\label{lem:Cdstar}
  Suppose that $\star \in \set{C^*,D^*}$. 
  Then the map \[
    \PBPes(\ckcO)\rightarrow \Z[\MYD_\star(\CO)],\quad \uptau\mapsto \ac_\uptau
    \]
    is injective and its image equals $\MYD_\star(\CO)$. 
  \end{prop}

\begin{prop}\label{lem:Cdstar2}
Suppose that  $\star= D^*$. Let $\uptau_i=(\tau_i, \emptyset)\in \PBPe_\star(\check \CO)$, and write $\uptau'_i=(\tau'_i, \emptyset)$ for its descent ($i=1,2$).   If
\[
  \ac_{\uptau_1}= \ac_{\uptau_2},
\]
then
\[
 ( p_{\tau'_1}, q_{\tau'_1})=( p_{\tau'_2}, q_{\tau'_2}).
\]
 \end{prop}

Assume that  $\star\in \{C^*,D^*\}$ in this subsection. When  $\check \CO$ is not quasi-distinguished, the sets $\PBPes(\ckcO)$ and $\MYD_\star(\CO)$
are both empty (see \Cref{prop:CD*} and the remark after \Cref{thmac0}), and so there is nothing to prove in Propositions \ref{lem:Cdstar} and \ref{lem:Cdstar2}. Thus in the rest of this subsection we assume that $\check \CO$ is  quasi-distinguished. 

First we suppose that $\star= C^*$. We recall from \cite[Section 10.5]{BMSZa} the definition of the tail of a painted bipartition. 
Put
\[
k := 
\lfloor\frac{\bfrr_{1}(\ckcO)-\bfrr_{2}(\ckcO)-1}{2}\rfloor\in \bN.
\]
Note that $k=\jmath_{\check \CO}- \imath_{\check \CO}$. 
Let $\ckcO_{\bftt}$ be the Young diagram consisting of  one row
with length  $2k+1$.
Note that
 $\check \CO_{\mathbf t}$ has good parity (with respect to $C^*$) and every element in $\PBP_{C^*}(\ckcO_\bftt)$ has the form
 \begin{equation}
 \label{tailcstar}
    \emptyset \times  \ytb{{x_1} , {x_2} , {\enon\vdots},{\enon{\vdots}},{x_k}  } \times
 C^*.
\end{equation}
 
\medskip

For every $\tau=(\imath_\ckcO,\cP)\times(\jmath_\ckcO,\cQ)\times C^* \in
\mathrm{PBP}_{C^*}(\check \CO) $, its tail $\tau_\bftt\in\PBP_{\star_\bftt}(\ckcO_\bftt) $  is defined to be the painted bipartition in
 \eqref{tailcstar} such that
\[
  (x_1, x_2, \cdots, x_k)= (\cQ(\bfcc_{1}(\imath_\ckcO)+1,1),\cQ(\bfcc_{1}(\imath_\ckcO)+2,1),\cdots, \cQ(\bfcc_{1}(\jmath_\ckcO),1)).
\]
By convention, $\tau_\bftt=\emptyset\times \emptyset \times C^*$ when $k=0$. 
It is clear that $(\tau_\bftt)_\bftt=\tau_\bftt$.

Recall that $\check \CO':=\check \nabla(\check \CO)$ is the dual descent of $\check \CO$ (see \eqref{duald}).  The key properties of the descent map 
when $\star\in \set{C^*,D^*}$ 
are summarized in the following two propositions.

\begin{prop}[{\cite{BMSZa}*{Proposition~10.8}}]\label{prop:CC.bijq}

Suppose that $\star =D^*$. Then   the descent map
\begin{equation}\label{eq:DD.CCq}
\nabla: \PBP_\star(\ckcO)\longrightarrow  \PBP_{\star'}(\ckcOp)
\end{equation}
 is bijective.

\end{prop}

\begin{prop}[{\cite{BMSZa}*{Proposition~10.9}}]
\label{lem:deltaq}
Suppose that $\star =C^*$, and put  $\ckcOpp := \ckDD(\ckcO')$.  Then the map
\[
  \PBP_\star(\ckcO)\longrightarrow
    \PBP_\star(\ckcOpp)\times \PBP_{\star_\bftt}(\ckcO_\bftt),
    \qquad \tau \mapsto (\DD(\DD(\tau)),\tau_\bftt)
\]
is bijective. 

\end{prop}

Using the induction formula \eqref{indac}, we easily obtain the following. 

\begin{lem}\label{indcdstar}
Suppose that $\uptau:=(\tau,\emptyset)\in \PBPes(\ckcO)$ and write $\uptau'\in \PBPesp(\ckcO')$ for its descent. 

 \noindent (a) 
  If $\star = C^{*}$, then
  \[
    \ac_{\uptau}=\ac_{\uptau'} \circledast  (p_{\tau_{\bftt}},q_{\tau_{\bftt}}).
\]

\noindent (b) 
If $\star = D^{*}$, then
  \[
    \ac_{\uptau}=\ac_{\uptau'} \circledast (n_{0},n_{0}),
  \]
  where $n_0:=\frac{\bfcc_1(\CO)-\bfcc_2(\CO)}{2}\  (\in \bN)$. 
\end{lem}

Using \Cref{indcdstar}, an inductive argument  easily shows that the map
\be\label{acquat}
    \PBPes(\ckcO)\rightarrow \Z[\MYD_\star(\CO)],\quad \uptau\mapsto \ac_\uptau
    \ee
    is injective and its image is contained in  $\MYD_\star(\CO)$.
    Using \Cref{lem:deltaq}, a similar argument as in the proof of \Cref{lem:BD22223} shows that the image of the map \eqref{acquat} equals $\MYD_\star(\CO)$ when $\star=C^*$. When $\star=D^*$, we use \Cref{prop:CC.bijq} and a similar argument as in \eqref{cd00} to conclude that the image of the map \eqref{acquat} also equals $\MYD_\star(\CO)$. This proves \Cref{lem:Cdstar}. 
    
Now let the notation and assumptions be as in \Cref{lem:Cdstar2}. 
Then \Cref{indcdstar} implies that 
\be\label{myd}
\ac_{\uptau_1'}=\ac_{\uptau_2'}\in \MYD_\star(\CO).
\ee
In the case under consideration, namely $\star\in \{C^*, D^*\}$, the set $\MYD_\star$ is obviously identified with the set $\SYD_\star$. Let $\sO\in \SYD_\star$ denote the element that is identified with \eqref{myd}. Then by \eqref{signactau}, we have that
\[
( p_{\tau'_1}, q_{\tau'_1})=(p_\sO, q_\sO)=( p_{\tau'_2}, q_{\tau'_2}).
\]
 This proves \Cref{lem:Cdstar2}.

 Finally, Propositions \ref{thmac1}-\ref{thmac5} follow by combining Propositions \ref{lem:dlift001}, \ref{lem:BD204} (a), \ref{lem:BD1111}, \ref{lem:BDinj0}, \ref{lem:C},  \ref{lem:Cdstar}, and \ref{lem:Cdstar2}.

\section{Special orthogonal groups and unitary  groups: modifications}\label{modify}

We retain the setting and notation of \Cref{sec:intro}.  

\subsection{Special orthogonal groups}\label{sorth}

First, we prove the following proposition. 

\begin{prop}\label{propd2}
Suppose that $\star=D$ and $\pi\in \Unip_{\ckcO }(G)$. Then $\pi\otimes \det\cong\pi$ if and only if $\check \CO$ has bad parity (namely all row lengths of $\check \CO$ are even). 
\end{prop}
\begin{proof}
Recall that $G=\oO(p,q)$ with $p+q$ even.  The proposition is trivial when $p+q=0$. Thus we assume that $p+q$ is a positive even integer. 

First suppose that $\check \CO$ has bad parity. Then we have a disjoint union:
\be\label{oo1o2}
  \check \CO=\check \CO_1\sqcup \check \CO_2,
\ee
where $\check \CO_1$ and $\check \CO_2$ are two distinct $\SO_{p+q}(\C)$-orbits in $\check \g$. 
Moreover,
\be\label{primitivei}
 I_{\check \CO}=I_{\check \CO_1}\cap I_{\check \CO_2},
\ee
where $I_{\check \CO_i}$ ($i=1,2$) is the maximal ideal of $\CU(\g)$ attached to $\check \CO_i$ (see \cite[Section 2.3]{BMSZa}).

Similar to $\Unip_{\ckcO }(G)$, we have two sets $\Unip_{\ckcO_1 }(\SO(p,q))$ and $\Unip_{\ckcO_2 }(\SO(p,q))$ of special unipotent representations of $\SO(p,q)$ (\cite[Section 2.3]{BMSZa}). Since primitive ideals are prime, the equality  \eqref{primitivei} implies that each irreducible subrepresentation of $\pi|_{\SO(p,q)}$ belongs to  either $\Unip_{\ckcO_1 }(\SO(p,q))$ or $\Unip_{\ckcO_2 }(\SO(p,q))$.
Note that if $\pi|_{\SO(p,q)}$ contains a subrepresentation in $\Unip_{\ckcO_1 }(\SO(p,q))$, then it must also contain a  subrepresentation in $\Unip_{\ckcO_2 }(\SO(p,q))$, and vice versa. Since the sets $\Unip_{\ckcO_1 }(\SO(p,q))$ and $\Unip_{\ckcO_2 }(\SO(p,q))$ are disjoint, we conclude that the representation  $\pi|_{\SO(p,q)}$ is reducible. Therefore $\pi\otimes \det\cong \pi$ by Clifford theory.

Recall from  the introductory section that
\begin{equation*}
     \ckcO=2\check \CO'_\mathrm b  \stackrel{r}{\sqcup} \check \CO_\mathrm g
\end{equation*}
and $n_\mathrm b:=\abs{\check \CO'_\mathrm b}$. 

Now we suppose that $\check \CO$ does not have bad parity. 
Then the $\oO_{p+q}(\C)$-orbit $\check \CO$ is also a $\SO_{p+q}(\C)$-orbit in $\check \g$, and  likewise $\check \CO_\mathrm g$ is also a $\SO_{p+q-2n_\mathrm b}(\C)$-orbit.
As before, we have a set $\Unip_{\ckcO }(\SO(p,q))$ of special unipotent representations of $\SO(p,q)$, and a set $\Unip_{\ckcO_\mathrm g} (\SO(p-n_\mathrm b,q-n_\mathrm b))$ of special unipotent representations of $\SO(p-n_\mathrm b,q-n_\mathrm b)$. By \cite[Theorem 2.20]{BMSZa}, every representation $\pi_0$ in $\Unip_{\ckcO }(\SO(p,q))$ has the form
\[
  \pi_0\cong \Ind_P^{\SO(p,q)}( \pi_\mathrm b\widehat \otimes \pi_\mathrm g), 
\]
where
\[
\pi_\mathrm b\in \Unip_{\check \CO'_\mathrm b}(\GL_{n_\mathrm b}(\R)),\qquad  \pi_\mathrm g\in \Unip_{\check \CO_\mathrm g}(\SO(p-n_\mathrm b,q-n_\mathrm b)),
\]
and
$P$ is a parabolic subgroup of $\SO(p,q)$ with its Levi component isomorphic to
$\GL_{n_\mathrm b}(\R)\times \SO(p-n_\mathrm b,q-n_\mathrm b)$.  

By \Cref{thm100} and \cite[Theorem 2.27]{BMSZa},  
\[
\#(\Unip_{\check \CO_\mathrm g}(\oO(p-n_\mathrm b,q-n_\mathrm b)))=2 \cdot \#(\Unip_{\check \CO_\mathrm g}(\SO(p-n_\mathrm b,q-n_\mathrm b))).
\]
Clifford theory then implies that $\pi_\mathrm g^{g_0}\cong \pi_\mathrm g$, where $g_0\in \oO(p-n_\mathrm b,q-n_\mathrm b)\setminus \SO(p-n_\mathrm b,q-n_\mathrm b)$, and $\pi_0^{g_0}$ is the twisting of the representation $\pi_0$ by the conjugation via $g_0$. View $g_0$ as an element of $\oO(p,q)$ as usual. Then $g_0$ normalizes $P$ and centralizes $\GL_{n_\mathrm b}(\R)\subseteq P$. Thus
the twisted representation 
\[
 \pi_0^{g_0}\cong \Ind_P^{\SO(p,q)}( \pi_\mathrm b^{g_0}\widehat \otimes \pi_\mathrm g^{g_0})\cong \Ind_P^{\SO(p,q)}( \pi_\mathrm b\widehat \otimes \pi_\mathrm g)\cong \pi_0. 
\]
Using Clifford theory, this completes the proof of the proposition. 
\end{proof}

We already know that the set $\Unip_{\ckcO }(G)$ is empty if \eqref{doubleco} or \eqref{existgl} does not hold. 
 So suppose that  both \eqref{doubleco} and \eqref{existgl} hold. Now we prove \Cref{thmB} in the introductory section.  
\begin{proof}[Proof of \Cref{thmB}]
Recall the notation of  \Cref{thmB}. We are to show that the normalized smooth parabolic induction from $Q$ to $G$ yields
   a well-defined bijection
   \[
    \begin{array}{rccc}
 &\Unip_{\ckcO'_{\mathrm b}}( G'_{\mathrm b}) \times   \Unip_{\ckcO_{\mathrm g}}( G_{\mathrm g})  &         \longrightarrow &\Unip_{\ckcO }(G), \\
                &   (\pi',\pi_\mathrm g) & \mapsto & \pi'\rtimes \pi_\mathrm g.
    \end{array}
 \]
When $\star\in\{A, \wtA, C, \wtC, C^*\}$, the result is proved in  \cite[Theorems 2.16 and 2.20]{BMSZa}. When $\star=B$, the result directly follows from \cite[Theorem 2.20]{BMSZa}. We now assume that $\star\in \{D, D^*\}$. Without loss of generality we further assume that $\abs{\check \CO}>0$.

First suppose that $\check \CO$ has bad parity so that the group $G_\mathrm g$ is trivial. Write 
\[
  \check \CO=\check \CO_1\sqcup \check \CO_2,
\]
as in \eqref{oo1o2}.  Then there is a unique $i\in \{1,2\}$ such that the complexified Lie algebra $\mathfrak q$ of $Q$ is   $\check \CO_i$-relevant in the following sense (see \cite[Section 2.6]{BMSZa}): the Barbasch-Vogan dual of $\check \CO_i$, which is an $\SO_{2n}(\C)$-orbit in $\mathrm{Nil}(\g^*)$,  equals the induction from $\mathfrak q$ to $\g$ of a nilpotent orbit  in $\mathfrak l^*$. Here $\mathfrak l$ is a Levi factor of $\mathfrak q$ which is isomorphic to $\mathfrak{gl}_n(\C)$. 
Without loss of generality we assume that $i=1$. 

If $\star=D^*$, then \cite[Theorem 2.20]{BMSZa} implies that
\[
  \Unip_{\ckcO_2 }(G)=\emptyset
\]
and hence 
\[
  \Unip_{\ckcO }(G)=\Unip_{\ckcO_1 }(G).
\]
Here the set $\Unip_{\ckcO_i }(G)$ ($i=1,2$) of special unipotent representations is defined in \cite[Section 2.3]{BMSZa}.  Therefore the result follows from \cite[Theorem 2.20]{BMSZa} in this case.

If $\star=D$,  by Clifford theory, \Cref{propd2} and its proof imply that the map
\[
  \Unip_{\ckcO_1 }(\SO(p,q))\rightarrow \Unip_{\ckcO }(\oO(p,q)), \quad \pi_0\mapsto \Ind_{\SO(p,q)}^{\oO(p,q)} \pi_0
\]
is well-defined and bijective. Therefore the result also follows from \cite[Theorem 2.20]{BMSZa} in this case.

Now we  suppose that $\check \CO$ does not have bad parity so that the group $G_\mathrm g$ is non-trivial.
Then $\check \CO$ is an $\SO_{2n}(\C)$-orbit, and likewise $\check \CO_\mathrm g$ is an $\SO_{2n-2n_\mathrm b}(\C)$-orbit. Thus $I_{\check \CO}$ is a maximal ideal of $\CU(\g)$, and likewise $I_{\check \CO_\mathrm g}$ is a maximal ideal of $\CU(\g_\mathrm g)$. Here  $\g_\mathrm g$ is the complexified Lie algebra of $G_\mathrm g$. 

If $\star=D^*$, the result is the same as  \cite[Theorem 2.20]{BMSZa}. 

If $\star=D$ so that $G=\oO(p,q)$ with $p+q$ even, then for every $\pi'\in \Unip_{\check \CO'_\mathrm b}(\GL_{n_\mathrm b}(\R))$ and  $\pi_\mathrm g\in \Unip_{\check \CO_\mathrm g}(\oO(p-n_\mathrm b,q-n_\mathrm b))$, we have that 
\[
 \left(\Ind_Q^G (\pi'\widehat \otimes \pi_\mathrm g)\right)|_{\SO(p,q)}\cong 
 \Ind_{Q_0}^{\SO(p,q)} (\pi'\widehat \otimes (\pi_\mathrm g)|_{\SO(p-n_\mathrm b,q-n_\mathrm b)}),
\]
where $Q_0:=Q\cap \SO(p,q)$. Clifford theory and \Cref{propd2} imply that 
$(\pi_\mathrm g)|_{\SO(p-n_\mathrm b,q-n_\mathrm b)}$ is irreducible and hence a representation in $\Unip_{\check \CO_\mathrm g}(\SO(p-n_\mathrm b,q-n_\mathrm b))$. Then \cite[Theorem 2.20]{BMSZa} implies that $\left(\Ind_Q^G (\pi'\widehat \otimes \pi_\mathrm g)\right)|_{\SO(p,q)}$ is irreducible and belongs to $\Unip_{\check \CO}(\SO(p,q))$.
Hence $\Ind_Q^G (\pi'\widehat \otimes \pi_\mathrm g)$ is irreducible and belongs to 
$\Unip_{\check \CO}(\oO(p,q))$.
We therefore have a commutative diagram
\[
 \begin{CD}
 \Unip_{\ckcO'_{\mathrm b}}( G'_{\mathrm b}) \times   \Unip_{\ckcO_{\mathrm g}}( G_{\mathrm g})  
                  @>   \textrm{parabolic induction}  >>   \Unip_{\ckcO }(G)\\
            @VV  \textrm{the identity map $\times$ the restriction map}  V         @ VV  \textrm{the restriction map} V \\
      \Unip_{\ckcO'_{\mathrm b}}( G'_{\mathrm b}) \times   \Unip_{\ckcO_{\mathrm g}}( \SO(p-n_\mathrm b,q-n_\mathrm b))  @> \textrm{parabolic induction} >> \Unip_{\ckcO }(\SO(p,q)).\\
  \end{CD}
\]
The bottom horizontal arrow is bijective by \cite[Theorem 2.20]{BMSZa}. This implies that the top horizontal arrow is also bijective, by considering the free action of $\Z/2\Z$ on the sets $\Unip_{\ckcO_{\mathrm g}}( G_{\mathrm g})$   and $\Unip_{\ckcO }(G)$ by the determinant twist (see 
\Cref{propd2}). This completes the proof of the theorem. 
\end{proof}

\subsection{Unitary groups}\label{subsec:unitary}

In this subsection, we suppose that $\star\in \{A, \wtA\}$ so that $G=\oU(p,q)$ or $\widetilde \oU(p,q)$. In view of \Cref{thmB}, we shall only consider the case when the $\check G$-orbit $\check \CO$ in $\mathrm{Nil}(\check \g)$ has good parity. In this case, every representation in $\mathrm{Unip}_{\check \CO}(G)$ is realized as a
derived functor module  and is in particular unitary. Furthermore it is determined by its wavefront set. See \cite[Section 4]{BV83} and
\cite[Section 2]{Tr}. We will outline how these representations can also be constructed by iterated theta lifts, as conjectured by Trapa (\cite{Tr}*{Conjecture~2.5}). 

We make the following definition.

\begin{defn}\label{defpbp1}
 A painted partition for unitary groups is a pair $(\imath, \CP)\times \gamma$ where $(\imath, \CP)$ is a painted Young diagram and $\gamma \in \{A, \wtA\}$, subject to the following conditions:
  \begin{itemize}
    \item the symbols of $\CP$ are in
                  $   \{\bullet, s, r\}$;
                    \item every row of $\imath$ has an even number of boxes painted with $\bullet$;
         \item every nonzero column length of $\imath$ has the same parity as 
         \[
         \left\{  \begin{array}{ll}
               \abs{\imath}, &\quad \textrm{if $\gamma =A$};  \\
               \abs{\imath}+1, &\quad \textrm{if $\gamma=\wtA$}. 
           \end{array}
           \right.
         \]
  \end{itemize}
Denote by $\mathrm{PUP}$ the set of all painted partitions for unitary groups. 
   \end{defn}

For every  $\tau=(\imath, \CP)\times \gamma\in \mathrm{PUP} $, define its signature to be the pair
\begin{equation}\label{eq:signature}
    (p_\tau, q_\tau): = \left (\frac{\#(\cP^{-1}(\bullet))}{2}+\#(\cP^{-1}(r)),\,
    \frac{ \#(\cP^{-1}(\bullet))}{2}+\#(\cP^{-1}(s))\right).
\end{equation}
Put
\[
  G_\tau:= \left\{  \begin{array}{ll}
               \oU(p_\tau, q_\tau), &\quad \textrm{if $\gamma=A$};  \\
               \widetilde \oU(p_\tau, q_\tau), &\quad \textrm{if $\gamma=\wtA$}.
           \end{array}
           \right.
\]

The following lemma is obvious. 
\begin{lem}\label{dppu}
For every  $\tau=(\imath, \CP)\times \gamma\in \mathrm{PUP}$, there is a unique element $\tau'=(\imath', \CP')\times \gamma' \in \mathrm{PUP}$
with the following properties: 
\begin{itemize}
    \item 
    \[
      \gamma'= \left\{  \begin{array}{ll}
               \wtA, &\quad \textrm{if $\gamma=A$ and  $\abs{\imath}$ is odd};  \\
               A, &\quad \textrm{otherwise};
           \end{array}
           \right.
    \]
    \item
   $\imath'=\nabla_{\mathrm{naive}}(\imath)$, and for all   $(i,j)\in \BOX(\imath')$,
   \[
     \cP'(i,j)=\begin{cases}
    \bullet \textrm{ or } s,&\textrm{ if  $\ \cP(i,j+1)\in \{\bullet, s\}$;} \smallskip \\
  r, & \textrm{ if $\ \cP(i,j+1)=r$}.  \end{cases}
   \]
  \end{itemize}
\end{lem}

We call the element $\tau'$ in \Cref{dppu} the descent of $\tau$, to be denoted by $\nabla(\tau)$. 

\begin{eg}
  Suppose
  that \[
  \tau=\ytb{\bullet\bullet s,\bullet\bullet , sr,s,r}\times A \in \mathrm{PUP}.
  \]
  Then 
  \[
 \nabla( \tau)=\ytb{\bullet\bullet ,s ,r}\times \wtA \in \mathrm{PUP},
  \]
  $G_\tau=\oU(4,5)$,  and $G_{\nabla(\tau)}=\widetilde \oU(2,2)$. 
\end{eg}

Let $\tau=(\imath, \cP)\times \gamma \in \mathrm{PUP}$ and write $\tau':=\nabla(\tau)$. Let $V_\tau$ denote the standard module of $G_\tau$, which is a Hermitian space of signature $(p_\tau, q_\tau)$. Likewise let $V_{\tau'}$ denote the standard module of $G_{\tau'}$. Then \[
  V_{\tau, \tau'}:=V_{\tau}\otimes V_{\tau'}
\]
is also a Hermitian space with the Hermitian form
\[
  \la u\otimes u', v\otimes v'\ra_{\tau, \tau'}:= \la u, v\ra_{\tau}\cdot  \la  u',  v'\ra_{\tau'}\quad \textrm{for all }\, u,v\in V_\tau, \,  u',v'\in V_{\tau'}.
\]
Here $\la \,, \, \ra_{\tau}$ and $\la \,, \, \ra_{\tau'}$ are the Hermitian forms on $V_\tau$ and $V_{\tau'}$ respectively. Define the Heisenberg group
\[
  \oH(V_{\tau, \tau'}):=V_{\tau, \tau'}\times \sqrt{-1}\R,
\]
with group multiplication 
\[
  (u,t)\cdot (u',t'):=(u+u', t+t'+\mathrm{Im}(\la u, u'\ra_{\tau, \tau'})), \quad u,u'\in V_{\tau, \tau'}, \ t,t'\in \sqrt{-1}\R.
\]
Here every $z\in \C$ is written as $z=\mathrm{Re}(z)+\mathrm{Im}(z)$, with  $\mathrm{Re}(z)$ real and $\mathrm{Im}(z)$ purely imaginary. Let $\Sp(V_{\tau, \tau'})$ denote the symplectic group of all $\R$-linear automorphisms of $V_{\tau, \tau'}$   that preserve the $\R$-bilinear map 
$\mathrm{Im}(\la \, ,\, \ra_{\tau, \tau'})$. Form the Jacobi group 
$\widetilde \Sp(V_{\tau, \tau'})\ltimes \oH(V_{\tau, \tau'})$, where 
$\widetilde \Sp(V_{\tau, \tau'})$ is the metaplectic double cover of $\Sp(V_{\tau, \tau'})$, which acts on $\oH(V_{\tau, \tau'})$ as group automorphisms through the action of $\Sp(V_{\tau, \tau'})$ on $V_{\tau, \tau'}$.

Up to isomorphism, there is a unique genuine smooth Fr\'echet representation of $\widetilde \Sp(V_{\tau, \tau'})\ltimes \oH(V_{\tau, \tau'})$ of moderate growth, to be denoted by $\omega_{\tau, \tau'}$, such that $(\omega_{\tau, \tau'})|_{  \oH(V_{\tau, \tau'})}$ is irreducible with central character $t\mapsto e^t$ (here ``$e$" denotes the  base of natural logarithms). Using the natural homomorphism $G_\tau\times G_{\tau'}\rightarrow \widetilde \Sp(V_{\tau, \tau'})$, we also view $\omega_{\tau, \tau'}$ as a representation of $G_\tau\times G_{\tau'}$. Finally, we inductively define a representation $\pi_\tau$ of $G_\tau$ by
\[
\pi_\tau:=\left\{  \begin{array}{ll}
               \textrm{the trivial one-dimensional representation}, &\quad \textrm{if $\gamma=A$ and  $\abs{\imath}=0$};  \\
               \textrm{the non-trivial one-dimensional representation}, &\quad \textrm{if $\gamma=\wtA$ and  $\abs{\imath}=0$};  \\
               (\omega_{\tau, \tau'}\widehat \otimes  \pi_{\tau'})_{G_{\tau'}}, &\quad \textrm{if   $\abs{\imath}>0$}.
           \end{array}
           \right.
\]

\begin{thm}\label{unitaryg}
Suppose that $\star\in \{A, \wtA\}$ so that $G=\oU(p,q)$ or $\widetilde \oU(p,q)$. Let $\check \CO$ be a $\check G$-orbit  in $\mathrm{Nil}(\check \g)$ that has good parity. Then the map
\[
\begin{array}{rcl}
  \{\tau=(\imath, \cP)\times \star \in \mathrm{PUP} \,:\, (p_\tau,q_\tau)=(p,q),\,  \textrm{$ \imath$ equals the transpose of $\check \CO$} \}&\rightarrow & \mathrm{Unip}_{\check \CO}(G),\\
  \tau&\mapsto & \pi_\tau
  \end{array}
\]
is well-defined and bijective. 
Moreover, every representation in $\Unip_{\ckcO}(G)$ is unitarizable.
\end{thm}

The method of this paper proves the above theorem.  We omit the details. In addition we remark that for $G=\oU(p,q)$ or $\widetilde \oU(p,q)$, there is precisely one admissible orbit datum over each $K_\C$-orbit in $ \CO\cap \p$, and the associated cycle map induces a bijection between $\Unip_{\ckcO}(G)$ and $\mathrm{AOD}(\CO)$. Here $\CO$ is the Barbasch-Vogan dual of $\check \CO$ (assumed to have good parity) and other notation is similarly defined as in \Cref{classical}. 

\appendix

\def\Iso{\mathrm{Iso}}
\def\Herm{\mathrm{Herm}}
\section{Induced nilpotent orbits: explicit construction} 
\label{sec:orbitinduce}

\subsection{Parabolic induction of nilpotent orbits and multiplicities}
\label{app:induction-def}

Let $G$ be a real reductive group with Lie algebra ${\mathfrak g_\R}$, and let $P\subset G$ be a parabolic subgroup with Levi decomposition $P=M\ltimes U$. Write $\mathfrak m_\R:=\Lie(M)$, $\mathfrak u_\R:=\Lie(U)$, and 
let $\sO_0\subset\mathfrak m_\R$ be a nilpotent $M$-orbit. 

\begin{defn}\label{def:Ind-orbit}
A nilpotent $G$-orbit $\sO\subset \g_\R$ is said to be induced  from $\sO_0$ if 
\[
\sO \cap (\sO_0+\mathfrak u_\R)
\]
is nonempty and open in $\sO_0+ \mathfrak u_\R$. 
\end{defn}

Let $\sO$ be a nilpotent $G$-orbit induced from $\sO_0$ (such an orbit always exists). 

For a closed subgroup $H\subseteq G$ and an element $x\in \g_\R$, write 
\[
\mathrm Z_H(x):=\{g\in H:\Ad(g)x=x\} \quad(\textrm{$\Ad$ stands for the adjoint representation}). 
\]
According to \cite[p.~103]{B.Orbit} (see also \cite[p.~443]{MW}), for every $x\in \sO \cap (\sO_0 +\fuu_\R)$, $\mathrm Z_G(x)/ \mathrm Z_P(x)$ is a finite group, and it is independent of $x$ up to conjugation by $P$.  

\begin{defn}\label{def:mult}
The \emph{multiplicity} of $\sO$ is defined to be 
\[
m(\sO) := \#\left(\mathrm Z_G(x)/ \mathrm Z_P(x)\right) \qquad (\# \textrm{indicates the cardinality of a finite set}),
\]
where $x\in {\sO}\cap(\sO_0+\mathfrak u_\R)$. 
\end{defn}

Define the induced set of $\sO_0$ with multiplicities to be the formal sum 
\[\Ind_P^G(\sO_0) := \sum_{\sO} m(\sO) \sO,\] 
where $\sO$ runs over all nilpotent $G$-orbits induced  from $\sO_0$. 

Fix a $G$-invariant non-degenerate symmetric bilinear form $\la\,,\,\ra_{\g_\R}$ on $\g_\R$. We identify $\g_\R$ with $\mathrm i\g_\R^*$ ($\mathrm i:=\sqrt{-1}\in \C$ and $\,^*$ indicates the dual space) via the linear isomorphism
\be\label{idg}
  \g_\R\rightarrow \mathrm i\g_\R^*, \quad x\mapsto (y\mapsto \mathrm i\la x,y\ra_{\g_\R}). 
\ee
Recall that the wavefront cycle of a Casselman-Wallach representation  of $G$ is a formal sum of nilpotent $G$-orbits in $\mathrm i\g_\R^*$. Via the above identification, it is also viewed as a formal sum of  nilpotent $G$-orbits in $\g_\R$. By using the restriction of the form $\la\,,\,\ra_{\g_\R}$ to $\mathfrak m_\R\times \mathfrak m_\R$, similarly we identify $\mathfrak m_\R$ with $\mathrm i\mathfrak m_\R^*$. 

We recall the following result of Barbasch. 
\begin{thm}[{\cite[Corollary 5.0.10]{B.Orbit}}]\label{thm:A3}
 Suppose a Casselman-Wallach  representation $\pi$ of $M$ has wavefront cycle $\sum_{\sO_i} c_{\sO_i} \sO_i$. 
 Then $\Ind_P^G (\pi)$ has wavefront cycle 
 $\sum_{\sO_i} c_{\sO_i}\Ind_P^G (\sO_i)$. 
\end{thm}

\subsection{Computation of the induced orbits and multiplicities}
In the rest of this appendix we spell out the computation of nilpotent orbit induction used in  the proof of  \Cref{lem:ind00}. Let $\epsilon,\dot \epsilon\in\{1,-1\}$ and let $(V, \la\,,\,\ra, J)$ be a triple  satisfying the following conditions:
  \begin{itemize}
   \item $V$ is a finite-dimensional complex vector space;
   \item $\la\,,\,\ra$ is an $\epsilon$-symmetric non-degenerate bilinear form on $V$;
   \item $J: V\rightarrow V$ is a conjugate linear automorphism of $V$ such that $J^2=\epsilon\cdot \dot \epsilon$;
\item   $\inn{J u}{Jv}=
  \overline{\inn{u}{v}},\quad$  for all $u,v\in V$.
  \end{itemize}
Note that the triple $(V, \la\,,\,\ra, J)$ is part of the data specifying a classical space (see Section \ref{classical}).

Denote by $G_{\C}$ the isometry group of $ (V, \la\,,\,\ra)$, which is a complex orthogonal group if $\epsilon=1$ and a complex symplectic group if $\epsilon=-1$. Denote by  $G:=\mathrm{G}(V):=G_{\C}^{J}$ the centralizer of $J$ in  $G_{\C}$, which is a classical group as in Section \ref{classical}.  

Suppose that we are given a $J$-stable decomposition
\begin{equation}\label{eq:WittDecomp}
V=E\oplus V_0\oplus E',
\end{equation}
where $E$ and $E'$ are non-zero, totally isotropic and paired perfectly by $\langle\ ,\ \rangle$, and $V_0$ is the orthogonal complement of $E\oplus E'$.
Let $P=M\ltimes U$ be the parabolic subgroup of $G$ stabilizing $E$, with unipotent radical $U$ and Levi subgroup
\[
M=G_0\times \GL(E)^J,
\]
where $G_0:=\mathrm{G}(V_0)$, and $ \GL(E)^J$ is the centralizer of $J$ in the general linear group $\GL(E)$ on $E$ (similar notation will be used without explanation). Set 
$\mathfrak u_\R:=\Lie(U)$ and $\mathfrak g_{0\R}:=\Lie(G_0)$.

A $(-\epsilon)$-symmetric bilinear form $T$ on $E'$ is called $J$-compatible if 
\begin{equation} \label{hermep}
T(J u, Jv)=
  \overline{T(u,v)}\quad\textrm{  for all $u,v\in E'$}.
\end{equation}
Let $\mathrm{Herm}_{-\epsilon}(E')$ denote the space of  $J$-compatible  $(-\epsilon)$-symmetric bilinear forms on $E'$.
By abuse of notation, for every $T\in \mathrm{Herm}_{-\epsilon}(E')$, we still use $T: E'\rightarrow E$ to denote the linear map such that
\[
  \langle T(u), v\rangle= T(u, v)\quad\textrm{  for all $u,v\in E'$.}
\]
With respect to the decomposition  \eqref{eq:WittDecomp}, we have an identification
\begin{equation}\label{eq:YCB}
\Hom_J(E',V_0)\oplus\mathrm{Herm}_{-\epsilon}(E')\xrightarrow{\sim} \mathfrak u_\R, \quad (S,T)\mapsto Y_{S,T}:=
\begin{bmatrix}
0&-S^* &T\\
0&0&S\\
0&0&0
\end{bmatrix}.
\end{equation}
Here $\Hom_J(E',V_0)$  denotes the space of $J$-intertwining linear maps
from $E'$ to $V_0$ (similar notation will be used without explanation), and 
 $S^*\in\Hom_J(V_0,E)$ is specified by
\begin{equation}\label{eq:Cvee}
\langle S^* v,u\rangle=\langle v,S u\rangle \qquad \textrm{for all } v\in V_0,\ u\in E'.
 \end{equation}

Fix a nilpotent $G_0$-orbit $\sO_0\subset\mathfrak g_{0\R}$, which is also viewed as a nilpotent $M$-orbit of $\mathfrak m_\R:=\Lie(M)$.  We also fix an element $X_0\in \sO_0$, which is a $J$-intertwining linear map $X_0: V_0\rightarrow V_0$. Denote by $R_0$   the radical of the $\epsilon$-symmetric bilinear form $\la\, ,\,\ra$ on $\Ker X_0$. 
By using the Jacobson–Morozov theorem, it is easy to see that 
\[
  \dim R_0=\dim(\Ker X_0^2/\Ker X_0).
\]
For application to the proof of Proposition~\ref{lem:ind00}, we assume in the rest of this appendix that \begin{equation}\label{eq:r-gd0}
\dim E\geq \dim R_0.
\end{equation}
This is the same as the assumption \eqref{eq:dimcon} in \Cref{subsec:induced}.

The conjugate linear automorphism $J:V\rightarrow V$ induces a conjugate linear automorphism $J:(\Ker X_0)/R_0\rightarrow (\Ker X_0)/R_0$, and the $\epsilon$-symmetric bilinear form $\la\,,\,\ra$ induces a non-degenerate  $\epsilon$-symmetric bilinear form on $(\Ker X_0)/R_0$. Similar to the group $G=\mathrm{G}(V)$, we have the classical group $\mathrm{G}((\Ker X_0)/R_0)$.

An  $\epsilon$-symmetric bilinear form $T'$ on a complex vector space $V'$ is said to be regular if one of the following two conditions is satisfied:   
\begin{itemize}
    \item $T'$ is non-degenerate;
    \item $\epsilon =-1$, $V'$ is odd dimensional, and the radical of $T'$ is one-dimensional.  
\end{itemize}
 When $V'$ is equipped with an $\epsilon$-symmetric bilinear form, a subspace $V''$ of $V'$ is said to be regular if the $\epsilon$-symmetric form is regular when restricted to $V''\times V''$.

When $\dim E \geq \dim (\Ker X_0)$, we define   
 \[
 \Hom_J^{\max}(E',V_0):=\{S\in \Hom_J(E',V_0)  \,:\,  S^*|_{\Ker X_0} \text{ is injective}\}. 
\]

When $\dim R_0\leq \dim E < \dim (\Ker X_0)$, we define $\Hom_J^{\max}(E',V_0)$ to consist of those $S\in \Hom_J(E',V_0)$ such that 
\[
\textrm{$S^*|_{R_0}$ is injective, and $S^*|_{\Ker X_0}$ is surjective with regular kernel}.
\]
For each $S\in \Hom_J^{\max}(E',V_0)$, let  
\[
V_S := \textrm{the image of $\Ker (S^*|_{\Ker X_0})$ under the quotient map $\Ker X_0\rightarrow  (\Ker X_0)/R_0$}.
\]
 This is a $J$-stable subspace of $(\Ker X_0)/R_0$ of dimension $d$, where 
\[
d:=\dim (\Ker X_0) -\dim E, 
\]
and it is regular.

Denote by 
 \[
\mathrm{Gr}_{d}^{\mathrm{reg}} ((\Ker X_0)/R_0):=\{\textrm{$J$-stable regular subspaces of $(\Ker X_0)/R_0$ of dimension $d$}\}.
\]
This set carries an obvious action of $\mathrm{G}((\Ker X_0)/R_0)$.

The set $X_0+ \Hom_J^{\max}(E',V_0)$, viewed as a subset of the  quotient set  $(X_0+\mathfrak u_\R)/\mathrm{Herm}_{-\epsilon}(E')$ via \eqref{eq:YCB}, carries an obvious action of 
\[
H:= (\GL(E')^J\times \mathrm Z_{G_0}(X_0))\ltimes U\subset P
\]
through the adjoint representation. 

We omit the proof of the following elementary result. 

\begin{prop}\label{prop:A4}
 \noindent (a) Under the Euclidean topology, the set $\Hom_J^{\max}(E',V_0)$ is open dense in $\Hom_J(E',V_0)$  and  every $H$-orbit  in $X_0+\Hom_J^{\max}(E',V_0)$  is open in $X_0+\Hom_J(E',V_0)$. 
 
 \noindent (b) Suppose that $\dim E \geq \dim (\Ker X_0)$. Then $H$ acts on $X_0+\Hom_J^{\max}(E',V_0)$ transitively. 
 
 \noindent (c) Suppose that $\dim R_0\leq \dim E < \dim (\Ker X_0)$. 
 Then the map $S\mapsto V_S$ yields a bijection between the set of $H$-orbits
 in $X_0+\Hom_J^{\max}(E',V_0)$ and the set of  $\mathrm{G}((\Ker X_0)/R_0)$-orbits in $\mathrm{Gr}_{d}^{\mathrm{reg}} ((\Ker X_0)/R_0)$. 
 
\end{prop}

Fix  a $J$-stable complement $L_0$ of the image $\Im X_0$ of $X_0$ so that
\begin{equation}\label{eq:L0}
V_0=L_0 \oplus \Im X_0\qquad (\textrm{$\Im$ indicates the image of a map}).
\end{equation}
View $\Hom_J(E',L_0)$ as a subspace of $\Hom_J(E',V_0)$ and 
let 
\[
\Hom_J^{\max}(E',L_0)  := \Hom_J(E',L_0) \cap  \Hom_J^{\max}(E',V_0) . 
\]
Then each $H$-orbit in  $X_0 + \Hom_J^{\max}(E',V_0)$ intersects 
$X_0 +\Hom_J^{\max}(E',L_0)$.

For every  $S\in \Hom_J^{\max}(E',L_0)$, define 
\[
\Herm_{-\epsilon}^{\mathrm{reg}}(E')_S := \{T \in \Herm_{-\epsilon}(E') \,:\, 
T|_{\Ker S\times \Ker S} \text{ is regular}\}.
\]
Write 
\[
\cU := \{Y_{S,T} : S\in \Hom_J^{\max}(E', L_0),\, T \in \Herm^{\mathrm{reg}}_{-\epsilon}(E')_S
\}.
\]

The following proposition also follows from elementary linear algebra. 

\begin{prop}\label{prop:A5}
\noindent (a) Under the Euclidean topology, the set $H\cdot  (X_0 + \cU)$ is open dense in $X_0 + \fuu_\bR$ and  every  $H$-orbit in $H\cdot (X_0+ \cU)$ is open in $X_0 + \fuu_\bR$.  

\noindent (b) 
Fix $S\in \Hom_J^{\max}(E', L_0)$ and let $T_1, T_2 \in \Herm_{-\epsilon}^{\mathrm{reg}}(E')_S$. Then 
$X_0 + Y_{S,T_1}$ and $X_0 + Y_{S,T_2}$ are in the same $H$-orbit 
if and only if $T_1|_{\Ker S\times \Ker S}$ and $T_2 |_{\Ker S\times \Ker S}$ are conjugate under $\GL(\Ker S)^J$. 
\end{prop}

We are now ready to enumerate the representatives of the induced orbits. 

Suppose that $\dim E \geq \dim (\Ker X_0)$. Then there exists a surjective linear map $S\in \Hom_J(E', L_0)$. Note that the surjectivity is equivalent to the condition that  $S\in \Hom_J^{\max}(E',L_0)$. 
Fix a $J$-stable decomposition 
\be\label{decep}
E' = \Ker S \oplus E''.
\ee
 Each  bilinear form $T \in \mathrm{Herm}_{-\epsilon}(\Ker S)$ uniquely extends to a bilinear form in $\mathrm{Herm}_{-\epsilon}(E')$ whose radical contains $E''$. By abuse of notation, still denote by $T$ this extended bilinear form. In this way we view $\mathrm{Herm}_{-\epsilon}(\Ker S)$ as a subspace of $\mathrm{Herm}_{-\epsilon}(E')$. 
 
Let $\mathrm{Herm}_{-\epsilon}^{\mathrm{reg}}(\Ker S)$ denote the set of all regular forms in $\mathrm{Herm}_{-\epsilon}(\Ker S)$, which carries an obvious action of $\GL(\Ker S)^J$. 
Set \[
X_{S,T}:=X_0+Y_{S,T}\in \g_\R,\qquad T\in \mathrm{Herm}_{-\epsilon}^{\mathrm{reg}}(\Ker S)\subset \mathrm{Herm}_{-\epsilon}(E').
\]

\begin{prop}\label{prop:A6}
Suppose that $\dim E \geq \dim (\Ker X_0)$. Let $S\in \Hom_J^{\max}(E',L_0)$. 

\noindent
(a) For each $T\in \mathrm{Herm}_{-\epsilon}^{\mathrm{reg}}(\Ker S)$, the orbit $G\cdot X_{S,T}$ is independent of the choice of the decomposition \eqref{decep}. Moreover, the  map $T\mapsto
 X_{S,T}$ induces a bijection between the set of  $\GL(\Ker S)^J$-orbits in $\mathrm{Herm}_{-\epsilon}^{\mathrm{reg}}(\Ker S)$ and the set of nilpotent $G$-orbits induced from $\sO_0$. 

\noindent
(b) For each $T\in \mathrm{Herm}_{-\epsilon}^{\mathrm{reg}}(\Ker S)$, the multiplicity  
\begin{equation}\label{countmul}
m(G\cdot X_{S,T})=
\begin{cases}
2, & \text{if $\epsilon = 1$ and $\Ker S$ is odd dimensional;}\\
1, & \text{otherwise}.
\end{cases}
\end{equation}
\end{prop}
\begin{proof}
The first assertion of part (a) is routine to check, while the second assertion  follows from Propositions \ref{prop:A4} and \ref{prop:A5}.

Now we consider part (b). 
In the second case of \eqref{countmul}, $\Ker (X_{S,T})=E$.  Every element of $\mathrm Z_G(X_{S,T})$ therefore preserves $E$, hence lies in $P$.  Thus $\mathrm Z_G(X_{S,T})=\mathrm Z_P(X_{S,T})$, and
\[
m(G\cdot X_{S,T})=1.
\]

In the first case of \eqref{countmul}, write $R$ for the radical of $T$, which is a one-dimensional subspace of $\Ker S$. 
Then $\Ker (X_{S,T}) = E \oplus R$. 
Set $Z:=\Ker X_{S,T} \cap \Im X_{S,T}$. 
Note that   
$\Ker ( X_{S,T})/ Z$ is two-dimensional and has two distinct one-dimensional subspaces
\[
E/Z\quad\textrm{and}\quad (R + Z)/Z.
\] 
The centralizer $\mathrm Z_G(X_{S,T})$ acts transitively on the set $\{E/Z, (R + Z)/Z\}$, while the subgroup preserving $E/Z$ is precisely $\mathrm Z_P(X_{S,T})$. Consequently
\[
m(G\cdot X_{S,T})=[\mathrm Z_G(X_{S,T}):\mathrm Z_P(X_{S,T})]=2.
\]
\end{proof}

Now we suppose that $\dim R_0\leq \dim E < \dim (\Ker X_0)$. 
The decomposition \eqref{eq:L0} induces a $J$-stable decomposition
\[
V_0 = L_0^\perp \oplus \Ker X_0,
\]
where 
\[
L_0^\perp:=\{u\in V_0\,:\, \la u, v\rangle=0 \textrm{ for all }v\in L_0\}.
\]
For a given $A\in \mathrm{Gr}^{\mathrm{reg}}_d((\Ker X_0)/R_0)$, fix a $J$-stable decomposition 
\be\label{decomaa}
\Ker X_0 = R_0 \oplus A' \oplus A''
\ee
such that the quotient map $\Ker X_0 \to \Ker X_0/R_0$ sends $A'$ onto $A$. 
Choose any $S_A\in \Hom_{J}(E',V_0)$ such that 
\be \label{chSA}
\textrm{$S_A^*|_{ L_0^\perp \oplus A' } = 0\quad$ and $\quad S_A^*|_{R_0\oplus A'' } $ is bijective.}
\ee
Then $S_A\in \Hom_J^{\max}(E',V_0)$.

As before, set
\[X_{S_A,0}:=X_0+Y_{S_A,0}\in \g_\R, \qquad A \in \mathrm{Gr}^{\mathrm{reg}}_d((\Ker X_0)/R_0).\]

\begin{prop}\label{prop:A7}
Suppose that $\dim R_0 \leq \dim E < \dim (\Ker X_0)$. 

\noindent (a) For each $A\in\mathrm{Gr}^{\mathrm{reg}}_d((\Ker X_0)/R_0) $, the orbit $G\cdot X_{S_A,0}$ is independent of the  choice of the decomposition \eqref{decomaa} and 
the choice of $S_A$ in \eqref{chSA}. Moreover, the map $A\mapsto
X_{S_A,0}$ induces a bijection between the set of $\mathrm{G}((\Ker X_0)/R_0)$-orbits in $\mathrm{Gr}^{\mathrm{reg}}_d((\Ker X_0)/R_0)$ and the set of nilpotent $G$-orbits induced  from $\sO_0$. 

\noindent (b) For each $A\in\mathrm{Gr}^{\mathrm{reg}}_d((\Ker X_0)/R_0) $, 
the multiplicity
\[
m(G\cdot X_{S_A,0}) = \begin{cases}
2 , & \text{if $\epsilon=-1$ and $d$ is odd};\\
1, & \text{otherwise}.
\end{cases}
\]
\end{prop}
\begin{proof}
The proof is similar to that of \Cref{prop:A6} and is omitted. 
\end{proof}

Note that the trace form on the Lie algebra $\g$ of $G_\C$ is real-valued when restricted to $\g_\R\times \g_\R$, where $\g_\R$ denotes the Lie algebra of $G$. Thus, by \eqref{idg}, the trace form yields an identification of $\fgg_\R$ with  $\mathrm i \g_\R^*$, and we may view the wavefront cycle of a Casselman-Wallach representation  of $G$ as a formal sum of nilpotent $G$-orbits in $\g_\R$. Thus, in view of \Cref{thm:A3}, Propositions \ref{prop:A6} and \ref{prop:A7} provide an explicit description of the wavefront cycle of the relevant induced representations.

\section*{Acknowledgements}

The authors are indebted to Roger Howe for his powerful insights on representations of classical groups, and to David Vogan for his inspiring ideas on representations of real reductive groups, which are indispensable to this series of two papers.  The authors are further grateful to the referee for helpful comments and suggestions, in particular for painstakingly checking the (inductive) computation of associated cycles of the constructed representations, which is rather involved. 

C.-B. Zhu is grateful to Max Planck Institute for Mathematics in Bonn, for its warm hospitality and conducive work environment, where he spent the academic year 2022/2023 as a visiting scientist.  

D. Barbasch is supported by the NSF grant, Award Number 2000254. J.-J. Ma is supported by the National Natural Science Foundation of China (Grant No. 11701364, 11971305 and 12571032) and Xiamen University
Malaysia Research Fund (Grant No. XMUMRF/2022-C9/IMAT/0019 and XMUMRF/2025-C15/IMAT/0036).
B. Sun is supported by  National Key R \& D Program of China No. 2022YFA1005300 and New Cornerstone Science Foundation.  
C.-B. Zhu is supported by MOE AcRF Tier 1 grants A-0004280-00-00 and A-8002490-00-00, and Provost's Chair grant E-146-00-0018-01 in NUS.

\end{document}